\documentclass[12pt]{article}
\usepackage{fullpage}
\usepackage{amsmath,amsfonts,amsthm,amssymb,epsfig}
\usepackage[
backend=biber,style=alphabetic,
]{biblatex}
\usepackage[all]{xy}
\usepackage{enumitem}
\usepackage[usenames,dvipsnames]{xcolor}
\usepackage[bookmarks=true,colorlinks=true, linkcolor=Mulberry, citecolor=Periwinkle]{hyperref}
\usepackage{graphicx}
\usepackage{tikz-cd}
\usepackage{float}
\usepackage{tipa}
\usepackage{MnSymbol}
\newtheorem{theorem}{Theorem}[section]
\newtheorem{lemma}[theorem]{Lemma}
\newtheorem{proposition}[theorem]{Proposition}
\newtheorem{corollary}[theorem]{Corollary}

\newtheorem{definition}[theorem]{Definition}
\newtheorem{example}[theorem]{Example}
\newtheorem{remark}[theorem]{Remark}
\newtheorem{conjecture}[theorem]{Conjecture}

\newcommand\bcdot{\ensuremath{%
  \mathchoice%
   {\mskip\thinmuskip\lower0.2ex\hbox{\scalebox{1.5}{$\cdot$}}\mskip\thinmuskip}}%
   {\mskip\thinmuskip\lower0.2ex\hbox{\scalebox{1.5}{$\cdot$}}\mskip\thinmuskip}%
   {\lower0.3ex\hbox{\scalebox{1.2}{$\cdot$}}}%
   {\lower0.3ex\hbox{\scalebox{1.2}{$\cdot$}}}%
   }

\newcommand{\on}{\operatorname}
\newcommand{\mc}{\mathcal}
\newcommand{\mbb}{\mathbb}
\newcommand{\solid}{\square}
\newcommand{\wh}{\widehat}
\newcommand{\heart}{\heartsuit}

\newcommand{\ul}{\underline}

\makeatletter
\def\legendre@dash#1#2{\hb@xt@#1{%
  \kern-#2\p@
  \cleaders\hbox{\kern.5\p@
    \vrule\@height.2\p@\@depth.2\p@\@width\p@
    \kern.5\p@}\hfil
  \kern-#2\p@
  }}
\def\@legendre#1#2#3#4#5{\mathopen{}\left(
  \sbox\z@{$\genfrac{}{}{0pt}{#1}{#3#4}{#3#5}$}%
  \dimen@=\wd\z@
  \kern-\p@\vcenter{\box0}\kern-\dimen@\vcenter{\legendre@dash\dimen@{#2}}\kern-\p@
  \right)\mathclose{}}
\newcommand\legendre[2]{\mathchoice
  {\@legendre{0}{1}{}{#1}{#2}}
  {\@legendre{1}{.5}{\vphantom{1}}{#1}{#2}}
  {\@legendre{2}{0}{\vphantom{1}}{#1}{#2}}
  {\@legendre{3}{0}{\vphantom{1}}{#1}{#2}}
}
\def\dlegendre{\@legendre{0}{1}{}}
\def\tlegendre{\@legendre{1}{0.5}{\vphantom{1}}}
\makeatother

\begin{document}

\title{Duality and linearization for $p$-adic Lie groups}
\author{Dustin Clausen}

\maketitle

\section{Introduction}

If $M$ is a compact oriented real manifold, Poincaré duality says that the homology of $M$ agrees with the cohomology of $M$ up to a shift by the dimension.  If $M$ is not oriented, the homology of $M$ agrees instead with the cohomology of $M$ with coefficients in a canonical twisted coefficient system $\on{or}_M$, the \emph{orientation local system}, again up to a shift by the dimension.  It is natural to incorporate the shift by dimension into the twist; then in terms of chain complexes, the statement is that
$$C_\cdot(M;\mbb{Z}) \simeq C^\cdot(M;\on{or}_M[d]).$$

One can better understand the geometric underpinnings of the orientation local system by passing from abelian group to spectrum coefficients.  Then the statement, known as Atiyah duality (\cite{AtiyahDuality}) or Milnor-Spanier duality (\cite{MilnorDuality}), is that
$$C_\cdot(M;\mbb{S}) \simeq C^\cdot(M;\mbb{S}^{TM}).$$
Here the relevant ``orientation local system'' is the local system of spectra $\mbb{S}^{TM}$ on $M$ gotten by taking a fiberwise one-point compactification of the tangent bundle.

\begin{remark}
Note that the form of $\mbb{S}^{TM}$ explains both the shift and the twist in Poincaré duality with $\mbb{Z}$-coefficients.  The shift is by the dimension $d$ of $TM$ (or equivalently of the sphere $\mbb{S}^{TM}$) and $\on{or}_M=H_d(\mbb{S}^{TM};\mbb{Z})$.

Atiyah duality gives Poincare duality in the most refined sense; one can then specialize to other coefficients. Besides the classical $\mbb{Z}$-coefficients, Poincaré duality with topological K-theory coefficients $KO$ or $KU$ is also quite geometrically relevant, for example in index theory.  In general, Poincare duality phenomena can be studied by combining Atiyah duality with a purely ``linearized'' problem: that of providing ``orientation'' properties (or in other words, Thom isomorphisms) for vector bundles with certain extra structure.  The model for this is \cite{AtiyahOrientation} for the case of $KO$ coefficients, where the extra structure is a spin structure.
\end{remark}

A different situation in which a version of Poincaré duality holds was discovered by Lazard (\cite{LazardGroupes}).  It concerns the continuous group cohomology of compact $p$-adic Lie groups $G$, with $p$-adic coefficients.  Lazard showed that if $G$ is sufficiently small, then $G$ has Poincaré duality. (And later Serre showed that ``sufficiently small'' just means ``$p$-torsionfree'', \cite{SerreDimension}.)  Again there is a shift and twist here, of similar form to Poincaré duality for real manifolds: the shift is by the dimension of $G$, and the twist is determined by the determinant of the adjoint representation.

\begin{example}
The simplest example is the compact abelian $p$-adic Lie group $G=(\mbb{Z}_p^d,+)$.  The continuous cohomology with $p$-adic coefficients of $G$ agrees with that of its discrete subgroup $\mbb{Z}^d$, which in turn agrees with the cohomology (not group cohomology) of the torus $(\mbb{R}/\mbb{Z})^d$, as this is a model for the classifying space of $\mbb{Z}^d$.  Thus the Lazard Poincaré duality for this $G=\mbb{Z}_p^d$ looks just like classical Poincaré duality for the $d$-dimensional real torus.
\end{example}

However, the analog of Atiyah duality had been missing.  More precisely, it's mostly a formal matter to deduce from Poincaré duality with $\mbb{Z}_p$-coefficients (or just $\mbb{F}_p$-coefficients) that Poincaré duality with $\mbb{S}_{\wh{p}}$-coefficients holds with respect to \emph{some} twist; the difficulty is in giving an explicit description of this twist analogous to the tangential description in Atiyah duality for real manifolds.

\begin{remark}
The method of Lazard for identifying the twist in terms of the adjoint representation cannot work with $\mbb{S}_{\wh{p}}$-coefficients.  Indeed, it is based on passage to $\mbb{Q}_p$-coefficients and a comparison with Lie algebra cohomology there.  But over the sphere spectrum inverting $p$ kills all the extra interesting information.
\end{remark}

The goal of this paper is to resolve this issue.  We prove that the twist in Poincaré duality for $p$-adic Lie groups with $p$-complete spectrum coefficients is determined by the adjoint representation, in a manner analogous to Atiyah duality for real manifolds.

\begin{remark}
The interest in duality for $p$-adic Lie groups with spectrum coefficients is not purely academic.  In fact actions of $p$-adic Lie groups on spectra are fundamental in chromatic homotopy theory, and duality properties from these perspectives were studied by Devinatz-Hopkins and Gross-Hopkins, see Strickland's paper \cite{StricklandDuality} for an account.  The statement we prove here was called the ``linearization hypothesis'' in the more recent work \cite{BeaudryLinearization} of Beaudry-Goerss-Hopkins-Stojanoska which explores some interesting consequences for duality in stable homotopy theory.
\end{remark}

There are three main things we need to do to accomplish our goal:

\begin{enumerate}
\item We need to choose a definition for what we mean by ``$p$-adic spectrum coefficient systems'' over objects like ``$BG$'' for a $p$-adic Lie group $G$.  (Part of this is saying what we mean by $BG$.)
\item We also need to set up some duality theory for such coefficient systems.  To do all the arguments in a flexible way, this should be encoded in a six functor formalism.
\item Again in the setting of such coefficient systems, we need to formulate and prove a result explaining how it's possible to translate data from the Lie algebra of a group to the group itself. (This is the ``linearization'' which gives the explicit description of the twist.)
\end{enumerate}

The main difficulty in 1 is psychological: there are so many possible options that it's hard to settle on just one.  To connect with 2 and 3 we need a fairly general solution including many more objects than just $BG$, but even with all the constraints there is a huge number of possibilities each of which would work just fine.  In the end we chose an option based on ``light condensed anima'' from our joint work with Peter Scholze, because ``why not?''.

The main difficulty in 2 is purely technical: it's not so easy to rigorously construct six functor formalisms, even when one knows in one's heart what all the functors are and how they behave.  Thankfully, we are saved by the recent article of Heyer-Mann (\cite{HeyerSixFunctors}), which contains very general and easy-to-cite results doing basically everything we want in terms of constructing six functor formalisms proving properties of them.

Point 3 is the main point, and the only truly novel contribution of this paper.  We actually prove something somewhat shocking: for certain sheaves on the topological space $\mathbb{Q}_p$, there is a canonical isomorphism from the stalk at $0$ to the stalk at $1$.  This is crazy because $\mathbb{Q}_p$ is totally disconnected.  But we have extra structure on our sheaves which comes to the rescue, namely we have $\mathbb{Q}_p^\times$-equivariance.  Even so the story is subtle, and we refer to Section \ref{pathsec} for details.  In any case, we combine this with the obvservation that any $p$-adic Lie group canonically deforms to its Lie algebra, in a manner which can be precisely encoded by a $\mbb{Q}_p^\times$-equivariant family of $p$-adic Lie groups over $\mbb{Q}_p$ (see Section \ref{defsec}), and that is how we get the linearization.

\begin{remark}
I would like to apologize for the extremely long delay in writing this article.  The main results were announced in 2018.  The delay was partly due to indecision regarding point 1 and partly due to despair in face of the task of writing up the technical details in 2.  Fortunately as mentioned above this second concern no longer applies thanks to \cite{HeyerSixFunctors}.  Also, in the meantime joint work with Scholze made light condensed anima seem like an attractive choice for 1.

In the last two sections we have also included some more recent results, designed to make the connection with the discussion in \cite{BeaudryLinearization}.
\end{remark}

Our main result was conjectured in the author's earlier paper \cite{ClausenJ}.  In section \ref{jsec} we also give a much nicer approach to the main result of that paper (the ``reciprocity law for J-homomorphisms'', which gives a relationship between the real and $p$-adic stories).

\subsection*{Acknowledgements}

I thank Agnès Beaudry, Shachar Carmeli, Ofer Gabber, Lars Hesselholt, Mike Hopkins, Peter Kropholler, Lucas Mann, Itamar Mor, and Peter Scholze for very helpful conversations and IHES for excellent writing conditions.

\tableofcontents

\section{Preliminaries on manifolds}\label{mansec}

Following Serre's book \cite{SerreLie}, let $(F,|\cdot |)$ be a complete valued field.  That means that $F$ is a field and $|\cdot|:F\rightarrow\mathbb{R}_{\geq 0}$ is a function such that:
\begin{enumerate}
\item $|x|=0 \Leftrightarrow x=0$ for $x\in F$;
\item $|x+y|\leq |x|+|y|$ for $x,y\in F$;
\item $|xy|=|x||y|$ for $x,y\in F$;
\item with respect to the metric $d(x,y)=|x-y|$, the field $F$ is \emph{complete}.
\end{enumerate}
We furthermore assume that $|\cdot|$ is not the trivial valuation: in other words, there should be a $\pi\in F$ with $0<|\pi|<1$, or, equivalently, the induced topology on $F$ should not be the discrete topology.  For brevity we will often write just $F$ instead of the pair $(F,|\cdot |)$, and we will call $F$ a \emph{nontrivial complete valued field}.

\begin{remark} In fact, the whole theory of $F$-analytic manifolds, which we are reviewing here, is independent of the valuation $|\cdot|$ in the sense that it only depends on the underlying topological field $F$.  This can be justified using the following fact, which is a standard exercise: given two complete valued structures $|\cdot|$ and $|\cdot |'$ on the same field $F$ inducing the same topology on $F$, there is some $\alpha\in\mathbb{R}_{>0}$ such that $|\cdot|' = |\cdot|^\alpha$.
\end{remark}

\begin{remark}
Recall Ostrowski's dichotomy: any nontrivial complete valued field is either \emph{archimedean} (the set of norms of natural numbers is unbounded), in which case it is isomorphic as a topological field to either $\mathbb{R}$ or $\mathbb{C}$, or else it is \emph{nonarchimedean} (the set of norms of natural numbers is bounded), in which case it satisfies the strong triangle inequality $|x+y|\leq \operatorname{max}\{|x|,|y|\}$.  In contrast to the archimedean case, there is a huge variety of different kinds of nonarchimedean nontrivial complete valued fields, most of which (for example $\mathbb{C}_p$ or $\mathbb{Q}((T))$) are not locally compact.
\end{remark}

For $d\in\mathbb{N}$ and an open subset $U\subset F^d$, recall the ring of \emph{$F$-analytic functions on $U$}: it is the set of those functions $\varphi:U\rightarrow F$ such that for all $x\in U$, there is a formal power series $f \in F[[t_1,\ldots,t_d]]$ and an $r=(r_1,\ldots,r_d) \in (\mathbb{R}_{>0})^d$ such that:
\begin{enumerate}
\item if $|f|\in \mathbb{R}[[t_1,\ldots,t_d]]$ denotes the result of applying $|\cdot|$ to each coefficient of $f$, then $|f|(r)<\infty$;
\item for all $y=(y_1,\ldots,y_d)\in F^d$ such that $|y_i|<r_i$ for all $i$, we have $x+y\in U$ and
$$\varphi(x+y) = f(y),$$
where $f(y)\in F$ makes sense as an absolutely convergent sum as a consequence of assumption 1.
\end{enumerate}
In other words, this is the set of functions $\varphi:U\rightarrow F$ which locally admit an absolutely convergent power series expansion.  

The set $\mathcal{O}(U)$ of $F$-analytic functions $\varphi:U\rightarrow F$ is an $F$-algebra under pointwise $+$ and $\times$, and $U\mapsto \mathcal{O}(U)$ forms a sheaf of $F$-algebras on the topological space $F^d$.  Any $\varphi\in\mathcal{O}(U)$ is continuous, and in fact infinitely differentiable, as a function $U\rightarrow F$.  Moreover, the power series $f=\varphi_x$ as above is uniquely determined by $\varphi$ and $x$.  We will call it the \emph{power series expansion of $\varphi$ at $x$}.

\begin{example}
For $d\in\mathbb{N}$ and $r\in(\mathbb{R}_{>0})^d$, denote by $\operatorname{D}(r)$ the open polydisk of radius $r$ at the origin: the set of $a\in F^d$ with $|a_i|<r_i$ $\forall i$.

Then if $f\in F[[t_1,\ldots,t_d]]$ is such that $|f|(r')<\infty$ for all $r'\in (\mathbb{R}_{>0})^d$ with $r'_i<r_i \forall i$, we have that the function $\varphi:\operatorname{D}(r)\rightarrow F$ defined by
$$\varphi(a) = f(a)$$
lies in $\mathcal{O}(\operatorname{D}(r))$.  In other words, an absolutely convergent power series on an open polydisk defines a locally analytic function in that polydisk.  This is shown in \cite{SerreLie}: the proof is by the familiar device of re-expanding a power series at a different point in the polydisk.
\end{example}

\begin{remark}\label{Cisspecial}
If $F=\mathbb{C}$, then the converse to the above example holds: every $\varphi\in \mathcal{O}(D(r))$ is given globally by an $f\in \mathbb{C}[[t_1,\ldots,t_d]]$ which is absolutely convergent on $D(r)$.  This follows from basic complex analysis, or more precisely, from Cauchy's formula for the power series coefficients.

We issue the following standard warning: if $F\not\simeq \mathbb{C}$, then this converse fails even for $d=1,r=1$: that is, not every $\varphi\in \mathcal{O}(\operatorname{D}(1))$ is given globally by a power series converging absolutely in $\operatorname{D}(1)\subset F$. Indeed, if $F=\mathbb{R}$ then we can have poles in the complex plane close to the real axis obstructing the existence of a global power series expansion, whereas if $F$ is nonarchimedean then we can write an open disk as disjoint union of smaller open disks which allows to make obvious counterexamples.
\end{remark}

Now we investigate the stalks of this sheaf of $F$-algebras $\mathcal{O}_{F^d}$ on $F^d$.  By translation, it suffices to consider the stalk at the origin $0\in F^d$, and we will do this for simplicity of notation.  The following is of course well-known.

\begin{proposition}\label{localrings}
Let $d\in\mathbb{N}$, let $F$ be a nontrivial complete valued field.  Let $\mathcal{O}_{F^d,0}$ denote the stalk of the sheaf of $F$-analytic functions on $F^d$ at the origin $0\in F^d$.  Let $\operatorname{ev}:\mathcal{O}_{F^d,0}\rightarrow F$ denote the $F$-algebra homomorphism $\varphi\mapsto \varphi(0)$, and let $\mathfrak{m}=\operatorname{ker}(\operatorname{ev})$.  Then:
\begin{enumerate}
\item the homomorphism $\mathcal{O}_{F^d,0}\rightarrow F[[x_1,\ldots,x_d]]$ of taking power series expansions (which is well-defined) is injective, and its image constists of those formal power series which converge absolutely in some neighborhood of $0$.
\item $\mathcal{O}_{F^d,0}$ is a local ring with maximal ideal $\mathfrak{m}$.
\item $\operatorname{ev}$ is the unique homomorphism of $F$-algebras $\mathcal{O}_{F^d,0}\rightarrow F$;
\item for $d\geq 1$, multiplication by the coordinate function $x_d$ is injective on $\mathcal{O}_{F^d,0}$, and $\mathcal{O}_{F^d,0}/x_d \simeq \mathcal{O}_{F^{d-1},0}$ via restriction of functions along $F^{d-1}\times\{0\}\subset F^d$;
\item the coordinate functions $x_1,\ldots,x_d\in\mathcal{O}_{F^d,0}$ form a regular sequence generating $\mathfrak{m}$.
\item the homomorphism $\mathcal{O}_{F^d,0}\rightarrow F[[x_1,\ldots,x_d]]$ of taking power series expansions identifies the target with the $\mathfrak{m}$-adic completion of the source.
\end{enumerate}
\end{proposition}
\begin{proof}

For 1, injectivity follows from the definition of locally analytic. As for the claimed characterization of the image, one containment is by definition, and the other follows from Remark \ref{Cisspecial}  

Now, for part 2, it suffices to see that if $\varphi\in\mathcal{O}_{F^d,0}$ has $\varphi(0)\neq 0$, then $\varphi$ is a unit.  This follows by formally inverting the power series using a geometric series and then checking convergence.

For part 3, suppose given an arbitrary $F$-algebra homomorphism $f:\mathcal{O}_{F^d,0}\rightarrow F$.  The kernel is a maximal ideal, hence equal to $\mathfrak{m}$ by 2.  Thus $f$ factors through $\operatorname{ev}$.  The only $F$-algebra map $F\rightarrow F$ is the identity, so this factoring implies that $f$ equals $\operatorname{ev}$, proving 3.

The injectivity statement in 4 is clear by thinking of power series. Further, the map $\mathcal{O}_{F^d,0}/x_d\rightarrow \mathcal{O}_{F^{d-1},0}$ is clearly well-defined and (split) surjective.  For injectivity it suffices to note that if $f=x_d\cdot g$ in $F[[x_1,\ldots,x_d]]$ and $f$ converges absolutely in some neighborhood of the origin, then so does $g$.

Part 5 follows immediately by induction on $d$.  Finally, 6 follows from 5 by considering the homomorphism $F[x_1,\ldots,x_d]\rightarrow \mathcal{O}_{F^d,0}$ and completing at $(x_1,\ldots,x_d)$.\end{proof}

This gives the following characterization of $F$-analytic functions between open subsets of the basic spaces $F^d,F^e$.

\begin{lemma}\label{analyticmaps}
Let $d,e\in\mathbb{N}$, let $F$ be a nontrivial complete valued field, and let $U\subset F^d$ and $V\subset F^e$ be open subsets.  We endow these with their respective sheaves $\mathcal{O}_U$ and $\mathcal{O}_V$ of $F$-algebras of $F$-analytic functions.

Suppose given a continuous map $f:U\rightarrow V$.  The following conditions are equivalent:
\begin{enumerate}
\item $f$ is $F$-analytic in the sense of \cite{SerreLie}: that is, each of the $e$-many coordinates of $f$ is an $F$-analytic function $U\rightarrow F$.
\item for every open subset $V'\subset V$ and every $F$-analytic function $\varphi:V'\rightarrow F$, the composition $\varphi\circ f:f^{-1}V'\rightarrow F$ is $F$-analytic.
\item there exists a homomorphism of sheaves $F$-algebras $f^\sharp:f^{-1} \mathcal{O}_V\rightarrow \mathcal{O}_U$.
\end{enumerate}
Moreover, under these conditions, the homomorphism $f^\sharp$ in 3 is unique, and is given by the operation of composition with $f$ from 2.
\end{lemma}
\begin{proof}
The implication 1 $\Rightarrow$ 2 is basic, see \cite{SerreLie}.  We have that 2 implies 3 because we can define $f^\sharp$ to be the operation of composition with $f$.  To simultaneously show that 3 implies 1 and that the uniqueness claim holds, thereby finishing the proof, it suffices to show that any $f^\sharp$ as in 3 must be given, at the level of functions, by composition with $f$.  For that, let $x\in U$, and consider the induced map
$$f^\sharp_x:\mathcal{O}_{V,f(x)}\rightarrow \mathcal{O}_{U,x}$$
of $F$-algebras.  By Lemma \ref{localrings}, the composition of $f^\sharp_x$ with the homomorphism $\mathcal{O}_{U,x}\rightarrow F$ of evaluation at $x$ identifies with evaluation at $f(x)$.  Thus, given any open subset $V'\subset V$, any $\varphi\in\mathcal{O}(V')$ and any $x\in f^{-1}(V')$, we have $(f^\sharp\varphi)(x)=\varphi(f(x))$.  This was exactly the claim.
\end{proof}

We recall from \cite{SerreLie} that, for $F$ a nontrivial complete valued field, an \emph{$F$-analytic manifold}, or just an \emph{$F$-manifold} for brevity, is, by definition, a topological space $X$ equipped with a maximal atlas consisting of charts taking values in some open subset of some $F^d$, with $F$-analytic transition maps.  (Note that we require neither that $X$ be Hausdorff nor that $X$ be paracompact.)  Moreover, an \emph{$F$-analytic map} between $F$-analytic manifolds is a map of underlying topological spaces with the property that, locally on charts, it is given by $F$-analytic functions.

We can use the above lemma to equivalently formalize this definition in the language of $F$-ringed spaces.  Namely, an \emph{$F$-ringed space} is a pair $(X,\mathcal{O}_X)$ where $X$ is a topological space and $\mathcal{O}_X$ is a sheaf of $F$-algebras.  A map $(X,\mathcal{O}_X)\rightarrow (Y,\mathcal{O}_Y)$ of $F$-ringed spaces is a pair $(f,f^\sharp)$ where $f:X\rightarrow Y$ is continuous and $f^\sharp:f^{-1}\mathcal{O}_Y\rightarrow \mathcal{O}_X$ is a map of sheaves of $F$-algebras.  If $X$ is an $F$-manifold, then defining $\mathcal{O}_X(U)$ to be the $F$-algebra of $F$-analytic functions $U\rightarrow F$ for $U\subset X$ open makes an $F$-ringed space $(X,\mathcal{O}_X)$.  Moreover, an $F$-analytic map $f:X\rightarrow Y$ gives rise to a map of corresponding $F$-ringed spaces, where $f^\sharp$ is given by composition with $f$.  This gives a functor from $F$-manifolds to $F$-ringed spaces.

\begin{lemma}
Let $F$ be a nontrivial complete valued field.   The above functor $X\mapsto (X,\mathcal{O}_X)$ from $F$-manifolds to $F$-ringed spaces is fully faithful.  The essential image consists of those $F$-ringed spaces which are locally isomorphic to $(U,\mathcal{O}_U)$ for some open $U\subset F^d$ and $d\in\mathbb{N}$ (where $\mathcal{O}_U$ is the sheaf of $F$-analytic functions on open subsets of $U$).
\end{lemma}

\begin{proof}
Suppose given $F$-manifolds $X$ and $Y$.  For full faithfulness, we need to show that a continuous map $f:X\rightarrow Y$ is $F$-analytic if and only if it promotes to a map of $F$-ringed spaces, and that such a promotion is necessarily unique.  For this, we can work locally on both $X$ and $Y$ and thereby reduce to the case covered by Lemma \ref{analyticmaps}.

Clearly every $F$-ringed space in the essential image is locally isomorphic to some $(U,\mathcal{O}_U)$.  Conversely, given an $F$-ringed space $(X,\mathcal{O}_X)$ which is locally isomorphic to some $(U,\mathcal{O}_U)$, we can use these local isomorphisms to put charts on $X$, and the transition functions will be $F$-analytic by Lemma \ref{analyticmaps}.  It is clear from the definition that the ring of $F$-analytic functions on this $F$-analytic manifold identifies with $\mathcal{O}_X$.
\end{proof}

We will also use the \emph{functor of points} perspective on manifolds, following \cite{GrothendieckAnalytic}.  Let $\operatorname{Man}_F$ denote the category of $F$-manifolds.  One way to produce an $F$-manifold is to first produce a functor (presheaf of sets) $\mathcal{F}:\operatorname{Man}_F^{op}\rightarrow\operatorname{Set}$, and then  to prove it is \emph{representable}, i.e.\ isomorphic to $h_M: N\mapsto \operatorname{Hom}(N,M)$ for some $M\in\operatorname{Man}_F$.

\begin{example}
Consider the presheaf of sets on $\operatorname{Man}_F$ which sends $M$ to the set $\mathcal{O}(M)$ of global $F$-analytic functions on $M$.  By definition, this is represented by the $F$-analytic manifold $F$.

Next, consider the subpresheaf which sends $M$ to $\mathcal{O}(M)^\times$, the set of units in the ring $\mathcal{O}(M)$.  This is representable by $F^\times$, by which we mean the $F$-manifold given as the open subset $F^\times =F\setminus \{0\}\subset F$.  The claim here is just that an $F$-analytic function $\varphi: M\rightarrow F$ is a unit in $\mathcal{O}_M(M)$ if and only if $\varphi(m)\neq 0$ for all $m\in M$.  This is a rephrasing of \ref{localrings}.
\end{example}

The following completely standard representability criterion is close to tautological, but still it is useful in organizing gluing procedures.  To set it up, we make a definition: a map $\mathcal{F}\rightarrow\mathcal{G}$ of presheaves of sets on $\operatorname{Man}_F$ is called an \emph{open immersion} if for any $M\in\operatorname{Man}_F$ and any map $h_M\rightarrow\mathcal{G}$, there is a (necessarily unique) open subset $U\subset M$ and an isomorphism
$$h_U\overset{\sim}{\rightarrow} h_M\times_{\mathcal{G}}\mathcal{F}$$
which projects to the inclusion $h_U\rightarrow h_M$.

Note that this implies that $\mathcal{F}\rightarrow\mathcal{G}$ is a monomorphism; beyond this, we can interpret $\mathcal{F}\hookrightarrow\mathcal{G}$ being an open immersion as saying that ``the condition that an $s\in\mathcal{G}(M)$ should lie in $\mathcal{F}(M)$ is an open condition (on $M$)''.

\begin{lemma}\label{localrep}
Let $\mathcal{F}:\operatorname{Man}_F^{op}\rightarrow\operatorname{Set}$ be a functor.  Suppose that:
\begin{enumerate}
\item for all $M\in\operatorname{Man}_F$, the restriction of $\mathcal{F}$ to the poset of open subsets of $M$ satisfies the sheaf condition with respect to open covers;
\item there is a collection $(M_i)_{i\in I}$ of objects of $\operatorname{Man}_F$ and open immersions $(h_{M_i}\rightarrow \mathcal{F})_{i\in I}$ such that for all $N\in\operatorname{Man}_F$ and all maps $h_N\rightarrow \mathcal{F}$, we have $N=\cup_i U_i$ where $U_i\subset N$ is the open subset for which $h_{U_i}=h_N\times_\mathcal{F}h_{M_i}$.
\end{enumerate}
Then $\mathcal{F}$ is representable.
\end{lemma}

In other words, if $\mathcal{F}$ is a sheaf, and if $\mathcal{F}$ admits an open cover by representable sheaves, then $\mathcal{F}$ is itself representable.

\begin{proof}
For each $i,j\in I$, because $h_{M_j}\rightarrow \mathcal{F}$ is an open immersion we have that $h_{M_i}\times_\mathcal{F} h_{M_j}$ identifies with the representable functor associated to an open subset $U_{ij}\subset M_i$.  By the Yoneda lemma we get an isomorphism $\alpha_{ij}:U_{ij}\simeq U_{ji}$ satisfying the cocycle condition, allowing to glue the $M_i$ to an $F$-analytic manifold $M$ equipped with a map $h_M\rightarrow\mathcal{F}$ which is an isomorphism over $h_{M_i}$.  Being a map of sheaves it must therefore be a global isomorphism.
\end{proof}

We recall that although the category of $F$-analytic manifolds does not have all pullbacks (there are transversality issues), it at least has pullbacks of diagrams $M\rightarrow N\leftarrow N'$ such that $M\rightarrow N$ is a \emph{submersion}.  More precisely, in this case the presheaf
$$S\mapsto \operatorname{Hom}(S,M)\times_{\operatorname{Hom}(S,N)}\operatorname{Hom}(S,N')$$
is representable.  Indeed, recall from \cite{SerreLie} that a submersion is a map $M\rightarrow N$ which, locally on $M$, identifies with the composition of a projection $U\times V\rightarrow U$ followed by an open inclusion.  Because of the locality principle Lemma \ref{localrep}, this lets us reduce to the case where $M\rightarrow N$ is such a projection, when the question for $M\rightarrow N\leftarrow N'$ is equivalent to that for $V\rightarrow \ast \leftarrow N'$.  Again working locally we can assume $V,N'$ are open subsets of $F^d,F^e$.  Then $V\times N'$ is representable by the corresponding open subset of $F^{d+e}$.  

\begin{remark}
We recall that the implicit function gives a convenient criterion for a map $f:M\rightarrow N$ of $F$-manifolds to be a submersion: $f$ is a submersion if and only if for all $x\in M$, the induced $F$-linear map $df|_x:T_xM\rightarrow T_{f(x)}N$ on tangent spaces is surjective.  Moreover, we have the following permanence properties: the class of submersions is closed under composition and pullback, and the question of whether $M\rightarrow N$ is a submersion is local on $M$.
\end{remark}

\section{The deformation to the tangent bundle}\label{defsec}

The deformation to the tangent bundle is a kind of geometric incarnation of the definition of the derivative as a limit of difference quotients, see Lemma \ref{localdef} for more precision.  Here we give a construction of this deformation from a functor of points perspective.

We start with just the tangent bundle itself.  The functor of points description will be an analog of the familiar picture from algebraic geometry, based on thinking of the ring $F[t]/t^2$ as incarnating ``a point together with a tangent vector at that point''.  More generally, if $S$ is an $F$-analytic manifold, we define a new $F$-ringed space $S_\rightarrow$ to have the same underlying topological space as $S$, but with structure sheaf $\mathcal{O}_{S_\rightarrow} = \mathcal{O}_S[t]/t^2$.  Note that $S\mapsto S_\rightarrow$ has the evident structure of a functor from $\operatorname{Man}_F$ to the category of $F$-ringed spaces.  Then we make the following definition.

\begin{definition}
Let $M\in\operatorname{Man}_F$.  Define a presheaf $TM$ on $F$-manifolds by
$$S\mapsto \operatorname{Hom}(S_\rightarrow,M),$$
where $\operatorname{Hom}(S_\rightarrow,M)$ stands for the set of maps of $F$-ringed spaces $S_\rightarrow\rightarrow M$.
\end{definition}

Note the following structure on $TM$:

\begin{enumerate}
\item There is a projection map
$$TM\rightarrow M,$$
sending $f:S_\rightarrow \rightarrow M$ to the composition $f_0$ of $f$ with the map $S\rightarrow S_\rightarrow$ which is induced by the $\mathcal{O}_S$-algebra map $\mathcal{O}_S[t]/t^2\rightarrow \mathcal{O}_S$ sending $t\mapsto 0$.
\item There is ``zero section'' map
$$M\rightarrow TM,$$
ending $f:S\rightarrow M$ to to its composition with the map $S_\rightarrow\rightarrow S$ induced by the unit map $\mathcal{O}_S\rightarrow\mathcal{O}_S[t]/t^2$.
\item There is a ``scalar multiplication'' map
$$F\times TM\rightarrow TM,$$
sending $(\lambda:S\rightarrow F, f:S_\rightarrow\rightarrow M)$ to $f\circ [\lambda]:S_\rightarrow \rightarrow M$, where
$$[\lambda]:S_{\rightarrow}\rightarrow S_{\rightarrow}$$
is given by the $\mathcal{O}_S$-algebra map $\mathcal{O}_S[t]/t^2\rightarrow \mathcal{O}_S[t]/t^2$ sending $t\mapsto \lambda t$.  Here we view $\lambda$ as a global section of $\mathcal{O}_S$.
\item There is an ``addition'' map
$$TM\times_M TM\rightarrow M,$$
given as follows: suppose we have $f:S_\rightarrow\rightarrow M$ and $g:S_\rightarrow\rightarrow M$ such that $f_0=g_0:S\rightarrow M$.  Using the homomorphism
$$\mathcal{O}_S[x]/x^2\times_{\mathcal{O}_S}\mathcal{O}_S[y]/y^2=\mc{O}_S[x,y]/(x^2,xy,y^2)\to \mathcal{O}_S[t]/t^2$$
defined by $x,y\mapsto t$ we make a new map
$$f+g:S_\to\rightarrow M.$$
\end{enumerate}

In total, this structure makes $TM$ into an $F$-module object over $M$.

\begin{lemma}
\begin{enumerate}
\item Let $M\in\operatorname{Man}_F$.  Then $TM:\operatorname{Man}_F^{op}\rightarrow\operatorname{Sets}$ is a sheaf.
\item Suppose $U\subset M$ is an open subset of an $F$-manifold $M$.  Then $TU\subset TM$ is also an open subset (in the sense of presheaves, as in Lemma \ref{localrep}).
\item  Suppose $\{U_i\}_{i\in I}$ is an open cover of an $F$-manifold $M$.  Then $\{TU_i\}_{i\in I}$ is an open cover of $\{TM\}_{i\in I}$ (in the sense of presheaves).
\end{enumerate}
\end{lemma}
\begin{proof}
Note that if $U\subset S$ is an open submanifold, then $U_{\rightarrow} \subset S_{\rightarrow}$ is the same open subset, and the structure sheaf of $U_{\rightarrow}$ is restricted from that of $S_{\rightarrow}$.  The claims follows immediately.
\end{proof}

Given this lemma, the local-global principle Lemma \ref{localrep} implies that in order to check that $TM$ is representable by an $F$-manifold, it suffices to work locally on $M$.  In particular it suffices to check it for open subsets of $F^d$ for any $d\geq 0$.  That case is handled by the following lemma.

\begin{lemma}\label{localtangent}
For each pair $(U,V)$ consisting of a finite-dimensional $F$-vector space $V$ and an open subset $U\subset V$, viewed as an $F$-manifold in the natural way, there is an isomorphism
$$\iota_{U,V}:TU\simeq U\times V,$$
satisfying the following functoriality property: if $(U',V')$ is another such pair and $\varphi:U\rightarrow U'$ is any $F$-analytic map, then the induced map $T\varphi:TU\rightarrow TU'$ corresponds, via $\iota_{U,V}$ and $\iota_{U',V'}$, to the map $U\times V\rightarrow U'\times V'$ described by
$$(u,v)\mapsto (\varphi(u),d\varphi|_u(v)),$$
where $d\varphi|_u:V\rightarrow V'$ denotes the first differential of $\varphi$ at $u$ (defined by the standard limit of difference quotients, or by looking at power series expanstions).  Moreover, via $\iota_{U,V}$ we have that:
\begin{enumerate}
\item the projection $TU\rightarrow U$ corresponds to the projection $(u,v)\mapsto u$;
\item the zero section $U\rightarrow TU$ corresponds to $u\mapsto (u,0)$;
\item the scalar multiplication $F\times TU\rightarrow TU$ corresponds to $(\lambda,u,v)\mapsto (u,\lambda v)$.
\item the addition $TU\times_U TU\rightarrow U$ corresponds to $((u,v),(u,v'))\mapsto (u,v+v')$.
\end{enumerate}
In other words, $TU\rightarrow U$, equipped with its $F$-module structure, identifies with the trivial vector bundle $U\times V\rightarrow U$ modelled on $V$.
\end{lemma}
\begin{proof}
Note that for $S\in\operatorname{Man}_F$ there is a natural map of $F$-ringed spaces $S_\rightarrow \rightarrow S\times F$, given by $s\mapsto (s,0)$ on underlying spaces, and on structure sheaves sending $\varphi\in \mathcal{O}(W\times D)$ (for $W\subset S$ open and $D\subset F$ open containing $0$) to
$$\varphi(-,0) + \varphi'(-,0)\cdot t \in \mathcal{O}_S(W)[t]/t^2,$$
where for $s\in W$, the expression $\varphi'(s,0)$ stands for the derivative at $0$ of the function $\varphi(s,-):D\rightarrow F$.

Now, let $U\subset V$ be as in the statement of the lemma, and suppose given $F$-analytic maps $a:S\rightarrow U$ and $b:S\rightarrow V$.  Then we can define an $F$-analytic map
$$a+bt: S\times F\rightarrow V$$
by sending $(s,t)\mapsto a(s)+b(s)\cdot t$.  Restricting along the above map $S_\rightarrow\rightarrow S\times F$, we get a map
$$a+bt: S_{\rightarrow}\rightarrow U.$$
This operation $(a,b)\mapsto a+bt$ defines the desired map
$$\iota_{U,V}: U\times V\rightarrow TU.$$
We can produce an inverse as follows: choose coordinates $x_1,\ldots,x_d$ for $V$.  Given a map $f:S_\rightarrow \rightarrow U$, write $f^\sharp(x_i) = a_i + b_i\cdot t$.  Then the $a_i$ are the coordinates of a map $a:S\rightarrow U$ and the $b_i$ are the coordinates of a map $b:S\rightarrow V$.  The argument that this describes an inverse of $\iota_{U,V}$ will be given later in a more general setting (Lemma \ref{technicaldeflemma}) so we skip it here.  Finally, the required functoriality properties are straightforward to check because we have direct definitions of both $\iota_{U,V}$ and its inverse.
\end{proof}

\begin{remark}
As mentioned before, the previous lemmas together imply that $TM$ is representable by an $F$-manifold.  Moreover, thanks to the described functoriality in the previous lemma, we also see that this $TM$ agrees with the tangent bundle as usually defined using charts.
\end{remark}

Now, the \emph{deformation to the tangent bundle} is a one-parameter family of manifolds $(D_tM)_{t\in F}$ attached to $M$, where $D_t=M\times M$ for $t\neq 0$ but $D_t = TM$ for $t=0$.  More precisely, this family is encoded as an $F$-manifold $DM$ together with a submersive map $DM\rightarrow F$, so that $D_tM$ is the fiber at $t\in F$.  Intuitively, one can imagine ``rescaling'' $M\times M$ near the diagonal $\Delta(M)\subset M\times M$ so that it resembles $TM$ more and more.  To give an explanation from the functor of points perspective, note that the $F$-algebra $F[t]/t^2$ incarnating tangent vectors is a specialization of the $F$-algebra $F[t]/t(t-\tau)$ for $\tau\in F$, which for $\tau\neq 0$ incarnates two points.  Mapping into $M$ we get the desired specialization of $M\times M$ to $TM$.

Now we turn this idea into a precise description of $DM$.  Given $S\in\operatorname{Man}_F$ and $\tau:S\rightarrow F$ we define a new $F$-ringed space $S_{0\cup \tau}$ as follows.

\begin{definition}
Let $S\in\operatorname{Man}_F$ and let $\tau:S\rightarrow F$.  Define an $F$-ringed space $S_{0\cup \tau}$ to have underlying topological space the closed subspace of $S\times F$ which is the union of the $0$-section and $\tau$-section of the projection $S\times F\rightarrow S$, and to have sheaf of $F$-algebras given by $\mathcal{O}_{S\times F}/\mathcal{I}_0\cdot\mathcal{I}_\tau$ where $\mathcal{I}_0$ (resp.\ $\mathcal{I}_\tau$) is the ideal sheaf of functions vanishing on the $0$-section (resp.\ the $\tau$-section).  (Note that this sheaf of $F$-algebras on $S\times F$ vanishes away from $S_{0\cup \tau}$ hence corresponds by restriction to a sheaf of $F$-algebras on $S_{0\cup \tau}$.)
\end{definition}

\begin{lemma}\label{cartiersection}
For $S\in\operatorname{Man}_F$ and $\tau:S\rightarrow F$, the ideal sheaf $\mathcal{I}_\tau\subset\mathcal{O}_{F\times S}$ of functions vanishing on $S_\tau$ (the $\tau$-section of the projection $\pi:S\times F\rightarrow S$) is generated by the single non-zero divisor $t-\pi^\ast\tau$ where $t:S\times F\rightarrow F$ denotes the projection to $F$.  Moreover the $F$-ringed space $(S_\tau,\mathcal{O}_{S\times F}/\mathcal{I}_\tau)$ is isomorphic to $S$ via the mutually inverse maps given by the $\tau$-section and $\pi$.
\end{lemma}
\begin{proof}
Using the automorphism $(s,x)\mapsto (s,x-\tau(s))$ of $S\times F$, we can reduce to the case $\tau=0$ which follows from Lemma \ref{localrings}.
\end{proof}

\begin{lemma}\label{quadraticalgebra}
Let $S\in\operatorname{Man}_F$ and $\tau:S\rightarrow F$.  Recall the projection $\pi:S\times F\rightarrow S$ and its restriction to $\pi:S_{0\cup \tau}\rightarrow S$.  The the sheaf of $F$-algebras $\pi_\ast \mathcal{O}_{S_{0\cup \tau}}$ on $S$ identifies with $\mathcal{O}_S[t]/t\cdot (t-\tau)$ where $t$ corresponds to the projection $S\times F\rightarrow F$.  In particular, as an $\mathcal{O}(S)$-module, $\mathcal{O}(S_{0\cup \tau})$ is free of rank 2 on $1$ and $t$.
\end{lemma}
\begin{proof}
The fact that $t\cdot (t-\pi^\ast \tau)\in \mathcal{I}_0\cdot\mathcal{I}_\tau$ furnishes the comparison map $\mathcal{O}_S[t]/t\cdot (t-\tau)\rightarrow \pi_\ast \mathcal{O}_{0\cup \tau}$, so it suffices to show the claim that $\pi_\ast \mathcal{O}_{S_{0\cup\tau}}$ is free on $1$ and $t$.  For this consider the short exact sequence of sheaves on $S\times F$ 
$$0\rightarrow \mathcal{O}_{S\times F}/\mathcal{I}_\tau\overset{\cdot t}{\rightarrow} \mathcal{O}_{S\times F}/\mathcal{I}_0\cdot\mathcal{I}_\tau\rightarrow\mathcal{O}_{S\times F}/\mathcal{I}_0\rightarrow 0$$
coming from Lemma \ref{cartiersection}.  Since these sheaves are supported on $S_{0\cup \tau}$ and $\pi:S_{0\cup \tau}\rightarrow S$ is proper with finite fibers, the pushforward to $S$ is still short exact.  But the outer terms identify with $\mathcal{O}_S$ by previous lemma, and the unit map $\mathcal{O}_S\rightarrow \pi_\ast \mathcal{O}_{S\times F}/\mathcal{I}_0\cdot\mathcal{I}_\tau$ yields a splitting, giving the claim.
\end{proof}

\begin{lemma}\label{specialcasedef}
Let $S\in\operatorname{Man}_F$ and $\tau:S\rightarrow F$.
\begin{enumerate}
\item If $\tau=0$, then $S_{0\cup \tau}=S_{\rightarrow}$.
\item If $\tau(s)\neq 0$ for all $s\in S$, then $S_{0\cup \tau}=S\sqcup S$ via the maps $0,\tau:S\rightarrow S_{0\cup \tau}$.
\end{enumerate}
\end{lemma}
\begin{proof}
For 1, note that when $\tau=0$, $\pi:S_{0\cup\tau}\rightarrow S$ is a homeomorphism with inverse the $0$-section, so that the calculation $\pi_\ast\mathcal{O}_{S\cup \tau}=\mathcal{O}_S[t]/t^2$ from the previous lemma amounts to a calculation of $\mathcal{O}_{S\cup \tau}$ itself, giving the claim.  For 2, the map $S\sqcup S\rightarrow S_{0\cup \tau}$ is clearly a homeomorphism, and the identification of structure sheaves results from Lemma \ref{cartiersection}.
\end{proof}

Now we try to understand maps of $F$-ringed spaces from $S_{0\cup \tau}$ to an open subset $U$ of $F^d$, viewed as an $F$-manifold.  Actually this can be reduced to the case of $U=F^d$ itself, because a map of $F$-ringed spaces to $U$ is the same as a map of $F$-ringed spaces to $F^d$ which set-theoretically lands in $U$ (and $S_{0\cup \tau}$ is set-theoretically covered by $0,\tau:S\rightarrow S_{0\cup\tau}$, allowing to easily understand the latter condition).

Instead of talking about $F^d$ it is slightly cleaner to talk about a finite-dimensional $F$-vector space $V$.  Let $S\in\operatorname{Man}_F$ and $\tau:S\rightarrow F$.  Suppose given $F$-analytic maps $a:S\rightarrow V$ and $b:S\rightarrow V$.  The $F$-analytic map
$$a+bt: S\times F\rightarrow V$$
defined by $(s,t)\mapsto a(s) + b(s)\cdot t$ can be restricted to $S_{0\cup \tau}$, giving a map
$$a+bt: S_{0\cup \tau}\rightarrow V.$$

\begin{lemma}\label{technicaldeflemma}
This association $(a,b)\mapsto a+bt$ gives a bijection between the set of $F$-analytic maps $S\rightarrow V\times V$ and the set of maps of $F$-ringed spaces $S_{0\cup \tau}\rightarrow V$.

Moreover, given an open subset $U\subset V$, the map $a+bt:S_{0\cup \tau}\rightarrow V$ lands in $U$ if and only if the map $(a,b):S\rightarrow V\times V$ is such that $a(s),a(s)+b(s)\tau(s)\in U$ for all $s\in S$.
\end{lemma}
\begin{proof}
As explained in the previous paragraph, the second claim is immedaite from the fact that the sections $0,\tau:S\rightarrow S_{0\cup \tau}$ of $\pi:S_{0\cup \tau}\rightarrow S$ are jointly surjective.
For the first claim, let us make a map backwards.  It will be convenient to use coordinates, so assume $V=F^d$ with coordinate functions $x_1,\ldots x_d\in\mathcal{O}(F^d)$.  Given a map of $F$-ringed spaces $f:S_{0\cup \tau}\rightarrow F^d$, we get $f^\sharp(x_1),\ldots,f^\sharp(x_d)\in \mathcal{O}(S_{0\cup \tau})$.  By Lemma \ref{quadraticalgebra}, we can uniquely write
$$f^\sharp(x_i) = a_i + b_i\cdot t$$
for $a_i,b_i\in \mathcal{O}(S)$, or equivalently, $a_i,b_i:S\rightarrow F$.  Then we can collect these $(a_i,b_i)_{i=1,\ldots,d}$ back up again as coordinates of functions $a,b:S\rightarrow F^d$, as desired.

It is clear that, given $a,b:S\rightarrow F^d$, if we apply the above procedure to the map $a+bt:S_{0\cup \tau}\rightarrow F^d$, then we recover the same $(a,b)$ we started with.  For the other direction, it suffices to show that if $f,g:S_{0\cup \tau}\rightarrow F^d$ are two maps of $F$-ringed spaces with $f^\sharp(x_i)=g^\sharp(x_i)$ for all $i=1,\ldots,d$, then $f=g$.  For that, first note that $f$ and $g$ must agree as maps of topological spaces, as one sees by composing with $0,\tau:S\rightarrow S_{0\cup \tau}$.  Moreover, and for the same reason, $f$ and $g$ agree on structure sheaves above the open locus in $S$ where $\tau\neq 0$ (as over that locus $S_{0\cup \sigma}$ breaks up into 2 copies of $S$ via $0$ and $\tau$, see Lemma \ref{specialcasedef}).  Thus, let $s\in S$ with $\tau(s)=0$, and let $a\in F^d$ denote the common point $a=f(s,0)=g(s,0)$.  Subtracting $a$, we can assume $a=0$.  We need to see that the two maps on stalks
$$f^\sharp,g^\sharp: \mathcal{O}_{F^d,0}\rightarrow \mathcal{O}_{S_{0\cup \tau},(s,0)}$$
agree.  By Lemma \ref{quadraticalgebra} and Lemma \ref{localrings}, choosing coordinates $y_1,\ldots,y_e$ on $S$ at $s$, the target injects into $F[[y_1,\ldots,y_e]][t]/t\cdot(t-\tau)$, which is an inverse limit of Artin local $F$-algebras with residue field $F$.  However, by Lemma \ref{localrings}, any $F$-algebra homomorphism from $\mathcal{O}_{F^d,0}$ to an Artin local $F$-algebra with residue field $F$ factors through $F[x_1,\ldots,x_d]/(x_1,\ldots,x_d)^N$ for some $N$.  As by assumption $f^\sharp(x_i)=g^\sharp(x_i)$ for all $i$, this implies that $f^\sharp=g^\sharp$ as desired.
\end{proof}

Let us also make a remark about functoriality.  Suppose given $S\in\operatorname{Man}_F$ and $\tau:S\rightarrow F$, with $\tau$ equivalently viewed as a section of $S\times F\rightarrow S$.  Given a map $f:S'\rightarrow S$ in $\operatorname{Man}_F$, the map $f\times \operatorname{id}_F$ sends the sections $0,\tau\circ f$ of $S'\times F\rightarrow S'$ to the sections $0,\tau$ of $S\times F\rightarrow S$, and hence induces a natural map
$$S'_{0\cup \tau\circ f}\rightarrow S_{0\cup \tau}$$
covering $f:S'\rightarrow S$ via the projections, and intertwining the sections $0,\tau$ with $0,\tau\circ f$.

Now we can define the deformation to the tangent bundle.

\begin{definition}\label{defstructure}
Let $M\in\operatorname{Man}_F$.  Define a presheaf $DM$ on $\operatorname{Man}_F$ by
$$S\mapsto (\tau:S\rightarrow F, f:S_{0\cup \tau}\rightarrow M)$$
where $\tau$ and $f$ are maps of $F$-ringed spaces.  The presheaf structure uses the functoriality explained in the previous paragraph.
\end{definition}

We have the following structure on $DM$.

\begin{enumerate}
\item There is a map $(p,\pi_0,\pi_1):DM\rightarrow F\times M\times M$, induced by sending $(\tau,f)$ to $(\tau,f_0, f_\tau)$ where $f_0$ (resp.\ $f_\tau$) is the composition of $f$ with the section of $S_{0\cup \tau}\rightarrow S$ induced by $0$ (resp.\ $\tau$).
\item There is a map $\sigma: F\times M \rightarrow DM$, sending $(\tau:S\rightarrow F,g:S\rightarrow M)$ to $(\tau, g\circ \pi)$ where $\pi:S_{0\cup \tau}\rightarrow S$ is the natural projection.
\item There is an action of the group object $F^\times$ on $DM$, described as follows.  Suppose given $S\in\operatorname{Man}_F$, $\lambda:S\rightarrow F^\times$, and a map $S\rightarrow DM$ corresponding to $(\tau:S\rightarrow F, f:S_{0\cup \tau}\rightarrow M)$ as in the definition.  We define $\lambda\cdot (\tau,f)$ to be $(\lambda\cdot \tau, f\circ [\lambda^{-1}])$, where $[\lambda^{-1}]:S_{0\cup \lambda\cdot\tau}\rightarrow S_{0\cup \tau}$ is the map induced by the map $S\times F\rightarrow S\times F$ given by $(s,t)\mapsto (s,\lambda^{-1}t)$.
\end{enumerate}

More prominent than the whole triple $(p,\pi_0,\pi_1)$ for us will be the first two components, the map $(p,\pi_0):DM\rightarrow F\times M$.  To emphasize this we introduce new notation, writing $\Pi = (p,\pi_0)$.  We note that $\sigma:F\times M\rightarrow DM$ is a section of $\Pi$, and that both $\sigma$ and $\Pi$ are $F^\times$-equivariant with respect to the $F^\times$-action on $DM$ described in 3 and the $F^\times$-action on $F\times M$ given by multiplication on the first coordinate.

\begin{lemma}\label{specialdefgeom}
Let $M\in\operatorname{Man}_F$.  Then:
\begin{enumerate}
\item The map $(\pi,\pi_0,\pi_1):DM\rightarrow F\times M\times M$ restricts to an isomorphism
$$DM\times_F F^\times \simeq F^\times \times M\times M.$$
Via this isomorphism, $\sigma$ corresponds to the map $\operatorname{id}_{F^\times}\times \Delta_M:F^\times \times M\rightarrow F^\times \times M\times M$, $\Pi$ corresponds to the map $F^\times \times M\times M\rightarrow F^\times \times M$ projecting onto the first two factors, and the $F^\times$-action on $DM$ from 3 goes to the action of $F^\times$ on $F^\times \times M\times M$ via multiplication on the first factor.
\item There is a natural isomorphism $DM\times_F \{0\} = TM$, via which the map $\Pi$ goes to the projection $TM\rightarrow M$ and $\sigma$ goes to the $0$-section $M\rightarrow TM$.  Moreover the $F^\times$-action on $DM$ from 3 goes to the $F^\times$-action on $TM$ given by the \emph{inverse} of the natural scalar multiplication action on the vector bundle $TM\rightarrow M$ (see Lemma \ref{localtangent}).
\end{enumerate}
\end{lemma}
\begin{proof}
This is an immediate translation of Lemma \ref{specialcasedef}.
\end{proof}

\begin{lemma}\label{coverdef}
Let $M\in\operatorname{Man}_F$.
\begin{enumerate}
\item The presheaf $DM$ is a sheaf.
\item If $U\subset M$ is an open subset, $DU\rightarrow DM$ is an open immersion.

\item If $\{U_i\}_{i\in I}$ is an open cover of $M$, then $\{DU_i\}_{i\in I}\cup DM\times_F F^\times$ is an open cover of $DM$.
\end{enumerate}
\end{lemma}
\begin{proof}
Note that if $S\in\operatorname{Man}_F$ and $\tau:S\rightarrow F$, then for an open subset $U\subset S$ we have that $\pi^{-1}U = U_{0\cup \tau|_U}$ where $\pi:S_{0\cup \tau}\rightarrow S$ is the projection.  This readily implies 1.

For 2, suppose $U\subset M$ is open, and let $S\in\operatorname{Man}_F$ with a map $S\rightarrow DM$, corresponding to $(\tau:S\rightarrow F,f:S_{0\cup \tau}\rightarrow M)$.  Then to lift $S\rightarrow DM$ along $DU\rightarrow DM$ is the same as to lift $f$ along $U\rightarrow M$, i.e.\ it is just the condition that $S_{0\cup \tau}\subset f^{-1}U$.  Let $S'\subset S$ denote the set of those $s\in S$ such that $\pi^{-1}(\{s\})\subset f^{-1}U$.  Then $S'$ is open because $\pi:S_{0\cup \tau}\rightarrow S$ is proper, $\pi^{-1}S'\subset f^{-1}U$ by definition, and the formation of $S'$ plainly commutes with pullback along maps $T\rightarrow S$. We deduce that $S'=S\times_{DM}{DU}$.

For 3, suppose again given $S\in\operatorname{Man}_F$ with a map $S\rightarrow DM$, corresponding to $(\tau:S\rightarrow F,f:S_{0\cup \tau}\rightarrow M)$.  Let $S'_i\subset S$ denote the open subset $\{s\in S: \pi^{-1}(s)\subset f^{-1}(U_i)\}$, so $S'_i=S\times_{DM}{DU_i}$ by the previous paragraph.  Since $S\times_{DM}(DM\times_F F^\times)$ identifies with $\tau^{-1}(F^\times)\subset S$, it suffices to show that the $S'_i$ cover $\tau^{-1}(\{0\})$.  But indeed, for $s\in \tau^{-1}(\{0\})$ the fiber $\pi^{-1}(\{s\})$ consists of a single point so $f(\pi^{-1}(\{s\}))\subset U_i$ for some $i$.
\end{proof}

Next we describe $DM$ for open subsets of $F$-vector spaces, together with the effect of coordinate transformations.

\begin{lemma}\label{localdef}
Suppose given a pair $(U,V)$ consisting of a finite-dimensional $F$-vector space $V$ and an open subset $U\subset V$.  We view $U$ as an $F$-manifold in the natural way.  Let $D_{U,V}$ denote the open subset of $F\times U\times V$ consisting of those $(t,u,v)$ with $u,u+tv\in U$.  Then there is an isomorphism
$$\iota_{U,V}:DU\simeq D_{U,V}$$
satisfying the following naturality property: given another pair $(U',V')$ as above, for any $F$-analytic map $\varphi:U\rightarrow U'$, the induced map $DU\rightarrow DU'$ corresponds, via $\iota_{U,V}$ and $\iota_{U',V'}$, to the map $D_{U,V}\rightarrow D_{U',V'}$ described as follows:
$$(t,u,v) \mapsto (t, \varphi(u),\frac{1}{t}\left(\varphi(u+tv)-\varphi(u)\right))$$
for $t\neq 0$, while
$$(0,u,v)\mapsto (0,\varphi(u),d\varphi|_u(v))$$
for $t=0$.

Moreover, under $\iota_{U,V}$, the map $(p,\pi_0,\pi_1):DU\rightarrow F\times U\times U$ corresponds to the map $D_{U,V}\rightarrow F\times U\times U$ given by $(t,u,v)\mapsto (t,u,u+tv)$, the map $\sigma:F\times U\rightarrow DU$ corresponds to $(t,u)\mapsto (t,u,0)$, and the $F^\times$-action on $DU$ corresponds to $\lambda\cdot(t,u,v) = (\lambda\cdot t,u,\lambda^{-1}\cdot v)$.
\end{lemma}
\begin{proof}
This is an immediate translation of Lemma \ref{technicaldeflemma}.  Note also that the formula for $t=0$ follows from that for $t\neq 0$ by continuity.
\end{proof}

\begin{remark}
We interpret this as saying that this association $M\mapsto DM$ geometrizes the definition of derivative as a limit of difference quotients.
\end{remark}

\begin{theorem}\label{defrepresentable}
Let $M\in\operatorname{Man}_F$.
\begin{enumerate}
\item The presheaf $DM$ is representable by an $F$-manifold.
\item The map $\Pi:DM\rightarrow F\times M$ is a submersion.
\item If $f:M\rightarrow N$ is a submersion (resp.\ a surjective submersion), then so is the induced map $Df:DM\rightarrow DN$.
\item If $M\rightarrow N$ is a submersion and $N'\rightarrow N$ is an arbitrary map of $F$-manifolds, then $D(M\times_N N')\overset{\sim}{\rightarrow} DM\times_{DN}DN'$.
\end{enumerate}
\end{theorem}
\begin{proof}
For 1, by the local-global principle Lemma \ref{localrep}, we can check representability by restricting to an open cover of $DM$.  Choose an open cover $\{U_i\}_{i\in I}$ of $M$ by open subsets of $F$-vector spaces, and let us use the open cover $\{DU_i\}_{i\in I}\cup DM\times_{F}F^\times$ provided by Lemma \ref{coverdef}.  We have that $DM\times_FF^\times = F^\times \times M\times M$ by Lemma \ref{specialdefgeom} and this is plainly representable.  On the other hand each $DU_i$ is representable by Lemma \ref{localdef}.  Thus $DM$ is representable.

For 2, we likewise argue as follows.  Over the locus $F^\times\subset F$, we see the projection map $F^\times \times M\times M\rightarrow F^\times\times M$ which is clearly a submersion.  Thus by Lemma \ref{coverdef} we can work locally on $M$, reducing to the case of an open subset of an $F$-vector space, which follows from Lemma \ref{localdef}.

For 3, again, over the locus $F^\times\subset F$ we see the map $F^\times\times M\times M\rightarrow F^\times \times N \times N$ given by $\operatorname{id}_{F^\times}\times f\times f$ which is a submersion.  Thus by Lemma \ref{coverdef} we can work locally and therefore assume that $f$ is a projection map $U\times U'\rightarrow U'$ for $U\subset V$ and $U'\subset V'$ open subsets of $F$-vector spaces.  Then Lemma \ref{localdef} shows that $Df$ is obtained from a projection map via restriction to open subsets of source and target, hence is a submersion as desired.  For surjectivity, we can check over $F^\times \subset F$ and $\{0\}\times F$  separately, where it is clear in each case by the identification with $(-)\times (-)$ and $T(-)$ respectively.

For 4, again, over the locus $F^\times\subset F$ this is clear from the identification of $DM$ with $M\times M$, and so Lemma \ref{localdef} lets us work locally, thereby (by the implicit function theorem) reducing us to the case where $M\rightarrow N$ is a projection $U\times V\rightarrow V$ induced by a projection map of ambient vector spaces.  Moreover we can assume $N'$ is also an open subset of a vector space.  Then the explicit description from Lemma \ref{localdef} gives the claim.
\end{proof}

We also need a relative version of this construction, for a submersion $f:M\rightarrow S$ (thought of as a family of $F$-manifolds parametrized by $S$).  We define this in terms of the absolute version as follows.

\begin{definition}
Let $f:M\rightarrow S$ be a submersion in $\on{Man}_F$.  Define
$$D(M/S):= DM\times_{DS} (F\times S),$$
where the map $DM\to DS$ is $Df$ and the map $F\times S\rightarrow DS$ is the section $\sigma$ from Definition \label{defstructure}.
\end{definition}

Note that this fiber product is representable by an $F$-manifold because $DM\rightarrow DS$ is a submersion (Theorem \ref{defrepresentable}).  Note also that the map $f$ is supposed to be implicit in the notation $D(M/S)$.

We have in-families analogs of the basic structure on $DM$:

\begin{enumerate}
\item There is an $F^\times$-action on $D(M/S)$, since $F^\times$ acts on the diagram defining the pullback.
\item There is an $F^\times$-equivariant map $\Pi:D(M/S)\rightarrow F\times M$ gotten by composing the tautological map $D(M/S)\rightarrow DM$ from the pullback and the map $\Pi:DM\rightarrow F\times M$ from the absolute case (Definition \ref{defstructure}).  Note that the other tautological map $D(M/S)\rightarrow F\times S$ from the pullback is recovered from this $\Pi$ by composition with $F\times f:F\times M\rightarrow F\times S$.
\item The map $\Pi$ has an $F^\times$-equivariant section $\sigma:F\times M\rightarrow D(M/S)$, which when composed to $DM$ is given by the $\sigma:F\times M\rightarrow DM$ from the absolute case and when composed to $F\times S$ is given by $\on{id}_F\times f$.  These two maps agree when composed to $DS$ by the functoriality of the $\sigma$ in the absolute case.
\item There is an $F^\times$-equivariant isomorphism $D(M/S)\times_F F^\times \simeq  (M\times_SM)\times F^\times $, under which $\Pi$ goes to the first projection map $M\times_S M\rightarrow M$ (crossed with $F^\times$) and $\sigma$ goes to the diagonal section.  This is easily deduced from the absolute case using $M\times_S M = S\times_{S\times S}(M\times M)$.
\item There is an isomorphism $D(M/S)\times_F \{0\}\simeq T(M/S)$ under which $\Pi$ goes to the natural projection map $T(M/S)\rightarrow M$ and $\sigma$ goes to the zero section.  Here $T(M/S)$ is the usual relative tangent bundle, i.e.\ the kernel of the surjective vector bundle map $TM\rightarrow f^\ast TS$.  This isomorphism is $F^\times$-equivariant with respect to the inverse of the usual scalar multiplication action on the relative tangent bundle.  This is immediate from the absolute case.
\end{enumerate}

\begin{remark}
A slightly nontrivial point is that the map $\Pi:D(M/S)\rightarrow F\times M$ is a submersion, just as in the absolute case.  This does not immediately follow from the absolute case because the map $D(M/S)\rightarrow DM$ is almost never a submersion.  However, we can argue analogously to the absolute case.  Over $F^\times \subset F$, we see the first projection map $M\times_S M\rightarrow M$ (crossed with $F^\times$).  This is a pullback of $f$, hence a submersion.  Thus, by the locality principle Lemma \ref{coverdef}, we can work locally on $M$ and $N$.  By the implicit function theorem, we can therefore assume that $f$ is obtained from a linear projection map of $F$-vector spaces $V\times W\rightarrow V$ via restricting to open subsets.  In that case the explicit description of $D$ from Lemma \ref{localdef} implies that our map $\Pi:D(M/S)\rightarrow F\times M$ is similarly obtained by restriction to open subsets from a surjective linear map of vector spaces, hence is a submersion.
\end{remark}

Beyond this, to see that $D(M/S)$ is a reasonable in-families construction, we need to check that it recovers the absolute construction when $S=\ast$, and that it is compatible under base-change.  The first claim is trivial because $\sigma:F\times M\rightarrow DM$ is an isomorphism when $M=\ast$.  The second claim, more precisely, is the following lemma.

\begin{lemma}
Let $f:M\rightarrow S$ be a submersion in $\on{Man}_F$, let $g:S'\rightarrow S$ be an arbitrary map to $S$ in $\on{Man}_F$, and let $f':M'\rightarrow S'$ denote the pullback of $f$ along $g$.  Then:
\begin{enumerate}
\item The natural map $D(M'/S')\rightarrow D(M/S)\times_S S'$ is an isomorphism.
\item The isomorphism in 1 intertwines all the structure discussed above: the $F^\times$-action, the submersion $\Pi$, the section $\sigma$, and the identifications over $F^\times\subset F$ and $\{0\}\subset F$.
\end{enumerate}
\end{lemma}
\begin{proof}
This follows immediately from the fact that $D(-)$ preserves the pullback square in question (Theorem \ref{defrepresentable}).
\end{proof}

We will also rely on a certain specialization of this deformation to the tangent bundle construction, namely the deformation of $M$ itself to a fixed tangent space.  More precisely, suppose given a point $m$ in an $F$-manifold $M$.  Then we can consider the fiber product
$$D_mM:= DM\times_{F\times M}F$$
where the map $DM\rightarrow F\times M$ is the submersion $\Pi$ and the map $F\rightarrow F\times M$ is $t\mapsto (t,m)$.  This $D_mM$ is an $F^\times$-equivariant manifold equipped with a submersion $D_mM\rightarrow F$ whose fiber over $0$ identifies with the tangent space $T_mM$ and whose general fiber identifies with $M$ itself.  Moreover this submersion $D_mM\rightarrow F$ has a section (induced by $\sigma$) which specializes to the point $0\in T_mM$ on the special fiber and to the point $m\in M$ on the general fiber.

The in-families version of this is the following.

\begin{definition}\label{deformman}
Let $f:M\rightarrow S$ be a submersion and suppose given a section $g:S\rightarrow M$.  Define
$$D_g(M/S):= D(M/S)\times_{F\times M}(F\times S)$$
where the map $D(M/S)\rightarrow F\times M$ is the relative $\Pi$ and $F\times S\rightarrow F\times M$ is $\on{id}_F\times g$.
\end{definition}

This is an $F^\times$-equivariant manifold equipped with a $F^\times$-equivariant submersion $\Pi$ to $F\times S$ together with an $F^\times$-equivariant section $\sigma$.  Over $F^\times\subset F$ it recovers our original $f$ and $g$ (crossed with $F^\times$), but over $\{0\}\times F$ it identifies with the $g$-pullback of the relative tangent bundle  $g^\ast T(M/S)$ (or, equivalently, the normal bundle of $g:S\rightarrow M$) together with its zero section.

It will be convenient to think of this construction as a functor from the category of manifolds over and under $S$, where the ``over'' map is a submersion, to the category of $F^\times$-equivariant manifolds over and under $F\times S$ where again the ``over'' map is a submersion.  By Theorem \ref{defrepresentable}, this functor commutes with base-change in $S$ and moreover it preserves finite products (which are calculated as fiber products over the given base manifold).

In particular, this functor sends group objects to group objects.  This fits into the following definition.

\begin{definition}
Let $S\in\on{Man}_F$.  A \emph{Lie group over $S$} is a group object in the category of manifolds equipped with a submersion to $S$.  Denote the category of Lie groups over $S$ by $\on{LieGp}_{S}$.  Note that pullback defines a natural functor $\on{LieGp}_S\rightarrow \on{LieGp}_{S'}$ for any map $S'\rightarrow S$.
\end{definition}
  Given a Lie group $G$ over $S$, we can apply the above relative deformation construction, specialized to the section given by the identity section, giving us an $F^\times$-equivariant manifold $D_e(G/S)$ over $F\times S$ equipped with an $F^\times$-equivariant section.  By the commutation with fiber products, the group structure passes through this construction and hence promotes $D_e(G/S)$ to an $F^\times$-equivariant Lie group over $F\times S$, such that the above section is the identity section for the group structure.  In other words:

  \begin{definition}\label{deformlie}
Let $S\in\on{Man}_F$.  We denote by $\on{LieGp}_{(F/F^\times)\times S}$ the category of $F^\times$-equivariant Lie groups over $F\times S$ (the notation is justified by the language of stacks, which will be reviewed later), and by
$$D_e(-/S):\on{LieGp}_{S}\rightarrow \on{LieGp}_{(F/F^\times)\times S}$$
the functor given by the relative deformation to the tangent bundle, specialized along the identity section, as discussed above.  Note that this functor commutes with base-change in $S$.
  \end{definition}

  \begin{remark}
  When restricted to $F^\times\subset F$, this deformation just recovers original Lie group $G\rightarrow S$.  However, over $\{0\}\subset F$ it gives us a certain Lie group structure (over $S$) on the pullback of the relative tangent bundle along the identity section of $G\rightarrow S$, or in other words the relative Lie algebra of $G\to S$.  Thus a priori this construction produces some exotic Lie group structure on the Lie algebra.  However, this Lie group structure is in fact the same as the usual additive Lie group structure one has on any vector bundle, as follows immediately from the Eckmann-Hilton argument.

  \end{remark}

  \begin{remark}

  One can be more precise at the level of Lie algebras: the induced deformation of Lie algebras corresponds to ``scaling the Lie bracket down to $0$''.  But we will not need this more precise remark.
\end{remark}

\section{Etale sheaves on light condensed anima}\label{etalesec}

To discuss Atiyah duality for $p$-adic Lie groups, we want to define a six functor formalism on some kinds of $p$-adic analytic stacks.  But it is a bit inconvenient to try to set up such a six functor formalism directly in analytic stacks.

There are two reasons for this.  First, the category $\on{Man}_F$ does not have all pullbacks (due to transversality issues).  Compatibility of the lower-$!$ functors with pullbacks is \emph{the} crucial aspect of a six functor formalism.  This means that, unless we are willing to make a technical modification to the standard set-up for six functor formalisms (which we are not), we can only encode lower-$!$ functors for a class of maps which is closed under arbitrary pullbacks.  In the case of $\on{Man}_F$, this requires us to restrict to submersions.  Now, for our intended application we do only need lower-$!$ for submersions, but such a restriction is still not desirable.

The second reason $\on{Man}_F$ is inconvenient is that it depends on $F$.  The category of sheaves on an $F$-manifold by definition only depends on the topological space underlying the manifold; the $F$-analytic structure should only be relevant for discussing things like tangent spaces and so on, not for setting up the six functor formalism.  Moreover, there are some very interesting ``global'' relations between the archimedean and non-archimedean six functor formalisms  (such as the reciprocity law in Section \ref{jsec}), and to express these relations it is convenient to have a common framework which both $\on{Man}_{\mathbb{Q}_p}$ and $\on{Man}_\mathbb{R}$ map to.

For this reason, we will take as our basic ``geometric'' objects not $F$-manifolds for some fixed $F$, but rather a certain class of compact Hausdorff spaces.  In fact there are two basic choices to make: one, the class of compact Hausdorff spaces, and two, the Grothendieck topology to put on the resulting full subcategory of compact Hausdorff spaces.  It turns out that many different choices would work equally well for present purposes.  But in recent years with Peter Scholze we have developed a certain specific choice which results in the theory of  \emph{light condensed anima}.  We came to that choice with a different motivation in mind, namely we wanted light condensed anima to serve as ``coefficient objects'' rather than ``geometric objects'', but light condensed anima can and do serve both purposes.

It seems reasonable enough for us to also use this theory of light condensed anima for our present purposes, so that is what we will do.  Thus, in this section we will recall the definition of light condensed anima, and then discuss the formalism of etale sheaves on light condensed anima.

\begin{remark}
Much of the material in this section is from joint work with Peter Scholze.  It is also very closely analogous to, and was inspired by, the theory of etale sheaves on v-stacks developed by Scholze in \cite{ScholzeDiamonds}.
\end{remark}

\begin{definition}
A \emph{light profinite set} is a topological space $S$ which can be written as a sequential inverse limit of finite discrete topological spaces, $S=\varprojlim_n S_n$.  The full subcategory of $\on{Top}$ spanned by the light profinite sets is denoted $\on{ProfSet}^{light}$.
\end{definition}
The full subcategory $\on{ProfSet}^{light}\subset \on{Top}$ is closed under countable limits and finite disjoint unions.  Moreover, if $S\in \on{ProfSet}^{light}$ and $T\subset S$ is any closed subset, then $T\in\on{ProfSet}^{light}$. 

\begin{remark}
While we defined $\on{ProfSet}^{light}$ as a full subcategory of topological spaces, it also embeds as a full subcategory of the Pro-category of finite sets, namely that spanned by the sequential inverse limits of finite sets.
\end{remark}

\begin{definition}
Let $S\in\on{ProfSet}^{light}$.
\begin{enumerate}
\item  Let $\mathcal{U}\subset \on{ProfSet}^{light}_{/S}$ be a sieve over $S$.  We say that $\mathcal{U}$ is a covering sieve if there exist finitely many $\{T_i\rightarrow S\}_{i\in I}$ in $\mathcal{U}$ such that the induced map $\sqcup_{i\in I} T_i\rightarrow S$ is surjective.
\item Let $S_\bullet \rightarrow S$ be a simplicial object in $\on{ProfSet}^{light}_{/S}$.  We say that $S_\bullet \rightarrow S$ is a \emph{hypercover} if for all $n\geq 0$, the map $S_n\rightarrow S_\bullet(\partial \Delta^n)$ is surjective.  (Here $S_\bullet(\partial \Delta^n)$ is calculated in $\on{ProfSet}^{light}_{/S}$.  Thus for $n=0$ the condition is $S_0\twoheadrightarrow S$, for $n=1$ it is $S_1\twoheadrightarrow S_0\times_S S_0$, etc.)
\end{enumerate}
\end{definition}

\begin{example}
Suppose $f:S_0\twoheadrightarrow S$ is a surjective map of light profinite sets.  Then the Cech nerve of $f$ is a hypercover of $S$.
\end{example}

\begin{definition}\label{lightdef}
A \emph{light condensed anima} is a functor $X:(\on{ProfSet}^{light})^{op}\rightarrow\on{An}$ such that the following conditions hold:
\begin{enumerate}
\item If $I$ is a finite set and $(S_i)_{i\in I}$ is an $I$-indexed collection of light profinite sets, then
$$X(\sqcup_{i\in I} S_i)\overset{\sim}{\rightarrow}\prod_{i\in I}X(S_i).$$
\item If $S_\bullet \rightarrow S$ is a hypercover of $S$ in $\on{ProfSet}^{light}$, then
$$X(S)\overset{\sim}{\rightarrow}\varprojlim_{n\in\Delta} X(S_n).$$
\end{enumerate}
The $\infty$-category of light condensed anima is the full subcategory $\on{CondAn}^{light}\subset\on{PSh}(\on{ProfSet}^{light})$ spanned by the $X$ satisfying the above two conditions.
\end{definition}

\begin{remark}
By the criterion of \cite{LurieSAG} Prop.\ A.5.7.2, $\on{CondAn}^{light}$ identifies with the collection of hypercomplete objects in the $\infty$-topos of sheaves of anima on $\on{ProfSet}^{light}$ with respect to the Grothendieck topology specified by the family of covering sieves described in Definition \ref{lightdef} part 1.  In particular, $\on{CondAn}^{light}$ is a hypercomplete $\infty$-topos.
\end{remark}

\begin{remark}
The Grothendieck topology on $\on{ProfSet}^{light}$ is subcanonical, so the Yoneda embedding $\on{ProfSet}^{light}\rightarrow \on{CondAn}^{light}$ is fully faithful.  This follows from the fact that any surjective map of compact Hausdorff spaces is a quotient map.
\end{remark}

The general notion of a \emph{surjective} (or \emph{$0$-connective}) map in an $\infty$-topos (called \emph{effective epimorphism} in \cite{LurieHTT}) unwinds to the following in $\on{CondAn}^{light}$: a map $X\rightarrow Y$ is surjective if and only if for all $S\in\on{ProfSet}^{light}$ and all maps $S\rightarrow Y$, there is a surjective map of light profinite sets $T\twoheadrightarrow S$ such that the composition $T\twoheadrightarrow S\rightarrow Y$ lifts along $X\rightarrow Y$.  From this we deduce the following.

\begin{lemma}
The $\infty$-topos $\on{CondAn}^{light}$ is \emph{replete} in the sense of \cite{BhattProEtale}: for any tower $\ldots\twoheadrightarrow X_1\twoheadrightarrow X_0$ in $\on{CondAn}^{light}$ with surjective transition maps, the map
$$\varprojlim_n X_n\rightarrow X_0$$
is surjective.
\end{lemma}
\begin{proof}
This follows from the corresponding fact for light profinite sets, which is obvious.
\end{proof}

In \cite{BhattProEtale} and \cite{MondalReplete}, it is shown that hypercomplete replete $\infty$-topoi have favorable exactness properties, at least assuming a defining 1-site.  We review this here, partly for later use and partly to simplify and generalize some arguments from \cite{MondalReplete}.

\begin{lemma}\label{replete}
Let $\mathcal{C}$ be an $\infty$-topos.  For the following conditions on $\mathcal{C}$, we have 
1 $\Leftrightarrow$ 2 $\Rightarrow$ 3 $\Leftrightarrow$ 4 $\Rightarrow$ 5, and 5 $\Rightarrow$ 6 and 7.
\begin{enumerate}
\item $\mathcal{C}$ is replete: the inverse limit of a tower of surjections maps surjectively to each term.
\item Let $(f_n):(X_n)\rightarrow (Y_n)$ be a map of towers in $\mathcal{C}$.  Suppose that $f_0$ is surjective, and that for all $n\geq 1$, the map $X_n\rightarrow Y_n\times_{Y_{n-1}} X_{n-1}$ is surjective.  Then $\varprojlim_n X_n\rightarrow \varprojlim_n Y_n$ is surjective.
\item A countable product of surjective maps in $\mathcal{C}$ is surjective.
\item A countable product of $d$-connective maps in $\mathcal{C}$ is $d$-connective, for all $d\geq 0$.
\item Let $(f_n):(X_n)\rightarrow (Y_n)$ be a map of towers in $\mathcal{C}$, and let $d\geq 0$.  Suppose that $f_n$ is $d+1$-connective for all $n\geq 0$.  Then $\varprojlim_n X_n\rightarrow \varprojlim_n Y_n$ is $d$-connective.
\item For all $X\in\mathcal{C}$, the map $X\rightarrow \varprojlim_n \tau_{\leq n}X$ is $\infty$-connective.
\item If $(X_n)_n$ is a tower in $\mathcal{C}$ such that the map $X_{n+1}\rightarrow X_n$ exhibits $X_n$ as $\tau_{\leq n}X_{n+1}$ for all $n\geq 0$, then the projection $\varprojlim_n X_n\rightarrow X_d$ exhibits $X_d$ as $\tau_{\leq d}\varprojlim_n X_n$ for all $d\geq 0$.
\end{enumerate}

In particular, any hypercomplete replete $\infty$-topos is \emph{Postnikov complete}: the functor
$$\mathcal{C}\overset{\sim}{\rightarrow}\varprojlim_n \mathcal{C}_{\leq n},$$
sending $X\in\mathcal{C}$ to its Postnikov tower, is an equivalence.
\end{lemma}
\begin{proof}
Assuming 1, consider 2.  Set $X_\infty=\varprojlim_n X_n$ and $Y_\infty=\varprojlim_n Y_n$.  We can identify $X_\infty\rightarrow Y_\infty$ with the inverse limit of the tower formed by the
$$X_n\times_{Y_n}Y_\infty\rightarrow Y_\infty,$$
so by repleteness it suffices to show that the transition maps in this tower are surjective.  But each transition map is a base-change of a map we've assumed surjective, hence is surjective.  Thus 1 $\Rightarrow 2$, and the converse follows by considering the case $Y_n=\ast$ for all $n$.

Now consider 3.  A finite product of surjections is always a surjection, because surjections are closed under base-change and composition.  Then assuming 2 we can pass to the inverse limit and get 3, so 2 $\Rightarrow$ 3.

Certainly 4 $\Rightarrow$ 3 because 3 is the special case $d=0$ of 4.  The converse follows by induction because for $d\geq 1$, a map $X\rightarrow Y$ in an $\infty$-topos is $d$-connective if and only if it is surjective and the diagonal map $X\rightarrow X\times_Y X$ is $d-1$-connective.  Thus $3\Leftrightarrow 4$.

Now consider 5.  Recall that in a general $\infty$-category we have the Milnor square for countable inverse limits:
$$\varprojlim_n Z_n = \prod_nZ_n\times_{(\prod_n Z_n)\times (\prod_{n}Z_{n})}\prod_n Z_n.$$
Here the first map $\prod_n Z_n\rightarrow (\prod_n Z_n)\times(\prod_n Z_n)$ is the diagonal and the second map $\prod_n Z_n\rightarrow (\prod_n Z_n)\times(\prod_n Z_n)$ is the identity on the first coordinate and the shift map on the second coordinate.  Thus, assuming 4, we deduce 5 from the fact that a pullback of $d+1$-connective maps is $d$-connective.  Hence 4 $\Rightarrow$ 5.

Next, consider 6.  In the map of towers from the constant tower on $X$ to the tower $(\tau_{\leq n}X)_n$, the maps in the tower are eventually as connective as we'd like.  Thus 5 $\Rightarrow$ 6.

Similarly, 7 also follows from 5: indeed, considering the natural map of towers from the tower $(X_n)_n$ to the tower $(\tau_{\leq d+1}X_n)_n$ we deduce from 3 that
$$\varprojlim_n X_n\rightarrow \varprojlim_n \tau_{\leq d+1}X_n$$
is $d+1$-connective, hence
$$\tau_{\leq d}\varprojlim X_n = \tau_{\leq d}\varprojlim_n \tau_{\leq d+1}X_n.$$
But the target tower is eventually constant with value $X_{d+1}$ and $\tau_{\leq d}X_{d+1}=X_d$, yielding the claim.

Finally, the remark about Postnikov completeness follows because unwinding the unit and counit for the adjunction between $\mathcal{C}$ and $\varprojlim_n \mathcal{C}_{\leq n}$ shows that $\mathcal{C}$ is Postnikov complete if and only if 7 holds, and the map $X\rightarrow\varprojlim_n \tau_{\leq n}X$ is an isomorphism for all $X\in\mathcal{C}$.  But in a hypercomplete $\infty$-topos that last condition is equivalent to 6.
\end{proof}

\begin{theorem}
The $\infty$-topos $\on{CondAn}^{light}$ is Postnikov-complete.  More generally, for any $X\in\on{CondAn}^{light}$, the $\infty$-topos $\on{CondAn}^{light}_{/X}$ is Postnikov-complete.
\end{theorem}
\begin{proof}
Note that repleteness and hypercompleteness pass to slice $\infty$-topoi, so this follows from Lemma \ref{replete}.
\end{proof}

For a light profinite set $S$, we can consider the usual $\infty$-category $\on{Sh}(S)$ of sheaves of anima on the topological space $S$.  A more convenient presentation of the same $\infty$-topos is as the sheaves on the site $\on{COp}_S$ of compact open subsets of $S$ with respect to the Grothendieck topology of finite disjoint unions.  The equivalence follows from the general criterion of \cite{HoyoisLefschetz} Lemma C.3, since the compact open subsets form a basis for the topology closed under finite intersection.  Note that this presentation in terms of compact open subsets makes it plain that $\on{Sh}(S)$ is Postnikov complete of homotopy dimension $0$ (indeed, taking sections over compact open subsets commutes Postnikov truncation, because the sheaf condition just involves finite products).

\begin{remark}
Equivalently, if we write $S=\varprojlim_n S_n$ as a countable limit of finite sets, then
$$\on{Sh}(S)\overset{\sim}{\rightarrow}\varprojlim_n \on{Sh}(S_n)=\varprojlim_n \prod_{S_n} \on{An}$$
via the pushforward maps.  Indeed, the site of compact open subsets of $S$ is the filtered colimit of the site of subsets of $S_n$, and these are finitary sites, so this follows from \cite{ClausenHyper} Lemma 3.3.
\end{remark}

We will now embed $\on{Sh}(S)$ as a full subcategory of $\on{CondAn}^{light}_{/S}$.  For this we will go through the category of \emph{finite etale} maps to $S$, for $S\in\on{ProfSet}^{light}$.  We can discuss these from 3 equivalent perspectives: topological, sheaf theoretic, or combinatorial.
\begin{enumerate}
\item Topologically, a finite etale map $T\rightarrow S$ is a map of topological spaces which, locally on $S$, is isomorphic to one of the form $F\times S\rightarrow S$ where $F$ is a finite set.  In other words, it's a finite covering map.
\item Sheaf theoretically, a finite etale map $T\rightarrow S$, being in particular a local homeomorphism, can be encoded in terms of its sheaf of sections.  The sheaves arising in this way are exactly those which are locally isomorphic to the constant sheaf on a finite set.
\item Combinatorially, a finite etale map $T\rightarrow S$ is a map in $\on{ProfSet}^{light}=\on{Pro}^{\aleph_0}(\on{fSets})$ which is a pullback of a map between finite sets.
\end{enumerate}

For future use, we will axiomatize the combinatorial perspective.  Given an $\infty$-category $\mathcal{C}$, denote by $\on{Pro}^{\aleph_1}(\mathcal{C})$ the full subcategory of the Pro-category of $\mathcal{C}$ spanned by the countable inverse systems of objects of $\mathcal{C}$.

\begin{lemma}\label{prononsense}
Let $\mathcal{C}$ be a small $\infty$-category admitting all pullbacks.  Define a map $f:X\rightarrow Y$ in $\on{Pro}^{\aleph_1}(\mathcal{C})$ to be a \emph{$\mathcal{C}$-map} provided it is a pullback of a map between objects in $\mathcal{C}$, and denote the full subcategory of $\mathcal{C}$-maps to $Y\in \on{Pro}^{\aleph_1}(\mathcal{C})$ by $\mathcal{C}_{/Y}$ (very abusively, as when $X\rightarrow Y$ is a $\mathcal{C}$-map, $X$ need not lie in $\mathcal{C}$).  Then:
\begin{enumerate}
\item The functor $(\on{Pro}^{\aleph_1}(\mathcal{C}))^{op}\rightarrow\mathcal{C}at_\infty$, sending $Y$ to $\mathcal{C}_{/Y}$ (with pullback functoriality), preserves countable filtered colimits.
\item For $Y\in\on{Pro}^{\aleph_1}(\mathcal{C})$, the full subcategory $\mathcal{C}_{/Y}\subset \on{Pro}^{\aleph_1}(\mathcal{C})_{/Y}$  is closed under pullbacks.
\item For $Y\in\on{Pro}^{\aleph_1}(\mathcal{C})$, we have $\on{Pro}^{\aleph_1}(\mathcal{C})_{/Y} = \on{Pro}^{\aleph_1}(\mathcal{C}_{/Y})$.
\end{enumerate}
\end{lemma}
\begin{proof}
For 1, it is equivalent to say that if $(X_n)_n$ is a tower of objects in $\mathcal{C}$ with limit $X \in \on{Pro}^{\aleph_1}(\mathcal{C})$, then the comparison functor
$$\varinjlim_n \mathcal{C}_{/X_n}\overset{\sim}{\rightarrow} \mathcal{C}_{/X}$$
is an equivalence in $\mathcal{C}at_\infty$.  This functor is essentially surjective by definition, and fully faithfulness follows from the fact that mapping anima in fitered colimits on $\mathcal{C}at_\infty$ are calculated by filtered colimits of mapping anima of the terms.
Claim 2 follows from claim 1, because we can lift a finite diagram to a finite stage in the colimit, and the transition functors all preserve pullbacks.
For claim 3, it suffices to show that we can write every map $X\rightarrow Y$ in $\on{Pro}^{\aleph_1}(\mathcal{C})$ as a sequential filtered limit of a tower $\mathcal{C}$-maps $(X_n\rightarrow Y)_n$, such that for any other $\mathcal{C}$-map $Y'\rightarrow Y$, we have
$$\varinjlim_n \on{Map}_Y(X_n,Y')\overset{\sim}{\rightarrow} \on{Map}_Y(X,Y').$$
For this, write $X=\varprojlim_n X'_n$ and $Y=\varprojlim Y_n$ as limits of objects in $\mathcal{C}$.  The map $f$ comes from a map of pro-objects $(X'_n)\rightarrow (Y_n)$, but by reindexing we can assume it is even realized by a map of towers $(X'_n)\rightarrow (Y_n)$.  Then we set $X_n = X'_n\times_{Y_n} Y$.  Clearly $X=\varprojlim_n X_n$, and the second claim is trivial to verify after descending $Y'\rightarrow Y$ to some $Y'_n\rightarrow Y_n$.
\end{proof}

\begin{lemma}\label{etaleonprof}
For $S\in\on{ProfSet}^{light}$, define a functor
$$\delta_S:\on{Sh}(S)\rightarrow \on{CondAn}^{light}_{/S}$$
as the left Kan extension of the functor $\on{COp}_S\rightarrow \on{CondAn}^{light}_{/S}$ which sends $T\subset S$ to the inclusion map $T\rightarrow S$.  Then:
\begin{enumerate}
\item $\delta_S$ preserves finite limits and arbitrary colimits, hence gives a map in $\mathcal{T}op^L$;
\item $\delta_S$ is fully faithful;
\item the essential image of $\delta_S$ consists of those light condensed anima $X\rightarrow S$ over $S$ such that the corresponding functor $X:(\on{ProfSet}^{light}_{/S})^{op}\rightarrow \on{An}$ (sending $T\rightarrow S$ to $\on{Hom}_S(T,X)$) preserves countable filtered colimits.
\end{enumerate}
\end{lemma}

\begin{proof}
Consider the following $\infty$-categories:
\begin{enumerate}
\item The $\infty$-category of functors $\on{COp}_S^{op}\rightarrow\on{An}$ satisfying the sheaf condition for finite disjoint unions, i.e.\ $\on{Sh}(S)$.
\item The $\infty$-category of functors $(\on{fEt}_{/S})^{op}\rightarrow\on{An}$ satisfying the sheaf condition for finite disjoint unions.
\item The $\infty$-category of functors $(\on{ProfSet}^{light}_{/S})^{op}\rightarrow\on{An}$ which preserve countable filtered colimits and satisfy the sheaf condition for finite disjoint unions.
\item The $\infty$-category of functors $(\on{ProfSet}^{light}_{/S})^{op}\rightarrow\on{An}$ which preserve countable filtered colimits, satisfy the sheaf condition for finite disjoint unions and send hypercovers to limit diagrams, i.e.\ the full subcategory of $\on{CondAn}^{light}_{/S}$ which part 3 of this Lemma claims is the essential image of $\delta_S$.
\end{enumerate}
We claim that restriction and right Kan extension identify 1 and 2, that restriction and left Kan extension identify 2 and 3, and that 3 and 4 are the same.  The claim about 1 and 2 follows from \cite{HoyoisLefschetz} Lemma C.3.  For the claim about 2 and 3, Lemma \ref{prononsense} part 2 shows the same without the conditions of satisfying the sheaf conditions, so it suffices to see that that a countable filtered colimit-preserving functor $(\on{ProfSet}^{light}_{/S})^{op}\rightarrow\on{An}$ satisfies the sheaf condition for finite disjoint unions if and only if its restriction to $(\on{fEt}_{/S})^{op}$ does, which is obvious because finite products commute with filtered colimits.  For the claim about 3 and 4, we need to see that every $X:(\on{ProfSet}^{light}_{/S})^{op}\rightarrow\on{An}$ as in 3 sends hypercovers to limit diagrams.  Note that the collection of $X$ as in 3 is closed under sectionwise Postnikov truncation, because Postnikov truncation is preserved by finite products and filtered colimits.  Thus we can reduce to truncated $X$.  In that setting it suffices to consider hypercovers given by the Cech nerve of a single surjective map $T'\twoheadrightarrow T$ (over $S$).  We can write this as the limit of a tower of surjective maps $T'_n\twoheadrightarrow T_n$ in $\on{fEt}_{/S}$.  Since limits over $\Delta$ are finite limits in the truncated setting, and finite limits commute with filtered colimits, we thereby reduce to the case of a finite etale surjective map of light profinite sets, where the claim is trivial because such a map admits a section.

To show that $\delta_S$ is a map in $\mathcal{T}op^L$ it suffices to see that $h_T\mapsto (T\rightarrow S)$ preserves finite limits and finite disjoint unions, which is obvious.  To prove the rest of the claims it suffices to show that $\delta_S$ identifies with the composition of the above identifications $1-2-3-4$.  For this, by transitivity of left Kan extensions, we need to check that $\delta_S$, applied to the sheaf of sections of a finite etale map $T\rightarrow S$, recovers $T\rightarrow S$. This is true for $T\subset S$ a compact open subset by definition, and it follows in general because $\delta_S$ preserves finite disjoint unions.
\end{proof}

\begin{theorem}\label{etale}
Let $f:X\to Y $ be a map in $\on{CondAn}^{light}$.  The following properties are equivalent.
\begin{enumerate}
\item For all towers $(S_n)_n$ of light profinite sets with limit $S=\varprojlim_n S_n$, the map $f$ has the unique right lifting property with respect to the map (of pro-objects) $S\rightarrow \underset{n}{``\varprojlim"} S_n$. In other words, $\varinjlim_n X(S_n)\overset{\sim}{\rightarrow} X(S)\times_{Y(S)} \left(\varinjlim_n Y(S_n)\right)$.
\item The functor $(\on{ProfSet}^{light}_{/Y})^{op}\rightarrow \on{An}$ represented by $f:X\to Y$ preserves countable filtered colimits.
\item For all $S\in\on{ProfSet}^{light}$ and all maps $S\rightarrow Y$, the pullback $X\times_YS \rightarrow S$ lies in the essential image of the fully faithful functor $\delta_S:\on{Sh}(S)\rightarrow \on{CondAn}^{light}_{/S}$ from Lemma \ref{etaleonprof}.
\end{enumerate}
\end{theorem}
\begin{proof}
The equivalence of 1 and 2 is purely formal, and the equivalence of 2 and 3 is a translation of Lemma \ref{etaleonprof}.
\end{proof}

\begin{definition}
A map $X\rightarrow Y$ in $\on{CondAn}^{light}$ is called \emph{etale} if it satisfies the equivalent properties listed in Theorem \ref{etale}.
\end{definition}

\begin{lemma}\label{etaleproperties}
\begin{enumerate}
\item The collection of etale maps of light condensed anima is closed under pullbacks and compositions, and contains all isomorphisms.  It is also closed under passage to diagonals; or, equivalently (given the previous properties), if $X\rightarrow Y\rightarrow Z$ are maps of light condensed anima such that $X\rightarrow Z$ and $Y\rightarrow Z$ are etale, then $X\rightarrow Y$ is etale.
\item For $S\in\on{CondAn}^{light}$, the full subcategory of $\on{CondAn}^{light}_{/S}$ spanned by the etale maps is closed under finite limits and arbitrary colimits.
\item A map $X\rightarrow Y$ of light condensed anima is etale if and only if its $n^{th}$ relative Postnikov truncation is etale for all $n$.
\item If $\{S_i\rightarrow S\}_{i\in I}$ is a collection of maps of light condensed anima such that $\sqcup_{i\in I} S_i\rightarrow S$ is surjective, then a map $X\rightarrow S$ is etale if and only if the pullback $X\times_S S_i\rightarrow S_i$ is etale for all $i\in I$.
\end{enumerate}
\end{lemma}
\begin{proof}
Property 1 is a general fact about a collection of maps defined by a unique right lifting property.  For property 2, by part 3 of Theorem \ref{etale} we can reduce to the case $S\in\on{ProfSet}^{light}$.  There it follows from Lemma \ref{etaleonprof}.  For property 3, the closure under Postnikov truncation follows from property 2 (in any $\infty$-topos, $\tau_{\leq n}X$ is the colimit over $\Delta^{op}$ of the $(n+1)^{st}$ coskeleton of the constant simplicial object on $X$).  For the converse, by Theorem \ref{etale} part 3 we can reduce to $S\in\on{ProfSet}^{light}$.  In that case, the claim follows from the fact that a map in $\mathcal{T}op^L$ between Postnikov-complete $\infty$-topoi preserves inverse limits of Postnikov towers. (Namely, we apply this to $\delta_S:\on{Sh}(S)\rightarrow\on{CondAn}^{light}_{/S}$.)

Finally, for property 4, we can reduce to the case where $S,S_i\in\on{ProfSet}^{light}$ for all $i$, and $I$ is finite.  We want to show that for $X\in\on{CondAn}^{light}_{/S}$, the associated $X:(\on{ProfSet}^{light}_{/S})^{op}\rightarrow \on{An}$ preserves countable filtered colimits if and only if it does so on restriction to light profinite sets over $S_i$ for all $i\in I$.  This is a problem of commuting countable filtered colimits with limits over $\Delta$.  However, by Property 3 we can reduce to the case where $X$ is truncated, in which case there is no problem because limits over $\Delta$ in $\on{An}_{\leq n}$ are finite limits.
\end{proof}

\begin{remark}\label{checkonstalks}
A useful property is the following.  If $S\in\on{CondAn}^{light}$ and $X\to S$ and $Y\to S$ are etale maps to $S$, then a map $X\to Y$ over $S$ is an isomorphism if and only if it is so on pullback along $\ast \to S$ for any map $\ast \to S$ from a point.  Indeed, we can reduce to $S$ light profinite and then this follows because we can check isomorphisms of sheaves on stalks (recall that we are homotopy dimension zero, so this reduces to the classical case of sheaves of sets).
\end{remark}

Sometimes, instead of thinking of maps $X\rightarrow S$ in $\on{CondAn}^{light}$ as geometric/topological objects over $S$, we'd rather think of them as a kind of sheaves on $S$ . For such a purpose, let us adopt some abbreviated notation, borrowed from \cite{ScholzeDiamonds}.

\begin{definition}
For $S\in\on{CondAn}^{light}$, let us denote by
$$\on{Sh}_v(S):= \on{CondAn}^{light}_{/S},$$
and by $\on{Sh}_{et}(S)\subset\on{Sh}_v(S)$ the full subcategory of etale maps to $S$.
\end{definition}

Note that $v$-sheaves have the obvious pullback functoriality, namely $\on{Sh}_v(-):(\on{CondAn}^{light})^{op}\rightarrow \mathcal{T}op^L$.  This functor preserves limits by descent applied to the $\infty$-topos $\on{CondAn}^{light}$.  The full subcategory $\on{Sh}_{et}(-)\subset\on{Sh}_v(-)$ is closed under pullback and is of local nature by Lemma \ref{etaleproperties}, so $\on{Sh}_{et}(-):(\on{CondAn}^{light})^{op}\rightarrow \mathcal{T}op^L$ inherits the functoriality and the limit-preservation.  It follows that $\on{Sh}_{et}(-)$ can also be described as the right Kan extension of its restriction to $(\on{ProfSet}^{light})^{op}$, where by Lemma \ref{etaleonprof} it identifies with the usual theory of sheaves of anima on the underlying topological space.

Thus, to understand $\on{Sh}_{et}(S)$ for a condensed anima $S$, one can write $S$ as a colimit of light profinite sets and try to identify the resulting limit of sheaf categories.  Let us carry this procedure out in three classes of examples:
\begin{enumerate}
\item The ``discrete'' case: the case where $S$ lies in the essential image of the unique map in $\mathcal{T}op^L$
$$\on{An}\to\on{CondAn}^{light}.$$
\item The ``topological'' case: the case where $S$ lies in the essential image of the functor
$$\on{Top}\to \on{CondAn}^{light}$$
sending $K\mapsto \ul{K}: S\mapsto \on{Cont}(S,K).$
\item The ``profinite'' case: the case where $S$ is a countable inverse limit of (discrete) $\pi$-finite anima.
\end{enumerate}

\noindent\textit{The discrete case:}\label{etaleonetale}\\

We start with the discrete case, which is quite simple.  The special case $S=\ast$ of Lemma \ref{etaleonprof} shows that the unique map in $\mathcal{T}op^L$
$$\delta: \on{An}\rightarrow \on{CondAn}$$
is fully faithful with essential image those light condensed anima which are \emph{discrete}, i.e.\ etale over $\ast$.  By the 2 out of 3 property of etale maps, it follows that for $X\in\on{An}$,
$$\on{Sh}_{et}(\delta(X)) = \on{An}_{/X} = \on{Fun}(X,\on{An}),$$
the usual $\infty$-category of ``local systems of anima over $X$''.\\

\noindent\textit{The topological case:}\\

Now we consider the topological case.  There is a functor $\on{Top}\rightarrow \on{CondAn}^{light}$, sending a topological space $K$ to the light condensed anima $\underline{K}$ defined by
$$\underline{K}(S) = \on{Cont}(S,K).$$
Note that this satisfies the sheaf condition because surjective maps of compact Hausdorff spaces are quotient maps.  We are interested in relating etale sheaves on $\underline{K}$ to usual sheaves on $K$, extending what we already know in the case $K\in\on{ProfSet}^{light}$.

\begin{lemma}\label{etaleontopological}
Let $K$ be a topological space.
\begin{enumerate}
\item Suppose $\{U_i\}_{i\in I}$ is an open cover of $K$.  Then $\{\underline{U_i}\rightarrow \underline{S}\}_{i\in I}$ is a cover in $\on{CondAn}^{light}$.
\item Suppose $f:K'\rightarrow K$ is a local homeomorphism.  Then $\underline{K'}\rightarrow \underline{K}$ is etale.
\item There is a unique map $\on{Sh}(K)\rightarrow \on{Sh}_{et}(\underline{K})$ in $\mathcal{T}op^L$ which sends the Yoneda image of an open subset $U\subset K$ to the inclusion $\underline{U}\rightarrow \underline{K}$.
\item Suppose every point $x\in K$ admits a neighborhood $L$ which is a second-countable compact Hausdorff space.  Then this functor $\on{Sh}(K)\rightarrow \on{Sh}_{et}(\underline{K})$ realizes the target as the Postnikov-completion of the source.
\end{enumerate}
\end{lemma}
\begin{proof}
For 1, we are formally reduced to the case where $K\in\on{ProfSet}^{light}$.  Then $\{U_i\}_{i\in I}$ admits a refinement which is a disjoint cover by compact open subsets, proving the claim.  For 2, because of 1 we can reduce to the case where $f$ is an open inclusion, or even a compact open inclusion of light profinite sets.  Then the claim follows by definition of the functor $\delta_S$ in Lemma \ref{etaleonprof}.  Claim 3 follows formally from 1 and 2, given that $U\mapsto \underline{U}$ obviously preseves finite limits.  Finally, for claim 4, by descent we can reduce to the case when $K$ is an open subset of an $L$ as in the statement.  Then $K$ is Hausdorff, so the collection of compact Hausdorff subsets is closed under finite intersections, so by descent again we can reduce to the case $K=L$.  Choose a light profinite set $S$ with a surjection $f:S\twoheadrightarrow K$ (e.g., take the spectral space of Example \ref{countableexamples} part 1 and equip it with the constructible topology).  It suffices to show that $\on{Sh}(K)_{\leq d}$ is the limit of $\on{Sh}(-)_{\leq n}$ over the Cech nerve of $f$ for all $n\geq 0$.  This is something that can be checked using the cohomological descent formalism set up in \cite{LurieHA} 4.7.5.  Namely, by \cite{LurieHTT} 7.3 the proper base change theorem holds for sheaves of anima on compact Hausdorff spaces, so it suffices to show that if $g:X\rightarrow Y$ is a surjective map of compact Hausdorff spaces, then $g^\ast:\on{Sh}(Y)_{\leq n}\rightarrow \on{Sh}(X)_{\leq n}$ is conservative and commutes with limits over $\Delta$.  Conservativity follows from the fact that isomorphisms in $\on{Sh}(Y)_{\leq n}$ are tested on stalks, and commutation with limits over $\Delta$ follows because we are truncated, so that limits over $\Delta$ are finite limits hence commute with pullbacks.
\end{proof}

\begin{remark}
We recall from \cite{LurieHTT} Thm.\ 7.2.3.6 that if $K$ is a compact Hausdorff space of finite covering dimension, then $\on{Sh}(K)$ identifies with its Postnikov-completion.  Thus, if $K$ is additionally assumed second-countable, then we get $\on{Sh}(K)\simeq \on{Sh}_{et}(\underline{K})$. See also Proposition \ref{chauscohdim} for a similar statement with spectrum coefficients under a weaker hypothesis.\end{remark}

\noindent\textit{The profinite case}:\label{profiniteanima}\\

Now we turn to the profinite setting.  This was notably considered by Barwick-Haine, \cite{BarwickPyknotic} Example 3.3.10.  We start with the unique map in $\mathcal{T}op^L$
$$\delta: \on{An}\rightarrow \on{CondAn}^{light}$$
considered above.  Then we restrict to the full subcategory $\on{An}_{\pi}\subset\on{An}$ of $\pi$-finite anima (those anima $X$ for which $\pi_i(X,x)$ is finite for all $x\in X$ and $i\in\mathbb{N}$, and $\ast$ for $i$ sufficiently large).  Then we formally extend by countable filtered inverse limits to get a countable-limit-preserving functor
$$\widehat{\delta}:\on{Pro}^{\aleph_1}(\on{An}_\pi)\rightarrow \on{CondAn}^{light}.$$
Barwick-Haine proved the theorem that this functor is fully faithful.  Actually, they proved not quite this statement, but its ``pyknotic" variant.  However the light condensed context is even easier to handle because countable pro-systems are simpler than arbitrary pro-systems.  Just to be safe we'll include a proof of this light variant: it follows from the following more general claim (which we'll generalize even further below).

\begin{lemma}\label{mapoutofprof}
Suppose $(X_n)_n$ is a tower of $\pi$-finite anima and that $A\in\on{An}$ is an anima such that for each $x\in \pi_0A$, the corresponding component $A_x\subset A$ is $d$-truncated for some $d\in\mathbb{N}$ (possibly depending on $x$).  Then if $X=\varprojlim_n X_n$ is the limit in $\on{CondAn}^{light}$, we have
$$\varinjlim_n \on{Map}(X_n,A)\overset{\sim}{\rightarrow}\on{Map}_{\on{CondAn}^{light}}(X,A).$$
Here we implicitly use the fully faithful embedding $\delta:\on{An}\rightarrow\on{CondAn}^{light}$.
\end{lemma}
\begin{proof}
By the Lemma which follows, every tower of $\pi$-finite anima admits a termwise surjective map from a tower of finite sets, such that the induced map on inverse limits (in $\on{CondAn}^{light}$) is also surjective.  It follows that any tower $(X_n)$ of $\pi$-finite anima admits a hypercover by towers of finite sets, such that on inverse limits we get a hypercover in $\on{CondAn}^{light}$.

Now, we know the lemma in the case of towers of finite sets (it is a special case of the discussion of etale sheaves in Theorem \ref{etale}), so by this hypercover we reduce to the question of commuting a filtered colimit by a limit over $\Delta$.  But in the truncated context this is no problem as limits over $\Delta$ are finite.  This handles the case of truncated $A$.  In particular, the lemma holds for sets $A$.  Then applying this to the set of components of our $A$, we deduce that $\on{Map}_{\on{CondAn}^{light}}(X,-)$ commutes with the filtered colimit writing $A$ as the union of its sub-anima with finitely many components, letting us reduce to the truncated case, already handled.
\end{proof}

We used the following simple lemma in the proof.

\begin{lemma}\label{surjectfromprofset}
Suppose $(X_n)$ is a tower of $\pi$-finite anima.  Then there exists a tower of finite sets $(S_n)$ and a map of towers $(S_n)\to (X_n)$ such that $S_0\to X_0$ is surjective and
$$S_{n+1} \to S_n\times_{X_n} X_{n+1}$$
is surjective for all $n\geq 0$.

Moreover, for such a map of towers we have:
\begin{enumerate}
\item Each $f_n:S_n\to X_n$ is surjective.
\item The induced map $\varprojlim_n S_n \to \varprojlim_n X_n$ in $\on{CondAn}^{light}$ is surjective.
\end{enumerate}
\end{lemma}
\begin{proof}
Every $\pi$-finite anima has finite $\pi_0$, hence admits a surjection from a finite set.  Moreover the class of $\pi$-finite anima is closed under pullbacks.  Thus we can build $(S_n)$ inductively in the obvious way.

Claim 1 follows from the fact that surjections are closed under base change and composition.  Claim 2 follows from repleteness, see Lemma \ref{replete}.
\end{proof}

\begin{example}
Suppose we take $A\in\on{An}_\pi$.  Then Lemma \ref{prononsense} implies that $\on{Pro}^{\aleph_0}(\on{An}_\pi)\rightarrow \on{CondAn}^{light}$ is fully faithful.  Call the essential image the $\infty$-category $\on{ProfAn}^{light}$ of \emph{light profinite anima}.  Note that the full subcategory of light profinite anima is closed under countable limits and finite disjoint unions.
\end{example}

We can also use Lemma \ref{mapoutofprof} to study the notion of \emph{$\pi$-finite etale} map of light condensed anima.  First, recall that a $\pi$-finite map of usual anima is a map $X\to Y$ in $\on{An}$ such that each fiber is $\pi$-finite etale, or equivalently, it is locally on $Y$ (with respect to the canonical Grothendieck topology on $\on{An}$) isomorphic to projection off some $\pi$-finite anima.  Then $\pi$-finite maps with target $Y\in \on{An}$ are classified by maps in $\on{An}$
$$Y\to \on{An}^{\simeq}_{\pi},$$
where $\on{An}^{\simeq}_{\pi}$ is the groupoid core of the $\infty$-category of $\pi$-finite anima.

\begin{definition}\label{pifiniteetale}
Let $Y\in\on{CondAn}^{light}$.  Say that a map $X\to Y$ in $\on{CondAn}^{light}$ is \emph{$\pi$-finite etale} if and only if it is, locally on $Y$ in the light condensed Grothendieck topology, isomorphic to projection off some $\pi$-finite anima.
\end{definition}

In particular, by descent it follows that every $\pi$-finite etale map is indeed etale in the sense of Theorem \ref{etale}.

\begin{lemma}\label{pifiniteproperties}
\begin{enumerate}
\item A composition or base-change of $\pi$-finite etale maps is $\pi$-finite etale.  Moreover the class of $\pi$-finite etale maps is closed under taking the diagonal.
\item For $Y\in \on{CondAn}^{light}$, $\pi$-finite etale maps to $Y$ are classified by maps
$$Y\to \on{An}_{\pi}^\simeq$$
in $\on{CondAn}^{light}$.  (As usual, we implicitly use the fully faithful functor $\delta:\on{An}\to\on{CondAn}^{light}$ from \ref{etaleonprof}.)
\item If $Y\in \on{ProfAn}^{light}$ , then a map $X\to Y$ in $\on{CondAn}^{light}$ is $\pi$-finite etale if and only if it is a pullback of a map between $\pi$-finite anima (in which case $X\in\on{ProfAn}^{light}$ as well).
\item If again $Y\in \on{ProfAn}^{light}$, then $\on{ProfAn}^{light}_{/Y}$ identifies with the countable pro-category of the $\infty$-category of $\pi$-finite etale maps to $Y$.
\end{enumerate}
\end{lemma}
\begin{proof}
For 1, the claims about base-change and closure under passage to the diagonal are obvious by descent.  We will return to closure under composition after proving Claim 3.

Claim 2 (as well as Definition \ref{pifiniteetale}) actually works in any $\infty$-topos.  Namely if $\mc{C}$ is an $\infty$-topos, then the anima of global sections of the constant sheaf of anima with value $\on{An}^\simeq_\pi$ on $\mc{C}$ identifies with the anima of $\pi$-finite etale maps $X\to \ast$ in $\mc{C}$ (those maps which locally are isomorphic to projection off some $\pi$-finite anima).  Indeed by descent theory $\pi$-finite etale maps $X\to \ast$ are classified by global sections of the object
$$\bigsqcup_{[K]} B\ul{\on{Aut}}(\underline{K})$$
of $\mc{C}$.  Here $K$ runs over a set of representatives for the isomorphism classes of $\pi$-finite anima, $\underline{K}$ denotes the constant sheaf on $K$, and $\ul{\on{Aut}}$ means the internal automorphism group.  In this language what we need to see is that $\ul{\on{Aut}}(\ul{K}) = \ul{\on{Aut}(K)}$, i.e.\ that the sheaf of automorphisms of the constant sheaf on $K$ identifies with the constant sheaf on the anima of automorphisms of $K$, when $K$ is $\pi$-finite (caution that the statement is false e.g.\ for an infinite discrete set $K$).  But this follows from the fact that $X\mapsto \underline{X}$, being a special case of pullback map of $\infty$-topoi, preserves finite limits and coproducts.

For 3, this follows from 2 and Lemma \ref{mapoutofprof}, noting that $A=\on{An}_\pi^\simeq$ has the property that each component is $\pi$-finite. As for 4, it follows directly from 3 and Lemma \ref{prononsense}.

What remains is the claim (from 1) that $\pi$-finite etale maps are closed under composition.  Assume given $X\to Y \to Z$ where both maps are $\pi$-finite etale.  Working locally we can assume $Z$ is a light profinite set.  By 3, the map $Y\to Z$ is pulled back from a map of $\pi$-finite anima, and then so is $X\to Y$, and in this way we reduce to the setting of $\on{An}$ where it is clear from the long exact sequence of homotopy groups.
\end{proof}

The following lemma will also be convenient.

\begin{lemma}\label{surjtower}
Let $X\in\on{ProfAn}^{light}$.  Then there is a tower $(X_n)_n$ of $\pi$-finite anima and an isomorphism $X\simeq \varprojlim_n X_n$ such that $X\rightarrow X_n$ is surjective for all $n$.
\end{lemma}
\begin{proof}
Start with any tower $(X'_n)$ of $\pi$-finite anima with $X\simeq\varprojlim_n X'_n$.  Then set $X_n = \on{im}(X\rightarrow X'_n)$.  Here the image is in the sense of the $\infty$-topos $\on{CondAn}^{light}$.  However, it's easy to see that if $A\in\on{An}\subset\on{CondAn}^{light}$ and $B\subset A$ is any subobject in $\on{CondAn}^{light}$, then $B\in\on{An}$ as well (by descent, we reduce to $A=\ast$, whose only subobjects are $\ast$ and $\emptyset$).  Thus $(X_n)_n$ is a tower of $\pi$-finite anima, and $X=\varprojlim_n X_n$.
\end{proof}

For $Y\in\on{CondAn}^{light}$, let $\on{fEt_\pi}_{/Y}$ denote the full subcategory of $\pi$-finite etale maps to $Y$, equipped with the Grothendieck topology inherited from $\on{CondAn}^{light}$.  The following theorem shows that, up to completeness issues, $\on{fEt_\pi}_{/Y}$ provides a defining site for $\on{Sh}_{et}(Y)$, for any light profinite anima $Y$.

\begin{theorem}\label{truncatedtheorem}
Let $S\in\on{ProfAn}^{light}$.  Consider the $\infty$-topos $\on{Sh}(\on{fEt_\pi}_{/S})$ of sheaves of anima on the site $\on{fEt_\pi}_{/Y}$ with respect to the Grothendieck topology inherited from $\on{CondAn}^{light}$.  Then the natural comparison map
$$\on{Sh}(\on{fEt_\pi}_{/S})\rightarrow \on{Sh}_{et}(S)$$
in $\mathcal{T}op^L$ (left Kan extending the inclusion map $\on{fEt_\pi}_{/S}\rightarrow \on{Sh}_{et}(S)$) identifies the target as the Postnikov-completion of the source.
\end{theorem}

\begin{proof}
Because $\on{Sh}_{et}(S)$ is Postnikov-complete (Theorem \ref{etale}), it suffices to show that the comparison map induces an equivalence on $d$-truncations for all $d\geq 0$.  Consider the following three $\infty$-categories:
\begin{enumerate}
\item The $\infty$-category of $\on{An}_{\leq d}$-valued sheaves on $\on{fEt_\pi}_{/S}$, i.e. $\on{Sh}(\on{fEt_\pi}_{/S})_{\leq d}$.
\item The $\infty$-category of $\on{An}_{\leq d}$-valued sheaves on $\on{ProfAn}^{light}_{/S}=\on{Pro}^{\aleph_0}(\on{fEt_\pi}_{/S})$ which preserve countable filtered colimits (as functors $(\on{ProfAn}^{light}_{/S})^{op}\to \on{An}_{\leq d}$).
\item The $\infty$-category of $\on{An}_{\leq d}$-valued sheaves on $\on{ProfSet}^{light}_{/S}$ which preserve countable filtered colimits (as functors $(\on{ProfSet}^{light}_{/S})^{op}\to\on{An}_{\leq d}$), i.e.\ $\on{Sh}_{et}(S)_{\leq d}$.
\end{enumerate}
We claim that restriction and left Kan extension identify 1 and 2, and restriction and right Kan extension identify 2 and 3.  Assuming this, the composition of these equivalences restricts to the inclusion $\on{fEt_\pi}_{/S}\rightarrow \on{Sh}_{et}(S)$, so it agrees with the comparison map, finishing the proof.
For the claim about 1 and 2, we need to show that a functor $\mathcal{F}:(\on{ProfAn}^{light}_{/S})^{op}\to \on{An}_{\leq d}$ preserving countable filtered colimits whose restriction to $\on{fEt_\pi}_{/S}$ is a sheaf is itself necessarily a sheaf.  Finite disjoint unions correspond to maps from some object to a finite set, hence come from a finite stage in a defining filtered inverse limit, so the sheaf property for finite disjoint unions follows from the fact that finite products commute with filtered colimits.  For surjections, using Lemma \ref{surjtower} we can write any surjection $T'\twoheadrightarrow T$ in $\on{ProfAn}^{light}_{/S}$ as a sequential filtered limit of surjective maps in $\on{fEt_\pi}_{/S}$, and the conclusion follows by commuting the (finite, by truncatedness) limit over $\Delta$ with the filtered colimit.   For the claim about 2 and 3, note that $\on{ProfSet}^{light}_{/S}\subset\on{ProfAn}^{light}_{/S}$ is closed under pullbacks, and every $X\in \on{ProfAn}^{light}_{/S}$ admits a surjection from a $Y\in \on{ProfSet}^{light}_{/S}$ by Lemma \ref{surjectfromprofset}.  By \cite{HoyoisLefschetz} Lemma C.3 this gives the claim if we remove the condition of preserving countable filtered colimits from both 2 and 3.  Thus, to finish it suffices to show that if a sheaf $\mathcal{F}: (\on{ProfAn}^{light}_{/S})^{op}\to\on{An}_{\leq d}$ is such that its restriction to $(\on{ProfSet}^{light}_{/S})^{op}$ preserves countable filtered colimits, then $\mathcal{F}$ also preserves countable filtered colimits.  For this, suppose given a tower $(X_n)_n$ over $S$ of light profinite anima.  Proceeding inductively using the fact that every $X\in \on{ProfAn}^{light}_{/S}$ admits a surjection from a $Y\in \on{ProfSet}^{light}_{/S}$, we can make a hypercover of this tower by a tower of light profinite sets.  Then the question amounts to commuting a filtered colimit with a limit over $\Delta$, which is handled as usual using truncatedness. \end{proof}

\begin{remark}
For $S\in\on{ProfAn}^{light}$, the Grothendieck topology on $\on{fEt}^\pi_{/S}$ inherited from $\on{CondAn}^{light}$ is simple to understand.  In fact, more generally, the Grothendieck topology on $\on{ProfAn}^{light}$ can be described as follows: a sieve over $S\in\on{ProfAn}^{light}$ is covering if and only if it contains finitely many $\{f_i:S_i\rightarrow S\}_{i\in I}$ such that for all $x:\ast\rightarrow S$, there is an $i\in I$ such that $x$ lifts along $f_i:S_i\rightarrow S$.  Note that this generalizes the defining Grothendieck topology on $\on{ProfSet}^{light}$.  The proof is straightforward using that every $S\in\on{ProfAn}^{light}$ admits a surjective map from an $T\in\on{ProfSet}^{light}$.
\end{remark}

In practice, a truncated variant of this theorem is useful.  Namely, for $d\geq 0$, consider the $\infty$-category of $d$-truncated $\pi$-finite anima, $\on{An}^\pi_{\leq d}$.  The induced functor $\on{Pro}^{\aleph_0}(\on{An}^\pi_{\leq d})\rightarrow \on{CondAn}^{light}$ is of course still fully faithful, and we call the essential image the full subcategory of \emph{$d$-truncated light profinite anima}.\footnote{There is in principle a question to be addressed, as to whether every light profinite anima which is $d$-truncated in the sense of the t-structure on $\on{CondAn}^{light}$ is actually a $d$-truncated light profinite anima in this sense.  This is true and follows from Lemma \ref{truncateprof}.}
For a $d$-truncated light profinite anima $S=\varprojlim_n S_n$, we can consider the full subcategory $\on{fEt^\pi_{\leq d-1}}_{/S}$ of $\on{fEt^\pi}_{/S}=\varinjlim_n \on{An}^\pi_{/S_n}$ consisting of the $(d-1)$-truncated maps $X\rightarrow S$, and equip it with the Grothendieck topology induced from $\on{CondAn}^{light}$.

Then we have the following.

\begin{corollary}\label{truncatedcorollary}
Let $d\geq 0$.  Suppose $S$ is a $d$-truncated light profinite anima.  Then the comparison morphism in $\mathcal{T}op^L$
$$\on{Sh}(\on{fEt^\pi_{\leq d-1}}_{/S})\rightarrow \on{Sh}_{et}(S)$$
exhibits the target as the Postnikov-completion of the source.
\end{corollary}
\begin{proof}
By the theorem, it suffices to show that the comparison functor $\on{Sh}(\on{fEt^\pi_{\leq d-1}}_{/S})\rightarrow \on{Sh}(\on{fEt^\pi}_{/S})$ is an equivalence.  But $\on{fEt^\pi_{\leq d-1}}_{/S}$ is closed under finite limits, so by \cite{HoyoisLefschetz} Lemma C.3 for this it suffices to see that for every $X\rightarrow S$ in $\on{fEt^\pi}_{/S}$, there is a finite jointly surjective family of maps $\{X_i\rightarrow X\}_{i\in I}$ such that $X_i\rightarrow S$ is in $\on{fEt^\pi_{\leq d-1}}_{/S}$ for all $i\in I$.  By base change we can assume $X,S\in \on{An}^\pi$.  Then take $I=\pi_0X$, take $X_i=\ast$ for all $i\in I$, and let the map $X_i\rightarrow X$ be any map hitting the $i^{th}$ component.
\end{proof}

\begin{example}
In the case $d=0$, the $0$-truncated light profinite anima are exactly the light profinite sets $S$, and the $-1$-truncated finite etale maps to $S$ are exactly the inclusions of compact open subsets.  Thus when $d=0$ this corollary recovers Theorem \ref{etaleonprof}.
\end{example}

\begin{example}\label{profgpexample}
Let $G$ be a light profinite group, written as the limit of a tower $(G_n)_n$ of finite groups.  We then have
$$\ast/G \simeq \varprojlim_n \ast/G_n.$$
Indeed, by repleteness the map $\ast \rightarrow \varprojlim_n \ast/G_n$ is surjective, so we can test on base change along it, and then we exactly reduce to $G\simeq \varprojlim_n G_n$.

Thus $\ast/G$ is a $1$-truncated light profinite anima.  A $0$-truncated finite etale map to $\ast/G$ is the same thing as a finite set with a continuous $G$-action.  Thus, in this case Corollary \ref{truncatedcorollary} shows that $\on{Sh}_{et}(\ast/G)$ is the Postnikov completion of the ``usual'' $\infty$-topos of sheaves of anima on the site of finite continuous $G$-sets (see \cite{ClausenHyper} 4.1).
\end{example}

\begin{remark}
We recall that classically, the 1-topos of sheaves of sets on the site of finite sets with continuous $G$-action has an alternate description: it is the category of (arbitrary) sets with continuous $G$-action.  The $\infty$-topos $\on{Sh}_{et}(\ast/G)$ is already the direct $\infty$-categorical generalization of this.  Indeed, by descent $\on{Sh}_v(\ast/G)$ identifies with the $\infty$-category of light condensed anima with $G$-action, and by definition $\on{Sh}_{et}(\ast/G)$ is the full subcategory where the underlying light condensed anima is ``discrete'', i.e.\ etale over $\ast$.
\end{remark}

We can also deduce the following convenient ``continuity'' result.

\begin{proposition}\label{procontinuity}
Let $S$ be a light profinite anima, and let
$$\ldots \to X_n \to X_{n-1} \to \ldots \to X_0 \to S$$
be a tower of light profinite anima over $S$, with limit $X=\varprojlim X_n$.  Then if $\mathcal{F}\in\on{Sh}_{et}(S)$ is $d$-truncated for some $d$, we have
$$\Gamma(X;\mathcal{F})=\varinjlim_n  \Gamma(X_n;\mathcal{F}).$$
\end{proposition}
\begin{proof}
By Lemma \ref{pifiniteproperties}, the (finitary) site $\on{fEt}^\pi_{/X}$ is the filtered colimit of the $\on{fEt}^{\pi}_{/X_n}$.  Thus the analog of this result for sheaves on the $\pi$-finite etale sites follows from \cite{ClausenHyper} Lemma 3.3.  Then we deduce the claim from Theorem \ref{truncatedtheorem}.
\end{proof}

To conclude this section, we prove a general ``$\on{lim}^1$-vanishing result'' in the setting of light profinite anima, showing that Postnikov truncations can be taken termwise.

\begin{proposition}\label{truncateprof}
Suppose $(X_n)_n$ is a tower of light profinite anima, with limit $X=\varprojlim_n X_n$ in $\on{CondAn}^{light}$.  Then for all $d\geq -1$, the map
$$\tau_{\leq d} X \rightarrow \varprojlim_n \tau_{\leq d}X_n$$
is an isomorphism in $\on{CondAn}^{light}$.

More generally, if $(f_n):(X_n)\rightarrow (Y_n)$ is a map of towers of light profinite anima, then taking the inverse limit in $\on{CondAn}^{light}$ commutes with relative Postnikov truncation.
\end{proposition}
\begin{proof}
Let us immediately handle the more general relative case.  By Lemma \ref{surjectfromprofset}, we can write $X_0 \to Y_0$ as the sequential inverse limit of a map of $\pi$-finite anima.  Then we do the same for $Y_1 \to Y_0$, but using the same tower for $Y_0$ as before (again this is possible by Lemma \ref{surjectfromprofset}).  Then we do the same again for $X_1 \to Y_1\times_{Y_0}X_0$, using the tower on the target coming from what we've already constructed.  Continuing in this manner, we write our map of towers as a sequential inverse limit of maps of towers of $\pi$-finite anima. Thus we can reduce our claim to the case where all $X_n,Y_n$ are $\pi$-finite anima.

Consider the factorization
$$X_n\rightarrow \tau_{\leq d}f_n \rightarrow Y_n$$
of $f_n$ into a $d+1$-connective map followed by a $d$-truncated map. Taking inverse limits, we get
$$X\rightarrow \varprojlim_n \tau_{\leq d}f_n\rightarrow Y.$$
The second map is $d$-truncated because a limit of $d$-truncated maps is $d$-truncated.  Thus it suffices to show that the first map is $d+1$-connective, or in other words, a limit of a tower of $d+1$-connective maps of $\pi$-finite anima is $d+1$-connective.  Choose a tower finite sets with a surjection to the target tower such that the map on inverse limits is still a surjection (using Lemma \ref{surjectfromprofset}).  Then by base change, we reduce to the case where the target tower consists of finite sets.

When $d=-1$, we need to show that if our map $(X_n)\rightarrow (S_n)$ (with $S_n$ a tower of finite sets and $X_n$ a tower of $\pi$-finite anima) is termwise surjective, then $\varprojlim_n X_n\rightarrow\varprojlim_n S_n$ is also surjective.  Choosing also a tower of finite sets surjecting onto $(X_n)$, we reduce to the case that both $(X_n)$ and $(S_n)$ are towers of finite sets.  Then it is a consequence of the fact that a filtered inverse limit of nonempty finite sets is nonempty (by a standard compactness argument).

For $d\geq 0$, we proceed inductively.  Suppose $(X_n)\rightarrow (S_n)$ is $d+1$-connective, in particular it induces $\pi_0(X_n)\overset{\sim}{\rightarrow}S_n$.  Using this we can define, proceeding inductively on the tower, a section of $(X_n)\rightarrow (S_n)$, giving in particular a section $\sigma:S\rightarrow X$ on inverse limits.  By the $d=-1$ case, $\sigma$ is surjective.  On the other hand, $S\times_X S\rightarrow S$ (the fiber product being formed using the map $\sigma$) is the inverse limit of a tower of $d$-connective maps of $\pi$-finite anima.  (We have in essence passed to the based loop spaces on fibers.)  Thus by induction $S\times_X S\rightarrow S$ is $d$-connective, so $X\rightarrow S$ is $d+1$-connective as desired.
\end{proof}

\section{Light condensed anima attached to sheaves on manifolds}\label{condmansec}

Above, we briefly discussed the light condensed sets $\underline{X}$ associated to topological spaces $X$ which are ``locally compact'' in the slightly nonstandard sense that each point admits a second-countable compact Hausdorff neighborhood.  We showed that for such an $X$, the category $\on{Sh}_{et}(X)$ of etale sheaves on $X$ identifies with the Postnikov completion of the category of sheaves of anima on the topological space $X$.

This result implies in particular to $F$-manifolds if $F$ is a \emph{local field} (a non-discrete locally compact field).  That case is actually even simpler: any compact subset of $F^d$ has finite covering dimension (less than or equal to $0$ if $F$ is nonarchimedean, $d$ if $F=\mathbb{R}$, and $2d$ if $F\simeq\mathbb{C}$.)  Thus $\on{Sh}(M)$ is Postnikov-complete by \cite{LurieHTT} Cor.\ 7.2.3.7, and so the result of Lemma \ref{etaleontopological} shows that
$$\on{Sh}(M)\overset{\sim}{\rightarrow}\on{Sh}_{et}(\underline{M})$$
for any $F$-manifold $M$: etale sheaves on $F$-manifolds identify with usual sheaves.

In this paper we are not directly interested in $F$-manifolds; rather our purported focus is on cohomology of $F$-Lie groups $G$.  For this it is convenient to pass from $G$ to some version of the classifying stack $BG=\ast/G$.  For some of our arguments, we will also have to consider more general quotient stacks.  To form such a quotient $M/G$, we could of course consider the underlying light condensed sets of $G$ and $M$ and then form this quotient in light condensed anima.  But for some purposes we'll want to remember the $F$-analytic structure on such an object.  For this we make the following standard definitions.

\begin{definition}
Let $F$ be a local field.  A \emph{sheaf on $F$-manifolds} is a sheaf (of anima) on the category $\on{Man}_F$ with respect to the Grothendieck topology generated by open covers.
\end{definition}

\begin{remark}
Since $\on{Man}_F$ is not a small category, in principle there are cardinality issues in this definition.  But since every manifold by definition admits an open cover by open subsets of some $F^d$, it follows from \cite{HoyoisLefschetz} Lemma C.3 that we could equivalently restrict to the full subcategory of $\on{Man}_F$ consisting of such open sets and the sheaf category would stay the same.  In particular $\on{Sh}(\on{Man}_F)$ is an $\infty$-topos.

Moreover, $\on{Sh}(\on{Man}_F)$ is hypercomplete.  This follows from the fact that $F$-manifolds have local finite covering dimension and \cite{LurieHTT} Thm.\ 7.2.3.6.  Note that every $\mathcal{F}\in\on{Sh}(\on{Man}_F)$ admits a hypercover by $F$-manifolds, or even by disjoint unions of open polydisks.
\end{remark}

Now we describe the ``forgetful functor'' from sheaves on $F$-manifolds to light condensed anima.

\begin{proposition}\label{underlying}
Let $F$ be a local field.
\begin{enumerate}
\item There is a unique colimit-preserving functor
$$\on{Sh}(\on{Man}_F)\rightarrow \on{CondAn}^{light}$$
whose restriction to $\on{Man}_F$ is the functor $M\mapsto \underline{M^{top}}: \on{Man}_F\rightarrow \on{CondAn}^{light}$, where $M^{top}$ denotes the underlying topological space of $M$.  By abuse of notation, denote this functor also by $X\mapsto \underline{X}$ (or sometimes, if we're being really abusive, just $X\mapsto X$ with the context being implicit).
\item This functor $X\mapsto \underline{X}$ preserves pullbacks of diagrams of the form $X\rightarrow S\leftarrow Y$ whenever $X\to S$ can, locally on $S$, be written as a colimit of submersions of manifolds $X_i \to S$.
\end{enumerate}
\end{proposition}
\begin{proof}
Part 1 follows formally from Lemma \ref{etaleontopological}.  For part 2, we can use that colimits distribute over pullbacks in any $\infty$-topos to reduce to the claim that pullbacks in $\on{Man}_F$ by submersions are representable and commute with the forgetful functor to condensed sets.
\end{proof}

\begin{remark}
A special case of 2 is that $X\mapsto \underline{X}$ preserves finite products.  In particular, it sends group objects to group objects.  Moreover, if $G$ is a group object in $\on{Sh}(\on{Man}_F)$, then
$$\underline{\ast/G} = \ast/\underline{G}$$
by commutation with colimits and finite products.
\end{remark}

\begin{remark}
It seems likely that $X\mapsto \underline{X}$ commutes with arbitrary pullbacks, hence gives a morphism in $\mc{T}op^L$.  But we won't need this.
\end{remark}

Up until now we've done things uniformly in $F$.  But of course there are major structural differences between $F$-manifolds for archimedean and non-archimedean $F$.  One such difference is that when $F$ is archimedean, every $F$-manifold is locally contractible.  Local contractibility gives a good theory of locally constant sheaves, see \cite{LurieHA} App.\ A, and relatedly shows that every $F$-manifold has an ``underlying anima''.  We now give an expression of this phenomenon in the setting of etale sheaves.  There is no real loss in restricting to $F=\mathbb{R}$, so we will.

\begin{lemma}\label{manifoldanima}
\begin{enumerate}
\item For $X\in\on{Sh}(\on{Man}_{\mbb{R}})$, there is an $|X|\in\on{An}$ and a map $\underline{X}\rightarrow |X|$ in $\on{CondAn}^{light}$ such that the pullback functor $\on{Sh}_{et}(|X|)\rightarrow\on{Sh}_{et}(\underline{X})$ is fully faithful.
\item Given $\underline{X}\rightarrow |X|$ as in 1, for arbitrary $A\in\on{An}$ we have
$$\on{Map}(|X|,A)\overset{\sim}{\rightarrow} \on{Map}(\underline{X},A).$$
In particular, $X\mapsto |X|$ is functorial in $\underline{X}$.
\item For $X\in\on{Sh}(\on{Man}_{\mbb{R}})$, the essential image of the fully faithful pullback functor $\on{Sh}_{et}(|X|)\rightarrow \on{Sh}_{et}(\underline{X})$ from 1 consists of the locally constant sheaves, meaning those sheaves $\mathcal{F}\in\on{Sh}_{et}(\underline{X})$ such that there exists a cover $\{f_i:X_i\rightarrow X\}_{i\in I}$ in $\on{Sh}(\on{Man}_{\mbb{R}})$ such that $f_i^\ast\mathcal{F}$ is pulled back along $X_i\rightarrow \ast$, for all $i\in I$.
\item The functor $X\mapsto |X|:\on{Sh}(\on{Man}_{\mbb{R}})\rightarrow\on{An}$ is uniquely characterized as follows: it preserves all colimits, and on the full subcategory of open polydisks it identifies with the constant functor with value $\ast\in\on{An}$.  More precisely, the full subcategory of $\on{Fun}(\on{Sh}(\on{Man}_{\mbb{R}}),\on{An})$ spanned by the functors satisfiying these two conditions is equivalent to the terminal $\infty$-category $\ast$.
\item The functor $X\mapsto |X|$ preserves finite products.
\end{enumerate}
\end{lemma}
\begin{proof}
Note that 2 follows from 1, since mapping to $A$ is the same as giving a section of the etale projection map $A\times (-)\rightarrow (-)$.  It follows that collection of $X$ satisfying 1 is closed under colimits.  Thus (using hypercovers) we can reduce to the case where $X$ is an open polydisk.  Then we claim $|X|=\ast$ works.  Indeed, this follows by contractibility, see \cite{LurieHA} A.2.

For 3, note that we just showed that if a sheaf is in the essential image, then it is a constant sheaf on any open polydisk, hence it is locally constant. For the converse, suppose $\mathcal{F}\in\on{Sh}_{et}(\underline{X})$ is locally constant.  Then $\mathcal{F}$ locally comes from $\on{Sh}_{et}(|X|)$.  But the property of lying in $\on{Sh}_{et}(|X|)\subset\on{Sh}_{et}(X)$ is a local one thanks to the colimit preservation of $X\mapsto |X|$, so $\mathcal{F}$ comes from $\on{Sh}_{et}(|X|)$ as required.

For 4, note that by hypercompleteness that every colimit-preserving functor out of $\on{Sh}(\on{Man}_{\mbb{R}})$ is the left Kan extension of its restriction to open polydisks.  But for any $\infty$-category $\mathcal{C}$, the full subcategory of functors $\mathcal{C}\to \on{An}$ spanned by the constant functor with value $\ast$ is equivalent to $\ast$, whence the claim.

For 5, since colimits distribute finite products (in any $\infty$-topos, so in particular in $\on{Sh}(\on{Man}_{\mbb{R}})$ and in $\on{An}$), we can reduce to the case of open polydisks, which is obvious because $\ast\times \ast=\ast$.
\end{proof}

\begin{remark}
We can interpret $|X|$ as the ``homotopy type'', or ``underlying anima'', of $X$.  Claim 4 makes it easy to identify this model for the homotopy type with any other one.  For example, given $M\in\on{Man}_{\mbb{R}}$ we can consider the usual Kan complex $\on{Sing}(M)$ of singular simplices on $M$ (so $\on{Sing}(M)_n = \on{Cont}(\Delta^n,M)$), and its realization $\varinjlim_{\Delta^{op}} \on{Sing}(M)$ in $\on{An}$.  The functor
$$M\mapsto \varinjlim_{\Delta^{op}} \on{Sing}(M)$$
from $\on{Man}_F$ to $\on{An}$ is a cosheaf, because of the standard Mayer-Vietoris property for binary covers and the fact that $\on{Sing}(-)$ preserves filtered unions of open subsets.  Thus it extends uniquely to a colimit-preserving functor $\on{Sh}(\on{Man}_F)\rightarrow\on{An}$.  On the other hand, its value on an open polydisk is $\ast$ by contractibility, so we deduce that the usual singular simplicial Kan complex of an $F$-manifold $M$ gives a model for the abstract $|M|$ considered in Lemma \ref{manifoldanima}.  Note, however, that the map $\underline{M}\rightarrow |M|$, which is a fundamental structure, is not at all clear from the perspective of the singular simplicial complex.
\end{remark}

This gives the following.

\begin{theorem}\label{locconst}
Let $X\in\on{Sh}(\on{Man}_{\mbb{R}})$ with underlying anima $|X|$ as defined above.  The functor
$$\on{An}_{/|X|}\rightarrow \on{Sh}_{et}(\underline{X}),$$
arising from pull back of etale sheaves under the map $\underline{X}\rightarrow |X|$, is fully faithful.  The essential image consists of those $\mathcal{F}\in\on{Sh}_{et}(\underline{X})$ such that there is a cover $\{f_i:X_i\rightarrow X\}_{i\in I}$ in $\on{Sh}(\on{Man}_{\mbb{R}})$ for which the pullback $f_i^\ast\mathcal{F}\in \on{Sh}_{et}(\underline{X_i})$ is constant for all $i\in I$.
\end{theorem}
\begin{proof}
This follows by combining Lemmas \ref{manifoldanima} and \ref{etaleonetale}.
\end{proof}

\begin{corollary}
Let $G$ be an $\mathbb{R}$-Lie group with underlying group anima $|G|$.  Then
$$\on{An}_{/(\ast/|G|)}\overset{\sim}{\rightarrow} \on{Sh}_{et}(\underline{\ast/G}).$$
In other words, an etale sheaf on $\ast/G$ is the same as an etale sheaf on $\ast/|G|$ which is the same thing as an anima with a $|G|$-action.
\end{corollary}
\begin{proof}
The map $\ast \rightarrow \ast/G$ gives a cover such that, tautologically, any sheaf is constant on pullback along it.  Thus this follows from Theorem \ref{locconst}.
\end{proof}

\begin{example}
We have $\on{Sh}_{et}(\ast/\mathbb{R})\overset{\sim}{\leftarrow}\on{An}$ via pullback, so the theory of etale sheaves on $\ast/\mathbb{R}$ is trivial... Or is it? [Cue dramatic music.] It turns out that the six functor formalism we consider will do something nontrivial even in this situation, and this is actually a crucial phenomenon, see Section \ref{jsec}.
\end{example}

It is convenient to make the following standard definition, as a reasonable category in which all our examples of interest will live.

\begin{definition}
Let $F$ be a local field.  An \emph{$F$-analytic smooth Artin stack} is an $X\in\on{Sh}(\on{Man}_F)$ such that there exists a surjective representable submersion $M\twoheadrightarrow X$ with $M\in\on{Man}_F$.
\end{definition}

\begin{remark}
A representable submersion is surjective in the sense of $\infty$-topos theory if and only if on pullback to any manifold it is surjective as a map of sets.  This follows because a submersion has open image, and a surjective submersion has local sections (by the implicit function theorem).
\end{remark}

\begin{example}
\begin{enumerate}
\item Suppose $G$ is an $F$-Lie group.  The \emph{classifying stack} $\ast/G$, also denoted $BG$, is an $F$-analytic smooth Artin stack.
\item More generally, suppose $G$ is an $F$-Lie group acting on an $F$-manifold $M$.  Then the quotient $M/G$ is an $F$-analytic smooth Artin stack, with cover $M\to M/G$.  This kind of example is called a \emph{quotient stack}.
\item Even more generally, if $N\in\on{Man}_F$, and $G\to N$ is an $F$-Lie group over $N$ (a group object in submersions with target $N$), then for any $F$-manifold $M\to N$ mapping to $N$ (not necessarily by a submerison) with an action by $G$ in the slice category of manifolds over $N$, the quotient $M/G$ is an $F$-analyic smooth Artin stack.  We'll call this a \emph{relative quotient stack}.
\end{enumerate}
\end{example}

In the $p$-adic case, we will have to make a study of families of Lie groups.  Recall that Lazard showed that every $p$-adic Lie group $G$ admits a compact open subgroup $H$ which is a \emph{uniform pro-$p$-group}.  Here we use the language of \cite{DixonAnalytic} instead of \cite{LazardGroupes}, and refer to the nice lecture notes \cite{Pstragowski} for more on this classical situation.

In particular, we recall the following properties of uniform pro-$p$-groups $H$:\label{scholiumuniform}
\begin{enumerate}
\item The subset of $p^n$-powers $H^{p^n}\subset H$ is an open normal subgroup, and $\cap_n H^{p^n}=\{e\}$.
\item Each $H^{p^n}/H^{p^{n+1}}$ is abelian, hence an $\mathbb{F}_p$-vector space, and has dimension $d$ where $d=\on{dim}(H)=\on{dim}(G)$;
\item For the continuous group cohomology of $H$ with constant $\mathbb{F}_p$-coefficients, we have:
\begin{enumerate}
\item The natural map $H^1(H;\mbb{F}_p)\to \on{Hom}_{\mbb{F}_p}(H/H^p,\mbb{F}_p)$ is an isomorphism.
\item For all $x\in H^1(H;\mbb{F}_p)$, we have $x^2=0$ in $H^2(H;\mbb{F}_p)$, and the induced map
$$\Lambda^i_{\mbb{F}_p}H^1(H;\mbb{F}_p)\to H^i(H;\mbb{F}_p)$$
is an isomorphism for all $i\geq 0$.
\end{enumerate}
\end{enumerate}

Now we go to the in-families version.  If $M$ is a $\mathbb{Q}_p$-manifold, recall that a $\mathbb{Q}_p$-Lie group over $M$ is a group object in the slice category of submersions with target $M$.

\begin{lemma}\label{nonarchimedeancontract}
Let $n,m\geq 0$ and let $\varphi: U\to\mathbb{Q}_p^m$ be a $\mathbb{Q}_p$-analytic map from an open subset $U\subset \mathbb{Q}_p^n$ containing $0$ such that:
\begin{enumerate}
\item $\varphi(0)=0$;
\item the $\mathbb{Q}_p$-linear map $d\varphi |_0:\mathbb{Q}_p^n \to \mathbb{Q}_p^m$ sends $\mathbb{Z}_p^n$ inside $\mathbb{Z}_p^m$.
\end{enumerate}
Then there is an $N\geq 0$ such that:
\begin{enumerate}
\item $p^N\mathbb{Z}_p^n\subset U$;
\item The power series expansion of $\varphi$ at the origin converges on $p^N\mathbb{Z}_p$;
\item If we re-coordinatize both source and target by replacing the old coordinate $X$ with the new coordinate $X'=p^NX$, then each coefficient in the Taylor expansion of $\varphi$ at the origin lies in $\mathbb{Z}_p$;
\item With respect to the new coordinates, $\varphi(\mathbb{Z}_p^n)\subset \mathbb{Z}_p^m$.
\end{enumerate}
\end{lemma}
\begin{proof}
We can assume $m=1$ by working coordinate-wise.  It is clearly no problem to ensure 1 and 2.  In fact if we write the power series expansion of $\varphi$ as
$$\varphi(X) = \sum_{i_1,i_2,\ldots,i_n} c_{i_1,\ldots,i_n}x_1^{i_1}\ldots x_n^{i_n},$$
then by definition of locally analytic we can ensure that
$$\sum_{i_1,\ldots,i_n} |c_{i_1,\ldots,i_n}|p^{-N(i_1+\ldots+i_n)}<\infty,$$
and in particular for all but finitely many $i_1,\ldots,i_n$ we have $c_{i_1,\ldots,i_n}p^{N(i_1+\ldots+i_n)}\in \mbb{Z}_p$.  Since the constant term vanishes, by increasing $N$ we can even ensure this for all $i_1,\ldots,i_N$.  If we make the re-coordinatizations $X'= p^MX$ on both source and target, this changes $c_{i_1,\ldots,i_n}$ to
$$c'_{i_1,\ldots,i_n} = \frac{1}{p^M}p^{M(i_1+\ldots+i_n)}c_{i_1,\ldots,i_n}.$$
The linear terms are unchanged, but we assumed those are in $\mathbb{Z}_p$ anyway.  Thus if we take $M=2N$ we get all the coefficients to lie in $\mathbb{Z}_p$, meaning we can ensure 1, 2, and 3.  But then 4 follows because $\mathbb{Z}_p$ is closed under convergent sums in $\mathbb{Q}_p$.

\end{proof}

\begin{proposition}\label{localstructurepadiclie}
Let $p$ be a prime, let $M$ be a compact Hausdorff $\mathbb{Q}_p$-manifold, and let $G\to M$ be a $\mathbb{Q}_p$-Lie group over $M$.  Then there exists a compact Hausdorff open sub-group object $H\subset G$ over $M$ such that for all $m\in M$, the fiber $H_m$ is a uniform pro-$p$-group.
\end{proposition}
\begin{proof}
Note that because $M$ is profinite, we are free to work locally on $M$; the gluing back up will be trivial.  In particular we can assume $M=\mathbb{Z}_p^e$ for some $e\geq 0$.  Using the implicit function theorem, again working locally on $M$, we can write $G\to M$ in local coordinates around the identity section as the projection
$$\mathbb{Z}_p^d \times \mathbb{Z}_p^e \to \mathbb{Z}_p^e,$$
such that the identity section corresponds to $t\mapsto (0,t)$.  Now consider the map $\varphi: G\times_M G\to G$ defined by $(g,h)\mapsto gh^{-1}$.  Note that an open subset $H\subset G$ containing the identity section is a sub-group object if and only if $\varphi(H\times_M H)\subset H$. Restricting to our coordinate chart $\varphi$ takes the form
$$\varphi: \mbb{Z}_p^d\times \mbb{Z}_p^d\times\mbb{Z}_p^e \to G,$$
lying over $M=\mbb{Z}_p^e$.  By continuity there is an $M\geq 0$ such that $\varphi$ sends $p^M(\mbb{Z}_p^{d+d}\times\mbb{Z}_p^e)$ inside our coordinate chart $\mbb{Z}_p^d\times\mbb{Z}_p^e$ again.  Now, we have $\varphi(0)=0$, and
$$d\varphi|_0(X,Y,Z) = X - Y + Z,$$
Thus we can apply Lemma \ref{nonarchimedeancontract} to conclude that there is some $N\geq 0$ such that after changing all the coordinates via $T' = p^NT$, we have
$$\varphi(\mbb{Z}_p^d\times \mbb{Z}_p^d\times\mbb{Z}_p^e)\subset \mbb{Z}_p^d\times \mbb{Z}_p^e,$$
and in fact the power series expansion of $\varphi$ at $0$ has $\mathbb{Z}_p$-coefficients.

It follows that $\mathbb{Z}_p^d\times \mathbb{Z}_p^e$ defines a compact open subgroup object of $G$.  Moreover, if we specialize to the fiber over any $t\in \mbb{Z}_p^e$, we get that the group law is defined by a (convergent) formal group law with $\mathbb{Z}_p$-coefficients.  By \cite{DixonAnalytic} 8.4, then, if we further restrict to $(2p\mathbb{Z}_p^d)\times\mathbb{Z}_p^e$, then this gives a uniform pro-$p$-group (and still a compact open subgroup object).
\end{proof}

Now let us describe some consequences for the associated light condensed anima.

\begin{proposition}\label{localstructurecorollary}
Let $M$ be a compact Hausdorff $\mathbb{Q}_p$-manifold, let $G\to M$ be a group object in sumbersions over $M$, and let $H\subset G$ be a compact Hausdorff open sub-group object such that each fiber $H_x$ is a uniform pro-$p$-group (whose existence is guaranteed by the previous result).  Let $BH \in \on{Sh}(\on{Man}_{\mbb{Q}_p})$ denote the classifying stack of $H$ relative to $M$, and similarly for $H^{p^n}$ or $G$.  

Then:
\begin{enumerate}
\item For $n\geq 1$, the subset $H^{p^n}\subset G$ is also a compact open sub-group object.
\item $\underline{BH}\to \underline{BG}$ is etale, as is $\underline{BH^{p^n}}\to \underline{BH}$ for all $n\geq 0$.
\item $\underline{BH} = \varprojlim_n \underline{B(H/H^{p^n})}$.
\item $\underline{BH}$ is a 1-truncated light profinite anima.
\end{enumerate}
\end{proposition}
\begin{proof}
By definition $H^{p^n}\subset H$ is the image of the $p$-power map $H \to H$, hence it is compact.  It is  a sub-group object because this is so fiberwise (\ref{scholiumuniform}).  Then $H/H^{p^n}$ is the quotient of a compact Hausdorff space by a closed equivalence relation, hence compact Hausdorff.  The map $H/H^{p^n}\to M$ of compact Hausdorff spaces has finite fibers, and the cardinality of the fibers is a locally constant function of the base (by \ref{scholiumuniform}); it follows that $H/H^{p^n}\to M$ is finite etale (a finite covering map), and a fortiori $H^{p^n}\subset H$ is open, proving 1.

Next we claim that if $G'\subset G$ is any open sub-group object, the quotient $G/G'\to M$ is etale.  Indeed because $G\to M$ is a submersion it admits local sections through any given point, hence so does $G/G'\to M$.  The conclusion follows easily.  In this situation we have that $BG'\to BG$ is a pullback of $G/G'\to M$, hence it is also etale.  Moreover this pullback is preserved by the functor to light condensed anima by Lemma \ref{underlying}, and we deduce 2.

Now consider 3.  By repleteness $M$ surjects onto the target, so we can test this map being an isomorphism on base change to $M$.  Thus we reduce to showing 
$$\varprojlim_n \ul{BH^{p^n}} \overset{\sim}{\rightarrow} \ul{M}.$$
Using the same trick again with repleteness and pullbacks, we further reduce to showing $M\overset{\sim}{\rightarrow}\varprojlim_n H^{p^n}$.  But this is a map of compact Hausdorff spaces and it's a bijection as we see looking fiberwise.

Finally, for 4, given 3 it suffices to show that each $\ul{B(H/H^{p^n})}$ is a 1-truncated light profinite anima.  But we saw above that $H/H^{p^n}\to M$ is finite etale.  It follows that as a group object it is locally constant with finite fibers, and the claim follows.
\end{proof}

Thus, in particular, for any $M\in\on{Man}_F$ and any Lie group object $G\to M$ over $M$, the relative $BG$ has the following property: etale locally, the associated light condensed anima is light profinite.  More generally one can deduce that the same claim holds for any quotient stack or relative quotient stack.  I conjecture that this actually holds for all $p$-adic analytic smooth Artin stacks:

\begin{conjecture}
Let $X$ be a $p$-adic analytic smooth Artin stack.  Then étale locally, the associated light condensed anima is a light profinite anima.
\end{conjecture}

If true, this would imply that every $p$-adic analytic smooth Artin stack is ``$!$-good'' in the sense of Theorem \ref{good}.

\section{Interlude on descendability in sheaf categories}\label{descsec}

For $\mathcal{C}\in\on{CAlg}(\on{Pr}_{st}^L)$, or equivalently $\mathcal{C}$ a symmetric monoidal presentable stable $\infty$-category where the tensor product distributes over colimits, Mathew in \cite{MathewGalois}, \cite{MathewDescent} defined and studied the condition of \emph{descendability} on a commutative algebra object $A\in\on{CAlg}(\mathcal{C})$.  This notion has important technical relevance in the construction of six functor formalisms, so we will prove some results giving criteria for descendability in (certain) sheaf categories.

We start with a review of some aspects of Mathew's theory.  For the basic definition, actually, the commutative algebra structure on $A$ is not relevant.

\begin{definition}
Let $\mathcal{C}$ be a symmetric monoidal stable $\infty$-category and let $A\in \mathcal{C}$.
\begin{enumerate}
\item Say a map $f:X\to Y$ in $\mathcal{C}$ is \emph{$A$-null} if the induced map $X\otimes A \to Y \otimes A$ in $\mathcal{C}$ is null, i.e.\ homotopic to $0$.
\item For $d\geq 0$, say that $A$ is \emph{descendable of index $\leq d$} if whenever $X_0\to X_1\to\ldots \to X_d$ is a sequence of $d$ composable $A$-null maps, the composition $X_0\to X_d$ is null.
\item Say that $A$ is descendable if if is descendable of index $\leq d$ for some $d\geq 0$.
\end{enumerate}
\end{definition}

\begin{example}
Suppose $\mathcal{C}=\on{D}(k)$ for a field $k$.  Then $A\in\mathcal{C}$ is descendable if and only if $A\neq 0$, in which case it is descendable of index $\leq 1$.  Indeed, every nonzero $A\in \on{D}(k)$ admits $k[n]$ as a summand for some $n\in\mathbb{Z}$.
\end{example}
\begin{remark}
If $A\in\mathcal{C}$ is descendable, then $-\otimes A:\mathcal{C}\to\mathcal{C}$ is conservative, because $X\otimes A=0$ means that $\on{id}_X$ is $A$-null.  The converse fails: in $\mathcal{C}=\on{D}(\mathbb{Z})$, the object $A=\mathbb{Z}[1/p]\times\mathbb{F}_p$ satisfies that $-\otimes A$ is conservative, but it is not descsendable, because
$$\mathbb{Z}/p\mathbb{Z}\overset{p}{\to}\mathbb{Z}/p^2\mathbb{Z}\overset{p}{\to} \ldots$$
gives an infinite sequence of $A$-null maps where no composition is null.
\end{remark}

\begin{remark}
In verifying the definition of descendability for an $A\in\on{CAlg}(\mathcal{C})$, it is enough to consider only a certain basic instance of a composition of $A$-null maps.  Namely, set $I=\on{Fib}(1\to A)$, and denote by $\alpha:I\to 1$ the canonical map.  Note that $\alpha$ is $A$-null.  Define inductively a sequence of maps
$$\alpha_d:I^{\otimes d}\to 1$$
by taking $\alpha_0$ to be the identity map $1\to 1$, and taking $\alpha_d$ for $d\geq 1$ to be the composition
$$\alpha_d: I^{\otimes d}\overset{\on{id}\otimes \alpha}{\longrightarrow}I^{\otimes d-1}\overset{\alpha_{d-1}}{\longrightarrow} 1.$$
Then $A$ is descendable of index $\leq d$ if and only if $\alpha_d$ is null.  This criterion in particular shows that symmetric monoidal functors send descendable algebras of index $\leq d$ to descendable algebras of index $\leq d$.
\end{remark}

How to prove descendability in a given case?  Sometimes one can use devissage to reduce complicated cases to some more basic ones, a nice example being given in \cite{BhattProjectivity} 11.2.  But in the end, one is faced with the problem of proving that some $\alpha_d$ map is null.  Sometimes, one can proceed by finding an auxiliary property $\mathcal{P}$ of maps, such that every sufficiently long composition of $A$-null maps satisfies $\mathcal{P}$, and any sufficiently long composition of maps satisfying $\mathcal{P}$ is null.  One model for this is Bousfield's argument in the context of spectra (see \cite{BousfieldLocalization} Section 6), where the relevant property $\mathcal{P}$ is that of being a \emph{phantom map}.  We would like to place Bousfield's argument in a more general context, starting with the following definition.

\begin{definition}\label{phantomdef}
Let $\mathcal{C}$ be a presentable stable $\infty$-category.  Say that a map $f:X\to Y$ is \emph{phantom} if for every colimit-preserving functor $F:\mathcal{C}\to\on{Sp}$, the induced map of abelian groups
$$\pi_0 F(X)\to \pi_0 F(Y)$$
is zero.
\end{definition}

\begin{example}
If $\mathcal{C}$ is compactly generated, then then the colimit preserving functors $F:\mathcal{C}\to\on{Sp}$ are exactly the filtered colimits of functors of the form $\on{map}(c,-)$ for some $c\in\mathcal{C}^{\aleph_0}$.  It follows that the above definition recovers the usual one: a map $f:X\to Y$ is phantom if and only if for every compact object $c$, the induced map $[c,X]\to[c,Y]$ is zero.
\end{example}

\begin{example}
More generally, suppose $\mathcal{C}$ is \emph{dualizable}. Then every colimit-preserving functor $\mathcal{C}\to\on{Sp}$ is a filtered colimit of functors of the form
$$\on{map}_{cpct}(c,-) = \on{map}(c,\hat{y}(-))$$
for $c\in\mathcal{C}$.  Here $\hat{y}:\mathcal{C}\to\on{Ind}(\mathcal{C})$ is the \emph{left} adjoint to the colimit functor, which exists when $\mathcal{C}$ is dualizable.  It follows that in the dualizable case, Definition \ref{phantomdef} agrees with that given in \cite{EfimovKtheory} App.\ E: a map $f:X\to Y$ is phantom if and only if for every $c\in \mathcal{C}$ and every compact map $g:c\to X$, we have $g\circ f\sim 0$.
\end{example}

\begin{lemma}\label{phantomobvious}
Let $\mathcal{C}$ be a presentable stable $\infty$-category.
\begin{enumerate}
\item Any shift or retract of a phantom map is phantom.
\item Phantom maps form a 2-sided ideal: they are closed under addition and under composition on either side with an arbitrary map.
\item If $A:\mathcal{C}\to\mathcal{D}$ is a colimit-preserving functor to another presentable stable $\infty$-category, then $A$ sends phantom maps to phantom maps.
\end{enumerate}
\end{lemma}
\begin{proof}
For closure under shifts note that since the class of colimit-preserving functors is closed under shifts, we could replace the condition that $\pi_0F(X)\to \pi_0F(Y)$ should be $0$ in the definition with the condition that $F(X)\to F(Y)$ should be zero on all homotopy groups.  The rest of the properties are straightforward from the definition.
\end{proof}

We can attempt to bridge the gap between phantom maps and nullhomotopic maps using the following defintion.

\begin{definition}
Let $\mathcal{C}$ be a presentable stable $\infty$-category, let $X\in\mathcal{C}$, and let $n\in\mathbb{N}$.

Say $X$ is a \emph{ghostbuster of index $\leq n$} if for any composable sequence $X_0 \to X_1\to\ldots\to X_n$ of $n$ consecutive phantom maps in $\mathcal{C}$ and any map $X\to X_0$ the composition $X \to X_n$ is homotopic to $0$.
\end{definition}

\begin{lemma}\label{ghostbuster}
Let $\mathcal{C}$ be a presentable $\infty$-category.
\begin{enumerate}
\item For $n\in\mathbb{N}$, the collection of ghostbusters of index $\leq n$ is closed under shifts, retracts, and arbitrary direct sums.
\item Suppose $X \to Y \to Z$ is a cofiber sequence in $\mathcal{C}$.  If $X$ is ghostbuster of index $\leq n$ and $Z$ is a ghostbuster of index $\leq m$, then $Y$ is a ghostbuster of index $\leq n+m$.
\item Any compact object $X\in\mathcal{C}$ is a ghostbuster of index $\leq 1$.
\item Suppose given another presentable stable $\infty$-category $\mathcal{D}$ and functors $A:\mathcal{C}\to\mathcal{D}$ and $B:\mathcal{D}\to\mathcal{C}$ such that $A$ preserves colimits, $B(0)=0$, and $B\circ A \simeq \on{id}_{\mathcal{C}}$.  Then for $X\in\mathcal{C}$, we have the implication $A(X)\in\mathcal{D}$ is a ghostbuster of index $\leq n$ $\Rightarrow$ $X\in\mathcal{C}$ is a ghostbuster of index $\leq n$.
\end{enumerate}
\end{lemma}
\begin{proof}
Part 1 is straightforward using the remark that the class of phantom maps is closed under shifts.  For part 2, it suffices to show that if we have maps $A \to B \to C$ in $\mathcal{C}$ such that $[X,A] \to [X,B]$ is $0$ and $[Z,B]\to [Z,C]$ is $0$, then $[Y,A]\to [Y,C]$ is $0$.  But this is clear.  Part 3 follows from the definition of phantom, as $\on{map}(X,-)$ preserves colimits when $X$ is compact.  For part 4, assume $A(X)$ is a ghostbuster of index $n$, let $X_0\to\ldots \to X_n$ be a string of $n$ composable phantom maps in $\mathcal{C}$, and let $X\to X_0$ be arbitary.  As $A$ preserves colimits, it sends phantom maps to phantom maps by Lemma \ref{phantomobvious}.  We deduce that $A(X)\to A(X_0)\to\ldots A(X_n)$ is homotopic to $0$.  Applying $B$ we get the conclusion.
\end{proof}

\begin{example}
Let $A$ be an associative ring, and let $M$ be a left $A$-module which has projective dimension $n$.  Then $M\in\on{D}(A)$ is a ghostbuster of index $\leq n+1$.  Indeed, a projective module is a ghostbuster of index $1$ by Lemma \ref{ghostbuster} part 3 , so this follows by induction Lemma \ref{ghostbuster} parts 1 and 2.
\end{example}

It turns out that ghostbusters of small index are prevalent under countability conditions.  Under a countability condition on the object, we can say the following.

\begin{example}
Let $\mathcal{C}$ be a presentable stable $\infty$-category.
\begin{enumerate}
\item Suppose $\mathcal{C}$ is compactly generated.  Then any $\aleph_1$-compact object is a ghostbuster of index $2$.  Indeed, an $\aleph_1$-compact object is a sequential colimit of compact objects, so by the Milnor sequence it fits into a cofiber sequence with direct sums of compact objects.
\item More generally, suppose $\mathcal{C}$ is dualizable (see \cite{EfimovKtheory} Section 1). Then again any $\aleph_1$-compact object is a ghostbuster of index $2$.  Indeed, consider the functor $B:\on{Ind}(\mathcal{C}^{\aleph_1})\to\mathcal{C}$ of taking colimits and its left adjoint $A$.  Then $A$ is fully faithful, so $B\circ A\simeq \on{id}$.  Also, since $B$ preserves colimits, $A$ sends $\aleph_1$-compact objects to $\aleph_1$-compact objects.  Thus Lemma \ref{ghostbuster} part 4 reduces the dualizable case to the compactly generated case.
\end{enumerate}
\end{example}

A decidedly deeper fact, due to Neeman (\cite{NeemanPhantom}), is that under a countability assumption on $\mathcal{C}$ itself \emph{every} object is a ghostbuster of index 2:

\begin{theorem}\label{countablephantom}
Suppose that $\mathcal{C}$ is a presentable stable $\infty$-category which satisfies the following countability criterion: there is a countable stable $\infty$-category $\mathcal{C}_0$ (meaning the set of isomorphism classes of objects in $\mathcal{C}_0$ is countable, and the set of homtopy classes of maps between any $X,Y\in\mathcal{C}_0$ is countable) such that $\mathcal{C}$ is a retract, in $\on{Pr}^L_{st}$, of $\on{Ind}(\mathcal{C}_0)$.  Then every $X\in\mathcal{C}$ is a ghostbuster of index 2.
\end{theorem}
\begin{proof}
Again by Lemma \ref{ghostbuster} part 4 we reduce to the case $\mathcal{C}=\on{Ind}(\mathcal{C}_0)$.  But in that case, by \cite{NeemanPhantom} Thm.\ 5.1 and Lemma 4.1, every object is the cofiber of a map between direct sums of compact objects, so this follows from Lemma \ref{ghostbuster}.
\end{proof}

\begin{example}\label{countableexamples}
\begin{enumerate}
\item Suppose $X$ is a second-countable compact Hausdorff space.  Then $\on{Sh}(X;\on{Sp})$ satisfies the countability criterion of Theorem \ref{countablephantom}, and hence every object in $\on{Sh}(X;\on{Sp})$ is a ghostbuster of index 2.  Indeed, we can prove this by the following variant of the argument in \cite{LurieHTT} 7.3.4.  Let $\mathcal{K}$ denote a countable collection of compact subsets of $X$ which forms a basis for the topology, closed under finite intersection and finite union, and view $\mathcal{K}$ as a poset.  Let $X_\mathcal{K}$ denote the spectral space whose poset of quasicompact open subsets identifies with $\mathcal{K}$, via a fixed bijection which we will denote $K\in \mathcal{K} \mapsto [K]\subset X_\mathcal{K}$.  There is a unique continuous map $\pi:X_\mathcal{K}\to X$ such that for $U\subset X$ open, we have
$$\pi^{-1}(U) = \cup_{K\subset U} [K].$$
For the induced pullback $\pi^\ast$ and pushforward $\pi_\ast$ functors for sheaves of anima, we have that:
\begin{enumerate}
\item $(\pi^\ast\mathcal{F})([K]) = \varinjlim_{K\subset U}\mathcal{F}(U)$, where $U$ runs over open neighborhoods of $K$.  Indeed, this describes the presheaf pullback, but one checks it's a sheaf.
\item $\on{id}\overset{\sim}{\rightarrow} \pi_\ast \pi^\ast$.  Indeed, this follows readily from (a) and the description of $\pi_\ast$ coming from the definition of $\pi$.
\item $\pi_\ast$ commutes with filtered colimits and Postnikov truncations.  Indeed, by (b) $\pi^\ast$ is conservative, and as $\pi^\ast$ preserves colimits and Postnikov truncation on general grounds we reduce to checking that $\pi^\ast \pi_\ast$ preserves filtered colimits and Postnikov truncations, which is clear from (a).
\end{enumerate}
It follows that on the level of sheaves of spectra we have the same properties, and $\pi_\ast$ commutes with arbitrary colimits.  Thus $\on{Sh}(X;\on{Sp})$ is a retract of $\on{Sh}(X_\mathcal{K};\on{Sp})$ in $\on{Pr}^L_{st}$.  To finish, we note that $\on{Sh}(X_\mathcal{K};\on{Sp})$ is compactly generated by the Yoneda images of the $[K]$, and that Homs in all degrees between these are countable, for example by writing the site of quasicompact opens as a countable filtered colimit of finitary sites associated to finite posets.
\item Let $X = \varprojlim X_n$ be a 1-truncated light profinite anima (see \ref{profiniteanima}).  Suppose that $X$ has finite cohomological dimension in the sense that there is a $d\in\mathbb{N}$ such that for all $M\in\on{Sh}_{et}(X;\on{D}(\mbb{Z}))$ in degree $0$, we have $\Gamma(X;M)\in \on{D}(\mbb{Z})_{\geq -d}$.  Then $\on{Sh}_{et}(X;\on{Sp})$ satisfies the countability criterion of Theorem \ref{countablephantom}, and hence every object is a ghostbuster of index $2$.

Indeed, Corollary \ref{truncatedcorollary} shows that $\on{Sh}_{et}(X;\on{Sp})$ identifies with the left-completion of the $\infty$-category of sheaves of spectra on the finite etale site of $X$.  Now note that the cohomological dimension estimate passes to any finite etale $f:Y\to X$ because $f_\ast$ is t-exact (check on pullback to light profinite sets).  Then from \cite{ClausenHyper} Cor. 2.20 we deduce that for sheaves of spectra on the finite etale site of $X$, the hypercomplete full subcategory is already left-complete and hence identifies with $\on{Sh}_{et}(X;\on{Sp})$.   Moreover \cite{ClausenHyper} Prop.\ 2.28 gives that this category is compactly generated by the $\mbb{S}[h_{Y\to X}]$ as $Y\to X$ runs over the finite etale maps to $X$.  To finish we note that homs between these are countable as one sees by reducing to cohomology using the Postnikov tower.
\end{enumerate}
\end{example}

We will be more concerned with a variant of part 1 of this example in the setting of etale sheaves.  Recall from Lemma \ref{etaleontopological} that if $X$ is a second-countable compact Hausdorff space, then $\on{Sh}_{et}(\underline{X})$ identifies with the Postnikov completion of $\on{Sh}(X)$.  It follows that $\on{Sh}_{et}(X;\on{Sp})$ identifies with with the left completion of $\on{Sh}(X;\on{Sp})$ equipped with its natural t-structure.  (See Section \ref{sixsec} for the notion of etale sheaves with coefficients.)  Using this we show the following.

\begin{proposition}\label{chauscohdim}
Let $X$ be a second-countable compact Hausdorff space which has finite cohomological dimension, in the sense that there is a $d\geq 0$ such that for any sheaf of abelian groups $\mathcal{A}$ on $X$, we have $H^i(X;\mathcal{A})=0$ for $i>d$.

Then $\on{Sh}_{et}(X;\on{Sp})$ satisfies the countability hypothesis from Theorem \ref{countablephantom}, so the composition of any two phantom maps is $0$.
\end{proposition}
\begin{proof}
We recall from Lemma \ref{etaleontopological} that $\on{Sh}_{et}(X;\on{Sp})=\widehat{\on{Sh}}(X;\on{Sp})$, the left completion of sheaves of spectra on $X$.  By the commutation with Postnikov towers in Example \ref{countableexamples} part 1, we deduce that $\on{Sh}_{et}(X;\on{Sp})$ is, via $\pi_\ast$ and $\pi^\ast$, a retract of the left completion $\widehat{\on{Sh}}(X_\mathcal{K};\on{Sp})$ of $\on{Sh}(X_{\mathcal{K}};\on{Sp})$.  Moreover $\pi_\ast$ still commutes with colimits on the level of left-completions: it suffices to show preservation of direct sums, but this follows again from the t-exactness.  Thus it suffices to show that $\widehat{\on{Sh}}(X_{\mathcal{K}};\on{Sp})$ is compactly generated and the full subcategory of compact objects is countable.

Since again $\pi_\ast$ is t-exact, we deduce that $X_\mathcal{K}$ is also of cohomological dimension $\leq d$.  But moreover, the cohomological dimension hypothesis passes to any closed subset $K$ of $X$, because the pushforward from a closed subset is $t$-exact.   Since $K_{\mathcal{K}'} = [K]\subset X_\mathcal{K}$ where $\mathcal{K}'$ denotes the set of elements of $\mathcal{K}$ which lie in $K$, we deduce in fact that for all $K\in\mathcal{K}$, the quasicompact open subset $[K]\subset X_{\mathcal{K}}$ has cohomological dimension $\leq d$.  Then from \cite{ClausenHyper} Cor 2.20 and Prop 2.28 we deduce that the left completion of $\on{Sh}(X_\mathcal{K};\on{Sp})$ agrees with the hypercompletion, and the $\mathbb{S}[h_{[K]}]$ form a set of compact generators.  To finish we need to see that Hom's between shifts of these compact generators are countable.  Replacing $X$ by $K$, we reduce to showing that $\Gamma(X_\mathcal{K};\varprojlim_d \tau_{\leq d}\mathbb{S}[h_{[K]}])$ has countable homotopy groups for any $K\in\mathcal{K}$.  By the cohomological dimension hypothesis this limit stabilizes in any degree, and then we can again write the site of quasicompact opens as a countable filtered colimit of finitary sites associated to finite posets to conclude.
\end{proof}

\begin{remark}
There do exist second-countable compact Hausdorff spaces which have infinite covering dimension, but finite cohomological dimension, see \cite{DranishnikovDimension} Section 7.  (See also \cite{GrothendieckTohoku} for the equality of the notion of cohomological dimension used in \cite{DranishnikovDimension} and the one we use here based on sheaf cohomology.)
\end{remark}

We have just seen that phantom maps are often close enough to being nullhomotopic.  For this to be useful we also need to give criteria for proving that a given map is phantom.  Here is a helpful statement for that.

\begin{lemma}\label{tstructureargument}
Let $\mathcal{C}$ be a stable $\infty$-category equipped with a t-structure $(\mathcal{C}_{\geq 0},\mathcal{C}_{\leq 0})$.  Denote the heart by $\mathcal{C}^\heart$, the homotopy group functors by $\pi^\heart_d:\mathcal{C}\to\mathcal{C}^\heart$ for $d\in\mathbb{Z}$, and the truncation functors by $\tau_{\geq d},\tau_{\leq d}:\mathcal{C}\to\mathcal{C}$ for $d\in\mathbb{Z}$.  Suppose $F:\mathcal{C}\to\on{Sp}$ is an exact functor (in the sense of preserving finite limits and colimits, not in the sense of the t-structures) which satisfies the following conditions:

\begin{enumerate}
\item $F(\mathcal{C}^{\heart}) \subset \on{Sp}_{[-n,0]}$
for some $n\in\mathbb{N}$;
\item $\varinjlim_d F(\tau_{\geq -d}X)\overset{\sim}{\rightarrow} F(X)\overset{\sim}{\rightarrow} \varprojlim_d F(\tau_{\leq d}X)$ for all $X\in\mathcal{C}$.
\end{enumerate}

Then if $X_0\to\ldots\to X_{n+1}$ is any sequence of $n+1$ composable maps in $\mathcal{C}$, each of which induces the $0$ map on $\pi_d^\heart$ for all $d\in\mathbb{Z}$, we have that the composition
$$\pi_0 F(X_0)\to\ldots \to \pi_0 F(X_n)$$
is $0$.
\end{lemma}
\begin{proof}
Let us first show that if $X\in\mathcal{C}$, then $\pi_0 F(\tau_{\geq n+1}X)=0$ and $\pi_0 F(\tau_{\leq -1}X)=0$.   Let $\alpha\in \pi_0 F(\tau_{\geq n+1}X)$.  By 1, the image of $\alpha$ in $\pi_0 F(\pi^{\heart}_{n+1}[n+1])$ is 0, hence $\alpha$ lifts to some $\alpha_1 \in \pi_0 F(\tau_{\geq n+2}X)$, which by the same reasoning lifts to $\alpha_2 \in \pi_0 F(\tau_{\geq n+3})$.  Continuing in this way, we find that $\alpha$ lifts to a class in $\pi_0 \varprojlim_n F(\tau_{\geq n}X)$.  But this group is $0$ by 2, hence $\alpha=0$ as desired.  Now let $\alpha \in \pi_0 F(\tau_{\leq -1}X)$.  By 2, $\alpha$ comes form some class in $\pi_0 F(\tau_{[-N,-1]}X)$, so it suffices to show that this group vanishes.  Here we can use descending induction on $N$ and 1 to handle the claim.

Now suppose $X\to Y$ is a map in $\mathcal{C}$ which induces $0$ on $\pi^{\heart}_0$.  By the above vanishing of $\pi_0 F(\tau_{\leq -1}X)$, any $\alpha \in \pi_0F(X)$ lifts to $\alpha' \in \pi_0F(\tau_{\geq 0}X)$.  Since $\tau_{\geq 0}X \to \tau_{\geq 0}Y \to \pi^\heart_0Y$ is $0$, the image of $\alpha'$ in $\pi_0 F(\tau_{\geq 0}Y)$ lifts to $\pi_0 F(\tau_{\geq 1}Y)$.  If $Y\to Z$ is a further map which is $0$ on $\pi^\heart_1$, applying the same argument to this lift in $\pi_0 F(\tau_{\geq 1}Y)$ we get that the image of $\alpha'$ in $\pi_0F(\tau_{\geq 0}Z)$ lifts to $\pi_0 F(\tau_{\geq 2}Z)$. Continuing in this manner, we deduce that if $X\to Y$ is a composition of $d+1$ many maps each of which is $0$ on $\pi^\heart_\ast$, then the image of $\alpha'$ in $\pi_0F (\tau_{\geq 0}Y)$ lifts to $\pi_0 F(\tau_{\geq d+1}Y)$.  But this group is $0$ by the previous paragraph, whence the conclusion.
\end{proof}

To proceed further we recall a class of presentable stable $\infty$-categories $\mathcal{C}$ which are canonically self-dual, permitting an efficient classification of the colimit-preserving functors $\mathcal{C}\to\on{Sp}$ (such as is needed to understand phantom maps).  The definition is due to \cite{GaitsgoryDerivedI}, 1.9.

\begin{definition}\label{defrigid}
Let $\mathcal{C}\in \on{CAlg}(\on{Pr}_{st}^L)$.  Say that $\mathcal{C}$ is \emph{rigid} if the following conditions are satisfied:
\begin{enumerate}
\item The unit object $1\in\mathcal{C}$ is compact.
\item The multiplication map $m:\mathcal{C}\otimes\mathcal{C}\to\mathcal{C}$ in $\on{CAlg}(\on{Pr}^L_{st})$ has a right adjoint which preserves colimits and satisfies the projection formula.
\end{enumerate}
\end{definition}

If $\mathcal{C}$ is rigid, then by \cite{GaitsgoryDerivedI} 9.2, $\mathcal{C}$ is canonically self-dual in $\on{Pr}_{st}^L$ via the coevaluation pairing given by the composition

$$\on{Sp} \overset{1}{\rightarrow} \mathcal{C} \overset{m^R}{\rightarrow} \mathcal{C}\otimes \mathcal{C}$$

\noindent where the first map is the unit and the second map is the right adjoint $m^R$ to the functor $m:\mathcal{C}\otimes\mathcal{C}\to\mathcal{C}$.  The evaluation pairing can be described as follows: if we define $\Gamma:\mathcal{C}\to\on{Sp}$ to be the functor $\Gamma = \on{map}(1,-)$ or in other words the right adjoint to the functor $\on{Sp}\to\mathcal{C}$ classifying the unit, then the evaluation is given by the composition
$$\mathcal{C}\otimes\mathcal{C}\overset{m}{\rightarrow}\mathcal{C}\overset{\Gamma}{\rightarrow}\on{Sp}.$$

\label{keypointrigid} In particular, and this is the key point for us, every colimit-preserving functor $F:\mathcal{C}\to\on{Sp}$ is of the form
$$F(-) \simeq \Gamma(X\otimes -)$$
for some (unique) $X\in\mc{C}$.

\begin{example}\label{rigidexamples}
\begin{enumerate}
\item Suppose $\mathcal{C}\in\on{CAlg}(\on{Pr}^L_{st})$ is compactly generated.  Then $\mathcal{C}$ is rigid if and only if for $X\in\mathcal{C}$, we have that $X$ is compact $\Leftrightarrow$ $X$ is dualizable.  (Or equivalently if $1\in\mathcal{C}$ is compact, and $\mathcal{C}$ is generated by objects which are both compact and dualizable.)  See \cite{GaitsgoryDerivedI} Lemma 9.1.5.
\item For a compact Hausdorff space $X$, the $\infty$-category $\on{Sh}(X;\on{Sp})$ is rigid.  This follows formally from the proper base-change theorem of \cite{LurieHTT} 7.3.1.18 and the Kunneth property $\on{Sh}(X;\on{Sp})\otimes\on{Sh}(X;\on{Sp})\overset{\sim}{\rightarrow}\on{Sh}(X\times X;\on{Sp})$ of \cite{LurieHTT} 7.3.1.11.
\item Suppose $X=\varprojlim X_n$ is a 1-truncated light profinite anima such that $X$ has finite cohomological dimension.  Then the $\infty$-category $\on{Sh}_{et}(X;\on{Sp})$ is rigid.

Indeed, recall from Example \ref{countableexamples} that $\on{Sh}_{et}(X;\on{Sp})$ is compactly generated by the $f_\natural\mbb{S}$ where $f_\natural$ is the left adjoint to $f^\ast$ for $f:Y\to X$ finite etale.  By example 1, to finish the proof it suffices to show that these are dualizable.  In fact this holds generally for an arbitrary finite etale map.  Indeed by descent we can reduce to the case where the target is light profinite set, and then it is obvious because such finite etale maps are split.

\item If $R$ is any $E_\infty$-ring, then $D(R)$ is rigid.  Indeed, the unit is compact and dualizable and generates, so this follows from 1.

\item If $\mathcal{C},\mathcal{D}\in\on{CAlg}(\on{Pr}^L_{st})$ are both rigid, then so is $\mathcal{C}\otimes\mathcal{D}$.  Indeed the projection formula is stable under tensor product.
\end{enumerate}
\end{example}

We will want to use the following etale sheaf variant of  part 2 of the above example.

\begin{proposition}
Suppose $X$ is a second-countable compact Hausdorff space of finite cohomological dimension.  Then $\on{Sh}_{et}(\underline{X};\on{Sp})$ is rigid.
\end{proposition}
\begin{proof}
First we note that for a map $f:X\to Y$ of compact Hausdorff spaces such that the fibers have uniformly bounded cohomological dimension $\leq d$, the proper base change theorem holds for the left completion $\widehat{Sh}(-;\on{Sp})$.  Indeed, for a truncated sheaf of spectra $\mathcal{F}$ on $X$, the pushforward $f_\ast\mathcal{F}$ is also truncated, and by proper base-change and the cohomological dimension hypothesis we deduce that $\tau_{\leq n} f_\ast \mathcal{F} \overset{\sim}{\rightarrow} \tau_{\leq n} f_\ast \tau_{\leq N}\mathcal{F}$ for $N> n+d$.  Thus for an object $(\mathcal{F}_n)_n$ in the left-completion $\widehat{\on{Sh}}(X;\on{Sp})$, the left-complete pushforward $f_\ast((\mathcal{F}_n)_n)$ identifies with the object $(\tau_{\leq n} f_\ast \mathcal{F}_{n+d+1})_n$, and so the conclusion follows again from (noncomplete) proper base-change.

Thus rigidity of $\on{Sh}_{et}(\underline{X})=\wh{\on{Sh}}(X;\on{Sp})$ will follow provided we can show that we can show the Kunneth property, namely that if $X$ and $Y$ are compact Hausdorff spaces of finite cohomological dimension, then
$$\wh{\on{Sh}}(X;\on{Sp})\otimes \wh{\on{Sp}}(Y;\on{Sp})\overset{\sim}{\rightarrow} \wh{\on{Sp}}(X\times Y;\on{Sp}).$$
For this, as in the proof of Proposition \ref{chauscohdim} it suffices to prove the following spectral space variant: if $X,Y$ are spectral spaces with a basis of quasi-compact open subsets of uniformly bounded cohomological dimension, then $\wh{\on{Sh}}(X;\on{Sp})\otimes\wh{\on{Sh}}(Y;\on{Sp})$ identifies with $\widehat{\on{Sh}}(X\times Y;\on{Sp})$.   Note that the non-complete variant holds, for example by writing $X$ and $Y$ as filtered inverse limits of finite posets and using \cite{ClausenHyper} Lemma 3.3.  Looking at $\mathbb{Z}$-coefficients, it follows in particular that the uniform boundedness of cohomological dimension passes to $X\times Y$.  Now we compare compact generators.  The objects $\mathbb{S}[h_U]$, $\mathbb{S}[h_V]$, and $\mathbb{S}[h_{U\times V}]$ give compact generators of respectively $\wh{\on{Sh}}(X;\on{Sp})$, $\wh{\on{Sh}}(Y;\on{Sp})$, and $\wh{\on{Sh}}(X\times Y;\on{Sp})$ by \cite{ClausenHyper} Prop 2.28, where $U,V$ run over quasicompact basic opens in $X,Y$ respectively.  Note that the mapping spectra between these compact generators are all bounded below, by Postnikov completeness and the finite cohomological dimension.  Thus we can check our Kunneth property after base-change to $D(\mathbb{Z})$.  But then there is no distinction between hypercomplete = Postnikov-complete sheaves and arbitrary sheaves because the generators are truncated and hence automatically hypercomplete, see \cite{ClausenHyper} Prop.\ 2.28.
\end{proof}

Putting things together, we get the following criterion for descendability.

\begin{theorem}\label{descendabletheorem}
Let $X$ be a light condensed anima.  Suppose that the following hypotheses are satisfied:
\begin{enumerate}
\item $\on{Sh}_{et}(X;\on{Sp})$ is rigid and satisfies the countability hypothesis from Theorem \ref{countablephantom}.
\item $X$ has finite cohomological dimension: there is a $d\in\mathbb{N}$ such that for all $M\in \on{Sh}_{et}(X;\on{D}(\mathbb{Z}))$ concentrated in degree $0$, we have $\Gamma(M)\in\on{D}(\mathbb{Z})_{\geq -d}$.
\end{enumerate}
Then for any $A\in \on{CAlg}(\on{Sh}_{et}(X;\on{Sp}))$, the following conditions are equivalent:
\begin{enumerate}
\item $A$ is descendable;
\item there is an $N\in\mathbb{N}$ such that for all maps $x:\ast\to X$, the pullback $x^\ast A \in \on{CAlg}(\on{Sp})$ is descendable of index $\leq N$.
\end{enumerate}
\end{theorem}
\begin{proof}
If condition 1 holds, then $A$ is descendable of index $\leq N$ for some $N$.  This condition is preserved by symmetric monoidal functors, so condition 2 holds.  For the converse, we use the usual t-structure on $\on{Sh}_{et}(X;\on{Sp})$ coming from sheafified Postnikov towers.  The heart is the category of etale sheaves of abelian groups.  Consider the following classes of maps:
\begin{enumerate}
\item The $A$-null maps.
\item The maps inducing $0$ on homotopy groups (with respect to the t-structure).
\item The maps $f$ such that for all $X\in\mathcal{C}$, the map $f\otimes X$ is $0$ on homotopy groups.
\item The phantom maps.
\item The nullhomotopic maps.
\end{enumerate}
We first claim that any $N$-fold composition of $A$-null maps induces $0$ on homotopy groups.  As pullbacks are t-exact and detect being $0$ on abelian group sheaves, this follows from hypothesis 2.  Since the class of $A$-null maps is closed under $-\otimes X$, we even see that any $N$-fold composition of $A$-null maps satisfies the stronger condition 3 just above.  Next we claim that any $d+1$-fold composition of maps satisfying 3 is phantom.  By \ref{keypointrigid}, all colimit-preserving functors to spectra are of the form $\Gamma(X\otimes -)$ for some $X\in\mathcal{C}$, so for this it suffices to show that a $d+1$-fold composition of maps which induce $0$ on homotopy groups gets sent to a null map by $\Gamma$.  For that we can invoke Lemma \ref{tstructureargument}.  Since $\Gamma$ preserves colimits (because the unit is compact) and the t-structure is right complete, we have $\Gamma(X) = \varinjlim_d \Gamma(\tau_{\geq -d}X)$.  On the other hand, $\Gamma$ also preserves limits (because it is a right adjoint) and the t-structure is left-complete, so we have $\Gamma(X) = \varprojlim_d \Gamma(\tau_{\leq d}X)$.  Finally, the last hypothesis in Lemma \ref{tstructureargument} is satisfied by our cohomological dimension assumption.  Thus all the hypotheses of Lemma \ref{tstructureargument} are satisfied and we get the desired conclusion.  Finally, we claim that any 2-fold composition of phantom maps is null.  This follows from Theorem \ref{countablephantom} and the countability hypothesis.

In total, we have seen that any $2N(d+1)$-fold composition of $A$-null maps is null, hence $A$ is descendable.
\end{proof}

\begin{remark}\label{pcompletesemirigid}
Let $p$ be a prime.  There is also a $p$-complete analog of Theorem \ref{descendabletheorem}, but since $\on{Sp}_{\wh{p}}$ is not itself rigid we have to state it a bit differently.  We can use the concept of a \emph{semi-rigid} category from \cite{ArinkinLanglands} App.\ C (also called \emph{locally rigid} in \cite{EfimovKtheory} and \cite{RamziRigid}). This is a $\mathcal{C}\in\on{CAlg}(\on{Pr}^L_{st})$ which satisfies condition 2 in Definition \ref{defrigid} and is moreover dualizable as an object in $\on{Pr}^L_{st}$.  Any $\mathcal{C}\in\on{CAlg}(\on{Pr}^L_{st})$ which is compactly generated by dualizable objects is semi-rigid, and if $\mathcal{C}$ is semi-rigid and $p$-complete and $1/p$ is compact, then $\mathcal{C}$ is self-dual via $\Gamma(-)[p^\infty]$, as follows from \cite{RamziRigid} Cor.\ 4.7.2. (This applies in particular to $\mathcal{C}=\on{Sp}_{\wh{p}}$.)

Then if we assume $\on{Sh}_{et}(X;\on{Sp}_{\wh{p}})$ is semi-rigid and $X$ has finite (mod $p$) cohomological dimension, we see from a Postnikov tower argument that $1/p$ is compact, and then we can run a similar argument as in the above proof to deduce once more that (uniform) descendability is detected on pullback to points.
\end{remark}

\begin{remark}\label{cohdimfibers}
The following criterion can be useful for establishing finite cohomological dimension.  Suppose $X$ is a light condensed anima which admits a hypercover by light profinite sets (for example, $X$ could be a second-countable compact Hausdorff space or a light profinite anima).  Suppose further that we have a map $f:X \to S$ to a light profinite set, and for every $s:\ast \to S$ the fiber $X_s$ has finite cohomological dimension $\leq d$, with $d$ being independent of $s$.  Then $X$ has cohomological dimension $\leq d$.

Indeed, for a truncated $\on{D}(\mbb{Z})$-valued sheaf $\mathcal{F}$ we can deduce base-change for $f_\ast\mathcal{F}$ by using the hypercover and the fact that base-change for lower-$\ast$ holds for maps of light profinite sets.  Then we can make use of the fact that pullbacks to points are t-exact (as always) and jointly detect isomorphisms (Remark \ref{checkonstalks}) to conclude.

The same remark applies to mod $p$ cohomological dimension.
\end{remark}

\section{The six functor formalism for etale sheaves}\label{sixsec}

In this section we produce a six functor formalism for etale sheaves on light condensed anima, following \cite{HeyerSixFunctors}.  Recall that in the set-up of that article, a six functor formalism consists of the following data:
\begin{enumerate}
\item An $\infty$-category $\mathcal{C}$ which has all finite products;
\item A collection $E$ of morphisms in $\mathcal{C}$ which is \emph{geometric}, i.e.\ closed under composition, closed under base-change, contains all isomorphisms, and is closed under passing to diagonals.
\item A lax symmetric monoidal functor $M:\on{Span}_E(\mathcal{C})\to \on{Pr}^L$.
\end{enumerate}

Here $\on{Span}_E(\mathcal{C})$ is a certain natural $\infty$-category whose objects are the objects in $\mathcal{C}$ and where morphisms from $X$ to $Y$ are spans
$$X\leftarrow T\rightarrow Y$$
with $T\rightarrow Y$ in $E$, and where composition is induced by pullbacks of spans.  This $\infty$-category $\on{Span}_E(\mathcal{C})$ carries a symmetric monoidal structure induced by the product in $\mathcal{C}$, and we view $\on{Pr}^L$ as symmetric monoidal via the Lurie tensor product.

Thus a six functor formalism assigns a presentable $\infty$-category $M(X)$ to every $X\in\mathcal{C}$ and a colimit-preserving functor $f^\ast:M(Y)\to M(X)$ to every $f:X\to Y$ in $\mathcal{C}$ and a colimit-preserving functor $g_!:M(X)\to M(Y)$ to every map $g:X\to Y$ in $E$.  The functoriality of $M$ implies that $(-)^\ast$ and $(-)_!$ are separately functorial, and (crucially) that the formation of $g_!$ commutes with $f^\ast$ in a base-change square.  The symmetric monoidality gives the exterior product operation:
$$M(X)\otimes M(Y)\to M(X\times Y),$$
such that composing with pullback along the diagonal promotes each $M(X)$ to a presentably symmetric monoidal $\infty$-category.  We refer to \cite{HeyerSixFunctors} for more details on all of this.

Generally, one produces six functor formalisms in a multi-step procedure.  First one starts with the situation in which $g_!$ can be taken to be $g_\ast$, and hence the whole data is in principle determined just by the underlying functor $M^\ast:\mathcal{C}^{op}\to\on{CAlg}(\on{Pr}^L)$ by some fancy procedure of passing to adjoints.  Then, in the subsequent steps, one applies various formal procedures allowing to enlarge the class of maps $E$ allowed in the formalism.  In \cite{HeyerSixFunctors} a convenient result is proved which axiomatizes this procedure in a way which is sufficient for our applications.

We are primarily interested in etale sheaves of $p$-complete spectra , but for extra flexibility we consider a more general set-up.  Following \cite{LurieSAG} 21.1.2, we single out a class of reasonable ``coefficient categories'' which includes all examples of interest.

\begin{definition}
Suppose $\mathcal{C}$ is a compactly assembled presentable $\infty$-category.  For $X\in\on{CondAn}^{light}$, define
$$\on{Sh}_{et}(X;\mathcal{C}) = \on{Sh}_{et}(X)\otimes\mathcal{C}$$
and
$$\on{Sh}_v(X;\mathcal{C}) = \on{Sh}_v(X)\otimes \mathcal{C}.$$
\end{definition}

\begin{remark}
By \cite{LurieSAG} 1.3.1.6 this is consistent notation: if $\mathcal{X}$ is an $\infty$-topos which identifies with the $\infty$-topos of hypercomplete sheaves of anima on some site, then $\mathcal{X}\otimes \mathcal{C}$ identifies with the $\infty$-category of hypercomplete $\mathcal{C}$-valued sheaves on that same site.  But we use the tensor product perspective to keep better track of the functoriality.
\end{remark}

The reason we restrict to compactly assembled $\mathcal{C}$ is because of the following favorable properties.

\begin{lemma}\label{compassgood}
Let $\mathcal{C}$ be a compactly assembled presentable $\infty$-category.
\begin{enumerate}
\item The functor
$$\on{Sh}_{et}(X;\mathcal{C})\to\on{Sh}_v(X;\mathcal{C}),$$
gotten by tensoring the inclusion $\on{Sh}_{et}(X;\mathcal{C})\to\on{Sh}_v(X;\mathcal{C})$ with $\mathcal{C}$, is fully faithful.
\item The functors $X\mapsto \on{Sh}_{et}(X;\mathcal{C})$ and $X\mapsto \on{Sh}_v(X;\mathcal{C})$ from $\on{CondAn}^{op}$ to $\on{Pr}^L$ preserve arbitrary limits.
\item Let $X\in\on{CondAn}^{light}$.  A map $\mathcal{F}\to\mathcal{G}$ in $\on{Sh}_{et}(X;\mathcal{C})$ is an isomorphism if and only if $x^\ast\mathcal{F}\to x^\ast\mathcal{G}$ is an isomorphism for all $x:\ast \to X$.
\end{enumerate}
\end{lemma}
\begin{proof}
Because every compactly assembled presentable $\infty$-category is a retract in $\on{Pr}^L$ of a compactly generated $\infty$-category (\cite{LurieSAG} 21.1.2), we can reduce to the compactly generated case, so assume $\mathcal{C}$ compactly generated.  Note that if $\mathcal{X}$ is a presentable $\infty$-category, then
$$\mathcal{X}\otimes\mathcal{C}\simeq \on{Fun}^{rex}((\mathcal{C}^{\aleph_0})^{op},\mathcal{X}),$$
where $\on{Fun}^{rex}$ means the $\infty$-category of functors which preserve finite limits.  Indeed we have $\mathcal{C}\simeq\on{Fun}^{rex}((\mathcal{C}^{\aleph_0})^{op},\on{An})$ by sending $c\in C$ to the functor $c_0\mapsto \on{Map}(c_0,c)$.  Using that $\mathcal{X}$ is tensored over $\on{An}$, this produces an obvious comparison functor $\mathcal{X}\otimes \mathcal{C}\to \on{Fun}^{rex}((\mathcal{C}^{\aleph_0})^{op},\mathcal{X})$.  It is an isomorphism by \cite{LurieHA} Prop.\ 4.8.1.17.  But note that from this definition of the comparison functor, we see that it is functorial in maps between presentable $\infty$-categories $\mathcal{X}\to\mathcal{Y}$ which preserve colimits and finite limits.  Every pullback map of $\infty$-topoi is of this form by definition, and thus we reduce all three claims 1,2,3 to the case $\mathcal{C}=\on{An}$.  Then claims 1 and 2 follow from the definition, while 3 is given by Remark \ref{checkonstalks}.
\end{proof}

\begin{remark}\label{etinsidev}
If $\mathcal{C}$ is compactly generated, then the essential image of the fully faithful functor in 1 consists of those $\mathcal{C}$-valued $v$-sheaves $\mathcal{F}$ such that for all compact objects $c\in\mathcal{C}$, the $v$-sheaf of anima $\on{Map}(c,\mathcal{F})$ lies in $\on{Sh}_{et}(X)$.  For example, a $v$-sheaf of $p$-complete spectra $\mathcal{F}$ is etale if and only if $\mathcal{F}/p$ is etale as a $v$-sheaf of spectra, which is if and only if $\Omega^\infty (\mathcal{F}/p[n])$ is etale as a $v$-sheaf of anima for all $n\in\mathbb{Z}$, or again (see Lemma \ref{etaleproperties}) if and only if each homotopy group sheaf  of $\mathcal{F}/p$ is etale as an abelian group valued sheaf.  This follows because $\mathbb{S}/p$ is a compact generator of $\on{Sp}_{\wh{p}}$.
\end{remark}

In the abstract context, our ``base category'' for our six functor formalism will therefore be some $\mathcal{R}\in\on{CAlg}(\on{Pr}^L)$ which is compactly assembled (as an underlying $\infty$-category).

Following the procedure outlined above, a sufficient supply of proper maps to start out with is guaranteed by the following simple result.

\begin{lemma}
Let $f:T\to S$ be a map of light profinite sets.  Then the right adjoint $f_\ast:\on{Sh}_{et}(T;\on{Sp})\to\on{Sh}_{et}(S;\on{Sp})$
to the pullback of etale sheaves of spectra preserves all colimits, satisfies the projection formula, and commutes with base-change (along any map $S'\to S$ in $\on{ProfSet}^{light}$).
\end{lemma}
\begin{proof}
By the comparison of etale sheaves on light profinite sets with sheaves on the underlying compact Hausdorff space (Lemma \ref{etaleontopological}), this follows from Lurie's proper base change theorem (\cite{LurieHTT} 7.3).
\end{proof}

\begin{remark}
Lurie's proper base change theorem is of course quite elementary in the context of profinite sets. See \cite{HeyerSixFunctors} 3.5 for another perspective, by identifying sheaves of spectra on a profinite set $S$ with modules over the ring $\Gamma(S;\mbb{S})$.  Yet another perspective, that of \cite{ScholzeDiamonds}, is the following: it is tautological from slice topos formalism that the pushforward of $v$-sheaves commutes with pullbacks and satisfies the projection formula; thus to see that the etale pushforward does as well, it suffices to show that the pushforward of $v$-sheaves sends etale sheaves to etale sheaves.  This is a simple consequence of the criterion of Theorem \ref{etale}.
\end{remark}
\begin{remark}
For any $\mathcal{R}\in\on{CAlg}(\on{Pr}_{st}^L)$, if we tensor this lemma with $\mathcal{R}$ we deduce the analogous lemma for $\mathcal{R}$-valued etale sheaves.
\end{remark}

From this result and \cite{HeyerSixFunctors} Prop.\ 3.3.3, we get a (completely explicit) six functor formalism on $\on{ProfSet}^{light}$ with $E$ the class of all maps such that:
\begin{enumerate}
\item The underlying functor $M^\ast:(\on{ProfSet}^{light})^{op}\to\on{CAlg}(\on{Pr}^L)$ is $\on{Sh}_{et}(-;\mathcal{R})$ (for any $\mathcal{R}\in\on{CAlg}(\on{Pr}^L_{st})$ compactly assembled).
\item For all maps $f:X\to Y$ in $\on{ProfSet}^{light}$, the functor $f_!$ identifies with $f_\ast$, and the base-change isomorphisms for $f_!$ (coming from the six functor formalism) and for $f_\ast$ (coming from adjunction) agree.
\end{enumerate}
There is a more precise version of 2: in this six functor formalism every map is \emph{proper} in the sense of \cite{HeyerSixFunctors} Def.\ 4.6.1 (we will also review this notion below).  It follows by \cite{DauserUniqueness} that this is the unique six functor formalism on $\on{ProfSet}^{light}$ such that 1 holds and every map is proper.

Now we apply \cite{HeyerSixFunctors} Thm.\ 3.4.11 to extend this six functor formalism to all of $\on{CondAn}^{light}$.  In the statement we make use the notion of a \emph{$!$-cover} associated to a given six functor formalism
$$M:\on{Span}_E(\mathcal{C})\to\on{Pr}^L.$$
The definition is as follows:  a collection of maps $F=\{f_i:X_i\to X\}_{i\in I}$ in $\mathcal{C}$ with common target is a \emph{$!$-cover} (called a universal $!$-cover in \cite{HeyerSixFunctors} Def.\ 3.4.6) if:
\begin{enumerate}
\item Each $f_i$ lies in $E$;
\item For any map $X'\to X$ in $\mathcal{C}$, if $F':=\{f'_i: X'_i\to X'\}_{i\in I}$ denotes the pullback, then the co-presheaf $M_!: E \to \on{Pr}^L$ satisfies the cosheaf condition for the sieve on $E_{/X'}$ generated by $F'$.  Here $E$ is the wide subcategory of $\mathcal{C}$ corresponding to $E$, and $M_!$ is the functor given by the $!$-pushforwards from the six functor formalism.
\end{enumerate}

\begin{theorem}\label{sixfunctorextension}
Let $\mathcal{R}\in\on{CAlg}(\on{Pr}_{st}^L)$ such that $\mathcal{R}$ is compactly assembled.  Then there exists a minimal geometric class of maps $E$ of light condensed anima satisfying the following conditions:
\begin{enumerate}
\item There is a unique six functor formalism $M:\on{Span}_E(\on{CondAn}^{light})\to \on{Pr}^L$ equipped with an identification of the underlying functor $M^\ast:(\on{CondAn}^{light})^{op}\to \on{CAlg}(\on{Pr}^L)$ with $\on{Sh}_{et}(-;\mathcal{R})$ such that $E$ contains all maps between light profinite sets, and all maps between light profinite sets are \emph{proper}.
\item Suppose that $f:X\to Y$ is a map of light condensed anima such for all light profinite sets $S$ and all maps $S\to Y$, the pullback $X\times_Y S \to S$ lies in $E$.  Then $f$ lies in $E$.
\item Suppose that $f:X\to Y$ is a map of light condensed anima which, $!$-locally on source or target (with respect to the six functor formalism $M$ from 1), lies in $E$.  Then $f$ lies in $E$.
\item If $f:X\to Y$ lies in $E$ and $Y\in\on{ProfSet}^{light}$, then $X$ admits a $!$-cover by light profinite sets.
\end{enumerate}
By definition, we will say that a map $f$ of light condensed anima is \emph{$\mathcal{R}$-$!$-able} (or \emph{$\mathcal{R}$-fine}, in the language of \cite{HeyerSixFunctors}) if it lies in this class $E$.
\end{theorem}
\begin{proof}
By Lemma \ref{compassgood}, $\on{Sh}_{et}(-;\mathcal{R}):(\on{CondAn}^{light})^{op}\to\on{CAlg}(\on{Pr}^L)$ is right Kan extended from $\on{ProfSet}^{light}$.  Thus condition 1 is equivalent to the following condition: there is a unique six functor formalism $M:\on{Span}_E(\on{CondAn}^{light})\to \on{Pr}^L$ such that $E$ contains all maps of light profinite sets, $M^\ast$ preserves limits, and $M$ is equipped with an isomorphism to the basic six functor formalism discussed above on restriction to light profinite sets.  Then the result follows from a direct application of \cite{HeyerSixFunctors} Thm.\ 3.4.11.
\end{proof}

\begin{remark}\label{actuallyconstruct}
From the proof of Thm.\ 3.4.11 of \cite{HeyerSixFunctors}, one sees that the class of maps $E$ is produced by starting with the class $E_0$ of maps of light condensed anima whose every pullback to a light profinite set is light profinite, and then iteratively extending $E_0$ by two simple extension procedures (and, when necessary, taking unions over a chain of previously constructed classes of maps).  For each class of maps produced one has the uniqueness claim as in 1, allowing to talk about $!$-covers, etc.  The extension procedures are as follows:
\begin{enumerate}
\item Given a class of maps $E'$, enlarge to $E'_!$, the class of maps which are $!$-locally on the source in $E'$.
\item Given a class of maps $E'$, enlarge it to $E'_\ast$, the class of maps such that each pullback to a light profinite set lies in $E'$.
\end{enumerate}
\end{remark}

\begin{remark}\label{changecoefficients}
Suppose that $\mathcal{R}\to \mathcal{R}'$ is a map in $\on{CAlg}(\on{Pr}^L_{st})$, where both $\mathcal{R}$ and $\mathcal{R}'$ are compactly assembled.  Let $E$ denote the collection of $\mathcal{R}$-$!$-able maps, and similarly $E'$ the collection of $\mathcal{R}'$-$!$-able maps. Then $E\subset E'$. Moreover, we have the following compatibility property.  Suppose we take the six functor formalism for $\mathcal{R}$-valued etale sheaves, viewed as a lax symmetric mononidal functor
$$\on{Span}_E(\on{CondAn}^{light})\to\on{Mod}_\mc{R}(\on{Pr}^L)$$
(see \cite{HeyerSixFunctors} Lemma 3.2.5), and then apply $-\otimes_\mathcal{R}\mathcal{R}'$.  Then the resulting six functor formalism agrees with the one for $\mathcal{R}'$-valued etale sheaves, restricted from $(\on{CondAn}^{light},E')$ to $(\on{CondAn}^{light},E)$.

Following the previous remark, and as in the proof of \cite{HeyerSixFunctors} Prop.\ 3.5.22, to justify this we need to see that if we have added a new $\mathcal{R}$-$!$-able map by the extension procedures 1 or 2, then it is also in the collection of maps which are $\mathcal{R}'$-$!$-able in the corresponding step for $\mathcal{R}'$.  For procedure 2 this is obvious.  For procedure 1, we need to show that a $!$-cover for $\mathcal{R}$ is also a $!$-cover for $\mathcal{R}'$.  However, $M_!:E \to \on{Pr}^L$ promotes naturally to $M_!: E \to \on{Mod}_\mathcal{R}(\on{Pr}^L)$ (\cite{HeyerSixFunctors} Prop.\ 3.5.22), and then we can apply $-\otimes_\mathcal{R}\mathcal{R}'$ to conclude (recalling that $\on{Mod}_\mathcal{R}(\on{Pr}^L)\to\on{Pr}^L$ preserves and detects colimits).
\end{remark}

There are some important classes of morphisms one can define associated to any six functor formalism.  We spell some of them out now in the context of the six functor formalism from Theorem \ref{sixfunctorextension}.  Recall first that a map in an $\infty$-category is called \emph{truncated} if some iterated diagonal of it is an isomorphism.  Some of these notions will only be defined for truncated maps.

\begin{remark}
In fact, $!$-able morphisms in the six functor formalism from Theorem \ref{sixfunctorextension} have a strong tendency to be truncated (or at least, unions of truncated objects).  Indeed, in practice the $!$-covers one uses to produce new examples of $!$-able maps (via property 3 in Theorem \ref{sixfunctorextension}) are also covers in the light condensed Grothendieck topology, hence new examples of $!$-covers one produces are at most one less degree of truncated than old ones.  This is not a completely satisfactory situation: for example one would like to have that an arbitrary etale map, not necessarily truncated, between light condensed anima is $\on{Sp}$-$!$-able, but this is not true for the six functor formalism given by Theorem \ref{sixfunctorextension}.  Some new principle of construction (or extension) of six functor formalisms would be necessary to ``reach'' such maps.
\end{remark}

\begin{definition}\label{etalepropersmoothdef}
Let $\mathcal{R}\in\on{CAlg}(\on{Pr}^L_{st})$ be compactly assembled.
\begin{enumerate}
\item (\cite{HeyerSixFunctors} Def.\ 4.6.1) A truncated map $f:X\to Y$ of light condensed anima is called \emph{$\mathcal{R}$-proper} if it is $\mathcal{R}$-$!$-able, and either:
\begin{enumerate}
\item $f$ is an isomorphism, in which case we define an isomorphism $f_!\simeq f_\ast$ using the pullback of $f$ along itself and base-change in the pullback square $X\overset{(\on{id},\on{id})}{\longrightarrow} X\times_Y X$;
\item $\Delta_f:X\to X\times_Y X$ is proper and the natural map $f_!\to f_\ast$, coming from the isomorphism $(\Delta_f)_\ast\simeq (\Delta_f)_!$ and base-change in the pullback $X\times_Y X$, is an isomorphism.
\end{enumerate}
Thus by definition a proper map comes with a canonical isomorphism $f_!\simeq f_\ast$.
\item (\cite{HeyerSixFunctors} Def.\ 4.6.1) A truncated map $f:X\to Y$ of light condensed anima is called \emph{$\mathcal{R}$-etale} if it is $\mathcal{R}$-$!$-able, and either:
\begin{enumerate}
\item $f$ is an isomorphism, in which case we define an isomorphism $f^\ast \simeq f^!$ using the pullback of $f$ along itself and $X\overset{(\on{id},\on{id})}{\longrightarrow} X\times_Y X$;
\item $\Delta_f$ is etale and the natural map $f^\ast \to f^!$, coming from the isomorphism $(\Delta_f)^!\simeq (\Delta_f)^\ast$ and the pullback of $f$ along itself, is an isomorphism.
\end{enumerate}
Thus by definition an etale map comes with a canonical isomorphism $f^\ast\simeq f^!$.  Here of course $f^!$ denotes the right adjoint to $f_!$.
\item (\cite{HeyerSixFunctors} Def.\ 4.5.1) A map $f:X\to Y$ of light condensed anima is called \emph{$\mathcal{R}$-smooth} if it is $\mathcal{R}$-$!$-able, and the following two conditions are satisfied:
\begin{enumerate}
\item The formation of $f^!$ commutes with arbitrary base-change (via the natural comparison map coming from base-change for $f_!$), and likewise for any base change of $f$.
\item The natural transformation $f^\ast(-)\otimes f^!(1)\rightarrow f^!(-)$ (adjoint to the projection formula for $f_!$) is an isomorphism.  (In particular, $f^!$ preserves colimits.)
\item If $1 \in \on{Sh}_{et}(Y;\mathcal{R})$ denotes the unit object, then $f^!1 \in \on{Sh}_{et}(X;\mathcal{R})$ is invertible.
\end{enumerate}
\end{enumerate}
\end{definition}
\begin{remark}
Every $f:X\to Y$ which is truncated and $\mathcal{R}$-etale is $\mathcal{R}$-smooth.  In this case $f^!1\simeq 1$.
\end{remark}
\begin{remark}
We have the following closure properties:
\begin{enumerate}
\item The classes of $\mathcal{R}$-etale and $\mathcal{R}$-proper maps are geometric, \cite{HeyerSixFunctors} Lemma 4.6.3.
\item The class of $\mathcal{R}$-smooth maps is not geometric (it is not closed under passage to diagonals), but it is closed under composition and pullback (clear from the definition).
\end{enumerate}

\end{remark}

\begin{remark}\label{etalepropersmoothlocal}
There are also many situations in which one can test properness or smoothness locally on the source or target in some sense or another.  Let $f:X\to Y$ be a map in $\on{CondAn}^{light}$.  We have:
\begin{enumerate}
\item Suppose $f$ is truncated.  If for all $S\in\on{ProfSet}^{light}$ with a map $S\to Y$ the map $X\times_Y S \to S$ is $\mathcal{R}$-proper (resp.\ $\mathcal{R}$-etale), then $f$ is $\mathcal{R}$-proper (resp.\ $\mathcal{R}$-etale); see \cite{HeyerSixFunctors} Lemma 4.6.3.
\item If for all $S\in\on{ProfSet}^{light}$ with a map $S\to Y$ the map $X\times_Y S \to S$ is $\mathcal{R}$-smooth, then $f$ is $\mathcal{R}$-smooth; see \cite{HeyerSixFunctors} Lemma 4.5.7.
\item Let $g:X'\to X$ be a truncated $\mathcal{R}$-proper map.  Suppose $g_\ast 1\in \on{CAlg}(\on{Sh}_{et}(X;\mathcal{R}))$ is descendable (See section \ref{descsec}).  If $f\circ g$ is $\mathcal{R}$-$!$-able, then $f$ is $\mathcal{R}$-$!$-able; and if $f$ is truncated and $f\circ g$ is truncated and proper, then $f$ is proper; see \cite{HeyerSixFunctors} Lemma 4.7.4 and  Cor.\ 4.7.5
\item Suppose $\{g_i:X_i\to X\}_{i\in I}$ is a collection of $\mathcal{R}$-smooth maps to $X$ such that the pullback functors $\{g_i^\ast\}_{i\in I}$ are jointly conservative.  If each $f\circ g_i$ is $\mathcal{R}$-$!$-able (resp.\ $\mathcal{R}$-smooth), then $f$ is $\mathcal{R}$-$!$-able (resp.\ $\mathcal{R}$-smooth); see \cite{HeyerSixFunctors} Lemmas 4.7.1 and 4.5.8.
\item Suppose $f$ is truncated.  Let $\{g_i:X_i\to X\}_{i\in I}$ be a collection of truncated $\mathcal{R}$-etale maps to $X$ such that the pullacks $\{g_i^\ast\}_{i\in I}$ are jointly conservative and each $f\circ g_i$ is $\mathcal{R}$-etale.  Then $f$ is $\mathcal{R}$-etale; see \cite{HeyerSixFunctors} Lemma 4.6.3.
\end{enumerate}
Of course, implicit here is the claim that the $g$ in 2 and the $\{g_i\}$ in 3 give examples of $!$-covers.
\end{remark}

\begin{remark}
Suppose $\mathcal{R}\to\mathcal{R}'$ is a map in $\on{CAlg}(\on{Pr}^L)$ such that both $\mathcal{R}$ and $\mathcal{R}'$ are compactly assembled.  If a map $f$ is truncated and $\mathcal{R}$-proper, then it is also $\mathcal{R}'$-proper.  Similarly if $f$ is truncated and $\mathcal{R}$-etale, then it is $\mathcal{R}'$-etale, and if $f$ is $\mathcal{R}$-smooth, then it is $\mathcal{R}'$-smooth.  Using the compatibility from Remark \ref{changecoefficients}, this is clear by applying $-\otimes_\mc{R}\mc{R}'$.
\end{remark}

\begin{remark}
By construction, any $f:X\to Y$ which is truncated and $\mathcal{R}$-proper comes with a natural isomorphism $f_!\simeq f_\ast$.  One can check inductively that these natural isomorphisms commute with the base-change comparison maps and the projection formula comparison maps, where the ones for $f_!$ come from the six functor formalism and the ones for $f_\ast$ come from adjunction.  In particular, for $\mathcal{R}$-proper maps the formation of $f_\ast$ commutes with base-change and satisfies a projection formula.  Thus we could build another six functor formalism with $E$ the class of all $\mathcal{R}$-proper maps based on this via the explicit procedure of \cite{HeyerSixFunctors} Prop.\ 3.3.3.  But it follows from the uniqueness principle \cite{DauserUniqueness} that this six functor formalism agrees with the one from Theorem \ref{sixfunctorextension}.  Thus the isomorphism $f_!\simeq f_\ast$ for $\mathcal{R}$-proper maps $f$ automatically comes with all the relevant higher coherences.

Similarly, when $f:X\to Y$ is truncated and $\mathcal{R}$-etale, the natural isomorphism $f_!\simeq f_\natural$ (where $f_\natural$ denotes the left adjoint to $f^\ast$), adjoint to $f^!\simeq f^\ast$, commutes with base change and projection formulas, again up to all higher coherences as encoded in six functor formalisms.
\end{remark}

Now we give the basic examples of $\on{Sp}$-etale and $\on{Sp}$-proper maps.  (We defer discussion of smoothness to the next section.)

\begin{lemma}\label{openisetale}
For every light profinite set $S$ and every open subset $U\subset S$ (viewed as a light condensed subset), the inclusion $U\to S$ is $\on{Sp}$-etale.
\end{lemma}
\begin{proof}
First suppose $U$ is also a light profinite set.  Then $U\to S$ is $\on{Sp}$-proper by definition, and if we set $V=S \setminus U$ so that $$\on{Sh}_{et}(S) = \on{Sh}_{et}(U)\times\on{Sh}_{et}(V)$$ via pullback, then it is easy to see in terms of this decomposition that $U\to S$ is $\on{Sp}$-etale.

In the general case, we can write $U=\sqcup_n U_n$ as a disjoint union of light profinite sets indexed by $\mathbb{N}$.  As $U_n\to U$ is $\on{Sp}$-etale on pullback to any light profinite set by the first part of this proof, we deduce that $U_n\to U$ is $\on{Sp}$-etale.  So is $U_n\to S$, again by the first part of this proof. Then by Remark \ref{etalepropersmoothlocal} we deduce that $U\to S$ is $\on{Sp}$-etale, as desired.
\end{proof}

Generalizing this, we have the following, ensuring in particular that there is no clash of terminology concerning the word ``etale''.

\begin{proposition}\label{etaleisetale}
Suppose $f:X\to Y$ is a truncated etale map of light condensed anima.  Then $f$ is $\on{Sp}$-etale.
\end{proposition}
\begin{proof}
We proceed by induction on the truncatedness of $f$.  If $f$ is $-2$-truncated, then it is an isomorphism, and the claim is trivial.  Now we perform the inductive step.  Suppose $f:X\to Y$ is $d+1$-truncated, and we know the claim for $d$-truncated maps.  By Remark \ref{etalepropersmoothlocal} we can assume $Y\in\on{ProfSet}^{light}$.  By the comparison of etale maps with topological sheaves (Lemma \ref{etaleontopological}), there is a collection of open subsets $\{U_i\subset Y\}_{i\in I}$ and maps $\{s_i:U_i\to X\}$ over $Y$ which are jointly surjective.  The map $\sqcup_i U_i\to X$ is $d$-truncated and etale, hence $\mathcal{R}$-etale by induction. As each $U_j\to \sqcup_i U_i$ and $U_j \to X$ is $\on{Sp}$-etale by Lemma \ref{openisetale}, we deduce from Remark \ref{etalepropersmoothlocal} that $f$ is $\on{Sp}$-etale, as desired.
\end{proof}

We also have the following examples of $\on{Sp}$-proper maps, proved using the descendability formalism from the previous section.

\begin{theorem}\label{chausproper}
Suppose $X$ is a light condensed anima such that either:
\begin{enumerate}
\item $X$ is a second-countable compact Hausdorff space of finite cohomological dimension; or
\item $X$ is a 1-truncated light profinite anima of finite cohomological dimension.
\end{enumerate}

Then the map
$$f:X\to \ast$$
is $\on{Sp}$-proper.

\end{theorem}\label{moreproper}
\begin{proof}
In both cases, we have the following properties:
\begin{enumerate}
\item $\on{Sh}_{et}(X;\on{Sp})$ is rigid and countably assembled (see Section \ref{descsec});
\item $X$ has finite cohomological dimension (by assumption);
\item there exists a light profinite set and a surjective $g:S\twoheadrightarrow X$, such that for any $T\to X$ with $T\in\on{ProfSet}^{light}$ the fiber product $T\times_X S$ is also a light profinite set.
\end{enumerate}
Let us show that for any such $X$ the map $X\to \ast$ is $\on{Sp}$-proper.

Indeed, by assumption 3 and Theorem \ref{chausproper} $g:S\to K$ is $\on{Sp}$-proper.  As $S\to \ast$ is likewise $\on{Sp}$-proper, using Remark \ref{etalepropersmoothlocal} it suffices to show that $g_\ast\mathbb{S}\in\on{CAlg}(\on{Sh}_{et}(K;\on{Sp}))$ is descendable.  Applying the criterion of Theorem \ref{descendabletheorem}, it suffices to show that for all points $x:\ast\to K$, the pullback $x^\ast g_\ast \mathbb{S}$ is descendable, of some index which is uniform in $x$.  But by base change $x^\ast g_\ast\mathbb{S}$ is $(g_x)_\ast\mathbb{S}$ for $g_x:g^{-1}(x)\to \ast$.  Since $g$ is surjective, $g_x$ has a section, hence so does the unit map $\mathbb{S}\to (g_x)_\ast\mathbb{S}$, and thus $(g_x)_\ast\mathbb{S}$ is descendable of index $\leq 1$, whence the claim.
\end{proof}

\begin{remark}\label{trulylocal}
The same argument shows that for any surjective map $f:T\twoheadrightarrow S$ of light profinite sets, $f_\ast\mbb{S}$ is descendable in $\on{Sh}_{et}(S;\on{Sp})$.  In particular $f$ is a $!$-cover by Remark \ref{etalepropersmoothlocal}.  Of course, finite disjont unions are also $!$-covers by Lemma \ref{openisetale}.  Thus the Grothendieck topology on light profinite sets has the property that every cover is a $!$-cover.  (Caution that this is only true for light profinite sets, not for general light condensed anima.)

 This implies the following convenient property: the condition that a map $f:X\to Y$ in $\on{CondAn}^{light}$ be $\mc{R}$-$!$-able is local on $Y$ in the light condensed Grothendieck topology, i.e.\ if $\{Y_i\to Y\}_{i\in I}$ is a jointly surjective family of maps to $Y$ and $X\times_Y Y_i\to Y_i$ is $\mc{R}$-$!$-able for all $i$, then $f$ is $\mc{R}$-$!$-able.  The same holds for the classes of $\mc{R}$-proper, $\mc{R}$-smooth, and $\mc{R}$-etale maps.
\end{remark}

\begin{remark}
The presence of finite dimensionality assumptions in Theorem \ref{moreproper} is unavoidable.  In fact, suppose $X$ is a light condensed anima such that $f:X\to \ast$ is $\on{D}(\mbb{Z})$-proper.  Then $X$ has finite cohomological dimension.

Indeed, suppose the contrary.  Then there is a sequence $\mathcal{A}_0,\mathcal{A}_1,\ldots$ of etale sheaves of abelian groups on $X$ and integers $i_0,i_1,\ldots$ tending to $\infty$ such that $H^{i_n}(X;\mathcal{A}_n)\neq 0$ for all $n\geq 0$.  Consider $\mathcal{A} = \oplus_{n\geq 0} \mathcal{A}_n[i_n]$.  Because $f_\ast$ preserves colimits (due to the properness assumption), we have
$$H^0(X;\mathcal{A}) \simeq \oplus_{n\geq 0} H^{i_n}(X;\mathcal{A}_n).$$
On the other hand we also have $\mathcal{A} = \prod_{n\geq 0} \mathcal{A}_n[i_n]$ because this holds after $N$-truncation for any $N$.  Then since $f_\ast$ preserves limits we also have
$$H^0(X;\mathcal{A})\simeq \prod_{n\geq 0} H^{i_n}(X;\mathcal{A}_n).$$
We deduce that the comparison map from the direct sum to the product is an iso.  Thus $H^{i_n}(X;\mathcal{A}_n)=0$ for all but finitely many $n$, a contradiction.

Note that this also means that if $f:X\to Y$ is any $D(\mbb{Z})$-proper map of light condensed anima, then the pullback $X\times_Y S$ along any map from a light profinite set $S\to Y$ has finite cohomological dimension, and for $s\in S$ there is a uniform bound on the cohomological dimensions of the fibers.
\end{remark}

\begin{remark}\label{tn}
By Theorem \ref{chausproper}, if $G$ is a light profinite group of finite cohomological dimension, then $BG\to \ast$ is $\on{Sp}$-proper.  Suppose now $G$ is only \emph{virtually} of finite cohomological dimension, and let $N$ be an open normal subgroup of finite cohomological dimension.  Consider
$$BG \overset{f}{\to} B(G/N) \overset{g}{\to} \ast.$$
The map $g$ has proper diagonal, so there is a natural map $g_! \to g_\ast$.  But $g$ is etale so $g_!$ is left adjoint to $g^\ast$, and this $g_! \to g_\ast$ is just the usual ``norm'' or transfer map from homotopy orbits to homotopy fixed points.

Thus if we pass to $\mathcal{R}=L_{T(n)}\on{Sp}$ coefficients, then by telescopic Tate vanishing (\cite{KuhnTate}, the norm map $g_! \to g_\ast$ is an iso hence $g$ is proper, and therefore so is $BG\to \ast$.

On the other hand, turning this around, if $G$ has finite cohomological dimension and hence $g\circ f$ is proper a priori, we deduce that the norm map $g_! \to g_\ast$ is an isomorphism on any object of the form $f_\ast\mc{F}$; this is the ``Theorem of Tate-Thomason''.
\end{remark}

Now we consider the topological situation.

\begin{proposition}\label{LCHaus}
Suppose $X,Y$ are topological spaces each of which admits a basis of second countable compact Hausdorff spaces of finite cohomological dimension.  Then any map $f:\underline{X}\to\underline{Y}$ is $\on{Sp}$-$!$-able.
\end{proposition}
\begin{proof}
Because $!$-able maps are a geometric class of maps, it suffices to show that $f:X \to \ast$ is $!$-able (and same for $Y$, but the hypotheses on $X$ and $Y$ are the same).  Open covers are $\on{Sp}$-etale by Lemma \ref{openisetale}, so by Remark \ref{etalepropersmoothlocal} we can reduce to the case where $X$ is second countable finite dimensional compact Hausdorff which is given by Theorem \ref{chausproper}.
\end{proof}

In particular, if $F$ is a local field, then this result applies to any $F$-manifold.

Note that the uniqueness result of \cite{DauserUniqueness} shows that the six functor formalism on the class of topological spaces in the previous proposition agrees with any other reasonable one made by potentially different means.  For example, on the Hausdorff full subcategory we could use the usual method of factorizations into open immersions and proper maps.  Let's just record the following which gives the essence of the comparison.

\begin{lemma}
Suppose $X$ is a light condensed anima and $j:U\hookrightarrow X$ is an open immersion with closed complement $i:Z\subset X$ (meaning, on pullback to any light profinite set we see an open inclusion and its closed complement).  Then for any $\mathcal{F}\in\on{Sh}_{et}(X;\on{Sp})$, there is a unique and functorial fiber sequence
$$j_! j^\ast \mathcal{F} \to \mathcal{F} \to i_\ast i^\ast \mathcal{F}$$
such that the second map is the unit for the adjunction and the first map is the counit for the adjunction (using that $j:U\hookrightarrow X$ is etale so $j^!=j^\ast$).
\end{lemma}
\begin{proof}
The spectrum of maps $j_!j^\ast \mathcal{F} \to i_\ast i^\ast \mathcal{F}$ is $0$ by adjunction, because $i^\ast j_!=0$ by base-change and the fact that $\on{Sh}_{et}(\emptyset;\on{Sp})=0$.  For the same reason the spectrum of natural transformations in $\mathcal{F}$ of that form is $0$.  Thus there is a unique and functorial nullhomotopy of our purported fiber sequence.  Because isomorphisms of etale sheaves are detected on points (Remark \ref{checkonstalks}), to prove it is indeed a fiber sequence it suffices to check on points, where $j$ becomes either $\emptyset \subset \ast$ or $\ast\subset \ast$, where in both cases the result is trivial.
\end{proof}

\begin{example}\label{recollement}
In particular, if $f:X\to S$ is a map of second countable finite dimensional compact Hausdorff spaces and $j:U\to X$ is as above, then for any sheaf $\mathcal{F}$ on $X$ we can calculate the lower-$!$ of its restriction to $U$ in terms of lower-$*$ functors as the fiber
$$(f\circ j)_!f^\ast\mathcal{F} = \on{Fib}(f_\ast\mathcal{F} \to (f\circ i)_\ast f^\ast \mathcal{F}).$$
Intuitively, compactly supported cohomology on $U$ is the fiber of cohomology of the compactification restricting to cohomology of the boundary.
\end{example}

\section{Smooth morphisms}\label{smoothsec}

Now we go a bit beyond the basic results proved at the end of the previous section.  We are interested in exploring the six functor formalism on those light condensed anima which are associated to $F$-analytic Artin stacks, as discussed in Section \ref{condmansec}.

\begin{lemma}\label{repsubsmooth}
Let $F$ be a local field, and suppose $f:X\to Y$ is a representable submersion in $\on{Sh}(\on{Man}_F)$.  Then $\underline{X}\to\underline{Y}$ is $\on{Sp}$-$!$-able.
\end{lemma}
\begin{proof}
By Remark \ref{trulylocal}, we can use descent on $Y$ to reduce to the case where $Y$ is an $F$-manifold, when this follows from Proposition \ref{LCHaus}.
\end{proof}

We would like to be able to say that if $f:X\to Y$ is a representable submersion which is \emph{surjective}, then $\ul{X}\to\ul{Y}$ is a $!$-cover, but this is perhaps too much to ask in full generality in the nonarchimedean setting.  However we can show it in the archimedean setting using \emph{smoothness} (Definition \ref{etalepropersmoothdef}).   In principle this requires checking a number of separate conditions.  But we now give a very convenient mechanism for reducing some of these conditions to others.

First, note that by descent it suffices to treat the case $Y=S\in\on{ProfSet}^{light}$.  Moreover, again by descent, to check that the formation of $f^!$ commutes with base-change it suffices to show this for base-changes along maps of light profinite sets.  In that setting the following condition will be useful.

\begin{definition}
Let $f:X\to S$ be a map of light condensed anima with $S\in\on{ProfSet}^{light}$, and let $\mathcal{R}\in\on{CAlg}(\on{Pr}^L_{st})$ be compactly assembled.  Say that $f$ \emph{satisfies categorical base-change} (with $\mathcal{R}$-coefficients) if for all maps $S'\to S$ of light profinite sets with pullback $X'=X\times_S S'$, the natural map in $\on{CAlg}(\on{Pr}_{st}^L)$
$$\on{Sh}_{et}(X;\mc{R})\otimes_{\on{Sh}_{et}(S;\mc{R})}\on{Sh}_{et}(S';\mc{R})\rightarrow \on{Sh}_{et}(X';\mc{R})$$
is an isomorphism.
\end{definition}

Note that this definition does not refer to the six functor formalism, but only to the symmetric monoidal pullback functoriality.  However, we have the following lemma.

\begin{lemma}
Let $f:X\to S$ be a map of light condensed anima with $S\in\on{ProfSet}^{light}$, and let $\mathcal{R}\in\on{CAlg}(\on{Pr}^L_{st})$ be compactly assembled.  If $f$ is $\mathcal{R}$-$!$-able, then $f$ satisfies categorical base-change with $\mathcal{R}$-coefficients.
\end{lemma}
\begin{proof}
By Remark \ref{actuallyconstruct}, it suffices to show that the class of maps of light condensed anima, each of whose base-change to a light profinite set satisfies categorical base-change, contains $E_0$ and is closed under the two extension procedures $E'\mapsto E'_!$ and $E'\mapsto E'_\ast$.

To see it contains $E_0$ we need to show that a map of light profinite sets $X\to S$ satisfies categorical base-change.  Using that $\on{Sh}_{et}(-;\on{Sp})$ on a light profinite set is compactly generated by the unit, this follows from proper base-change.

Closure under $E'\mapsto E'_\ast$ is tautological.  For closure under $E'\mapsto E'_!$,  suppose $X\to S$ admits a $!$-cover $\{X_i\to X\}_{i\in I}$ in $E'$ such that each $X_i\to S$ lies in $E'$.  Each iterated fiber product of the $X_i\to X$ mapping to $S$ also lies in $E'$ and hence satisfies categorical base change.  Then we write $\on{Sh}_{et}(X;\mathcal{R})$ as the colimit of the Cech nerve via $!$-functors.  By definition this is a colimit in $\on{Pr}^L$, also after base-change to $S'$.  But the diagram lies in $\on{Mod}_{\on{Sh}_{et}(S;\mathcal{R})}(\on{Pr}^L)$ hence also gives a colimit there.   Using that relative tensor products preserve colimits we conclude.
\end{proof}

Using categorical base-change we get the following criterion for smoothness.

\begin{proposition}\label{smoothlemma}
Let $f:X\to S$ be a map of light condensed anima with $S\in\on{ProfSet}^{light}$ and let $\mathcal{R}\in\on{CAlg}(\on{Pr}^L_{st})$ such that $\mathcal{R}$ is \emph{semi-rigid} (Remark \ref{pcompletesemirigid}, \cite{ArinkinLanglands} App.\ C).  Suppose that $f$ is $\mathcal{R}$-$!$-able.  Then the following properties are equivalent:

\begin{enumerate}
\item $f$ is $\mathcal{R}$-smooth;
\item $f^!$ commutes with colimits and $f^!1$ is invertible.
\end{enumerate}

Furthermore, if we just assume that $f^!$ commutes with colimits, then the formation of $f^!$ commutes with base-change, so we can also test invertibility of $f^!1$ locally on $S$.
\end{proposition}
\begin{proof}
Suppose $f^!$ preserves colimits.  Since $\mathcal{R}$ and $\on{Sh}_{et}(S;\on{Sp})$ are both semi-rigid, so is their tensor product $\on{Sh}_{et}(S;\mc{R})$.  Thus by \cite{ArinkinLanglands} Prop.\ C.5.5 the projection formula for $f^!$ is automatic.

Now, suppose $f':X'\to S'$ is the base-change of $f$ along a map of light profinite sets $S'\to S$.  Then $(f')_!$ is a $\on{Sh}_{et}(S';\mathcal{R})$-linear map in $\on{Pr}^L$ which agrees with $f_!$ as a $\on{Sh}_{et}(S;\mathcal{R})$-linear map on restriction to $\on{Sh}_{et}(X;\mathcal{R})$.  By categorical base-change it follows that
$$(f')_! = f_! \otimes_{\on{Sh}_{et}(S;\mathcal{R})}\on{Sh}_{et}(S';\mathcal{R}).$$
By preservation of colimits and the projection formula, $f^!$ is the right adjoint to $f_!$ in the 2-categorical sense of $\on{Mod}_{\on{Sh}_{et}(S;\mathcal{R})}(\on{Pr}^L)$ and hence this adjunction passes through the tensor product, establishing the claim.
\end{proof}

\begin{remark}\label{smoothremark}
Suppose $\on{Sh}_{et}(X;\mc{R})$ is compactly generated.  Then $f^!$ preserves colimits if and only if $f_!$ sends compact objects to compact objects.  It even suffices to just consider a generating set of compact objects.
\end{remark}

We apply this, first in the archimedean case.

\begin{theorem}\label{realsmooth}
Suppose that $f:X\to Y$ is a representable submersion in $\on{Sh}(\on{Man}_\mathbb{R})$.  Then $f:\ul{X}\to\ul{Y}$ is $\on{Sp}$-smooth.
\end{theorem}
\begin{proof}
Working locally on $Y$ (Remark \ref{trulylocal}), we reduce to the case where $Y$ is a manifold.  We can also work locally on the source in the sense of open covers by Remark \ref{etalepropersmoothlocal}.  Thus we can further reduce to the case of a projection off an open polydisk, $S\times D \to S$.  By base-change and composition this in turn reduces to showing that the open interval $f:(-1,1)\to \ast$ is $\on{Sp}$-smooth.  By Lemma \ref{smoothlemma}, for this it suffices to show that $f^!$ preserves colimits and $f^!\mathbb{S}$ is invertible.  Since $\on{Sh}_{et}((-1,1);\on{Sp})$ identifies with usual sheaves of spectra on the topological space $(-1,1)$ (Lemma \ref{etaleontopological}), and the open sub-intervals form a basis for the topology closed under finite intersection, to prove $f^!$ preserves colimits it suffices to show that $f_!(h_U)$ is a compact object in $\on{Sp}$ for any open subinterval $U\subset (-1,1)$.  But calculating using Example \ref{recollement} and the compactification $[-1,1]$ gives $f_!(h_{U})\simeq \Sigma^{-1}\mbb{S}$, whence the conclusion.  To prove that $f^!\mbb{S}$ is invertible, we can replace $(-1,1)$ by $\mbb{R}$; then by base-change for $f^!$ and the fact $\mbb{R}$ has a group structure we deduce that $f^!\mbb{S}$ is a constant sheaf.  Thus it suffices to show that $\Gamma((-1,1);f^!\mbb{S})$ is an invertible spectrum.  But this is by definition the dual of what we calculated earlier to be $\mbb{S}[-1]$, whence the claim.
\end{proof}

\begin{corollary}\label{realartinsmooth}
\begin{enumerate}
\item If $Y\twoheadrightarrow X$ is any surjective representable submersion in $\on{Sh}(\on{Man}_{\mbb{R}})$, then $\ul{Y}\to\ul{X}$ is a $\on{Sp}$-$!$-cover.
\item Every map between $\mbb{R}$-analytic smooth Artin stacks is $\on{Sp}$-$!$-able.
\item For every $\mathbb{R}$-analytic smooth Artin stack $X$ the map $\underline{X}\to\ast$ is $\on{Sp}$-smooth.
\end{enumerate}
\end{corollary}
\begin{proof}
For 1, we have seen it's $\on{Sp}$-smooth.  It's also conservative on pullbacks because of surjectivity.  Thus the claim follows from Remark \ref{etalepropersmoothlocal}.

For 2, because the $!$-able maps form a geometric class, it suffices to show 3.

For 3, let $M\to X$ be a surjective representable submersion.  This is smooth and a cover by 1 and $M\to \ast$ is also smooth by Lemma \ref{repsubsmooth}.  Thus $X\to \ast$ is smooth by Remark \ref{etalepropersmoothlocal}.
\end{proof}

Now we move to the non-archimedean case, which is more technically demanding.

To apply Proposition \ref{smoothlemma} (in conjunction with Remark \ref{smoothremark}) we need criteria detecting when an object in $\on{Sh}_{et}(S;\on{Sp}_{\wh{p}})$ is compact, and also when it is invertible.  To complete the picture we also discuss dualizable objects.  (Dualizable is in some sense a bit nicer than compact because being dualizable is a local condition, by descent.)

\begin{lemma}\label{compactinvertible}
Let $p$ be a prime, $S$ a light condensed anima, and $\mathcal{F}\in \on{Sh}_{et}(S;\on{Sp}_{\wh{p}})$.  Then the following are equivalent:
\begin{enumerate}
\item $\mathcal{F}$ is dualizable.
\item $\mathcal{F}/p$ is locally bounded below in the t-structure, $\pi_\ast (\mathcal{F}\otimes\mbb{F}_p)$ locally lives in only finitely many degrees, and each $\pi_i (\mathcal{F}\otimes\mbb{F}_p)$ is represented by a ($0$-truncated) finite etale map $X\to S$.
\end{enumerate}
Moreover, the following are equivalent:
\begin{enumerate}
\item $\mathcal{F}$ is invertible.
\item $\mathcal{F}$ is dualizable, for any point $s:\ast \to S$ we have $s^\ast\mc{F}\otimes\mbb{F}_p\simeq \mbb{F}_p[d]$ for some $d\in\mbb{Z}$.
\end{enumerate}
Finally, suppose $S$ is a light profinite set.  Then the following are equivalent:
\begin{enumerate}
\item $\mathcal{F}$ is compact.
\item $\mathcal{F}$ is dualizable and killed by some power of $p$.
\end{enumerate}
\end{lemma}
\begin{proof}
For 1, by descent we can assume $S$ is a light profinite set.  Suppose $\mc{F}$ is dualizable.  As the unit is compact mod $p$, it follows that $\mc{F}/p$ is a compact object, hence lies in the thick subcategory generated by $\mbb{S}/p$, hence is bounded below.  Moreover $\mc{F}\otimes \mbb{F}_p \in \on{Sh}_{et}(S;\on{D}(\mbb{F}_p))$ is dualizable and therefore compact, hence it comes by pullback from a finite set where it has finite dimensional stalks whence $\mc{F}$ satisfies 2.  For the converse, since each homotopy group is represented by a finite etale map, we see it is locally constant with finite dimensional stalks, whence it lies in the thick subcategory generared by $\mbb{F}_p$ (we are still in the situation where $S$ is a light profinite set).  Applying Lemma \ref{checkperfmodp} we deduce that $\mc{F}$ lies in the thick subcategory generated by the unit $\mbb{S}_{\wh{p}}$, whence $\mc{F}$ is dualizable.

For the claim about invertible objects, again we can reduce to light profinite sets.  Clearly 1 implies 2, and for the converse note that we can check $\mc{F}^\vee\otimes \mc{F}\overset{\sim}{\rightarrow}\mbb{S}_{\wh{p}}$ by applying $-\otimes \mbb{F}_p$ because both sides are bounded below (as we saw above, a dualizable $\mc{F}$ lies in the thick subcategory generated by $\mbb{S}_{\wh{p}}$ hence so does its dual).

Finally, for the claim about compact objects, compact implies dualizable and killed by a power of $p$ because the generating compact object $\mbb{S}/p$ is dualizable and killed by a power of $p$.  For the converse, we saw above that any $\mc{F}$ which is dualizable lies in the thick subcategory generated by $\mbb{S}_{\wh{p}}$.  If $\mc{F}$ is killed by a power of $p$, then we deduce that $\mc{F}$ lies in the thick subcategory generated by some $\mbb{S}/p^n$ which is compact.
\end{proof}

We will now consider the setting of $F=\mbb{Q}_p$, with $\on{Sp}_{\wh{p}}$-coefficients.  Our results are not quite so strong as in the Archimedean case: we will not show that all $\mbb{Q}_p$-analytic smooth Artin stacks are $!$-able, but only a (fairly large) subclass.

We will rely on Proposition \ref{localstructurepadiclie}, showing that a $p$-adic Lie group $G\to S$ in families locally on $S$ admits a compact open subgroup object $H\subset G$ such that each fiber is a uniform pro-$p$ group.  We can use this in conjunction with the following:

\begin{lemma}\label{uniformisgood}
Let $S$ be a compact Hausdorff (hence light profinite) $\mbb{Q}_p$-manifold.  Let $H\to S$ be a compact Hausdroff $\mbb{Q}_p$-Lie group over $S$ such that each fiber is a uniform pro-$p$-group.  Denote by $BH\to S$ the relative classifying stack.  Then:
\begin{enumerate}
\item $BH$ has finite cohomological dimension.
\item $\on{Sh}_{et}(BH;\on{Sp})$ is rigid, with a family of compact generators provided by the $(j_n)_!\mbb{S}$ for $j_n:BH^{p^n}\to BH$.
\item $\on{Sh}_{et}(BH;\on{D}(\mbb{F}_p))$ is compactly generated by the unit $\mbb{F}_p$.
\end{enumerate}
\end{lemma}
\begin{proof}
For 1, by Remark \ref{cohdimfibers} it suffices to show that each fiber of $BH\to S$ has uniformly bounded finite cohomological dimension.  We recalled this with mod $p$ coefficients in \ref{scholiumuniform}.  With $\mbb{Z}[1/p]$-coefficients the cohomological dimension of each fiber is $0$ because they are pro-$p$ groups.

For 2, the rigidity claim follows from 1 and Example \ref{rigidexamples}, and we get compact generation by the $j_!\mbb{S}$ where $j$ runs over all finite etale maps to $BH$.  To see that the $(j_n)_!\mbb{S}$ are enough to generate, note by Proposition \ref{localstructurecorollary} every finite etale map to $BH$ is dominated by $BH^{p^n}\times_ST$ for some $T\subset S$ compact open and $n\geq 0$.  But the $j_!$ of this is a retract of the one for $j_n$, proving the claim.

For 3, from 2 we get that the $(j_n)_!\mbb{F}_p$ give compact generators.  To show that $\mbb{F}_p$ on its own generates, it suffices to show that each $(j_n)_!\mbb{F}_p$ lies in the thick subcategory generated by $\mbb{F}_p$.  For that, because $H/H^{p^n}\to S$ is a finite etale group object with $p$-group fibers (\ref{scholiumuniform}), by base-change and working locally on $S$ it suffices to show that if $Q$ is a finite $p$-group then for the map $f:\ast \to \ast/Q$ we have that $f_!\mbb{F}_p$ lies in the thick subcategory generated by $\mbb{F}_p$.  This holds because the group ring of a finite $p$-group with $\mbb{F}_p$-coefficients, as a representation, admits a finite filtration with associated gradeds the trivial representation $\mbb{F}_p$.
\end{proof}
Here is our main result on $!$-ability in the nonarchimedean setting.  We do not yet have to pass to $p$-complete coefficients; that will come when we discuss smoothness.

\begin{theorem}\label{good}
Let $X$ be a $\mbb{Q}_p$-analytic smooth Artin stack. If there exists a surjective representable submersion $M\to X$ with $M\in\on{Man}_F$ which is a $\on{Sp}$-$!$-cover, then every surjective representable submersion $Y\to X$ is a $\on{Sp}$-$!$-cover.

Call $X$ ``$!$-good'' if it satisfies this condition.  Then $\ast$ is $!$-good, and we have the following closure properties:
\begin{enumerate}
\item Suppose $Y\to X$ is a map such that there exists a surjective representable submersion $M\to X$ from a manifold such that the pullback $Y\times_XM$ is representable by a manifold. (For example $Y\to X$ could be a representable submersion.) Then $X$ is $!$-good $\Rightarrow$ $Y$ is $!$-good.
\item If $X$ is $!$-good and $G\to X$ is a group object in representable submersions over $X$, then the relative classifying stack $BG$ is $!$-good.
\item If $X$ admits an etale cover by $!$-good stacks, then $X$ is $!$-good.
\end{enumerate}
Finally, for any $!$-good $X$, the map $X\to \ast$ is $\on{Sp}$-$!$-able. (Hence any map between such $X$'s is likewise $!$-able.)
\end{theorem}
\begin{proof}
Suppose given a surjective representable submersion $M\to X$ with $M\in\on{Man}_{\mbb{Q}_p}$ which is a $!$-cover, and let $Y\to X$ be an arbitrary surjective representable submersion.  We already know $Y\to X$ is $!$-able by Lemma \ref{repsubsmooth}, so we can test whether $Y\to X$ is a $!$-cover on pullback to $M$.  Thus it suffices to show that any surjective submersion $M\to N$ of manifolds is a $!$-cover.  But it's etale locally split.

Now consider 1.  Assuming $X$ is $!$-good, we have that $M\to X$ is a $!$-cover, hence so is its pullback $M\times_XY \to Y$ (note also that this pullback commutes with the functor to light condensed anima by Proposition \ref{underlying}).  Thus $Y$ is $!$-good.

Now consider 2.  This is the only one where we have to work a bit.  Let $M\to X$ be a surjective representable submersion, so a $!$-cover by goodness of $X$.  The composition $M\to X \to BG$ is a surjective representable submersion, and we claim it's a $!$-cover.  Choose another surjective representable submersion $N\to X$.  (We could choose the same one, but we separate the notation for clarity.)  By hypothesis $N\to X$ is a $!$-cover.  Since $BG$ maps to $X$ we can pull back our situation along $N\to X$ and therefore reduce to the case where $X=N$ is a manifold.  Then it suffices to show that $N\to BG$ is a $!$-cover.  Working etale locally, we can assume $N$ is compact Hausdorff.  Then by Proposition \ref{localstructurepadiclie} there is a compact open sub-group object $H\subset G$ such that each fiber $H_x$ is a uniform pro-$p$-group and by Proposition \ref{localstructurecorollary} $BH \to BG$ is (surjective and) etale.  Thus we reduce to $H=G$.  Then $f:N\to BH$ is proper because $H$ is light profinite (\ref{localstructurecorollary}) and so by Remark \ref{etalepropersmoothlocal} it suffices to show that $f_\ast\mbb{S}$ is descendable.  Then by Lemma \ref{uniformisgood} we can apply Theorem \ref{descendabletheorem}  and work pointwise, where we have a section, to conclude.

Finally, suppose $\{U_i\to X\}_{i\in I}$ is an etale cover such that each $U_i$ is $!$-good.  Choose a surjective representable submersion $M_i \to U_i$ for each $i$, so by hypothesis this is a $!$-cover.  Then $\sqcup_i M_i \to X$ is a surjective representable submersion and it suffices to show it's a $!$-cover.  But we can check this on each $U_i$ where it's refined by $M_i \to U_i$ as required.

It remains to show that if $X$ is $!$-good then $X\to \ast$ is $!$-able (as always, with $\on{Sp}_{\wh{p}}$-coefficients).  Let $M\to X$ be a $!$-cover from a manifold.  We know $M\to \ast$ is $!$-able by Lemma \ref{repsubsmooth}, and the claim follows by Remark \ref{etalepropersmoothlocal}
\end{proof}

\begin{example}\label{lotsaregood}
Any quotient stack is $!$-good.  Indeed, $M/G \to \ast/G$ is a representable submerion, so this follows from 1 and 2.  More generally we can take local quotient stacks, or local quotient stacks over local quotient stacks, etc.  This generality will certainly cover all the Artin stacks we will need to consider.
\end{example}

Smoothness is a lot less common in the nonarchimedean setting: already of course $\mathbb{Q}_p\to \ast$ is not smooth. But Lazard proved both that Poincaré duality holds (with $p$-adic coefficients) for the continuous cohomology of sufficiently small $p$-adic Lie groups, and that every $p$-adic Lie group has a sufficiently small (in the same sense) open subgroup.  Later Serre clarified that ``sufficiently small'' exactly means compact and $p$-torsionfree. This suggests that classifying stacks are smooth, and proper under a hypothesis like Serre's.  In fact we have the following.

\begin{theorem}\label{padicsmooth}
Suppose that $X$ is a $\mathbb{Q}_p$-analytic smooth Artin stack which is $!$-good.  Let $G\to X$ be a group object in representable submersions over $X$, and let $BG\to X$ denote the relative classifying stack.  Then:
\begin{enumerate}
\item $BG \to X$ is $\on{Sp}_{\wh{p}}$-smooth. 
\item If $G\to X$ is proper and each fiber $G_x$ is $p$-torsionfree, then $BG \to X$ is $\on{Sp}_{\wh{p}}$-proper.
\end{enumerate}
\end{theorem}
\begin{proof}
By Example \ref{lotsaregood} we have that $BG$ is also $!$-good.  In particular $BG\to X$ is $!$-able.  Let us first show Claim 2.  By Remark \ref{trulylocal} we can check properness on pullback to any light profinite set $S$ (after implicitly passing to the light condensed anima).  By refining via a surjection we can even assume $S$ is a $\mbb{Q}_p$-manifold and the map $S\to X$ comes from a map in $\on{Sh}(\on{Man}_{\mbb{Q}_p})$.  Now we replace $X$ by $S$.  Thus $G\to S$ is a group object in proper submersions over the compact Hausdorff $p$-adic manifold (hence light profinite set) $S$, and we are assuming also that each fiber is $p$-torsionfree.

By Proposition \ref{localstructurepadiclie} we can find a compact open subgroup $H\subset G$ such that each fiber $H_s$ is a uniform pro-$p$-group.  By Lemma \ref{uniformisgood} $\on{Sh}_{et}(\ul{BH};\on{Sp})$ is compactly generated by dualizable objects.  However, the map $f:BH\to BG$ is etale (Proposition \ref{localstructurecorollary}), hence $f_!$ sends compact objects to compact objects, and $f$ is even finite etale, hence $f_!$ sends dualizables to dualizables (check on pullback to a light profinite set).  Moreover, since $f^\ast$ is conservative, $f_!$ sends a set of generators to a set of generators.  Thus we deduce that $\on{Sh}_{et}(\ul{BG};\on{Sp})$ is likewise compactly generated by dualizable objects.  But Serre showed (\cite{SerreDimension}) that every compact $p$-torsionfree $p$-adic Lie group has finite mod $p$ cohomological dimension, bounded by the dimension of the manifold.  It follows from Remark \ref{cohdimfibers} that $\ul{BG}$ has finite mod $p$ cohomological dimension, and then we can apply Theorem \ref{descendabletheorem} (or rather Remark \ref{pcompletesemirigid}) to see that for the (proper, as $G$ is light profinite) map $e:S\to \ul{BG}$, the object $e_\ast\mbb{S}_{\wh{p}}$ is descendable.  Thus $\ul{BG}\to S$ is proper by Remark \ref{etalepropersmoothlocal}, as desired.

For claim 1 we proceed similarly.  We can check smoothness $!$-locally (Remark \ref{trulylocal}) and hence it suffices to consider the case where $X=S$ is a $\mathbb{Q}_p$-manifold, which we can even assume compact Hausdorff (hence light profinite).  Then by Proposition \ref{localstructurepadiclie} there exists a compact open subgroup object $H\subset G$ such that each fiber $H_s$ is a uniform pro-$p$-group.  Because $BH\to BG$ is etale and surjective (Proposition \ref{localstructurecorollary}), we can reduce to the case $H=G$.  Then the $(j_n)_!\mbb{S}/p$ from Example \ref{countableexamples} give compact generators, so applying Proposition \ref{smoothlemma}, it suffices to show that each of these goes to a compact object under $f_!$, and that $f^!\mbb{S}_{\wh{p}}$ is invertible.

We put aside the invertibility for now and come back to it later. Thus consider the problem of proving $f_!$ sends these compacts to compacts.  As $H^{p^n}$ is uniform whenever $H$ is, the problem is actually the same for all $n$: we just need to show that if $H$ is family of uniform pro-$p$-groups over $S$, then $f_!\mbb{S}/p$ is compact in $\on{Sh}_{et}(S;\on{Sp}_{\wh{p}})$.  By 2, $f$ is proper, so we can replace $f_!$ by $f_\ast$.  Note that $f_\ast\mbb{S}/p$ is bounded below by the finite cohomological dimension (\ref{uniformisgood}), so applying Lemma \ref{compactinvertible} we reduce to showing the following: the homotopy group sheaves of $f_\ast \mbb{F}_p$ are $0$ outside finitely many degrees and in all degrees are locally trivial with finite fibers.

We start by analyzing $H^1 f_\ast \mbb{F}_p = \pi_{-1}f_\ast \mbb{F}_p$.  The global sections are the homomorphisms of $S$-group objects from $H$ to the constant object $\mbb{F}_p$, and similarly for sections over any $S'\subset S$.  As $H^p\subset H$ is a normal subgroup object, this is the same as the homomorphisms $H/H^p\to \mathbb{F}_p$.  Recall (Proposition \ref{localstructurecorollary}) that $H/H^p$ is finite etale, and in fact a finite locally constant $\mbb{F}_p$-vector space sheaf.  This implies that $H^1 f_\ast \mbb{F}_p$ is just the $\mbb{F}_p$-linear dual of $H/H^p$, and the claim holds in degree $1$.  Now we recall that $f_\ast\mbb{F}_p$ is a commutative algebra object.  Note that for any local section $x\in H^1 f_\ast \mbb{F}_p$, the square $x^2$ vanishes in $H^2 f_\ast \mbb{F}_p$.  Indeed, as $f_\ast$ commutes with base-change we can check this fiberwise where it reduces to the classical fact about uniform pro-$p$-groups (\ref{scholiumuniform}).  It follows we can make a comparison map
$$\Lambda^i_{\mbb{F}_p} H^1f_\ast\mbb{F}_p \to H^if_\ast \mbb{F}_p$$
of sheaves of $\mbb{F}_p$-vector spaces on $S$ extending the identity map in degree $1$.  Again checking on stalks and using the classical fact we deduce that this is an isomorphism.  As the left hand side is clearly finite locally constant (given we already know this in degree 1) and vanishes for all but finitely many $i$, we have thus proved the desired claim that $f_\ast\mbb{S}_{\wh{p}}$ is dualizable, and so the formation of $f^!$ commutes with base-change and satisfies the projection formula by Lemma \ref{smoothlemma}.

What remains to be shown is that $f^!\mbb{S}_{\wh{p}}$ is an invertible object.  For this we first claim that (mod $p$) it is bounded below in the t-structure.  By base-change and testing on points we can reduce to the case $S=\ast$, and we only need see that $e^\ast f^!\mbb{S}/p$ is bounded below.  But by Proposition \ref{procontinuity} and a Postnikov tower argument we have that $e^\ast f^!\mbb{S}/p$ is the colimit of the sections of $f^!\mbb{S}/p$ over the $BH^{p^n}$. These are the duals to $(f_n)_\ast \mbb{S}/p$ which are bounded below and have a uniform bound on the $\mathbb{F}_p$-homology by the above calculations, implying the claim.

Thus by Lemma \ref{compactinvertible} it suffices to show that $f^!\mbb{F}_p$ is invertible in $\on{Sh}_{et}(BH;\on{D}(\mbb{F}_p))$ (back in the setting of general $S$).  For this, recall from Lemma \ref{uniformisgood} that $\on{Sh}_{et}(BH;\on{D}(\mbb{F}_p))$ is compactly generated by the unit.  It follows that the comparison map in $\on{CAlg}(\on{Pr}^L)$
$$\on{Mod}_{f_\ast\mbb{F}_p}(\on{Sh}_{et}(S;\on{D}(\mbb{F}_p)))\to \on{Sh}_{et}(BH;\on{D}(\mbb{F}_p))$$
is an isomorphism.  Now $f^!$ is just the double right adjoint to the pullback functor $\on{Sh}_{et}(S;\on{D}(\mbb{F}_p))\to\on{Sh}_{et}(BH;\on{D}(\mbb{F}_p))$.  Translating this over via the isomorphism, we find that it suffices to show that the double right adjoint to the base-change functor
$$-\otimes f_\ast \mbb{F}_p: \on{Sh}_{et}(S;\on{D}(\mbb{F}_p))\to \on{Mod}_{f\ast\mbb{F}_p}(\on{Sh}_{et}(S;\on{D}(\mbb{F}_p)))$$
sends the unit $\mbb{F}_p$ to an invertible object in the target.  This double right adjoint sends
$$M\mapsto \ul{\on{Map}}(f_\ast\mbb{F}_p;M)$$
where we take internal Hom in $\on{Sh}_{et}(S;\on{D}(\mbb{F}_p))$, and give this the $f_\ast\mbb{F}_p$-module structure coming from the source $f_\ast\mbb{F}_p$.

Thus it suffices to show that the dual $\ul{\on{Map}}(f_\ast\mbb{F}_p;\mbb{F}_p)$ of $f_\ast\mbb{F}_p$ is invertible as a module over $f_\ast \mbb{F}_p$.  But we calculated the homotopy of $f_\ast\mbb{F}_p$ as the exterior algebra on a locally free sheaf of finite rank $d$ in degree $-1$.  Thus the claim follows from the familiar self-duality (up to shift) of exterior algebras.
\end{proof}

We have identified various situations where a certain map $f$ is smooth.  The next important thing is to identify the dualizing object $f^!1$ in more explicit or tractable terms.  This is in fact the main point of this article.  But to accomplish it we need a digression.

\section{On nonstandard ``paths"}\label{pathsec}

Our main theorem says that there is an identification between two a priori different sheaves on $*/G$, where $G$ is a $p$-adic Lie group.  To accomplish this, we will construct, using the deformation to the tangent bundle, a certain sheaf on $(*/G)\times (\mathbb{Q}_p/\mathbb{Q}_p^\times)$ whose specialization to $1\in \mathbb{Q}_p$ gives the first sheaf and whose specialization to $0\in \mathbb{Q}_p$ gives the second sheaf.  In this way, our problem will be translated into the question of how to identify the stalks at two different points of a certain sheaf on the stack $\mathbb{Q}_p/\mathbb{Q}_p^\times$ (or rather, we need an in-families analog of this, over $\ast/G$, but we start with the simplest instance).

To explain our solution to this problem, let's look at an analogous but much more familiar situation.  Suppose we have a connected real manifold $M$ (or more generally a path-connected topological space $M$), two points $m_0,m_1\in M$, and a sheaf (say, of anima) $\mathcal{F}$ on $M$.  We want to try to identify the stalks $\mathcal{F}_0=m_0^\ast\mathcal{F}$ and $\mathcal{F}_1=m_1^\ast\mathcal{F}$.  One classical way to proceed is to choose a path $\gamma:[0,1]\rightarrow M$ with $\gamma(0)=m_0$ and $\gamma(1)=m_1$.  Replacing $\mathcal{F}$ with its pullback along $\gamma$ we can reduce to considering the problem for $M=[0,1]$.  The key property of $[0,1]$ that makes this an advantageous reduction is that the pullback functor $\pi^\ast:\operatorname{An}=\operatorname{Sh}(\ast)\rightarrow \operatorname{Sh}([0,1])$ is fully faithful; this is an expression of the contractibility of $[0,1]$, see \cite{LurieHA} A.2.  Moreover, a sheaf $\mathcal{F}$ lies in the essential image of $\pi^\ast$ if and only if $\pi^\ast\pi_\ast\mathcal{F}\overset{\sim}{\rightarrow}\mathcal{F}$, and it follows in particular that the essential image of $\pi^\ast$ is closed under all colimits and finite limits (in fact all limits, due to the existence of a left adjoint to $\pi^\ast$, a reflection of the fact that $[0,1]$ is not just contractible but also \emph{locally} contractible, again see \cite{LurieHA} A.2).

What this all means is that that the condition of ``being a constant sheaf'' on $[0,1]$ really is just a condition, not extra structure, and it is moreover a very robust condition.   On the full subcategory of constant sheaves, there is a canonical isomorphism $m_0^\ast\simeq m_1^\ast$, because both describe inverses to the isomorphism $\pi^\ast$. Thus, after a choice of path connecting $m_0$ and $m_1$ in $M$, we get a robust condition on $\mathcal{F}$ (that of having constant pullback to the path) which ensures a functorial identification of the stalks of $\mathcal{F}$ at these two different points.

We would like the same thing for the stack $\mathbb{Q}_p/\mathbb{Q}_p^\times$ instead of $M$, with respect to the points coming from $0,1\in\mathbb{Q}_p$.  Unfortunately, there is no map $[0,1]\rightarrow \mathbb{Q}_p/\mathbb{Q}_p^\times$ (say, of light condensed anima) sending $0$ to $0$ and $1$ to $1$: the totally disconnected nature of $\mathbb{Q}_p$ obstructs this rather severely.

However, in our setting we are not actually working with sheaves of anima; we are only working with $p$-complete sheaves of spectra.  This raises the possibility of repeating the above story but with a different compact Hausdorff space $\widetilde{I}$ instead of $[0,1]$, one for which $\pi^\ast:\on{Sh}(\ast;\mathcal{C})\rightarrow\on{Sh}(\widetilde{I};\mathcal{C})$ is not necessarily fully faithful for $\mathcal{C}=\on{An}$, but maybe only for $\mathcal{C}=\operatorname{Sp}_{\hat{p}}$, the $\infty$-category of $p$-complete spectra.

Remarkably, this idea works: there is such an $\widetilde{I}$, equipped with two points $i_0,i_1$, mapping to $\mathbb{Q}_p/\mathbb{Q}_p^\times$ in the appropriate way.  This $\widetilde{I}$ is a one-dimensional compact Hausdorff space which, while not path connected, is contractible from the perspective of torsion coefficients.  It is a variant of the famous ``topologist's sine curve'', ending in the solenoid $\varprojlim_N\mathbb{R}/N\mathbb{Z}$ instead of the usual closed interval.  We will describe $\widetilde{I}$ and the map $\widetilde{I}\rightarrow\mathbb{Q}_p/\mathbb{Q}_p^\times$ explicitly towards the end of the section; for now we start with some abstract generalities about such exotic ``intervals'', using the notion of etale sheaves with coefficients from the previous section (\ref{compassgood}).

\begin{definition}
Let $f:X\to S$ be a map light condensed anima and $\mathcal{C}\in\on{Pr}^L$ compactly assembled.  Say that $f$ is a \emph{$\mc{C}$-blob} if for every base-change $f'$ of $f$, the functor $(f')^\ast$ (on etale sheaves with $\mc{C}$-coefficients) is fully faithful.
\end{definition}

Under some hypotheses, we get reasonable fiberwise criteria.

\begin{lemma}\label{blob}
Suppose $\mathcal{C}\in\on{Pr}^L$ is compactly assembled, and let $f:X\to S$ be a map of light condensed anima.  Suppose that:
\begin{enumerate}
\item either the right adjoint to pullback on etale $\mc{C}$-sheaves exists for all base-changes of $f$ and commutes with (further) base-change, or the left adjoint to pullback on $\mc{C}$-sheaves exists for all base-changes of $f$ and commutes with base-change;
\item for all $s:\ast \to S$, the pullback $f_s:X_s\to \ast$ satisfies that $(f_s)^\ast$ is fully faithful.
\end{enumerate}
Then $f$ is an $\mc{C}$-blob.  Moreover, we have that a sheaf $\mathcal{F}\in \on{Sh}_{et}(X;\mc{C})$ lies in the essential image of $f^\ast$ if and only if for every $s:\ast \to S$, the pullback of $\mc{F}$ to $X_s$ is a constant sheaf, i.e.\ is pulled back from $\ast$.

In this situation we will call $f$ a \emph{right $\mathcal{C}$-blob} or a \emph{left $\mc{C}$-blob} depending on which version of condition 1 holds.
\end{lemma}
\begin{proof}
We need to see that for every base-change $f'$ of $f$, the pullback $(f')^\ast$ is fully faithful.  In the right blob case, we use the criterion which says that $(f')^\ast$ is fully faithful if and only if the unit map
$$\on{id}\rightarrow (f')_\ast (f')^\ast$$
is an isomorphism.  By Lemma \ref{compassgood}, we can test this on pullback to points.  But by assumption the formation of this unit map commutes with base-change.  Similarly in the left blob case we can use the similar criterion involving the left adjoint to $(f')^\ast$.
\end{proof}
\begin{remark}\label{propersmoothblob}
If $\mathcal{C}\in\on{CAlg}(\on{Pr}_{st}^L)$ is compactly assembled, then the right adjoint condition in 1 holds if $f$ is $\mc{C}$-proper, and the left adjoint condition holds if $f$ is $\mc{C}$-smooth.
\end{remark}

The following simple observation is then our main result in this abstract context.

\begin{theorem}\label{blobsections}
Let $\mathcal{C}$ be a compactly assembled presentable $\infty$-category, let $f:X\to S$ be either a left $\mc{C}$-blob or a right $\mc{C}$-blob, let $\mathcal{F}\in\on{Sh}_{et}(X;\mc{C})$, and let $\sigma_1,\sigma_2:S\to X$ be two sections of $f$.

If $\mathcal{F}$ restricted to the fiber $X_s$ is a constant sheaf for all $s:\ast\to S$, then there is a natural isomorphism $(\sigma_1)^\ast \mc{F}\simeq (\sigma_2)^\ast\mathcal{F}$.

\end{theorem}
\begin{proof}
By Lemma \ref{blob}, $\mathcal{F}\simeq f^\ast\mc{G}$ for some unique and functorial $\mc{G}\in\on{Sh}_{et}(S;\mc{C})$.  It follows that for any section $\sigma$ of $f$,
$$\sigma^\ast\mathcal{F}\simeq \mc{G},$$
giving the conclusion.
\end{proof}

The following criterion will help recognize blobs and constant sheaves for our purposes.

\begin{lemma}\label{explicitblob}
Let $p$ be a prime number and $K$ a second-countable compact Hausdorff space of finite (mod $p$) cohomological dimension.  Suppose that $R\Gamma(K;\mathbb{F}_p)=\mathbb{F}_p$.
Then $K$ is a right $\on{Sp}_{\wh{p}}$-blob, i.e.\ the map $f:K\to \ast$ is a right $\on{Sp}_{\wh{p}}$-blob.  

Moreover, if we further suppose that $\Gamma(K;BG)=BG$ for all finite groups $G$, then for $\mc{F}\in\on{Sh}_{et}(K;\on{Sp}_{\wh{p}})$ the following conditions are equivalent:
\begin{enumerate}
\item $\mathcal{F}$ is constant;
\item $\mc{F}$ is a colimit of dualizable objects.
\end{enumerate}
\end{lemma}
\begin{proof}
By Theorem \ref{chausproper} and Remark \ref{propersmoothblob}, $K$ will be a right $\on{Sp}_{\wh{p}}$-blob provided that $f^\ast:\on{Sp}_{\wh{p}}\to\on{Sh}_{et}(K;\on{Sp}_{\wh{p}})$ is fully faithful.  Thus it suffices to show that if $A$ is a $p$-torsion spectrum, then
$$A \overset{\sim}{\rightarrow} R\Gamma(K;A).$$
(See also \ref{etaleontopological} for the comparison of etale sheaves and usual sheaves.) Note that the mod $p$ cohomological dimension hypothesis passes to $p$-torsion abelian group sheaves by extensions and filtered colimits.  Thus we can argue using the Postnikov tower to reduce to the case where $A$ lives in degree $0$.  Then another devissage reduces to $A=\mathbb{F}_p$ which was the hypothesis.

Since every object in $\on{Sp}_{\wh{p}}$ is a filtered colimit of dualizable objects (the compact generator $\mbb{S}/p$ is dualizable), by pullback we see that 1 implies 2.  For the converse, we first claim that the collection of constant sheaves is closed under colimits.  This follows because $\mc{F}$ is constant if and only if the counit is an isomorphism:
$$f^\ast f_\ast \mc{F}\simeq \mc{F}.$$
Indeed, both $f^\ast$ and $f_\ast$ preserve colimits (the latter by the cohomological dimension estimate and a Postnikov tower argument).  Thus it suffices to show that any dualizable $\mc{F}$ is constant.  Since dualizable objects are bounded below (\ref{compactinvertible}) we can use the counit criterion again to see that it suffices to show $\mc{F}\otimes\mbb{F}_p$ is constant.  For this, since pullbacks are t-exact it suffices to show that each $\pi_i (\mc{F}\otimes \mbb{F}_p)$ is constant.  But by Lemma \ref{compactinvertible} each of these represented by a finite etale map $X\to K$, hence is locally constant with finite dimensional stalks.  From $H^0(K;\mbb{F}_p)=\mbb{F}_p$ we see that $K$ is connected so the dimension of the stalks must be constant. Thus $\pi_i (\mc{F}\otimes \mbb{F}_p)$ is classified by a map $K \to B\on{GL}_d(\mbb{F}_p)$ and hence is constant by assumption.
\end{proof}

Now we will describe our exotic ``interval''.  We start with the product space
$$(\mathbb{Z}\cup\{+\infty\})\times \mathbb{R},$$
equipped with the $\mathbb{Z}$-action where $k\in\mathbb{Z}$ acts by
$$(n,t) \mapsto (n+k,t-k),$$
where of course we say $+\infty+k=+\infty$, as required by continuity.  This is a free and properly discontinuous action, and we denote by
$$X:=\left((\mathbb{Z}\cup\{+\infty\})\times \mathbb{R}\right)/\mathbb{Z}$$
the quotient topological space.  The open-closed decomposition 
$$\mathbb{Z}\times\mathbb{R}\hookrightarrow (\mathbb{Z}\cup\{+\infty\})\times \mathbb{R} \hookleftarrow \{+\infty\}\times \mathbb{R}$$
is preserved by the $\mathbb{Z}$-action, and therefore descends to an open-closed decomposition of $X$.  The closed part $\{+\infty\}\times \mathbb{R}/\mathbb{Z}$ identifies with the circle $\mathbb{R}/\mathbb{Z}$; call the inclusion $i:\mathbb{R}/\mathbb{Z}\rightarrow X$.  On the other hand, the open part $(\mathbb{Z}\times\mathbb{R})/\mathbb{Z}$ identifies with $\mathbb{R}$ via the map
$$j:\mathbb{R}\overset{t\mapsto (0,t)}{\longrightarrow} \mathbb{Z}\times\mathbb{R}\rightarrow (\mathbb{Z}\times\mathbb{R})/\mathbb{Z},$$
with inverse induced by $(n,t)\mapsto t-n$.

\begin{remark}
In terms of the set-theoretic decomposition $X=\mathbb{R}\sqcup\mathbb{R}/\mathbb{Z}$ coming from this open-closed decomposition, we can describe the topology of $X$ as follows:
\begin{enumerate}
\item For $t\in\mathbb{R}$, a neighborhood base of $t\in X$ is given by the sets $U_{\epsilon,t}=(t-\epsilon,t+\epsilon)\subset\mathbb{R}$ for varying $\epsilon>0$.
\item For $\alpha\in\mathbb{R}/\mathbb{Z}$, let $\epsilon>0$ and $T\in\mathbb{R}$.  Define the set $U_{T,\epsilon,\alpha}\subset X$ to consist of those points in $\mathbb{R}/\mathbb{Z}$ of $\epsilon$-distance to $\alpha$, together with those points in $\mathbb{R}$  which are within $\epsilon$-distance of a coset representative $\widetilde{\alpha}\in\mathbb{R}$ of $\alpha$ satisfying $\widetilde{\alpha}>T$.  Then these $U_{T,\epsilon,\alpha}$ form a neighborhood basis of $\alpha$.
\end{enumerate}
From this description we easily see that $X$ is locally compact Hausdorff, and that a basis of open neighborhoods of $\mathbb{R}/\mathbb{Z}$ in $X$ is given by the $(T,+\infty)\sqcup \mathbb{R}/\mathbb{Z} \subset \mathbb{R}\sqcup\mathbb{R}/\mathbb{Z}=X$ for varying $T\in\mathbb{R}$.
\end{remark}

\begin{remark}\label{explicittopology}
We can give a more pictorial description of $X$ as follows.  consider the embedding of $\mathbb{R}$ into $\mathbb{R}^3$ via
$$t\mapsto (t,\on{cos}(t),\on{sin}(t)).$$
This gives a helix.  Then $X$ is the closure this helix inside the partial compactification $(\mathbb{R}\cup\{+\infty\})\times\mathbb{R}^2$ of $\mathbb{R}^3$.  In essence we add a boundary circle at $t=+\infty$.  This identification is implemented by the map
$$(\mathbb{Z}\cup\{+\infty\})\times\mathbb{R}\rightarrow (\mathbb{R}\cup\{+\infty\})\times\mathbb{R}^2$$
sending $(n,t)\mapsto (t+n,\on{cos}(t),\on{sin}(t))$  for $n\in\mathbb{Z}$ and $(+\infty,t)\mapsto (+\infty,\on{cos}(t),\on{sin}(t))$.

Of course, we could apply some increasing homeomorphism $\mathbb{R} \simeq (-\infty,0)$ to the $t$-axis to make the boundary circle lie above the point $t=0$ instead of $t=+\infty$ and in this way get a closed embedding of $X$ into $\mathbb{R}^3$, resembling the topologist's sine curve, but with a limiting circle instead of a limiting line segment.
\end{remark}

Now, let $N\in\mathbb{Z}_{\geq 1}$.  Consider the map
$$(\mathbb{Z}\cup\{+\infty\})\times\mathbb{R}\rightarrow (\mathbb{Z}\cup\{+\infty\})\times\mathbb{R}$$
given by multiplication by $N$ on each factor, so $(n,t)\mapsto (N\cdot n,N\cdot t)$ and $(+\infty,t)\mapsto (+\infty,N\cdot t)$.  This passes through the quotients, giving an induced map
$$[N]:X\rightarrow X.$$
This map $[N]$ respects the open-closed decomposition.  On the open part it is the map $\cdot N:\mathbb{R}\rightarrow\mathbb{R}$: in particular, a homeomorphism.  On the closed part, it is the map $\mathbb{R}/\mathbb{Z}\rightarrow\mathbb{R}/\mathbb{Z}$ induced by $\cdot N$.  This is the standard $N$-fold covering map of the circle.

For $N,N'\in\mathbb{Z}_{\geq 1}$ we have $[N]\circ [N']=[N\cdot N']$, whence a functor
$$(\mathbb{N},|)^{op}\rightarrow \on{Top}$$
from the opposite of the poset of natural numbers under divisibility to topological spaces, sending every $N$ to $X$ and where for $N|M$ the induced map $X\rightarrow X$ is given by $[M/N]$.  Then we take the filtered limit of this diagram, getting
$$\widetilde{X} := \varprojlim_N X.$$
This gets an open-closed decomposition inherited from $X$
$$\mathbb{R}\overset{\widetilde{j}}{\hookrightarrow} \widetilde{X}\overset{\widetilde{i}}{\hookleftarrow} \varprojlim_N \mathbb{R}/\mathbb{Z}$$
into the real line and a solenoid.

Now let's get rid of the non-compactness of these spaces, which comes from the open part.  Denote by $I\subset X$
the closed subset which is gotten by removing $(-\infty,0)$ from the open locus $j:\mathbb{R}\hookrightarrow X$.  We have $[N]^{-1}(I)=I$ for $N\in\mathbb{Z}_{\geq 1}$, and in the inverse limit we get a closed subset $\widetilde{I}\subset\widetilde{X}$, again just obtained from the open locus $\mathbb{R}$ by removing $(-\infty,0)$.

\begin{lemma}
Let $K$ denote either $I$ or $\widetilde{I}$.  Then $K$ is second countable compact Hausdorff and has covering dimension 1.
\end{lemma}
\begin{proof}
We have already seen that $X$, hence $\widetilde{X}$, is Hausdorff.  Second countability is clear from Remark \ref{explicittopology}. To show that $\widetilde{I}$ is compact, because $\widetilde{I}$ is an inverse limit of copies of $I$ it suffices to see that $I$ is.  However the composition
$$(\mathbb{N}\cup\{+\infty\})\times[0,1]\subset (\mathbb{Z}\cup\{+\infty\})\times \mathbb{R}\rightarrow X$$
has image $I$, so this follows from compactness of $\mathbb{N}\cup\{+\infty\}$ and $[0,1]$.  For the covering dimension assertion, clearly our space does not have covering dimension $0$, so it suffices to bound the covering dimension by 1.  Again because $\widetilde{I}$ is a filtered inverse limit of copies of $I$, to see that $\widetilde{I}$ has covering dimension $\leq 1$ it suffices to see the same for $I$.  For that we use the neighborhood bases from Remark \ref{explicittopology} (say with $\epsilon<1/2$ for clarity).  The boundary of the neighborhood of a point $t\in [0,\infty)$ is either 2 points if $t\neq 0$ or 1 point if $t=0$, and the boundary of the neighborhood of a point $\alpha\in\mathbb{R}/\mathbb{Z}$ is a disjoint union of two copies of $\mathbb{N}\cup\{+\infty\}$.  In all cases the boundary is 0-dimensional, so our space has inductive dimension $1$, hence covering dimension $\leq 1$ as desired.
\end{proof}

\begin{lemma}\label{reducetosolenoid}
Let $K$ denote either $I$ or $\widetilde{I}$, let $i:Z\subset X$ denote the closed subset from the open-closed decomposition, and let $j:U\subset X$ denote the open complement.  (So for $K=I$ we have $U= [0,+\infty)$ and $Z=\mathbb{R}/\mathbb{Z}$, while for $K=\widetilde{I}$ we have $U=[0,+\infty)$ and $Z=\varprojlim_N \mathbb{R}/\mathbb{Z}$.)

Let $\mathcal{F}$ be a sheaf of anima on $K$ such that $j^\ast \mathcal{F}$ is a constant sheaf on $U$.  Then
$$\Gamma(K,\mathcal{F})\overset{\sim}{\rightarrow}\Gamma(Z;i^\ast\mathcal{F}).$$
\end{lemma}
\begin{proof}
Since $K$ is paracompact Hausdorff, by \cite{LurieHTT} 7.1.5.6 we have
$$\Gamma(Z;i^\ast\mathcal{F})=\varinjlim_V \Gamma(V;\mathcal{F})$$
where $V$ runs over open neighborhoods of $Z$ in $K$.  Thus we get
$$\Gamma(X;\mathcal{F})\overset{\sim}{\rightarrow}\Gamma(Z;i^\ast\mathcal{F})\times_{\varinjlim_V \Gamma(V\cap U;\mathcal{F})}\Gamma(U;\mathcal{F}).$$
It follows that a sufficient condition for $\Gamma(K;\mathcal{F})\overset{\sim}{\rightarrow}\Gamma(Z;i^\ast\mathcal{F})$ is that $Z$ should have a neighborhood basis of open subsets $V$ such that $\Gamma(U;\mathcal{F})\overset{\sim}{\rightarrow}\Gamma(U\cap V;\mathcal{F})$.  In our setting, such a neighborhood basis is given, in as in Remark \ref{explicittopology}, by the $(T,\infty)\sqcup Z$ for $T>0$.  As both $[0,\infty)$ and $(T,\infty)$ are contractible, they have the same sections for a constant sheaf, whence the claim.
\end{proof}

\begin{remark}
In my original approach to the main result in this section (Theorem \ref{deformiso}), I was using the topos $\mathcal{T}=\on{Sh}(\mathbb{Z}\cup\{+\infty\})^\mathbb{Z}$ and its inverse limit $\widetilde{\mathcal{T}}$ instead of these topological spaces $X$ and $\widetilde{X}$.  The relationship between the two is that $\mathcal{T}$ is the category of constructible sheaves on $X$ (or $I$) with respect to the stratification with 2 strata given by the open-closed decomposition.

In the context of $\mathcal{T}$, Lemma \ref{reducetosolenoid} was pointed out to me by Lucas Mann.  Using it allowed to greatly simplify my original construction of the comparison isomorphism between the stalks at $0$ and $1$ from Theorem \ref{deformiso}.  The idea to switch from the topos $\mathcal{T}$ to the corresponding topological space $X$ was suggested to me by Shachar Carmeli.  It makes the presentation more accessible.  I thank them both for their very helpful remarks.
\end{remark}

The following attempts to express the extent to which $\widetilde{I}$ is ``contractible'' from the perspective of sheaf theory.

\begin{lemma}\label{itscontractible}
Let $A\in\on{An}$.  Suppose that for each $a\in A$ and each $i\geq 1$, the homotopy group $\pi_i(A,a)$ is torsion (i.e.\ every element has finite order).  Then $\Gamma(\widetilde{I};A)=A$.
\end{lemma}
\begin{proof}
By Lemma \ref{reducetosolenoid}, we reduce to the analogous claim for the solenoid $S=\varprojlim_N \mathbb{R}/\mathbb{Z}$.  The functor $A\mapsto \Gamma(S;A)$ preserves finite limits, as with any topological space.  Because $S$ is paracompact Hausdorff of finite covering dimension, it also preserves limits of Postnikov towers (\cite{LurieHTT} 7.2).  Thus, by the usual induction with the Postnikov tower, it suffices to show:
\begin{enumerate}
\item For any set $A$ we have $\Gamma(S;A)=A$.
\item For any torsion group $G$, every $G$-torsor on $S$ is trivial.
\item For any torsion abelian group $A$ we have $H^i(S;A)=0$ for $i>1$ (the case $i=1$ being handled by 2).
\end{enumerate}
Point 1 is the same as the connectedness of the solenoid, which follows from the connectedness of the circle.  For point 2, by compactness a $G$-torsor on $S$ descends to the $N^{th}$ stage $\mathbb{R}/\mathbb{Z}$ for some $N$.  There, by contractibility of $\mathbb{R}$, it is induced by a homomorphism $\mathbb{Z}\rightarrow G$.  If $G$ is torsion then this homomorphism becomes trivial on composing with $\cdot M:\mathbb{Z}\rightarrow \mathbb{Z}$ for some $M\in\mathbb{Z}_{\geq 1}$, whence triviality of our descended $G$-torsor on pullback to the $(N\cdot M)^{th}$ stage, whence triviality of our original $G$-torsor on the solenoid.  Finally, 3 is true for any sheaf of abelian groups by the fact that $S$ has covering dimension 1.
\end{proof}

\begin{remark}\label{itsablob}
In particular, the previous results imply that the hypotheses of Lemma \ref{explicitblob} are satisfied.  Thus $\widetilde{I}$ is a $\on{Sp}_{\wh{p}}$-blob and an object $\mc{F}\in\on{Sh}_{et}(\widetilde{I};\on{Sp}_{\wh{p}})$ is constant if and only if it is a colimit of dualizable objects.  
\end{remark}

Now we describe the map of light condensed anima $\gamma:\widetilde{I}\rightarrow\mathbb{Q}_p/\mathbb{Q}_p^\times$ which we use to ``connect'' the points $0$ and $1$.  More generally, we will give such a map for any local field $F$ and any choice of $q\in F$ with $0<|q|<1$.  (For $F=\mathbb{Q}_p$ we can of course choose $q=p$).

For this, we take the composition
$$\widetilde{I}\to I \to \mathbb{Z}\cup\{+\infty\}/\mathbb{Z} \to F/F^\times,$$
where:
\begin{enumerate}
\item The first map is the natural projection to the $1^{st}$ term in the inverse limit defining $\widetilde{I}$;
\item The second map is projection off $\mathbb{R}$ in terms of the definition $I=((\mathbb{Z}\cup\{+\infty\})\times \mathbb{R})/\mathbb{Z}$;
\item The third map is induced by $n\mapsto q^n$.
\end{enumerate}

Note that the points $0,1\in F$ lift along $\gamma$ to points $i_0,i_1\in \widetilde{I}$.  For example we can take $i_0$ to be the identity element of the solenoid closed subspace of $\widetilde{I}$ and we can take $i_1$ to be $0$ inside the $[0,+\infty)$ open subspace of $\widetilde{I}$.

We get the following, which is our main result in this section. Write $0,1:\ast \to F$ for the specified points of $F$ (or their composition to $F/F^\times$).

\begin{theorem}\label{deformiso}
Let $F$ be a local field and $p$ a prime.  Suppose given $X\in\on{CondAn}^{light}$ and $\mathcal{F}\in\on{Sh}_{et}(X\times(F/F^\times);\on{Sp}_{\wh{p}})$ satisfying the following condition: for all $x:\ast\rightarrow X$, the sheaf $x^\ast\mathcal{F}\in\on{Sh}_{et}(F/F^\times;\on{Sp}_{\wh{p}})$ is a colimit of dualizable objects.  (We will only apply this when $\mc{F}$ itself is invertible.)

Then there is a natural isomorphism $0^\ast\mathcal{F}\simeq 1^\ast\mathcal{F} \in \on{Sh}_{et}(X;\on{Sp}_{\widehat{p}})$.
\end{theorem}
\begin{proof}
Let $\mathcal{G} = \gamma^\ast \mathcal{F} \in \on{Sh}_{et}(X\times \widetilde{I};\on{Sp}_{\wh{p}})$.  By Remark \ref{itsablob}, $\widetilde{I}$ is a right $\on{Sp}_{\wh{p}}$-blob and $\mathcal{G}$ lies in the essential image of the pullback functor $\on{Sh}_{et}(X;\on{Sp}_{\wh{p}})\to\on{Sh}_{et}(X\times \widetilde{I};\on{Sp}_{\wh{p}})$.  Thus by Theorem \ref{blobsections} we have $i_0^\ast\mathcal{G}\simeq i_1^\ast\mathcal{G}$, which gives the claim.
\end{proof}

This was valid for all complete valued fields $F$, but of course if $F$ is archimedean (i.e.\ $F=\mathbb{R}$ or $F\simeq\mathbb{C}$), then we can do better, because already $F$ itself is a left $\mathcal{C}$-blob for all $\mathcal{C}$ (see \cite{LurieHA} A.2); we don't even need the $F^\times$-equivariance.  Thus the same arguments as above give the following analog of the previous theorem.

\begin{theorem}
Let $F$ be an \emph{archimedean} complete normed field, let $X\in\on{CondAn}^{light}$, and let $\mathcal{F}\in\on{Sh}_{et}(X\times F;\on{Sp})$.  Suppose that for all $x:\ast\rightarrow X$, the sheaf $x^\ast\mathcal{F}\in\on{Sh}(F;\on{Sp})$ is is locally constant.  Then there is a natural isomorphism $0^\ast\mathcal{F}\simeq 1^\ast\mathcal{F} \in \on{Sh}_{et}(X;\on{Sp})$.
\end{theorem}

\section{Duality and linearization for real and p-adic Lie groups}\label{dualsec}

Recall from Theorem \ref{realsmooth} that for a submersion $f:X\to Y$ in $\on{Sh}(\on{Man}_\mathbb{R})$, the map $f:\underline{X}\to\underline{Y}$ of light condensed anima is $\on{Sp}$-smooth.  In particular we have a dualizing object $f^!\mbb{S}\in \on{Sh}_{et}(\ul{X};\on{Sp})$.  Now we want to identify this in terms of the relative tangent bundle $T_{X/Y}$, as in Atiyah duality.  First let's explain what construction associated to the tangent bundle we're aiming for.  It is defined for an arbitrary vector bundle.  We start with a description in terms of the six functor formalism.

\begin{definition}
Let $X\in\on{Sh}(\on{Man}_\mbb{R})$ and let $f:V\to X$ be a vector bundle on $X$, with zero section $e:X \to V$.  Define
$$\mathbb{S}^V = e^\ast f^!\mbb{S} \in \on{Sh}_{et}(\ul{X};\on{Sp}).$$
\end{definition}

By smoothness of $f$, this is an invertible object.  It follows formally that it is locally constant in the sense of \ref{locconst} (the proof of smoothness of $f$ also shows this directly).  Thus it lifts to $\on{Sh}_{et}(|X|;\on{Sp}) = \on{An}_{/|X|}\otimes\on{Sp}$ by \ref{locconst}: in other words it comes from a local system of spectra on the anima $|X|$ underlying $X$.  Let us now prove this in a different, more informative way, by explicitly describing a lift of $\mbb{S}^V$ to $\on{An}_{/|X|}$.

\begin{lemma}
Let $f_\natural$ denote the left adjoint to $f^\ast$ (it exists and commutes with base change by smoothness).  Then there is a natural isomorphism (compatible with base change)
$$\mathbb{S}^V \simeq f_\natural e_\ast\mbb{S}.$$
\end{lemma}
\begin{proof}
We have $f_\natural(-) = f_!(-\otimes f^!\mbb{S})$ because $f^!(-)\simeq f^\ast(-)\otimes f^!\mbb{S}$.  Thus
$$f_\natural e_\ast \mbb{S} = f_!(f^!\mbb{S}\otimes  e_\ast \mathbb{S}) = f_!(e_\ast(e^\ast f^!\mbb{S}))$$
by the projection formula for the closed inclusion $e$.  However $e_\ast = e_!$ again because it's a closed inclusion.  The claim follows.
\end{proof}

To move forward, we consider the open complement $j:V\setminus 0 \to V$ to $e$.

\begin{lemma}
There is a natural isomorphism
$$\mbb{S}^V \simeq \on{cofib}((p\circ j)_\natural\mbb{S} \to p_\natural\mbb{S}).$$
\end{lemma}
\begin{proof}
We have a cofiber sequence
$$j_\natural\mathbb{S} \to \mathbb{S} \to e_\ast\mathbb{S}$$
by Lemma \ref{recollement}.  The result follows by applying $p_\natural$.
\end{proof}

Now that we've translated everything in terms of the left adjoint to pullback, we can move to the homotopy theory context using the following.

\begin{lemma}
Let $f:Y \to X$ be a map in $\on{Sh}(\on{Man}_{\mbb{R}})$.  Suppose $f$ is a \emph{fiber bundle} in the sense that, locally on $X$, it is isomorphic to a projection map $U\times F\to U$ for some $U\in\on{Man}_{\mbb{R}}$ and $F\in\on{Sh}(\on{Man}_\mbb{R})$.  Then:
\begin{enumerate}
\item the left adjoint $f_\natural$ to $f^\ast:\on{Sh}_{et}(X;\on{An})\to\on{Sh}_{et}(Y;\on{An})$ exists and commutes with all base-change;
\item this left adjoint $f_\natural$ sends locally constant sheaves to locally constant sheaves (\ref{locconst}), and hence a fortiori restricts to the left adjoint to the pullback $\on{An}_{/|X|}\to\on{An}_{/|Y|}$ via the comparison \ref{locconst}.
\end{enumerate}
\end{lemma}
\begin{proof}
For 1, note that such claims can be checked locally by descent, so it suffices to handle the case of $f:F\to \ast$ for some $F\in\on{Sh}(\on{Man}_F)$.  Resolving $F$ by open polydisks we further reduce to open polydisks, and then to the unit interval.  In that case the claim is elementary, see \cite{LurieHA} A.2.

Again for 2, we reduce to the case of a projection which, by commutation with base-change, reduces to $f:F\to \ast$, where the claim is trivial.
\end{proof}

In total we deduce the following:

\begin{proposition}\label{homotopytwists}
Let $X\in\on{Sh}(\on{Man}_F)$ and let $V\to X$ be a vector bundle.  Write $V\setminus 0$ for the complement of the zero section.  View $|V\setminus 0|$ and $|V|$ as objects in $\on{An}_{/|X|}$, and write $\Sigma^\infty$ for the suspension spectrum functor from pointed objects in $\on{An}_{/|X|}$ to its stabilization.  Then
$$\mathbb{S}^V \simeq \Sigma^\infty(|V|/|V\setminus 0|)$$
\end{proposition}
\begin{proof}
This follows by combining the two previous lemmas, noting that $V\to X$ and $V\setminus 0 \to X$ are both fiber bundles.
\end{proof}

Now that we've discussed $\mathbb{S}^V$, we can give the proof of the Atiyah duality statement.  But first we isolate a very simple lemma which we will use repeatedly.

\begin{lemma}\label{universalsection}
Let $f:X\to Y$ be a smooth map (with respect to some six functor formalism).  Then we have
$$f^!1 \simeq \delta^\ast \pi^! 1,$$
where $\delta: X\to X\times_YX$ is the diagonal and $\pi:X\times_YX\to X$ is the first projection.
\end{lemma}
\begin{proof}
By base change for upper shriek in the pullback $X\times_Y X$, we have $(\pi_2)^\ast f^!1 \simeq \pi^! 1$ where $\pi_2$ is the second projection.  Applying $\delta^\ast$ we deduce the claim.
\end{proof}

Thus dualizing objects can be simply recovered from what should just be a specialization of them, namely their pullback along a section, as long as we allow ourselves to work in families and change the base.  Of course this is a general principle: $X$ is the ``moduli space of points in $X$'', and the universal section proving this is the diagonal section of the first projection.  But combined with the deformation to the tangent space, it gives a very useful technique for calculating dualizing objects, as in the next proof.

\begin{theorem}
Let $f:X \to Y$ be a representable submersion in $\on{Sh}(\on{Man}_\mbb{R})$.  Write $T_f$ for the relative tangent bundle.  Then there is a natural isomorphism
$$f^!\mbb{S} \simeq \mbb{S}^{T_f}$$
in $\on{Sh}_{et}(\ul{X};\on{Sp})$.
\end{theorem}
\begin{proof}
By the lemma, $f^!\mbb{S}\simeq \delta^\ast \pi^!1$.  Now we consider the deformation to the tangent bundle from \ref{deformman}, which we can extend to representable submersions by descent.  This gives a representable submersion

$$\Pi:D(X/Y) \rightarrow X \times \mbb{R}$$

\noindent equipped with a section $\Sigma$, such that over $1\in \mbb{R}$ we recover $\pi$ with its section $\delta$, while over $0\in \mbb{R}$ we recover $T_f \to X$ together with its zero section.  Because $\Pi$ is a representable submersion, by Theorem \ref{realsmooth} we have that $\Pi^!\mbb{S}$ is invertible (and indeed locally constant), and hence so is $\Sigma^\ast \Pi^!\mbb{S}$.  Therefore by Theorem \ref{deformiso} we get a canonical isomorphism
$$1^\ast \Sigma^\ast \Pi^!\mbb{S}\simeq 0^\ast \Sigma^\ast \Pi^!\mathbb{S}.$$
By base-change (which for upper $!$ functors is a part of smoothness), this gives the desired isomorphism
$$\delta^\ast \pi^\ast \mbb{S}\simeq \mbb{S}^{T_f}.$$\end{proof}

Now let's discuss classifying stacks of Lie groups.  Recall from Corollary \ref{realartinsmooth} that these are also smooth.  The same argument as above gives the following.

\begin{lemma}
Let $X\in\on{Sh}(\on{Man}_\mbb{R})$, and let $G\to X$ be a group object in representable submersions over $X$.  Write $f:BG\to X$ for the relative classifying stack, and $e:X\to BG$ for the canonical quotient map.  Let also $\mathfrak{g}\to X$ denote the Lie algebra of $G$, viewed as a vector bundle on $X$.

Then there is a natural isomorphism
$$e^\ast f^!\mbb{S} \simeq \mathbb{S}^{-\mathfrak{g}},$$
where we write $\mathbb{S}^{-\mathfrak{g}}$ for the inverse of $\mbb{S}^{\mathfrak{g}}$.
\end{lemma}
\begin{proof}
We deform $G$ to its Lie algebra (\ref{deformlie}), giving the group object in representable submersions
$$DG\to X\times \mbb{R}.$$
Pass to relative classifying stacks.  Applying Theorem \ref{deformiso} this gives
$$e^\ast f^!\mbb{S} \simeq (e')^\ast (f')^!\mbb{S}$$
where we write $f':B\mathfrak{g}\to X$ and $e':X\to B\mathfrak{g}$ for the maps analogous to $f$ and $e$ but applied to the additive Lie group $\mathfrak{g}$.  Thus it suffices to show for an arbitrary vector bundle $V$ over $X$ that upper shriek followed by upper star in
$$X \to BV \to X$$
gives $\mbb{S}^{-V}$.  But upper shriek followed by upper shriek obviously gives $\mbb{S}$, so by projection formula for upper shriek it suffices to see that upper shriek of $\mathbb{S}$ along $X \to BV$ gives $\mbb{S}^V$.  This follows by Lemma \ref{universalsection}.
\end{proof}

Now by combining with another general principle we deduce the following.

\begin{theorem}\label{realgrouplinearize}
Let $G$ be a real Lie group.  Then for the projection $f:BG \to \ast$ (which we recall is smooth by Corollary \ref{realartinsmooth}), we have
$$f^!\mbb{S} \simeq \mbb{S}^{-\on{ad}_G}$$
in $\on{Sh}_{et}(BG;\on{Sp})$, where $\on{ad}_G$ denotes the adjoint representation of $G$, viewed as a vector bundle on $\ast/G$.
\end{theorem}
\begin{proof}
By Lemma \ref{universalsection}, $f^!\mbb{S}$ identifies with upper shriek followed by upper star in
$$BG\to BG\times BG \to BG$$
where the first map is the diagonal $\delta$ and the second map is the first projection $\pi$.  On the other hand this diagram is isomorphic to $f$ and $e$ of the previous lemma in the following situation: take $G$ with its adjoint action and view it as a group $G^{ad}$ over $BG$.  We conclude by applying the previous lemma.
\end{proof}

\begin{remark}
We used an identification of $BG \times BG$ with the relative classifying stack (over $BG$) associated to the adjoint action of $G$ on itself.  This may seem opaque.  We offer the following perspective.  For any groupoid $X$, there is a canonical functor
$$\on{Ad}_X: X \to \on{Gp}$$
to the category of groups.  It sends a point $x\in X$ to the group $\on{Aut}(x)$, and a map $\gamma:x\to y$ to the isomorphism $\on{Aut}(x)\to \on{Aut}(y)$ given by conjugation by $\gamma$.  The composition with the functor from groups to groupoids given by taking the classifying space has a different description: it sends $x\in X$ to the component of $X$ in which $x$ lives.  The identification is gotten by the obvious inlusion functor $B\on{Aut}(x)\rightarrow X$ based at $x$.

In particular, if $X$ is connected this composed functor to groupoids is constant with value $X$.  On the other hand if $X$ is also pointed, hence identifies with $BG$ for a group $G$, then $\on{Ad}_X$ by construction encodes the adjoint action of $G$ on itself.

Taking Grothendieck constructions, this provides the desired canonical identification.  It is functorial in $G$ and therefore passes to topoi.
\end{remark}

\begin{example}
Suppose that $G$ is a real Lie group which is \emph{compact}.  Then the diagonal of $p:BG \to \ast$ is $\on{Sp}$-proper (since it is a locally isomorphic to projection off $G$).  Thus by Definition \ref{etalepropersmoothdef} we get a natural transformation
$$p_! \to p_\ast.$$
On the other hand, passing to left adjoints in $f^!\simeq f^\ast \otimes f^!\mbb{S}$ and using Theorem \ref{realgrouplinearize} shows that
$$p_!(-) \simeq p_\natural(- \otimes \mbb{S}^{\on{ad}_G}).$$

Thus we obtain a natural transformation
$$p_\natural(-\otimes \mbb{S}^{\on{ad}_G}) \to p_\ast(-)$$
of functors $\on{Sh}_{et}(BG;\on{Sp})\to \on{Sp}$, relating the left and right adjoints to the pullback.  But recall from \ref{locconst} that $\on{Sh}_{et}(BG;\on{Sp}) = \on{Sh}_{et}(B|G|;\on{Sp})$ identifies with the $\infty$-category of spectra with $|G|$-action.  In these terms $p_\natural$ stands for the $|G|$-orbit functor and $p_\ast $ for the $|G|$-fixed point functor.  Thus we recover the twisted transfer maps of \cite{AdemTransfer}.

Note that the comparison of the two a priori different maps can be made by seeing they both satisfy the conditions of \cite{NikolausTC} Thm.\ I.4.1.  For the map constructed above, we can argue for this as follows:  because $e:\ast \to \ast/G$ is proper we have that $p_!\to p_\ast$ is an isomorphism on any $e_\ast X$, but by duality for $e$ (a proper representable submersion) this is the same as the $e_\natural X[d]$, hence one sees that $p_!\to p_\ast$ is an iso on compact objects and hence $p_!$ is the left Kan extension of the restriction of $p_\ast$ to the compact objects.
\end{example}

Now we discuss the $p$-adic case.  It's pretty similar, but one difference is that we can't define $\mbb{S}^V$ exactly as in the real case, because $\mbb{Q}_p\to \ast$ is not smooth.  This problem goes away when we pass to classifying stacks, though.

\begin{definition}\label{spherefromvb}
Let $X\in \on{Sh}(\on{Man}_{\mbb{Q}_p})$, and let $V\to X$ be a vector bundle on $X$.  Define
$$\mathbb{S}_{\wh{p}}^{V} = e^\ast f^! \mbb{S}_{\wh{p}},$$
where $X\overset{e}{\rightarrow} BV \overset{f}{\rightarrow} X$ are the quotient and projection maps associated to the relative classifying stack $BV$.
\end{definition}

Note that $BV \to X$ is indeed ($!$-able and) smooth: the universal case is $X=B\on{GL}_d(\mbb{Q}_p)$ which is $!$-good, so this follows from Theorem \ref{padicsmooth} and closure under pullbacks of $!$-able smooth maps.  Thus, as in the real case, $\mathbb{S}_{\wh{p}}^{V}$ is invertible.  Moreover its $\mathbb{F}_p$-homology lives only in degree $d=\on{dim}(V)$, again just as in the real case (see the proof of Theorem \ref{padicsmooth}).

Proceeding exactly as in the real case we get the following.

\begin{lemma}\label{padiclemma}
Let $X$ be a $!$-good $p$-adic analytic smooth Artin stack (\ref{good}), and let $G\to X$ be a group object in representable submersions over $X$.  Write $f:BG\to X$ for the relative classifying stack, and $e:X\to BG$ for the canonical quotient map.  Let also $\mathfrak{g}\to X$ denote the Lie algebra of $G$, viewed as a vector bundle on $X$.

Then there is a natural isomorphism
$$e^\ast f^!\mbb{S}_{\wh{p}} \simeq \mathbb{S}_{\wh{p}}^{\mathfrak{g}}.$$
\end{lemma}
\begin{proof}
We once again apply the deformation to the tangent bundle, specialized along the identity section, see \ref{deformlie}.  This gives a relative Lie group
$$DG \to X\times (F/F^\times)$$
which over $1$ recovers $G$ and over $0$ recovers $\mathfrak{g}$ with its additive Lie group structure.  Pass to the relative classifying stack of $DG$, call it $BDG$.  Note that every stack which appears is still $!$-good, \ref{lotsaregood}.  Then by Theorem \ref{deformiso} and base-change we deduce the claim.
\end{proof}

Now we can deduce the main result of this article.

\begin{theorem}\label{maintheorem}
Let $G$ be a $p$-adic Lie group.  Then for $f:BG\to \ast$, we have
$$f^!\mbb{S}_{\wh{p}} \simeq \mbb{S}_{\wh{p}}^{\on{ad}_G}$$
where $\on{ad}_G$ denotes the adjoint representation of $G$, viewed as a vector bundle on $BG$.
\end{theorem}
\begin{proof}
By base-change, $f^!\mbb{S}_{\wh{p}}$ identifies with the shriek pullback followed by star pullback in
$$BG \to BG \times BG \to BG$$
which in turn identifies with the $e$ and $f$ from Lemma \ref{padiclemma} applied to the group over $BG$ corresponding to the adjoint action of $G$.  Thus Lemma \ref{padiclemma} gives the conclusion.
\end{proof}

\begin{example}
Suppose furthermore that $G$ is compact and $p$-torsionfree.  Then $f_!\simeq f_\ast$ by Theorem \ref{padicsmooth}.  Thus we deduce that
$$f_\natural(-) \simeq f_\ast(-\otimes \mathbb{S}_{\wh{p}}^{\on{ad}_G}).$$
This means that homology and cohomology on $BG$ with $p$-complete spectrum coefficients differ by this explicit twist.
\end{example}

Given this result, one would like to get some handle on the construction $V\mapsto \mbb{S}_{\wh{p}}^V$ from Definition \ref{spherefromvb}.  If we change coefficients from $\on{Sp}_{\wh{p}}$ to $\on{D}(\mbb{Z})_{\wh{p}}$, a very explicit description is possible.  This also shows that our duality does recover Lazard's duality with the same explicit twist.

\begin{proposition}\label{classicaltwists}
Let $V$ be a vector bundle over an $X\in\on{Sh}(\on{Man}_{\mbb{Q}_p})$, say of constant dimension $d$.  Write $\Lambda^d V$ for the top exterior power of $V$.  This is a one-dimensional vector bundle, hence is classified by a map
$$X\to B\mathbb{Q}_p^\times$$
in $\on{Sh}(\on{Man}_{\mbb{Q}_p})$.  Composing with the homomorphism
$$\mbb{Q}_p^\times \to \mbb{Z}_p^\times$$
which is the identity on $\mbb{Z}_p^\times$ and sends $p\mapsto 1$, we deduce that $\Lambda^d V$ admits a canonical $\mbb{Z}_p$-structure; denote this by $(\Lambda^d V)_{\mbb{Z}_p}$.  Writing this as the limit of its mod $p^n$ versions, which are finite etale, we can view $(\Lambda^d V)_{\mbb{Z}_p}$ as an object in $\on{Sh}_{et}(X;\on{D}(\mbb{Z})_{\wh{p}})$.  Then we have
$$(\mbb{S}_{\wh{p}}^V\otimes \mathbb{Z})_{\wh{p}}\simeq (\Lambda^d V)_{\mbb{Z}_p}[d].$$
\end{proposition}
\begin{proof}
The universal case is $X= B\on{GL}_d(\mbb{Q}_p)$.  We know that $(\mbb{S}_{\wh{p}}^V\otimes \mathbb{Z})_{\wh{p}}$ is invertible and lives in degree $d$, so its $(-d)$-shift is classified by a map of light condensed anima
$$B\on{GL}_d(\mbb{Q}_p) \to B\mbb{Z}_p^\times,$$
or in other words a homomorphism $\on{GL}_d(\mbb{Q}_p)\to \mbb{Z}_p^\times$.  We need to see that this identifies with the determinant map followed by the homomorphism $\mathbb{Q}_p^\times \to \mathbb{Z}_p^\times$ considered in the statement.  Because the abelianization of $\on{GL}_d(\mbb{Q}_p)$ is realized by the determinant, we can restrict along $\on{GL}_1(\mathbb{Q}_p)\to \on{GL}_d(\mathbb{Q}_p)$ and thereby reduce to the case $d=1$.

Thus we're left with showing the following: consider the trivial one-dimensional vector bundle $\mbb{Q}_p\to \ast$ over the point, and use this to form
$$(\mbb{S}_{\wh{p}}^{\mbb{Q}_p} \otimes \mbb{Z})_{\wh{p}}[-1].$$
By functoriality we have an action of $\mbb{Q}_p^\times$ on this free $\mbb{Z}_p$-module of rank $1$, induced by the action by scalar multiplication on the group object $\mathbb{Q}_p$.  We need to show that $p$ acts trivially and $\mathbb{Z}_p^\times$ acts by scalar multiplication.

To see that $p$ acts trivially, note $p\in\mbb{Z}[1/p]^\times$, so by the reciprocity explained in the next section (Section \ref{jsec}), we can reduce to showing the analogous result for the the action of $p\in\mbb{R}^\times$ on $\mbb{S}^{\mbb{R}}$.  But then it follows because $p$ lies in the same connected component of $\mbb{R}^\times$ as $1$.

To see that $\mbb{Z}_p^\times$ acts by scalar multiplication, we can pass from $B\mathbb{Q}_p$ to $B\mathbb{Z}_p$, using that $B\mbb{Z}_p\to B\mbb{Q}_p$ is etale.  For clarity of notation let's replace $\mbb{Z}_p$ by an arbitrary free $\mbb{Z}_p$-module of rank one $L$.  We're trying to show that the twist for $\mbb{Z}_p$-coefficients of $BL$ identifies with $L[1]$.

This twist is $\mbb{D}_L:=e^\ast f^!\mbb{Z}_p$ in the diagram
$$\ast \to BL \to \ast.$$
But since $L$ is abelian, $BL$ gets a group structure, so $f^!\mbb{Z}_p$ identifies with the constant sheaf on our twist, say $f^!\mbb{Z}_p = f^\ast \mbb{D}_L$.  Adjunction gives a natural map
$$f_! f^! \mbb{Z}_p \to \mbb{Z}_p$$
which is split surjective (by considering $e$).  On the other hand by properness of $f$ and the above we can rewrite this split surjective map as
$$(f_\ast \mbb{Z}_p)\otimes \mbb{D}_L \to \mbb{Z}_p.$$
Looking in degree $0$ we deduce that $\on{Hom}(L,\mbb{Z}_p)[-1]\otimes \mbb{D}_L\simeq \mathbb{Z}_p$, so that $\mbb{D}_L\simeq L[1]$ as desired.

\end{proof}

\begin{remark}
This is analogous to the following well-known Archimedean fact: if $V$ is a vector bundle of dimension $d$ over $X\in\on{Sh}(\on{Man}_{\mbb{R}})$, then $\mbb{S}^V\otimes\mathbb{Z}$ can be described as follows: classify the top exterior power $\Lambda^dV$ by a map
$$X \to B\mathbb{R}^\times,$$
then project to $B(\pm 1) = B\mathbb{Z}^\times$ to deduce a local system of free $\mathbb{Z}$-modules of rank $1$ on $X$.  Shifting up by $d$, this identifies with $\mbb{S}^V\otimes\mathbb{Z}$.  In the case of the tangent bundle of a manifold, this is of course just the usual description of the ``orientation'' local system.
\end{remark}

\begin{remark}
We will get more information about this operation $V\mapsto \mbb{S}_{\wh{p}}^V$ in the next section.
\end{remark}

\section{J-homomorphisms from six functor formalisms}\label{jsec}

In the previous section, we recorded the ``twists'' relevant for Poincaré duality using the following construction: given a smooth map $f:X\to Y$ equipped with a section $s:Y\to X$, form
$$s^\ast f^!1,$$
which is an invertible object in the category which is value of our six functor formalism on $Y$.  We applied this in particular when $f:X\to Y$ is a vector bundle and $s$ its zero section, or when $f:X\to Y$ is the relative classifying stack of a vector bundle and $s$ is the natural quotient map.

In this section we would like to encode relationships between these twists for different $X\to Y$, in the general abstract context of a six functor formalism in the sense of \cite{HeyerSixFunctors}.  We will keep $Y$ fixed, so passing to the slice category we may as well assume $Y=\ast$.  Three basic examples of such relationships we want to encode are the following:
\begin{enumerate}
\item If $V\to \ast$ is a finite dimensional real vector space, then the twist for $V$ (with its zero section) identifies with the inverse of the twist for $BV$.
\item If $V \to W \to W/V$ is a short exact sequence of vector bundles, then the twist for $W$ is the tensor product of the twists for $V$ and $W/V$.
\item If $f:X\to Y$ is etale, then the twist for $f$ is canonically trivial.
\end{enumerate}
There are also more subtle relationships between the real and $p$-adic situations, which we will explore, but in the end they are of similar flavor to the above: they are all in some way ``K-theoretic'' in nature.

However, to express these relationships in a simple way it pays to change the setting a bit.  Once we are in a ``linear'' context like vector bundles, it happens surprisingly often that the smooth map $f:X\to \ast$ in question satisfies the property that $f^\ast$ is fully faithful.  For example this holds for real vector bundles (as is well-known), but also for relative classifying stacks of $p$-adic vector bundles (as we will check later).  Then if we pass to the classifying stack one more time, we get an even simpler situation: the pullback $f^\ast$ is not just fully faithful, but also an equivalence.  By smoothness it then follows that $f^!$ is also an equivalence, and therefore so is its adjoint $f_!$.  Then the twist is, up to passing to the inverse (which is a matter of sign convention), directly encoded by the six functor formalism: it is the image of the span
$$\ast \overset{f}{\leftarrow} X \overset{f}{\rightarrow} \ast$$
under the lax symmetric monoidal functor which ``is'' our six functor formalism.

This leads to the following definition.

\begin{definition}
Let $(\mathcal{C},E,M)$ be a six functor formalism in the sense of \cite{HeyerSixFunctors} (see Section \ref{sixsec}).  Say that a map $f:X\to Y$ in $\mc{C}$ is \emph{vectorial} if it lies in $E$, and for every base-change $f'$ of $f$, the functors $(f')^\ast$ and $(f')_!$ are equivalences.

Say that an object $X\in\mc{C}$ is \emph{vectorial} if the map $f:X\to \ast$ is vectorial.
\end{definition}

\begin{lemma}\label{vectorialproperties}
Let $(\mc{C},E,M)$ be a six functor formalism.  Then:
\begin{enumerate}
\item The class of vectorial maps in $\mc{C}$ contains all isomorphisms and is closed under composition and base-change.
\item Given composable maps $f$ and $g$ with composition $g\circ f$, if $f$ is vectorial and $g\circ f$ is vectorial, then $g$ is vectorial.
\item If $f:X\to S$ and $g:Y\to S$ are vectorial, then for any map $h:X\to Y$ with $g\circ h = f$, the map $h$ lies in $E$ and $h^\ast$ and $h_!$ are both equivalences (but $h$ need not be vectorial).
\end{enumerate}
\end{lemma}
\begin{proof}
Part 1 is clear from the definition.  Parts 2 and 3 follow from the 2-of-3 property for isomorphisms.  To see that $h$ in 3 need not be vectorial, consider $\ast \to B\mathbb{R} \to \ast$ in light condensed anima.  The map $\ast \to B\mathbb{R}$ is not vectorial because it has $\mbb{R}\to\ast$ as a base-change.
\end{proof}

\begin{remark}
Given a six functor formalism $(\mc{C},E,M)$ such that $\mc{C}$ has finite limits and any object $S\in \mc{C}$, we can restrict to the slice $\mc{C}_{/S}$ to get a new six functor formalism.  A map $f:X\to S$ in $\mathcal{C}$ is vectorial as a map in $\mc{C}$ if and only if it is vectorial as an object in $\mc{C}_{/S}$.  Thus we can often reduce discussion of vectorial maps to vectorial objects.
\end{remark}

\begin{remark}
Every vectorial map $f:X\to Y$ is smooth.  Indeed since $f_!$ is an equivalence (also after any base-change), the required properties of the right adjoint $f^!$ follow from those of $f_!$.
\end{remark}

\begin{definition}\label{jdef}
Let $(\mc{C},E,M)$ be a six functor formalism.  For a vectorial object $X$, we define the twist $J(X)\in M(\ast)$ by
$$J(X) = f_!f^\ast 1$$
where $f:X\to \ast$.
\end{definition}

In other words, $J(X)$ is the ``compactly supported cohomology of $X$''.  Note that $J(X)$ is an invertible object, because $f_!$ and $f^\ast$ are $M(\ast)$-linear isomorphisms and thus so is their composition $f_!f^\ast:M(\ast)\to M(\ast)$.  

The relation of this $J(X)$ with the construction of twists in the previous section is as follows.

\begin{lemma}\label{comparetwists}
Suppose $X$ is vectorial, and that we are given a section $s:\ast \to X$ of $f:X\to \ast$.  Then there is a natural isomorphism
$$J(X) \simeq (s^\ast f^! 1)^{-1}$$
of invertible objects in $M(\ast)$.

Moreover, if we further assume $\ast \to X$ is smooth and we set $X':=\ast\times_X \ast$ with its natural projection $f':X'\to \ast$ and diagonal section $s':\ast\to X'$, we have
$$J(X) \simeq (s')^\ast (f')^! 1.$$
\end{lemma}
\begin{proof}
Both $s^\ast$ and the left adjoint $f_\natural$ of $f^\ast$ are inverse to the equivalence $f^\ast$, hence $s^\ast \simeq f_\natural$.  Thus 
$$s^\ast f^! 1 \simeq f_\natural f^! 1.$$
But passing to right adjoints we see that $f_\natural f^!$ is inverse to $f_! f^\ast$, whence the first claim.  For the second claim, we certainly have $(s')^! (f')^!1=1$, but then we can apply the projection formula for $(s')^!$ and base-change of $\ast \to X$ along itself to conclude.
\end{proof}

From the definition it is mercifully straightforward to encode all the multiplicativity properties of $X\mapsto J(X)$.  For this we make the following definition, following \cite{QuillenQ}.

\begin{definition}\label{q}
Let $(\mc{C},E,M)$ be a six functor formalism.
\begin{enumerate}
\item Define $\on{Vect}_M\subset \mathcal{C}$ to be the full subcategory of $\mc{C}$ spanned by the vectorial objects.
\item Define $Q(\on{Vect}_M)$ to be the span category $\on{Span}_V(\on{Vect}_M)$ where $V$ is the class of vectorial maps between vectorial objects.  The closure properties required to form the span category follow from Lemma \ref{vectorialproperties}, and $Q(\on{Vect}_M)$ is symmetric monoidal via cartesian product in $\on{Vect}_M$.
\item Denote by $| Q(\on{Vect}_M) |$ the (possibly large) $E_\infty$-anima gotten by geometric realization;
\item Define $\on{K}(\on{Vect}_M) := \Omega | Q(\on{Vect}_M) |$ to be the loop space, again viewed as a (group-like) $E_\infty$-monoid in (possibly large) anima.
\end{enumerate}
\end{definition}

The following gives the multiplicativity of the $J$-construction ``with all higher coherences''.

\begin{theorem}\label{jcoherent}
Let $(\mc{C},E,M)$ be a six functor formalism.  Then there is a natural map of group-like $E_\infty$-anima
$$J:\on{K}(\on{Vect}_M)\to \on{Pic}(M(\ast))$$
refining the $J$ construction of Definition \ref{jdef}.
\end{theorem}
\begin{proof}
Start with the lax symmetric monoidal functor $M:\on{Span}_E(\mc{C})\to \on{Mod}_{M(\ast)}\on{Pr}^L$ which ``is'' our six functor formalism.  Restrict to $Q(\on{Vect}_M)$.  Then by vectoriality, $M$ becomes strongly symmetric monoidal, and it lands inside the full subcategory of the target spanned by the invertible objects (indeed, by objects isomorphic to the unit $M(\ast)$).  In fact by Lemma \ref{vectorialproperties} it lands inside the further subcategory of this where all morphisms are invertible, which is by definition $\on{Pic}(\on{Mod}_{M(\ast)}\on{Pr}^L)$.  Thus by adjunction we get an induced map
$$| Q(\on{Vect}_M) | \to \on{Pic}(\on{Mod}_{M(\ast)}\on{Pr}^L).$$
Taking loop spaces, the left hand side becomes $K(\on{Vect}_M)$ and the right hand side becomes $\on{Pic}(M(\ast))$, so we get $J$.
\end{proof}

Let's unravel what this is encoding a little bit.  There is a natural $E_\infty$-map
$$[-]:\on{Vect}_M^\simeq \to \on{K}(\on{Vect}_M)$$
sending a vectorial object $X$ to the loop corresponding to the span
$$\ast \leftarrow X \rightarrow \ast.$$
Thus, from $J$ we get an $E_\infty$ map
$$\on{Vect}_M^\simeq \to \on{Pic}(M(\ast))$$
encoding that $J(X\times Y)\simeq J(X)\otimes J(Y)$ in a symmetric monoidal fashion.

To see what else Theorem \ref{jcoherent} is giving, let's describe $\pi_0\on{K}(\on{Vect}_M)$ (basically following \cite{QuillenQ}).

\begin{lemma}
The abelian group $\pi_0\on{K}(\on{Vect}_M)$ is freely generated by the above $[X]$ for $X\in\on{Vect}_M$ modulo the following relations:
\begin{enumerate}
\item $[\ast]=0$;
\item $[X]=[Y]$ if $X\simeq Y$;
\item $[X\times_S Y] + [S] = [X] + [Y]$
whenever we have a pullback of the form $X\to S\leftarrow Y$ where all objects are vectorial and the map $X\to S$ is vectorial.
\end{enumerate}
\end{lemma}
\begin{proof}
It's clear that $Q(\on{Vect}_M)$ is connected, and we know that $\pi_0\on{K}(\on{Vect}_M)=\pi_1 Q(\on{Vect}_M)$ is abelian (by the $E_\infty$-structure).  Thus it suffices to show that for an abelian group $A$, the set of homotopy classes of functors
$$F:Q(\on{Vect}_M)\to BA$$
is in bijection with the set of elements $\{a_X\}_{X\in\on{Vect}_M}$ of $A$ indexed by $X\in\on{Vect}_M$ subject to the relations 1,2 and 3, via sending a functor $F$ to the image of $[X]$ under $F$ (note there are no base-point issues as $A$ is abelian).

The wide subcategory of $Q(\on{Vect}_M)$ corresponding to the spans of the form
$$X \leftarrow C\overset{\sim}{\to}{Y}$$
is equivalent to $\on{Vect}_M^{op}$, hence has an initial object, hence is contractible.  Thus any $F$ as above is homotopic to one which is the constant functor with value $\ast$ on this subcategory, i.e.\ $F$ takes every object to the basepoint and takes every morphism given by a span as above to the identity.  Any homotopy between two such functors can likewise be taken constant on this subcategory; but then it has to be the identity homotopy.

Thus, in this setting everything simplifies: our functor is completely determined by its value on maps of the form $X\overset{=}{\leftarrow} X\to Y$, and this is an element of $A$ (because every object is sent to the base-point).  Checking what's necessary for this to be a functor which is trivial on the wide subcategory discussed above, we see that giving $F$ as above is equivalent to giving an element $a_f$ of $A$ for every vectorial map $f:X\to Y$ between vectorial objects of $\mc{C}$, such that the following relations hold:
\begin{enumerate}
\item If $f'$ is a pullback of $f$ then $a_{f'}=a_f$;
\item If $f$ is an isomorphism then $a_f=0$.
\item If $f$ and $g$ are composable then $a_{g\circ f}=a_g + a_f$.
\end{enumerate}
Now, we can go back and forth between these $\{a_f\}$ and our desired $\{a_X\}$ via the following: given $\{a_f\}$, define $a_X = a_{X\to \ast}$; and given $\{a_X\}$, define $a_{X\to Y} = a_X - a_Y$.
\end{proof}

\begin{example}
As a special case, we see that if $f:X\to Y$ is a vectorial map between vectorial objects, and if $\ast \to Y$ is any map, then
$$J(X)\simeq J(F)\otimes J(Y)$$
where $F=\on{Fib}(f)$, a kind of ``twisted multiplicativity''.
\end{example}

Next we want to encode that $J$ admits canonical trivializations when restricted to certain special kinds of vectorial objects, namely the proper ones and the etale ones.  To explain this at the naive level, suppose $X$ is vectorial and $X$ is proper (i.e.\ the map $f:X\to \ast$ is proper).  Then $f_!\simeq f_\ast$ is right adjoint to $f^\ast$, but as $f^\ast$ is an equivalence it follows that $f_!$ is the inverse of $f^\ast$, so that $f_!f^\ast \simeq \on{id}$ and hence we get a trivialization
$$J(X)\simeq 1.$$
Similarly, if $X\to \ast$ is etale, then $f_!\simeq f_\natural$ is left adjoint to $f^\ast$ and the same argument gives the same conclusion $J(X)\simeq 1$. 

To give the higher coherences for this, consider the symmetric monoidal subcategories
$$Q(\on{Vect}^\pi_M),Q(\on{Vect}^\epsilon_M)\subset Q(\on{Vect}_M)$$
where the objects are the vectorial $X$ which are moreover proper (resp.\ etale), and the morphisms are the spans $X\leftarrow C\to Y$ giving morphisms in $Q(\on{Vect}_M)$ where the object $C$ (and hence each leg) is also required to be proper (resp.\ etale).

Then exactly as in Definition \ref{q} we define group-like $E_\infty$-anima
$$K(\on{Vect}^\pi_M), K(\on{Vect}^\epsilon_M)$$
equipped with a natural map to $K(\on{Vect}_M)$, by geometrically realizing the $Q$-categories and then taking loop spaces.

\begin{theorem}\label{eilenberg}
Let $(\mc{C},E,M)$ be a six functor formalism.  The $E_\infty$-map
$$J:K(\on{Vect}_M)\to \on{Pic}(M(\ast))$$
of \ref{jcoherent} admits a canonical trivialization on restriction to $K(\on{Vect}^\pi_M)$ (resp.\ $K(\on{Vect}^\epsilon_M)$).
\end{theorem}
\begin{proof}
We explain the proper case; the etale case is handled in the same way.  We can restrict our six functor formalism to the full subcategory of $\mathcal{C}$ spanned by the proper $X$, which is closed under finite limits.  Then the $Q(\on{Vect}_M)$-category for the restricted formalism contains the $Q(\on{Vect}_M^\pi)$ for the original one.  Thus we see that it suffices to show that if every map in $\mc{C}$ is proper, then $J\sim 0$.  (Unfortunately we can't assume every object is vectorial, because vectorial objects are not closed under pullback.)

In principle one can prove this just by examining the construction of the six functor formalism in the situation where every map is proper.  But that would require ``opening the hood'' and we don't want to do that.  Instead we will use a dirty trick: the Eilenberg swindle.

Consider $\mc{C}'=\on{Pro}(\mc{C})$.  Extending $M^\ast:\mc{C}^{op} \to \on{CAlg}(\on{Pr}^L)$ by filtered colimits, we get a new lax symmetric monoidal functor $(M')^\ast:(\mc{C}')^{op}\to\on{CAlg}(\on{Pr}^L)$.  It follows formally by taking limits along right adjoints that for this extended functor, we still have the property that ``every map is proper'' in the weak sense that for every map $f:X\to Y$ in $\mc{C}'$ the functor $f^\ast:M'(Y)\to M'(X)$ has a right adjoint satisfying the projection formula and commuting with pullback.  Thus we can apply \cite{HeyerSixFunctors} Prop 3.3.3 to extend $(M')^\ast$ to a six functor formalism on $\mc{C}'$.  When we restrict back to $\mc{C}$, it must agree with our original six functor formalism by the uniqueness result \cite{DauserUniqueness}.

Moreover, the inclusion $\mc{C}\to\mc{C}'$ preserves finite limits and sends vectorial maps to vectorial maps.  It follows that our $J$ for $M$ factors through the $J$ for $M'$, and hence it suffices to show $K(\on{Vect}_{M'})=0$.  But now consider the functor $\mc{C}'\to\mc{C}'$ defined by
$$X\mapsto \prod_\mbb{N} X.$$
This preserves finite limits; moreover it sends vectorial maps to vectorial maps.  Thus it induces a map
$$S:\on{K}(\on{Vect}_{M'})\to \on{K}(\on{Vect}_{M'}).$$
On the other hand by Hilbert hotel $S \simeq S\times \on{id}$ as limit-preserving functors $\mc{C}'\to\mc{C}'$.  Passing to K-theory we deduce that
$$S \sim S + \on{id}$$
whence, by subtracting in the group-like $E_\infty$-space of endomorphisms of $K(\on{Vect}_{M'})$, we have a homotopy $0\sim \on{id}$ and thus $\on{K}(\on{Vect}_{M'})=0$ as desired.
\end{proof}
\begin{remark}
Although produced in the same way, these two nullhomotopies are not compatible.  More precisely, on  $K(\on{Vect}^\pi_M\cap \on{Vect}^\epsilon_M)$ we get two nullhomotopies of $J$.  This gives rise to a map
$$K(\on{Vect}^\pi_M\cap \on{Vect}^\epsilon_M) \to \Omega \on{Pic}(M(\ast))=\on{Aut}_{M(\ast)}(1).$$
This map is not generally trivial.  Instead it gives a measure of the ``cardinality'' of the object $X\in \on{Vect}^\pi_M\cap \on{Vect}^\epsilon_M$ which is ``both compact and discrete, hence finite''.  Compare \cite{CarmeliCardinality}.
\end{remark}

This describes our J-homomorphism construction and its fundamental properties at the most primitive level.  But in practice, vectorial objects tend to have abelian group structures allowing to infinitely deloop, and this can be used to simplify the theory a bit.  To proceed along these lines, we will start to assume that the $\infty$-category $\mathcal{C}$ on which our six functor formalism lives is an $\infty$-topos (so that we have a correspondence between connected pointed objects and group objects given by looping and delooping), and that we have some reasonable interactions between descent in $\mc{C}$ and the six functor formalism as in the following definition.

\begin{definition}
Say that a six functor formalism $(\mc{C},E,M)$ \emph{has good descent} if the following conditions are satisfied:
\begin{enumerate}
\item $\mathcal{C}$ is an $\infty$-topos;
\item $M^\ast:\mathcal{C}^{op}\to\on{Pr}^L$ preserves limits;
\item the condition that a map $f:X\to Y$ in $\mathcal{C}$ lie in $E$ is local on $Y$.
\end{enumerate}
\end{definition}

\begin{example}
This condition holds for the six functor formalism on light condensed anima considered in Section \ref{sixsec} (for any coefficient category $\mathcal{R}$ as in that section).  Indeed, 1 and 2 hold by construction, and for 3 see Remark \ref{trulylocal}.

Moreover, the condition of having good descent evidently passes to slice topoi.
\end{example}

\begin{remark}
If $(\mc{C},E,M)$ is a six functor formalism with good descent, the the property of a map $f:X\to Y$ in $\mc{C}$ being smooth (resp.\ proper, resp.\ etale, resp.\ vectorial) is local on $Y$.  This follows from the results in \cite{HeyerSixFunctors} quoted in Remark \ref{etalepropersmoothlocal}.
\end{remark}

The following lemma gives some ways in which passing to classifying spaces simplifies the situation.

\begin{lemma}\label{bvectorial}
Let $(\mc{C},E,M)$ be a six functor formalism with good descent.
\begin{enumerate}
\item Suppose $G$ is a group object in $\mathcal{C}$ which is vectorial (forgetting the group structure).  Then $BG$ is also vectorial.  More generally, suppose $G\to \ast$ is smooth, and pullback along it is fully faithful after any base-change.  Then $BG\to \ast$ is vectorial.
\item Suppose $G$ and $G'$ are vectorial group objects and $f:G\to G'$ is any homomorphism.  Then $BG\to BG'$ is vectorial.
\item Suppose $G$ is a group object which is etale and vectorial.  Then $BG$ is etale and vectorial.
\item Suppose $G$ is a group object such that $f:G\to \ast$ is proper, and pullback along $f$ is fully faithful, also after any base-change (for example $G$ could be proper and vectorial).  Then $BG$ is proper and vectorial.
\end{enumerate}
\end{lemma}
\begin{proof}
For 1, let's prove the more general claim.  By descent pullback along $\ast \to BG$ is fully faithful, also after any base-change.  Since obviously pullback along $\ast \to \ast$ is an equivalence also after any base-change, we deduce that pullback along $BG\to \ast$ is an equivalence, also after any base-change.  The same argument applies to $!$-pullback, but we first need to show that $BG\to \ast$ is $!$-able.  But $\ast \to BG$ is vectorial and hence smooth, while pullback along it is conservative, hence $BG\to \ast$ is $!$-able (and smooth) by \cite{HeyerSixFunctors} Lemma 4.7.4 and Cor.\ 4.7.5.

For 2, again by descent $\ast \to BG$ and $\ast \to BG'$ are vectorial, therefore so is $BG\to BG'$ by Lemma \ref{vectorialproperties}.

For 3, $\ast \to BG$ is an etale cover and $\ast \to \ast$ is etale, so this follows from \cite{HeyerSixFunctors} Lemma 4.6.3.

For 4, the map $e:\ast \to BG$ is proper and vectorial by descent, and $e_\ast 1$ is descendable because $e_\ast 1\simeq 1$ by vectoriality.  Since $\ast \to \ast$ is proper, by \cite{HeyerSixFunctors} 4.7.4 and 4.7.5 it follows that $f:BG\to \ast$ is proper, and in particular $!$-able.  The same argument as in 1 shows that $\ast$- or $!$-pullback along $BG\to \ast$ is an equivalence also after any base-change. (Caution that even though in this case $f^!\simeq f^\ast$, the map $f$ is not generally etale, because its diagonal is generally not etale.) \end{proof}

\begin{definition}
Suppose given a six functor formalism $(\mc{C},E,M)$ which has good descent.  Say that a spectrum object $X\in \mc{C}\otimes\on{Sp}$ is \emph{vectorial} if the following conditions are satisfied:
\begin{enumerate}
\item $X$ is bounded below in the t-structure.
\item For $n>>0$, the object $\Omega^\infty \Sigma^n X \in \mc{C}$ is vectorial.
\end{enumerate}
Denote the full subcategory of vectorial objects in $\mc{C}\otimes\on{Sp}$ by $\on{vect}_M$.
\end{definition}

\begin{remark}
Suppose $X$ is connective.  If $\Omega^\infty \Sigma^nX$ is vectorial for some $n\geq 0$, then it is vectorial for all $n'\geq n$.  This follows from Lemma \ref{bvectorial}.
\end{remark}

\begin{lemma}\label{vectstable}
Let $(\mc{C},E,M)$ be a six functor formalism with good descent.  Then $\on{vect}_M\subset \mathcal{C}\otimes\on{Sp}$ is a stable subcateogry, i.e.\ it is closed under fibers and shifts.
\end{lemma}
\begin{proof}
Closure under shifts is immediate from the definition.  For closure under fibers, let $f:X\to Y$ be a map between vectorial spectrum objects in $\mc{C}$.  Shifting we can assume that $X,Y\in (\mc{C}\otimes\on{Sp})_{\geq 1}$ and $\Omega^\infty\Sigma^{-1}X$ and $\Omega^\infty\Sigma^{-1}Y$ are vectorial.  Then by Lemma \ref{bvectorial} the map $\Omega^\infty X \to \Omega^\infty  Y$ is vectorial.  Therefore by Lemma \ref{vectorialproperties} its fiber is also vectorial. But this fiber is $\Omega^\infty F$ where $F=\on{Fib}(f)$ and we conclude.
\end{proof}

For this stable category $\on{vect}_M$, we can define its Q-construction to simply be the span category with all maps allowed in the spans, then take the geometric realization, then take the loop space to define $K(\on{vect}_M)$.  But now this is actually a standard construction of Barwick known to be equivalent to all other possible definitions of the (connective) algebraic K-theory of a stable $\infty$-category (\cite{BarwickQ}).

Our next purpose is to show that $J$ can also be defined on this stable level.  For this we build a bridge between the stable and unstable situations as follows.  Consider the symmetric monoidal subcategory
$$Q(\on{vect}_{M,\geq 1})\subset Q(\on{vect}_M)$$
where the objects are the $X\in\on{vect}_M$ such that $\pi_iX=0$ for $i<1$ and $\Omega^\infty X$ is vectorial, and the the maps are the spans $X\leftarrow C\to Y$ such that $\pi_iC=0$ for $i<1$ and the map $C\to Y$ is surjective on $\pi_1$ and vectorial on $\Omega^\infty$. 

Then besides this symmetric monoidal inclusion there is also the symmetric monoidal functor
$$\Omega^\infty: Q(\on{vect}_{M,\geq 1})\to Q(\on{Vect}_M)$$
induced by $X\mapsto \Omega^\infty(X)$ on objects and spans.

\begin{theorem}\label{jextend}
Let $(\mc{C},E,M)$ be a six functor formalism with good descent.  Then there is a unique map of (group-like) $E_\infty$-anima
$$J^{new}:\on{K}(\on{vect}_M)\to \on{Pic}(M(\ast))$$
equipped with an isomorphism $J^{new}\mid_{\on{K}(\on{vect}_{M,\geq 1})}\simeq  J\circ \Omega^\infty$, where $J$ is the earlier construction \ref{jcoherent}.

From now on we will drop the ``new'' superscript and just call both of these maps $J$.
\end{theorem}
\begin{proof}
Clearly, it suffices to show that
$$K(\on{vect}_{M,\geq 1})\overset{\sim}{\rightarrow} K(\on{vect}_M).$$

For this, note that we have
$$(\mc{C}\otimes \on{Sp})_{>>-\infty} = \varinjlim ((\mc{C}\otimes \on{Sp})_{\geq 1})$$
where the colimit on the right is the sequential colimit where each transition map is given by the shift functor $X\mapsto X[1]$. Passing to span categories and then restricting to vectorial subcategories, we get

$$Q(\on{vect}_M) = \varinjlim Q(\on{vect}_{M,\geq 1}).$$

Thus it suffices to show that the functor $X\mapsto X[1]$ induces an isomorphism
$$K(\on{vect}_{M,\geq 1})\overset{\sim}{\rightarrow} K(\on{vect}_{M\geq 1}).$$
In fact we will show it induces $-\on{id}$.  Note that these $K(\on{vect}_{M,\geq 1})$ fit into the framework of Quillen K-theory of exact $\infty$-categories developed in \cite{BarwickQ}: the underlying additive category is that of spectrum objects $X$ with $\pi_i(X)=0$ for $i<1$ such that $\Omega^\infty X$ is vectorial; the class of ``egressive'' maps is that of maps surjective on $\pi_1$ whose underlying map of objects in $\mc{C}$ is vectorial, and the class of ``ingressive'' maps is all maps.  In particular we have access to the additivity theorem.  However, we have a functorial fiber sequence
$$X \mapsto (X \to 0 \to X[1])$$
where $0\to X[1]$ is egressive so from the additivity theorem we deduce the claim.
\end{proof}

\begin{remark}
This theorem explains how to calculate $J(X)$ for $X\in\on{vect}_M$.  Namely, find an $n\in 2\mbb{Z}$ such that $X[n]$ lies in $\on{vect}_{M,\geq 1}$.  Then
$$J(X) = J(X[n]) = J(\Omega^\infty(X[n])).$$
\end{remark}

We can completely analogously translate the canonical trivializations of $J$ from \ref{eilenberg} to the stable setting.

\begin{theorem}\label{eilenbergstable}
Let $(\mc{C},E,M)$ be a six functor formalism with good descent.  Let $\on{vect}_M^\pi$ denote the full subcategory of objects $X\in\on{vect}_M$ such that $\Omega^\infty X[n] \in\mc{C}$ is proper (and vectorial) for $n>>0$, and similarly for $\on{vect}_M^\epsilon$ with etale instead of proper.

Then $\on{vect}_M^\pi$ and $\on{vect}_M^\epsilon$ are stable subcategories of $\on{vect}_M$, and
$$J:K(\on{vect}_M)\to\on{Pic}(M(\ast))$$
admits a canonical trivialization on restriction to each of $K(\on{vect}_M^\pi),K(\on{vect}_M^\epsilon)$.
\end{theorem}
\begin{proof}
The fact that these are stable subcategories follows exactly as in the proof of \ref{vectstable}, and the canonical trivializations are deduced from their unstable versions (\ref{eilenberg}) by exactly the same method as in the proof of \ref{jextend}.
\end{proof}

Finally, we return to light condensed anima with their six functor formalism of etale sheaves with coefficients in $\on{Sp}_{\wh{p}}$ for a fixed prime $p$, as in Section \ref{sixsec}.  We will produce many examples of vectorial light condensed spectra.  They will actually be spectrum objects concentrated in degree $0$, that is, light condensed abelian groups.  By definition, such an $A$ is vectorial (resp.\ proper and vectorial, resp.\ etale and vectorial) as a light condensed spectrum if and only if some iterated classifying space $B^nA$ is vectorial (resp.\ proper and vectorial, resp.\ etale and vectorial) as a light condensed anima.  In all examples we consider, $n=2$ will suffice, and often $n=1$ will suffice.

\begin{theorem}\label{lotsofvectorial}
Let $p$ be a prime. In the six functor formalism for etale sheaves on light condensed anima with $p$-complete spectrum coefficients, we have:
\begin{enumerate}
\item For any (discrete) $\mbb{Z}[1/p]$-module $M$, the corresponding light condensed abelian group $M$ is etale and vectorial.
\item For any compact Hausdorff $\mbb{Z}[1/p]$-module $M$ which is second-countable and finite dimensional, the corresponding light condensed abelian group $M$ is proper and vectorial.
\item For any local field $F$ of characteristic $\neq p$ and any finite dimensional $F$-vector space $V$, the corresponding light condensed abelian group $V$ is vectorial.
\end{enumerate}
\end{theorem}
\begin{proof}
For 1, we claim that $f:BM\to \ast$ is etale and pullback along it is fully faithful.  This fully faithfulness automatically passes to all base-changes by Lemma \ref{blob}, so by Lemma \ref{bvectorial} that will show that $M$ is etale and vectorial.  Etaleness is Lemma \ref{etaleisetale}.  For full fidelity, we need to show that
$$f_\natural f^\ast X = X$$
for all $X\in\on{Sp}_{\wh{p}}$.  Working mod $p$ and using the Postnikov tower, it suffices to treat $X=\mbb{F}_p$.  So we need to see that the group homology of $M$ with $\mbb{F}_p$-coefficients is $\mbb{F}_p[0]$.  But by resolutions and Kunneth we can reduce to $M=\mbb{Z}[1/p]$ where it's a simple calculation.

For 2, because of the short exact sequence
$$0\to M^0 \to M \to \pi_0M\to 0$$
and Lemma \ref{vectstable} we can reduce to the connected and totally disconnected cases.  In the totally disconnected case, $M$ is Pontryagin dual to a countable torsion discrete $\mbb{Z}[1/p]$-module, hence $M\simeq \varprojlim M_n$ is a sequential limit of abelian groups of prime-to-$p$ order.  By Lemma \ref{bvectorial} it suffices to show that $BM\to \ast$ is proper and pullback along it is fully faithful (this will then automatically be so after any base change by Lemma \ref{blob}).   The properness follows from Theorem \ref{chausproper} because prime-to-$p$ groups have mod $p$ cohomological dimension $0$.  For full fidelity, arguing as usual with the Postnikov tower we reduce to showing that the $\mbb{F}_p$-cohomology is just $\mbb{F}_p$ in degree $0$, but again this follows from the same remark about cohomological dimension.  For connected $M$, we claim that $M\to \ast$ is proper and fully faithful on pullback.  Properness is Theorem \ref{chausproper}.  For full fidelity, note by Pontryagin duality that $M$ is dual to a torsionfree countably generated $\mbb{Z}[1/p]$-module, hence identifies with a sequential inverse limit of finite products of copies of the solenoid $C=\varprojlim\mbb{R}/p^n\mbb{Z}$.  By Corollary \ref{procontinuity} we reduce to $C$ itself, in which case the claim follows from a simple (mod $p$) cohomology calculation (exactly as in \ref{itscontractible}).

For 3, first consider nonarchimedean local fields $F$ of residue characteristic $\neq p$.  In the short exact sequence
$$\mc{O}_F\to F\to F/\mc{O}_F$$
the first term is vectorial by 2 and the last term is vectorial by 1, so $F$ and hence $V$ (as a direct sum of copies of $F$) is also vectorial by Lemma \ref{vectstable}.

Next consider archimedean local fields, which reduces to showing that $\mbb{R}$ is vectorial.  But $\mbb{R}\to \ast$ is smooth and pullback along it is fully faithful (also after any base change) so this follows from Lemma \ref{bvectorial}.

Finally we consider nonarchimedean local fields of residue characteristic $p$.  This reduces to showing that $F=\mbb{Q}_p$ is vectorial.  However we have a short exact sequence
$$0\to \mbb{Z}[1/p] \to \mbb{Q}_p\oplus \mbb{R} \to C\to 0$$
where $C  = \varprojlim_n \mbb{R}/p^n\mbb{Z}$ is the $p$-solenoid.  By 1 the first term is vectorial and by 2 the last term is vectorial.  Thus $\mbb{Q}_p\oplus\mbb{R}$ is vectorial.  As $\mbb{R}$ is vectorial, so is $\mbb{Q}_p$ as desired.
\end{proof}

By the general theory (Theorem \ref{jextend}) we have a map
$$J:\on{K}(\on{vect}_{\on{Sh}_{et}(-;\on{Sp}_{\wh{p}})})\to \on{Pic}(\on{Sp}_{\wh{p}})= \on{Pic}(\mbb{S}_{\wh{p}})$$
(for the last equality, see Lemma \ref{checkperfmodp}).  For any local field $F$ of characteristic $\neq p$, by Lemma \ref{lotsofvectorial} we can restrict to finite free $F$-modules and therefore get
$$J_F:\on{K}(F) \to \on{Pic}(\mbb{S}_{\wh{p}}).$$
This is our replacement for the piecemeal constructions of J-homomorphisms in \cite{ClausenJ}.

\begin{remark}
Of course, as the above proofs show, $\mbb{R}$ is already vectorial with $\on{Sp}$-coefficients, so we get a more refined $J_{\mbb{R}}:\on{K}(\mbb{R})\to \on{Pic}(\mbb{S})$.  We can also work in families as in Section \ref{dualsec} to directly see that this $J_{\mbb{R}}$ factors through the connective topological K-theory spectrum $ko$ (although such a factoring is also automatic and unique by \cite{SuslinK} and finiteness of homotopy groups of spheres).

Similarly, for $F$ nonarchimedean but of residue characteristic $\ell\neq p$ we actually get $J_F:\on{K}(F)\to \on{Pic}(\mbb{S}[1/\ell])$ which again is more refined.
\end{remark}

\begin{remark}
One should be careful with sign conventions.  By Lemma \ref{comparetwists} we have that for a finite dimensional $\mbb{R}$-vector space $V$, the invertible spectrum $J_\mbb{R}(V)$ identifies with the \emph{inverse} of the dualizing spectrum $0^\ast f^!\mbb{S}$ associated to
$$\ast\overset{0}{\to} V\overset{f}{\to}\ast,$$
and by Lemma \ref{homotopytwists} this dualizing spectrum $0^\ast f^!\mbb{S}$ identifies with the suspension spectrum of $V/(V\setminus 0)$, which is one of the usual ways of getting the real J-homomorphsim.  These identifications were made using the six functor formalism which is lax symmetric monoidal by definition, and it follows that as maps of spectra, we get an identification of $J_\mbb{R}$ with the \emph{negative} of the usual J-homomorphism $K(\mbb{R})\to\on{Pic}(\mbb{S})$.

On the other hand $J_{\mbb{Q}_p}$ sends $V$ to the dualizing object $\mbb{S}_{\wh{p}}^V$ for $BV$ considered in Section \ref{dualsec}, which lives in positive degrees.  The fact these behaviors are ``opposite'' is inevitable by reciprocity, \ref{reciprocity}.  But which one is positive and which one is negative is a matter of sign convention which (like in the analogous situation of local Artin maps) has no mathematical significance.
\end{remark}

The main result on these J-homomorphisms is their reciprocity law (from \cite{ClausenJ}), which we reprove here.  We state it in a bit more elementary finitary way:

\begin{theorem}\label{reciprocity}
Let $p$ be a prime, $F$ a global field of characteristic $\neq p$, and $S$ a finite set of places of $F$ containing all the infinite places and all the places above $p$.  Then there is a canonical null-homotopy of the composition
$$K(\mathcal{O}_{F,S}) \to \oplus_{\nu \in S}K(F_\nu)\to \on{Pic}(\mbb{S}_{\wh{p}}),$$
where the first map is induced by the natural homomorphism $\mathcal{O}_{F,S}\to F_\nu$ on the $\nu$-component and the second map is given by $J_{F_\nu}$ on the $\nu$-component.
\end{theorem}
\begin{proof}
Consider the short exact sequence of light condensed $\mathcal{O}_{F,S}$-modules
$$0\to\mathcal{O}_{F,S} \to \prod_{\nu\in S}F_\nu \to C_{F,S}\to 0.$$
The third term is another solenoidal compact Hausdorff space.  By Lemma \ref{lotsofvectorial} all objects appearing are vectorial, the first term is etale, and the last term is proper.

For any finitely generated projective $\mbb{Z}[1/S]$-module $M$ we can tensor this short exact sequence with $M$ to get another short exact sequence with the same properties.  In K-theory we have that the middle term is the sum of the two outer terms, but these outer terms go to $0$ under $J$ by Theorem \ref{eilenbergstable}.  Thus the middle term goes to $0$ under $J$.  But the middle term breaks up as the sum over $\nu$ of the $\nu$-adic J-homomorphisms, whence the claim.\end{proof}

\begin{remark}
We already used one simple corollary of this in the proof of Proposition \ref{classicaltwists}: namely, the image of $p\in \mbb{Q}_p^\times = \pi_1 K(\mbb{Q}_p)$ under $J_{\mbb{Q}_p}$ is trivial because this is so for the image of $p\in\mbb{R}^\times$ under $J_{\mbb{R}}$.
\end{remark}

\begin{remark}
The nicest way to state this kind of result for an aribitrary (possibly infinite) set of primes $S$ was given in \cite{BraunlingLocalCompact}, in terms of restricted direct products of exact categories.
\end{remark}

It's perhaps not so worthwhile to prove that the new definition of $J_F$ agrees with the old one from \cite{ClausenJ}: one should just use the new definition instead (though in principle arguments like those in Appendix C of \cite{ClausenJ} are sufficient to establish the comparison).  In the paper \cite{ClausenArtin}, which showed that Theorem \ref{reciprocity} (combined with cohomological dimension estimates for number fields and local fields) recovers the Artin reciprocity law, we only needed minimal information on the $J_F$: besides the reciprocity law \ref{reciprocity} and certain functorial properties which clearly hold for any definition, we just needed to know what $J_{\mbb{Q}_p}$ was in degrees $0$ and $1$.  For the new definition, this is already covered by Proposition \ref{classicaltwists}.

\begin{remark}
We can at least ask to characterize the $J_F$ after $K(1)$-localization, and this is a tractable problem which again only requires knowing $J_F$ in degrees $0$ and $1$.

 The result can be expressed in terms of the following calculation: when $F$ is a field of characteristic $\neq p$ and of virtual cohomological dimension $\leq 2$, the co-descent spectral sequence from \cite{ClausenArtin} 2.4 gives a natural short exact sequence
$$0 \to H^0(G_F;\mbb{Q}_p/\mbb{Z}_p)^\vee \to [K(F),L_{K(1)}\on{Pic}(\mbb{S}_{\wh{p}})] \to H^2(G_F;\mbb{Q}_p/\mbb{Z}_p(1))^\vee\to 0$$
where $(-)^\vee$ denotes Pontryagin duality (turning discrete $p$-torsion abelian groups into profinite abelian $p$-groups).

Then in these terms, for any local field $F$ of characteristic $\neq p$, our J-homomorphism
$$J_F \in [K(F),L_{K(1)}\on{Pic}(\mbb{S}_{\wh{p}})]$$
projects to the fundamental class
$$\on{inv}_F\in H^2(G_F;\mbb{Q}_p/\mbb{Z}_p(1))^\vee$$
for local Tate duality (and therefore Theorem \ref{reciprocity} recovers and refines the Hasse reciprocity law).  Indeed, pairing with a class $u\in F^\times =K_1(F)$ this reduces to the result in \cite{ClausenArtin} that the local Artin map is recovered one term up in the spectral sequence where we have
$$[\Sigma K(F), L_{K(1)}\on{Pic}(\mbb{S}_{\wh{p}})] = H^1(G_F;\mbb{Q}_p/\mbb{Z}_p)^\vee.$$

Moreover, among all elements of $[K(F),L_{K(1)}\on{Pic}(\mbb{S}_{\wh{p}})]$ which project to $\on{inv}_F$, we can pin down $J_F$ uniquely by its effect on $\pi_0$.
\end{remark}

\section{Morava E-theory in the light condensed framework}\label{esec}

Let us start with some recollections on chromatic homotopy theory, from the perspective of the ``$v_n$-self maps'' of Hopkins-Smith (\cite{HopkinsNilpotence}).  Let $p$ be a prime number.  The $\infty$-category $\on{Sp}_{(p)}$ of $p$-local spectra is in some ways similar to the derived $\infty$-category $\on{D}(\mathbb{Z}_{(p)})$.  In the latter, one has a distinguished self-map of any object, given by multiplication by $p$.  For $M\in\on{D}(\mathbb{Z}_{(p)})$, the object
$$M[p^{-1}] = \varinjlim(M\overset{\cdot p}{\to}M\overset{\cdot p}{\to}\ldots)$$
promotes to an object in $\on{D}(\mathbb{Q})$.  Thus it is controlled by linear algebra over a field.  The ``difference'' between $M$ and $M[p^{-1}]$, or more precisely the cofiber of the natural map $M \to M[p^{-1}]$, is the colimit of $M/p^n = \on{cofib}(M\overset{p^n}{\to}M)$.  This in turn promotes to an object of $D(\mathbb{Z}/p^n\mathbb{Z})$, and hence is controlled by a nilpotent thickening of a field.

In this way, the canonical self-map $\cdot p$ allows to decompose $D(\mathbb{Z}_{(p)})$ into pieces, each of which is in some way controlled by linear algebra over nilpotent thickenings of fields.  The story with $\on{Sp}_{(p)}$ is similar, except we need to iterate this process, and the meaning of being ``controlled by a nilpotent thickening of a field'' is a bit more abstract.  We still have a self-map $\cdot p$, and for $M\in\on{Sp}_{(p)}$ we still have that $M[p^{-1}]$ promotes to an object in $\on{D}(\mathbb{Q})$.  (This is a reflection of Serre's theorem that the stable homotopy groups of spheres in positive degrees are finite.)  On the other hand, $M/p^n$ does not behave like it lives over a field anymore.  Instead, we get a new functorial self-map of $M/p^n$ called $v_1$.  It's not literally a map $M/p^n\to M/p^n$, but rather a map from some shift of $M/p^n$ to $M/p^n$, but the formalism of modding out and inverting works in that context as well.

There is also a subtlety that $v_1$ is really only well-defined as a homotopy class, and even then only up to replacing $v_1$ to $v_1^N$ for some large $N$.  But if all we want to do with $v_1$ is complete along it or invert it, this ambiguity doesn't matter.

Then it turns out that for all $M\in\on{Sp}_{(p)}$, the spectrum $M/p[v_1^{-1}]$ in some sense ``lives over a field'', or at least over something that formally behaves like a nil-thickening of a field.  This can be codified in terms of the Balmer spectrum, but let's not get into the details of that.

Thus we have gone one step further than in the algebraic case of $\on{D}(\mathbb{Z}_{(p)})$, but we are still not done, because once again $M/(p,v_1)$ does not generally ``live over a field''.  Instead we get a new functorial degree-shifting self-map $v_2$, and $M/(p,v_1)[v_2^{-1}]$ behaves like it lives over a nilpotent thickening of a field.  And so on and so forth.

If this process applied to a spectrum $M$ stops after $n$ steps, meaning that $M/(p,v_1,\ldots, v_n)=0$, then $M$ is called \emph{$L_n^f$-local}.  In some sense this means that $M$ lives over a Zariski open subset of $\on{Spec}(\mathbb{S}_{(p)})$, namely the one described informally as the complement of the common vanishing locus of $(p,v_1,\ldots,v_n)$.  But it's a bit subtle because the $v_i$ are not globally defined.  In any case, there is a corresponding localization functor $L_n^f$ on spectra, which fits into a tower under the identity functor
$$\ldots L_n^f\to L_{n-1}^f \to \ldots \to L_0^f\to L_{-1}^f=0,$$
called the (finitary, or telescopic) chromatic tower.

In fact we have $L_n^fM = M\otimes L_n^f\mathbb{S}$ for all $M\in\on{Sp}$, just like with Zariski localizations.  Thus the most basic $L_n^f$-local spectrum is the $L_n^f$-local sphere.  It is very difficult to calculate, due to the intricate nature of the $v_i$ and the stable homotopy groups of spheres.

Fortunately, however, there is a simpler $L_n^f$-local $E_\infty$-ring spectrum which avoids these complications, but still controls a large sector of $L_n^f$-local stable homotopy theory.  (More precisely, it controls the $L_n$-local spectra, the more calculable portion of $L_n^f$-local spectra.)  This is the \emph{Morava $E$-theory}, denoted $E$ (we suppress the implicit dependence on $p$ and $n$).  Its homotopy groups are $0$ in odd degrees, and are 2-periodic, via an invertible class $\sigma\in\pi_2 E$.  Finally, its $\pi_0$ is isomorphic, non-canonically, to a power series ring over a Witt vector ring:
$$\pi_0 E \simeq W(\mathbb{F}_{p^n})[[u_1,\ldots,u_{n-1}]].$$
The classes $p,u_1,\ldots,u_{n-1}\in \pi_0E$ are some kind of incarnation of $p,v_1,\ldots,v_{n-1}$ (at least for the purposes of completion or inversion), and the class $\sigma\in \pi_2E$ is some kind of incarnation of $v_n$.  But now we see all these classes globally and, in some sense, algebraically.

With this as motivation, we now turn to a more thorough discussion of this Morava E-theory spectrum $E$.  We start with some algebraic preliminaries, as Morava E-theory (and indeed the whole chromatic story) is closely tied to the algebraic geometry of formal groups, though we avoided mentioning this connection in the above motivation.

More precisely, we are interested in one-dimensional formal groups over $\mathbb{Z}_{(p)}$-algebras $R$.  For this our main reference is \cite{GoerssMfg}; we give small recap of some material there, but incorporating the perspective of Cartier's theory \cite{CartierModules}, for which \cite{MumfordBiextensions} is a good reference.

Recall that if $\widehat{G}$ is a one-dimensional formal group over $R$, then we get:
\begin{enumerate}
\item An $R$-module $\omega=\on{coLie}(\widehat{G})$ which is locally free of rank one, namely the cotangent bundle to $\widehat{G}$ at the identity section;
\item An inductively defined sequence of elements $p=v_0,v_1,\ldots$ with
$$v_i\in \omega^{\otimes (p^i-1)}\otimes_R R/(p,v_1,\ldots, v_{i-1}),$$
called the \emph{Hasse invariants}.
\end{enumerate}

Namely, we recall that for a homomorphism $f:\widehat{G}\to\widehat{H}$ of one-dimensional formal groups over an $\mathbb{F}_p$-algebra, the induced map $f^\ast:\omega_{\widehat{H}}\to\omega_{\widehat{G}}$ is $0$ if and only if $f$ factors (necessarily uniquely) through the relative Frobenius $\widehat{G}\to\widehat{G}^{(p)}$.  In particular, the map $[p]:\widehat{G}\to\widehat{G}$ admits a factoring through the Frobenius, whence a map $V_1:\widehat{G}^{(p)}\to\widehat{G}$.  The section $v_1\in \omega^{\otimes(p-1)}$ corresponds to the effect of $V_1$ on the cotangent space at the identity, using $\omega_{\widehat{G}^{(p)}}=\omega_G^{\otimes p}$.  Then if $v_1=0$, we can further factor through to $V_2:\widehat{G}^{(p^2)}\to \widehat{G}$, whose effect on the contangent space at the identity is $v_2$, etc.

A one-dimensional formal group over a $\mathbb{Z}_{(p)}$-algebra $R$ is said to have \emph{exact height $n$} if $v_0,v_1,\ldots,v_{n-1}$ are all $0$ and $v_n$ is invertible.  Exact height $0$ just means that $R$ is a $\mathbb{Q}$-algebra.  For $n\geq 1$, by definition $\widehat{G}$ over $R$ has exact height $n$ if and only if $R$ is an $\mathbb{F}_p$-algebra, and there exists a (necessarily unique) isomorphism $G^{(p^n)}\simeq G$ identifying the relative Frobenius $G\to G^{(p^n)}$ with the multiplication by $p$ map $G\to G$.

One can classify the one-dimensional formal groups of exact height $n$ in terms of the Cartier theory of $p$-typical curves.  Namely, by Cartier theory, formal groups over a $\mathbb{Z}_{(p)}$-algebra $R$ correspond to certain $V$-reduced $V$-complete $p$-typical Cartier modules $M$ over $R$: to $\widehat{G}$ we assign the Cartier module of $p$-typical curves in $\widehat{G}$.  (Recall that Cartier modules are abelian groups equipped with left operators $V$, $F$, and $[r]$ for $r\in R$ satisfying certain properties.)   The one-dimensional formal groups correspond to those $M$ such that $M/VM$, which carries a natural $R$-module structure via the $[r]$, is locally free of rank one (it corresponds to $\omega_{\widehat{G}}^{-1}$).  In such a case we have that, Zariski locally on $\on{Spec}(R)$, there is an $e\in M$ which forms a $V$-basis, meaning the image of $e$ in $M/VM$ is a basis for that $R$-module.  Given such a $V$-basis $e\in M$, elements $m\in M$ correspond uniquely to sequences $(r_d)_{d\geq 0}$ of elements of $R$ via an expansion of the form $m=\sum_{d\geq 0} V^d[r_d]e$, and the structure of $M$ as a Cartier module (and hence of $\widehat{G}$ as a formal group) is uniquely specified by a single relation of the form
$$Fe = \sum_{d\geq 0} V^d [u_{d+1}]e,$$
where $u_1,u_2,\ldots \in R$ can be arbitrary.

On the formal group side, the generator $e\in M$ corresponds to a $p$-typical coordinate $t$.  As for the $u_d\in R$ giving the defining relation, when $R$ is an $\mathbb{F}_p$-algebra we have $p=VF$ in the Cartier ring, so multiplying the defining relation by $V$ on the left we deduce the following expansion of the $p$-series of $\widehat{G}$:
$$[p]_{\widehat{G}}(t) = \sum_{d\geq 1}^{\widehat{G}} u_d t^{p^d}.$$
Here the sum on the right is taken with respect to the group structure on $\widehat{G}$.  In particular we see that if we set $u_0:=p$, then these $u_0,u_1,u_2,\ldots$, which are globally defined as elements of $R$ given the $p$-typical coordinate $t$, lift the Hasse invariants $v_0,v_1,\ldots$ via the trivialization of $\omega^{p^n-1}$ coming from the coordinate.  Thus a formal group presented by a Cartier module as above has exact height $n$ if and only if $u_0,u_1,\ldots,u_{n-1}=0$ and $u_n$ is a unit.

With this in mind, we now write down the simplest possible Cartier module yielding a formal group of exact height $n\geq 1$.

\begin{definition}
Let $n\geq 1$.  Define the \emph{Honda formal group} $\widehat{H}_n$ over $\mathbb{F}_{p}$ to be the formal group whose associated Cartier module is presented by a single generator $e$ with the single relation $Fe=V^{n-1}e$.  Thus, by the above discussion, $\widehat{H}_n$ has a $p$-typical coordinate $t$ with $p$-series given by
$$[p](t) = t^{p^n},$$
and this uniquely specifies $\widehat{H}_n$.
\end{definition}

As is well-known, we can use $\widehat{H}_n$ to understand all formal groups of exact height $n$.  We give the explanation from the perspective of Cartier modules.

\begin{definition}
Let $\mathbb{F}_{p^n}$ be a finite field of order $p^n$.  Define a noncommutative ring $\mathcal{O}_n$ equipped with a distinguished element $\pi\in \mathcal{O}_n$ and a distinguished homomorphism $W(\mathbb{F}_{p^n})\to \mathcal{O}_n$ by the following prescriptions:
\begin{enumerate}
\item $\mathcal{O}_n$ is free of rank $n$ as a right $W(\mathbb{F}_{p^n})$-module on $1,\pi,\ldots,\pi^{n-1}$;
\item $\pi \alpha = \varphi(\alpha)\pi$, where $\varphi$ is the Frobenius on $W(\mathbb{F}_{p^n})$;
\item $\pi^n=p$.
\end{enumerate}
Note that this is consistent because $\varphi^n=\on{id}$ on $\mathbb{F}_{p^n}$.  Also, by property 2 we can equally well replace ``right'' with ``left'' in property 1.
\end{definition}

This is a sort of noncommutative Witt vector ring: note that the left and right ideals generated by $\pi$ agree, left or right multiplication by $\pi$ is injective, and
$$\mathcal{O}_n = \varprojlim_i \mathcal{O}_n/\pi^i,$$
with $\mathcal{O}_n/\pi=\mathbb{F}_{p^n}$.  Thus every element of $\mathcal{O}_n$ admits a unique expansion in the form $\sum_{d\geq 0} \pi^d [c_d]$ with $c_d\in\mathbb{F}_{p^n}$, where $[c_d]$ stands for the Teichmueller lift in $W(\mathbb{F}_{p^n})$.

We will be interested in the group of units $\mathcal{O}_n^\times$.  This is the subset of elements of $\mathcal{O}_n$ whose image in $\mathcal{O}_n/\pi=\mathbb{F}_{p^n}$ is nonzero, and it can be viewed as a profinite group via
$$\mathcal{O}_n^\times = \varprojlim_i (\mathcal{O}_n/\pi^i)^\times.$$
Likewise, $\mathcal{O}_n$ is a profinite ring.

\begin{theorem}\label{heightn}
Let $n\geq 1$.  Then:
\begin{enumerate}
\item Every one-dimensional formal group $\widehat{G}$ of exact height $n$ over an $\mathbb{F}_p$-algebra $R$ is pro-etale locally isomorphic to $\widehat{H}_n$.
\item The pro-etale sheaf of rings $\on{End}(\widehat{H}_n)$ of endomorphisms of $\widehat{H}_n$ over $\mathbb{F}_p$, when base-changed to $\mathbb{F}_{p^n}$, is represented by the pro-constant ring scheme on the profinite ring $\mathcal{O}_n$.
\item The identification in 2 is equivariant with respect to $\on{Gal}(\mathbb{F}_{p^n}/\mathbb{F}_p)$, which acts on $\mathcal{O}_n$ since this ring was functorially constructed in terms of $\mathbb{F}_{p^n}$.
\end{enumerate}
\end{theorem}
\begin{proof}
First consider 1.  Working Zariski locally on $\on{Spec}(R)$ we can assume that the Cartier module of $\widehat{G}$ is presented by a single generator $e$ with the single relation
$$Fe = \sum_{d\geq 0} V^{d+n-1} [c_d]e$$
where $c_d\in R$ for all $d\geq 0$ and $c_0\in R^\times$.  We need to find, pro-etale locally, a new $V$-basis vector $e'$ such that $Fe'=V^{n-1}e'$.  Write
$$e' = \sum_{k\geq 0} V^k[b_k]e.$$
The condition that $e'$ be a $V$-basis is $b_0\in R^\times$.  Furthermore, using relations in the Cartier ring and the above expression for $Fe$ we compute that
$$Fe' = \sum_{d,k\geq 0} V^{d+k+n-1}[\varphi^{d+n}(b_k)\cdot c_d]e$$
whereas
$$V^{n-1}e' = \sum_{k\geq 0} V^{k+n-1}[b_k]e.$$
Equating coefficients, we find that $e'$ does what we want if and only if
$$b_0 = \varphi^n(b_0)\cdot c_0,$$
$$b_1 = \varphi^{n+1}(b_0)\cdot c_1 + \varphi^n(b_1)\cdot c_0,$$
$$b_2 = \varphi^{n+2}(b_0)\cdot c_2 + \varphi^{n+1}(b_1)\cdot c_1 + \varphi^n(b_2)\cdot c_0,$$
etc.  Divide through by $c_0$.  Then when we try to inductively solve for the $b_k$, we find that each equation amounts to a monic polynomial that $b_k$ should be a root of, and that this monic polynomial has derivative the constant polynomial on the unit $-c_0^{-1}$, hence describes a finite etale extension.  To enforce that $b_0$ be a unit, we can divide the first equation by $b_0$ and see that we're still left with a finite etale equation $b_0^{p^n-1}=c_0^{-1}$ for $b_0$, but one which forces $b_0$ to be a unit.  Thus at the top of this tower of finite etale extensions we get an isomorphism $\widehat{G}\simeq \widehat{H_n}$, as required.

More precisely, what the above argument showed was that the functor on $R$-algebras $\on{Hom}(\widehat{H_n},\widehat{G})$ parametrizing the set of homomorphisms $\widehat{H_n}\to\widehat{G}$ is canonically the inverse limit
$$\on{Hom}(\widehat{H_n},\widehat{G})\simeq \varprojlim_i \on{Hom}_i(\widehat{H_n},\widehat{G})$$
where $\on{Hom}_i(\widehat{H_n},\widehat{G})$ parametrizes homomorphisms of the associated Cartier modules mod $V^i$ (so in particular $\on{Hom}_0(\widehat{H_n},\widehat{G})=\on{Spec}(R)$), and
$$\on{Hom}_{i+1}(\widehat{H_n},\widehat{G})\to \on{Hom}_i(\widehat{H_n},\widehat{G})$$
is finite etale and surjective for all $i\geq 0$, described by the sequence of separable monic polynomials discussed above; moreover the subfunctor $\on{Isom}(\widehat{H_n},\widehat{G})\subset \on{End}(\widehat{H_n},\widehat{G})$ parametrizing isomorphisms comes by pullback from $\on{Isom}_1(\widehat{H_n},\widehat{G})$, which is the finite etale extension of $\on{Spec}(R)$ described by $c_0X^{p^n-1}=1$ inside $\on{Hom}_1(\widehat{G},\widehat{H_n})=\on{Spec}(R[X]/(c_0X^{p^n}-X))$.

Now, moving on to 2, we specialize the above to $\widehat{G}=\widehat{H}_n$. Then $c_0=1$ and $c_d=0$ for $d>0$, so the above equations just amount to $b_k=\varphi^n(b_k)$ for all $k\geq 0$.  The polynomial $x^{p^n}-x$ splits completely over $\mathbb{F}_{p^n}$, so all the finite etale extensions appearing above are constant over $\mathbb{F}_{p^n}$, proving the pro-constancy of $\on{End}(\widehat{H}_n)$ over $\mathbb{F}_{p^n}$.  To finish the proof of 2, it suffices to identify each ring $\on{End}(\widehat{H}_n)_i(\mathbb{F}_{p^n})$ with $\mathcal{O}_n/\pi^i$, compatibly in $i$, and to prove 3 we have to make this identification functorial in the choice of $\mathbb{F}_{p^n}$.

By the above analysis, if $M$ denotes the Cartier module of $\widehat{H}_n$, then the set of endomorphisms of $M/V^i M$ over $\mathbb{F}_{p^n}$ is in bijection with the set of sequences
$$b_0,b_1,\ldots,b_{i-1}$$
of elements of $\mathbb{F}_{p^n}$, where to such a sequence we assign the endomorphism $f$ characterized uniquely by
$$f(e) = ([b_0] + V[b_1] + \ldots + V^{i-1}[b_{i-1}])e.$$
This is also in bijection with elements of the Cartier ring (mod $V^i$) of the form
$$[b_0] + V[b_1] + \ldots + V^{i-1}[b_{i-1}].$$
Using the injective homomorphism of $W(\mathbb{F}_{p^n})$ into the Cartier ring, assigning to a Witt vector with Witt coordinates $(a_0,a_1,\ldots)$ the element $\sum_{d\geq 0} V^d [a_d] F^d$ of the Cartier ring, we easily calculate that the ring structure on endomorphisms, via this parametrization, indeed corresponds to the ring structure on $\mathcal{O}_n/\pi^i$, with $\pi$ corresponding to $V$ (which, as an endomorphism of $\widehat{H}_n$, corresponds to the relative Frobenius $t\mapsto t^p$).  The key relation $V^n=p$ amounts to the definition of $\widehat{H}_n$.
\end{proof}

The upshot of this is that the structure of the moduli stack of one-dimensional formal groups of exact height $n$ can be understood in terms of the following (light) profinite group, which is in fact naturally a compact $p$-adic Lie group.

\begin{definition}\label{moravastab}
Let $n\geq 1$.  Define the \emph{(extended) Morava stabilizer group} to be the profinite group given as the semidirect product
$$\mathbb{G}_n:=\mathcal{O}_n^\times\rtimes \on{Gal}(\mathbb{F}_{p^n}/\mathbb{F}_p)$$
with respect to the natural action of $\on{Gal}(\mathbb{F}_{p^n}/\mathbb{F}_p)$ on $\mathcal{O}_n^\times$.
\end{definition}

We can consider $\mathbb{G}_n$ as a pro-constant group scheme over $\mathbb{F}_p$, and as such it acts on $\on{Spec}(\mathbb{F}_{p^n})$ via the natural projection $\mathbb{G}_n\to \on{Gal}(\mathbb{F}_{p^n}/\mathbb{F}_p)$.  Then Theorem \ref{heightn} has the following well-known consequence:

\begin{theorem}\label{mfgn}
The pro-etale quotient $\on{Spec}(\mathbb{F}_{p^n})/\mathbb{G}_n$ is the moduli stack of one-dimensional formal groups of exact height $n$ over $\mathbb{F}_p$:
$$\mathcal{M}_{fg,n} = \on{Spec}(\mathbb{F}_{p^n})/\mathbb{G}_n.$$
\end{theorem}
\begin{proof}
By general formalism of semidirect products, the quotient $\on{Spec}(\mathbb{F}_{p^n})/\mathbb{G}_n$ identifies with the quotient $\on{Spec}(\mathbb{F}_p)/\mathcal{G}$, where $\mathcal{G}$ is the nontrivial form of $\mathcal{O}_n^\times$ gotten by twisting via the action of $\on{Gal}(\mathbb{F}_{p^n}/\mathbb{F}_p)$ on $\mathcal{O}_n^\times$.  But $\mathcal{G}=\on{Aut}(\widehat{H}_n)$ by Theorem \ref{heightn} parts 2 and 3, and the quotient stack $\on{Spec}(\mathbb{F}_p)/\on{Aut}(\widehat{H}_n)$ identifies with $\mathcal{M}_{fg,n}$ by Theorem \ref{heightn} part 1.
\end{proof}

Now, let $\on{Perf}_{/\mathcal{M}_{fg,n}}$ denote the slice category of affine perfect $\mathbb{F}_p$-schemes equipped with a map to $\mathcal{M}_{fg,n}$.  Equivalently, this is the opposite of the category of pairs $(R,\widehat{G})$ consisting of a perfect $\mathbb{F}_p$-algebra $R$ and a formal group of exact height $n$ over $R$, where morphisms $(R,\widehat{G})\to (R',\widehat{G}')$ are pairs $(f,\alpha)$ where $f:R\to R'$ and $\alpha:\widehat{G}\times_{\on{Spec}(R)}\on{Spec}(R')\simeq \widehat{G}'$ is an isomorphism of formal groups over $R'$.  We will equip this slice category with the pro-etale Grothendieck topology (inherited from $\on{Perf}$).

A theorem of Lurie, refining previous work of Morava, Hopkins-Miller, and Goerss-Hopkins, produces a canonical sheaf of $E_\infty$-rings on $\on{Perf}_{/\mathcal{M}_{fg,n}}$, given by the \emph{Morava $E$-theories}.

\begin{definition}
A \emph{Morava E-theory of height $n\geq 1$} is an $E_\infty$-ring spectrum $E$ such that:
\begin{enumerate}
\item $\pi_iE=0$ for $i$ odd.
\item $\pi_iE\otimes_{\pi_0E}\pi_2E \overset{\sim}{\to}\pi_{i+2}E$ for all $i\in\mathbb{Z}$.  Thus $\pi_2 E$ is an invertible $\pi_0E$-module, and  $\pi_{2i}E = (\pi_2E)^{\otimes i}$ for all $i\in\mathbb{Z}$.
\item The Quillen formal group
$$\on{Spf}E^0(K(\mathbb{Z};2))=\varinjlim_n \on{Spec}(E^0(\mathbb{C}\mathbb{P}^n))$$
over $\pi_0 E$ has the following properties:
\begin{enumerate}
\item The associated Hasse invariants $(p,v_1,\ldots,v_{n-1})$ form a regular sequence.  (Recall that $v_i$ is only defined mod $(p,v_1,\ldots,v_{i-1})$, and is a section of a line bundle and not an honest element, but still the regular sequence condition makes sense, as do the following conditions.)
\item $\pi_0E$ is $(p,v_1,\ldots,v_{n-1})$-complete.
\item $v_n$ is invertible on $\pi_0E/(p,v_1,\ldots,v_{n-1})$.
\item $\pi_0E/(p,v_1,\ldots,v_{n-1})$ is a perfect $\mathbb{F}_p$-algebra.
\end{enumerate}
\end{enumerate}
\end{definition}

\begin{remark}
Recall in the introduction to this section that we discussed a similar sequence of elements $p=v_0,v_1,\ldots$ which are even defined over the sphere spectrum, unique up to replacing $v_i$ with some power $v_i^N$ (and up to homotopy).  The notation is consistent, in that if we base-change these $v_0,v_1,\ldots$ from the sphere spectrum to a Morava E-theory (or more generally to any complex orientable ring spectrum), then we get the above $v_0,v_1,\ldots$ coming from the Quillen formal group, again up to replacing each $v_i$ by some power and up to homotopy.  The degrees are consistent because the cotangent space of the Quillen formal group identifies with $\pi_2 E$.

In particular, while the operation on $E$-modules (for a Morava $E$-theory of height $n$) of sequentially modding out the Hasse invariants $(-)/(p,v_1,\ldots,v_{n-1})$ is not necessarily defined over $\mathbb{S}$, this problem goes away after raising to powers, and in particular the the operation of $(p,v_0,\ldots,v_{n-1})$-completion is defined over $\mathbb{S}$.
\end{remark}

Given a Morava $E$-theory, we get an object of $\on{Perf}_{/\mathcal{M}_{fg,n}}$, namely we take the $\mathbb{F}_p$-algebra $R=\pi_0E/p,\ldots,v_{n-1}$ equipped with the base-change of the Quillen formal group.  The following theorem of Lurie gives the converse, see \cite{LurieE} Section 5.

\begin{theorem}\label{constructe}
There is a fully faithful functor
$$(R,\widehat{G})\mapsto E(R,\widehat{G})$$
from $(\on{Perf}_{/\mathcal{M}_{fg,n}})^{op}$ to $E_\infty$-rings, whose essential image consists exactly of the Morava E-theories as defined above.  The inverse functor sends a Morava E-theory to $R=\pi_0E/p,\ldots,v_{n-1}$ equipped with the base-change of the Quillen formal group.  Moreover, this functor $E$ has the following properties:
\begin{enumerate}
\item It is a hypercomplete sheaf for the pro-etale topology on $\on{Perf}_{/\mathcal{M}_{fg,n}}$.
\item Suppose given two objects $X,Y\in\on{Perf}_{/\mathcal{M}_{fg,n}}$.  Then the natural map
$$E(X)\otimes E(Y)\to E(X\times_{\mathcal{M}_{fg,n}}Y)$$
of $E_\infty$-rings realizes the target as the $(p,v_1,\ldots,v_{n-1})$-completion of the source.  (Note that $X\times_{\mathcal{M}_{fg,n}}Y$ is indeed affine and perfect when $X$ and $Y$ are, e.g.\ by Theorem \ref{mfgn}.)  
\end{enumerate}
\end{theorem}
\begin{proof}
The main statement is \cite{LurieE} Thm.\ 5.1.5 and Claim 2 is \cite{LurieE} Remark 5.1.6: they both follow from a universal property of Morava E-theory.  For Claim 1, since Morava E-theories are $(p,v_1,\ldots,v_{n-1})$-complete, we can test this after modding out by $(p,v_1,\ldots,v_{n-1})$.  Then since
$$\pi_\ast (E(R,\widehat{G})/(p,v_1,\ldots,v_{n-1})) = \oplus_{i\in\mathbb{Z}} \omega_{\widehat{G}}^{\otimes_R i}[2i],$$
each homotopy group is a quasicoherent sheaf hence satisfies derived descent, so passing to the limit using the Postnikov tower we deduce the hyperdescent claim.
\end{proof}

Now, consider the functor
$$\on{ProfSet}^{light}\to\on{Perf}$$
which sends a light profinite set $S=\varprojlim_n S_n$ to the affine perfect $\mathbb{F}_p$-scheme
$$\on{Spec}(C(S,\mathbb{F}_p)) = \varprojlim_n \sqcup_{S_n} \on{Spec}(\mathbb{F}_p),$$
the pro-constant $\mathbb{F}_p$-scheme on $S$.  This functor preserves finite disjoint unions, sends pullbacks to pullbacks, and sends surjective maps to pro-etale covers.   Thus it induces a pullback map of $\infty$-topoi from light condensed anima to pro-etale hypersheaves of anima on $\on{Perf}_{\mathbb{F}_p}$.  Now, recall the Morava stabilizer group $\mathbb{G}_n$ from Definition \ref{moravastab}, a light profinite group.  Then we get an induced a pullback functor
$$\on{CondAn}^{light}_{/(\ast/\mathbb{G}_n)}\to \on{Sh}^{hyp}_{proet}(\on{Perf})_{/(\ast/\mathbb{G}_n)}$$
where on the right we view $\mathbb{G}_n$ as a pro-constant perfect $\mathbb{F}_p$-group scheme.

Next, recall that $\mathcal{M}_{fg,n}$ identifies with the quotient $\on{Spec}(\mathbb{F}_{p^n})/\mathbb{G}_n$, whence a natural map $\mathcal{M}_{fg,n} \to \ast/\mathbb{G}_n$.  Pulling back along this, we further deduce a pullback functor
$$\alpha^\ast: \on{CondAn}^{light}_{(\ast/\mathbb{G}_n)} \to \on{Sh}^{hyp}_{proet}(\on{Perf})_{/\mathcal{M}_{fg,n}}$$
(meaning, the ``left adjoint'' part of a geometric morphism of $\infty$-topoi).  Now, Lurie's theorem \ref{constructe} gives a sheaf of $E_\infty$-rings on $\on{Sh}^{hyp}_{proet}(\on{Perf})_{/\mathcal{M}_{fg,n}}$.  Taking $\alpha_\ast$ we therefore deduce a sheaf of $E_\infty$-rings on $\on{CondAn}^{light}_{(\ast/\mathbb{G}_n)}$, or in other words a light condensed $E_\infty$-ring with an action of the light condensed group $\mathbb{G}_n$.

\begin{definition}\label{lighte}
Define \emph{the Morava E-theory of height $n$} to be the light condsensed $E_\infty$-ring spectrum with $\mathbb{G}_n$-action, denoted $\mathbb{E}_n$, gotten by applying $\alpha_\ast$ to Lurie's sheaf of $E_\infty$-rings on $\on{Perf}_{/\mathcal{M}_{fg,n}}$ via the geometric morphism $\alpha$ defined above.
\end{definition}

By construction, as a hypersheaf of $E_\infty$-rings on the site of light profinite sets $S$ equipped with a map $S\to \ast/\mathbb{G}_n$ (or equivalently, with $\mathbb{G}_n$-torsor), $\mathbb{E}_n$ is described as follows:
$$\mathbb{E}_n(S\to\ast/\mathbb{G}_n) = E(S\times_{\ast/\mathbb{G}_n}\mathcal{M}_{fg,n}\to \mathcal{M}_{fg,n}).$$
In particular, the underlying $E_\infty$-ring is given by
$$E_n:=\mathbb{E}_n(\ast\to \ast/\mathbb{G}_n) = E(\on{Spec}(\mathbb{F}_{p^n})\to\mathcal{M}_{fg,n}) = E(\mathbb{F}_{p^n},(\widehat{H}_n)_{\mathbb{F}_{p^n}}).$$
In words, it is the Morava E-theory based on the perfect $\mathbb{F}_p$-algebra $\mathbb{F}_{p^n}$, equipped with the Honda formal group.

Note that $\mathbb{E}_n$ provides extra structure on $E_n$ in two directions: first, it promotes it to a light condensed $E_\infty$-ring; explicitly, for $S\in\on{ProfSet}^{light}$ we have
$$\mathbb{E}_n(S\overset{triv}{\longrightarrow}\ast/\mathbb{G}_n) = E(C(S,\mathbb{F}_{p^n}),(\widehat{H}_n)_{C(S;\mathbb{F}_{p^n})}).$$
But second, it also provides a $\mathbb{G}_n$-action on this light condensed $E_\infty$-ring.  This uses that we can more generally consider arbitrary $\mathbb{G}_n$-torsors over $S$, not just trivial ones.

If we ignore the $\mathbb{G}_n$-action, the light condensed structure on $\mathbb{E}_n$ is in fact determined by the underlying $E_\infty$-ring $E_n$.  This is because the ``topology'' on $\mathbb{E}_n$ is somehow algebraically determined as the ``$(p,v_1,\ldots,v_{n-1})$-adic topology''.  For the precise statement, recall the unique geometric morphism
$$\delta^\ast:\on{An}\to\on{CondAn}^{light},$$
with right adjoint $\delta_\ast(X) = X(\ast)$.  In particular, given an $E_\infty$-ring $R$, we get a light condensed $E_\infty$-ring $\delta^\ast R$, morally by ``equipping $R$ with the discrete topology''.

\begin{lemma}
The natural map $\delta^\ast E_n\to \mathbb{E}_n$ of light condensed $E_\infty$-rings realizes the target as the $(p,v_1,\ldots,v_{n-1})$-completion of the source.
\end{lemma}
\begin{proof}
By construction the target is $(p,v_1,\ldots,v_{n-1})$-complete.  Thus we can check mod $(p,v_1,\ldots,v_{n-1})$.  Because
$$\pi_\ast (E(R,\widehat{G})/(p,v_1,\ldots,v_{n-1})) = \oplus_{i\in\mathbb{Z}} \omega_{\widehat{G}}^{\otimes_R i}[2i]$$
and $\delta^\ast$ commutes with homotopy groups, we therefore reduce to the claim that for $S\in\on{ProfSet}^{light}$ and $A\in\on{Ab}$, we have $(\delta^\ast A)(S)=C(S;A)$, which is a special case of Theorem \ref{etale}.
\end{proof}

\begin{remark}\label{eisetale}
In terms of the formalism of etale sheaves from Section \ref{sixsec}, we can reinterpret this lemma as saying that the object
$$\mathbb{E}_n\in\on{CAlg}(\on{Sh}_v(\ast/\mathbb{G}_n;\on{Sp}^\wedge_{(p,v_1,\ldots,v_{n-1})}))$$
lies in the full subcategory $\on{CAlg}(\on{Sh}_{et}(\ast/\mathbb{G}_n;\on{Sp}^\wedge_{(p,v_1,\ldots,v_{n-1}})))$.  (Note that $\on{Sp}^\wedge_{(p,v_1,\ldots,v_{n-1})}$ is compactly generated, with generator $\mathbb{S}/(p,v_1,\ldots,v_{n-1})$, so that Remark \ref{etinsidev} does apply.)
\end{remark}

While the light condensed structure on $\mathbb{E}_n$ is in this manner completely determined by the underlying ordinary Morava E-theory $E_n$, the ``continuous'' $\mathbb{G}_n$-action on $\mathbb{E}_n$ is truly extra structure on $E_n$, and is not so simple to describe in classical terms.  Only the action of the underlying discrete group $\mathbb{G}_n(\ast)$ is directly visible without passing to a more refined context.

The first solution (and a completely satisfactory one at that) to this issue was provided by Devinatz-Hopkins (\cite{DevinatzE}).  They produced a hypercomplete sheaf $\mathcal{E}_n^{DH}$ of $(p,v_1,\ldots,v_{n-1})$-complete $E_\infty$-rings on the site of finite continuous $\mathbb{G}_n$-sets such that Morava E-theory identifies with the $(p,v_1,\ldots,v_{n-1})$-completion of the stalk $e^\ast \mathcal{E}_n^{DH}$, where $e$ is the canonical basepoint of the site of finite $\mathbb{G}_n$-sets (corresponding the forgetful functor to sets).  By the comparison \ref{profgpexample} between hypercomplete sheaves on the site of finite continuous $\mathbb{G}_n$-sets and etale sheaves on $\ast/\mathbb{G}_n$, and Remark \ref{eisetale} showing that $\mathbb{E}_n$ is etale as a $(p,v_1,\ldots,v_{n-1})$-complete sheaf, it makes sense to ask whether $\mathcal{E}_n^{DH}$ and $\mathbb{E}_n$ agree.  They do indeed, and we will prove this presently.

Recall that by construction, $\mathbb{E}$, as a $v$-sheaf on $\ast/\mathbb{G}_n$, is determined by its values on light profinite sets $S$ mapping to $\ast/\mathbb{G}_n$, namely in terms of Lurie's generalized Morava E-theories from \ref{constructe}, we have
$$\mathbb{E}(S\to\ast/\mathbb{G}_n) = E(S\times_{\ast/\mathbb{G}_n} \mathcal{M}_{fg,n} \to \mathcal{M}_{fg,n}).$$
To compare with Devinatz-Hopkins' construction, we should understand more generally its values on objects of the form $T/\mathbb{G}_n \to \ast/\mathbb{G}_n$, where $T$ is a light profinite set with $\mathbb{G}_n$-action.  The previous case corresponds to the situation where the action is free, with $S=T/\mathbb{G}_n$, and the current case of interest is when $T$ is finite.  The key is the following Kunneth/Galois property.

\begin{proposition}\label{galoiskunneth}
Suppose $T,T'$ are light profinite sets with $\mathbb{G}_n$-action.  Then the natural map of $E_\infty$-rings
$$\mathbb{E}(T/\mathbb{G}_n)\otimes \mathbb{E}(T'/\mathbb{G}_n) \to \mathbb{E}((T\times T')/\mathbb{G}_n)$$
exhibits the target as the $(p,v_1,\ldots,v_{n-1})$-completion of the source.
\end{proposition}
\begin{proof}
In the case where $T/\mathbb{G}_n$ and $T'/\mathbb{G}_n$ are light profinite sets $S$ and $S'$ respectively, this is a translation of Theorem \ref{constructe} part 2.   Now consider a general $T\in\on{ProfSet}^{light}$ with $\mathbb{G}_n$-action, but keep $T'$ as before.  Varying over $S\to T/\mathbb{G}_n$ with $S\in\on{ProfSet}^{light}$ we deduce that

$$\mathbb{E}\mid_{T/\mathbb{G}_n} \otimes \mathbb{E}(T'/\mathbb{G}_n) \to P_\ast P^\ast (\mathbb{E}\mid_{T/\mathbb{G}_n})$$
is a $(p,v_1,\ldots,v_{n-1})$-completion of $v$-sheaves on $T/\mathbb{G}_n$, where $P:(T\times T')/\mathbb{G}_n \to T/\mathbb{G}_n$ is the map induced by projection.  Now we take global sections on both sides.  Since $\mathbb{G}_n$ is of virtual finite (mod $p$) cohomological dimension (being a compact $p$-adic Lie group) and $\mathbb{E}$ is $T(n)$-local, the tensor product commutes with global sections mod $(p,v_1,\ldots,v_{n-1})$ by Remark \ref{tn}.  Thus we deduce the claim for general $T$.  Then applying a similar argument in the other coordinate we get the claim for general $T'$ as well, finishing the proof.\end{proof}

\begin{remark}\label{differentgalois}
One can push this Galois property to different formulations.  As we already noted in the above proof, if we vary over $S\to \ast/\mbb{G}_n$ but keep $T'/\mbb{G}_n$ fixed as the basepoint $\ast \to \ast/\mbb{G}_n$, we get
$$L_{K(n)}(\mbb{E}_n\otimes E_n)\simeq e_\ast e^\ast \mbb{E}_n$$
as $\mbb{E}_n$-algebras in $\on{Sh}_{et}(\ast/\mbb{G}_n;L_{K(n)}\on{Sp})$.  By the projection formula for $e$ we can also rewrite this as
$$L_{K(n)}(\mbb{E}_n\otimes E_n)\simeq L_{K(n)}(\mbb{E}_n\otimes e_\ast\mbb{S})$$
\end{remark}

\begin{corollary}\label{globalsec}
We have $\mathbb{E}(\ast/\mathbb{G}_n)=L_{K(n)}\mathbb{S}$, the $K(n)$-local sphere spectrum.
\end{corollary}
\begin{proof}
As E-theories are $T(n)$-local and complex orientable, they are $K(n)$-local, and thus so is $\mathbb{E}(\ast/\mathbb{G}_n)$ (as a limit of E-theories).  Since furthermore $K(n)=E_n/(p,v_1,\ldots,v_{n-1})$, it suffices to see that the unit map
$$E_n \to E_n\otimes \mathbb{E}(\ast/\mathbb{G}_n)$$
is an isomorphism mod $(p,v_1,\ldots,v_{n-1})$.  But the special case of Proposition \ref{galoiskunneth} where $T=\mathbb{G}_n$ and $T'=\ast$ shows that a retract of this map is a mod $(p,v_1,\ldots,v_{n-1})$-isomorphism.
\end{proof}

More generally, the same argument gives the following comparison with Devinatz-Hopkins' construction $\mathcal{E}_n^{DH}$.

\begin{theorem}
Recall from \ref{profgpexample} the pullback map $f^\ast$ from the $\infty$-topos of sheaves of anima on finite sets with continuous $\mathbb{G}_n$-action to the $\infty$-topos of v-sheaves on $\ast/\mathbb{G}_n$, induced by sending $T$ with $\mathbb{G}_n$-action to $T/\mathbb{G}_n$.

We have $f_\ast \mathbb{E}_n = \mathcal{E}_n^{DH}$.
\end{theorem}
\begin{proof}
Let $T$ be a finite set with continuous $\mathbb{G}_n$-action.  By Proposition \ref{galoiskunneth}, we have that the natural map
$$\mathbb{E}(T/\mathbb{G}_n)\otimes \mathbb{E} \to P_\ast P^\ast (\mathbb{E})$$
is a $(p,v_1,\ldots,v_{n-1})$-completion of v-sheaves of $\mathbb{E}$-algebras on $\ast/\mathbb{G}_n$, where $P:T/\mathbb{G}_n\to \ast/\mathbb{G}_n$.  Taking homotopy groups, we deduce that the algebra in Morava modules associated to the $K(n)$-local $E_\infty$-ring spectrum $\mathbb{E}(T/\mathbb{G}_n)$ identifies with the same one used to construct $\mathcal{E}_n^{DH}(T)$ in \cite{DevinatzE}.  As by \cite{MathewGalois} Thm.\ 10.9 $\mathcal{E}_n^{DH}$ is uniquely determined by its associated sheaf of algebras in Morava modules, this implies the claim.
\end{proof}

\begin{remark}
Recall from \ref{eisetale} that $\mathbb{E}_n$ is etale as a $K(n)$-local v-sheaf on $\ast/\mathbb{G}_n$.  Thus by \ref{profgpexample}, we can restate this Theorem in the reverse direction: we have
$$f^\ast \mathcal{E}_n^{DH} = \mathbb{E}_n,$$
where the pullback is in the sense of $K(n)$-local sheaves.

In \cite{MorPicard}, Mor used the left-hand side as the definition of the $\mathbb{G}_n$-equivariant condensed $E_\infty$-ring structure on $E_n$.  Thus the approach in \cite{MorPicard} and the approach taken here are equivalent.
\end{remark}

Recall the interpretation that $E_n$ is a ($K(n)$-local)  $\mbb{G}_n$-Galois extension of $L_{K(n)}\mbb{S}$ from \cite{RognesGalois}.  An important and well-known application of this is the resulting ``Galois descent'', the idea being that a $K(n)$-local spectrum $X$ should be recovered from $X\otimes \mathbb{E}_n$ with all its structure.  One formulation of this is the following (compare \cite{MathewGalois} Prop.\ 10.10, \cite{MorPicard} Thm.\ A.II).

\begin{theorem}\label{kndescent}
The base-change functor
$$L_{K(n)}\on{Sp}=\on{Sh}_{et}(\ast;L_{K(n)}\on{Sp}) \to \on{Mod}_{\mathbb{E}_n}(\on{Sh}_{et}(\ast/\mathbb{G}_n;L_{K(n)}\on{Sp}))$$
is an equivalence, with inverse given by taking global sections.
\end{theorem}
\begin{proof}
Let $f:\ast/\mathbb{G}_n\to\ast$.  We need to show that if $X\in L_{K(n)}\on{Sp}$, then
$$X\overset{\sim}{\rightarrow} f_\ast(f^\ast X \otimes \mathbb{E}_n),$$
and that if $M\in\on{Mod}_{\mathbb{E}_n}(\on{Sh}_{et}(\ast/\mathbb{G}_n;L_{K(n)}\on{Sp}))$, then
$$\mathbb{E}_n\otimes f^\ast f_\ast M \overset{\sim}{\rightarrow} M,$$
where all operations are in the sense of $K(n)$-local etale sheaves.

For the first claim, when $X=L_{K(n)}\mathbb{S}$ this is Corollary \ref{globalsec}.  The general case then follows from the projection formula for $f$, which by Remark \ref{tn} holds because $\mathbb{G}_n$ has virtual finite (mod $p$) cohomological dimension.

For the second claim, note again that $f_\ast$ preserves colimits and tensoring with any $X\in L_{K(n)}\on{Sp}$.  Thus it suffices to check $\mathbb{E}_n\otimes f^\ast f_\ast M\overset{\sim}{\to} M$ for a collection of $M$ which generates $\on{Mod}_{\mathbb{E}_n}(\on{Sh}_{et}(\ast/\mathbb{G}_n;L_{K(n)}Sp))$ under colimits and tensoring with $L_{K(n)}\on{Sp}$.  By the colimit diagram given by the standard resolution of an $\mathbb{E}_n$-module induced by the adjunction forgetting the $\mathbb{E}_n$-module structure, it therefore suffices to consider $M$ of the form $\mathbb{E}_n \otimes N$, where $N\in\on{Sh}_{et}(\ast/\mathbb{G}_n;L_{K(n)}\on{Sp})$.  But by the descendability argument in \ref{chausproper} the latter category is generated under colimits by the $g_\ast X$ for $X\in L_{K(n)}\on{Sp}$, where $g:\ast \to \ast/\mathbb{G}_n$.  Thus we reduce to checking one single case, namely
$$M = \mathbb{E}_n\otimes g_\ast L_{K(n)}\mathbb{S}.$$
But this $M$ identifies with $g_\ast E_n$ by the projection formula for $g$, and for $M=g_\ast E_n$ the desired claim follows from Proposition \ref{galoiskunneth}.
\end{proof}

\begin{remark}
Theorem \ref{kndescent} implies the well-known ``formula'' for the $K(n)$-localization of a spectrum $X$, namely $L_{K(n)}X$ identifies with the ``continuous homotopy fixed point spectrum'' of $\mathbb{G}_n$ acting on the $(p,v_1,\ldots,v_{n-1})$-completion of $X\otimes E_n$.  The continuous homotopy fixed point spectrum is, in the present language, modeled by the $f_\ast$ functor appearing in the above proof.

One can ask what this means from a calculational perspective.  Let's start with the most basic example $X=\mbb{S}$ for orientation.  Taking the Postnikov tower of $\mbb{E}_n\in\on{Sh}_v(\ast/\mbb{G}_n;\on{Sp})$ and then applying $f_\ast$, we deduce a conditionally convergent spectral sequence of light condensed abelian groups
$$H^p(\mbb{G}_n; \pi_q\mbb{E}_n)\Rightarrow \pi_{q-p} f_\ast(\mbb{E}_n)$$
where by definition $H^p(\mbb{G}_n;\pi_q\mbb{E}_n) := \pi_0f_\ast((\pi_q\mbb{E}_n)[p])$.  The pushforward $f_\ast$ on $v$-sheaves restricts to the pushforward on etale sheaves as one sees by resolving $\ast/\mbb{G}_n$ by light profinite sets and using \ref{etale}, so by \ref{globalsec} we see that $f_\ast (\mbb{E}_n)=L_{K(n)}\mbb{S}$ (where you take the $K(n)$-localization in light condensed spectra, but on underlying spectra it's the usual thing, compare Remark \ref{localizesolid}).  Thus this spectral sequence converges to the homotopy groups of the $K(n)$-local sphere (appropriately ``topologized'', i.e.\ equipped with light condensed structure).

As for the $E_2$-page, note (as in the proof of Lemma \ref{eisetale}) that the underlying light condensed abelian group $\pi_q\mbb{E}_n$ is profinite (it is the inverse limit of itself mod powers of $(p,v_1,\ldots,v_{n-1})$ and these are finite). Thus by Proposition \ref{truncateprof} there is no higher cohomology for its sections over a light profinite set $S$, and the $H^0$ is just the continuous functions on $S$ with values in $\pi_q \mbb{E}_n$.  Thus when we resolve $\ast/\mbb{G}_n$ by the bar resolution to calculate $f_\ast$, we just get the usual complex of continuous cochains.  We deduce the $E_2$-page is the usual continuous group cohomology (with profinite coefficients), and this gives another perspective on the spectral sequence from \cite{DevinatzE}.

For an arbitrary spectrum $X$ we can do the same thing to get a spectral sequence of light condensed abelian groups
$$H^p(\mbb{G}_n;\pi_q((X\otimes \mbb{E}_n)^\wedge_{(p,v_1,\ldots,v_{n-1})})) \Rightarrow \pi_{q-p} L_{K(n)}X,$$
but now $\pi_q((X\otimes \mbb{E}_n)^\wedge_{(p,v_1,\ldots,v_{n-1})})$ need not be profinite so in principle one has to take care with $\on{lim}^1$-terms when calculating the $E_2$-page (of course if $X$ is a finite spectrum this is not an issue).
\end{remark}

\begin{remark}\label{alternatedescent}
Suppose $\mathcal{C}\in\on{CAlg}(\on{Pr}^L_{st})$.  For a finite group $G$, Mathew in \cite{MathewGalois}, generalizing \cite{RognesGalois}, defined the notion of a faithful $G$-Galois algebra $R$ in $\mathcal{C}$.  Such an $R$ is a particular kind of commutative algebra object in $\mathcal{C}$ with $G$-action, or in other words a particular kind of sheaf on $\ast/G$ with values in $\on{CAlg}(\mathcal{C})$.  Galois descent in this setting says that
$$\mathcal{C}\overset{\sim}{\rightarrow}\on{Mod}_R(\on{Sh}(\ast/G;\mathcal{C})).$$
For a profinite group $G$, we can define a faithful $G$-Galois algebra $R$ to be a compatible collection of faithful $G/N$-Galois algebras $R_N$ as $N$ runs over the open normal subgroups of $G$.  As in \cite{ClausenHyper} 4.1, we can equivalently view this as a sheaf $R$ on the site $B^{fin}(G)$ of finite continuous $G$-sets with values in $\on{CAlg}(\mathcal{C})$, and the finite Galois descent above tautologically induces
$$\mathcal{C}\overset{\sim}{\rightarrow}\on{Mod}_R(\on{Sh}(B^{fin}G;\mathcal{C})).$$
When $R$ is hypercomplete and $G$ has finite cohomological dimension, using \ref{profgpexample} we can translate between sheaves on $B^{fin}G$ and etale sheaves on $\ast/G$ to arrive at a statement like \ref{kndescent}.  This gives another perspective.
\end{remark}

\section{Solid refinements of duality}\label{solidsec}

The duality theory of \ref{dualsec} takes place in the setting of etale coefficients.  For the purposes of the next section, it will be convenient to extend this to the more general \emph{solid} coefficients.

We start with a small recap of the solid theory from \cite{ClausenCondensed}.

\begin{definition}
A light condensed abelian group $M$ is called \emph{solid} if for every $S\in\on{ProfSet}^{light}$, every map $\mathbb{Z}[S]\to M$ admits a unique extension to a map $\mathbb{Z}[S]^\square\to M$.

Here
$$\mathbb{Z}[S]^\square := \underline{\on{Hom}}(\underline{\on{Hom}}(\mathbb{Z}[S],\mathbb{Z}),\mathbb{Z}) = \varprojlim_n \mathbb{Z}[S_n]$$
if $S=\varprojlim_n S_n$ is written as a sequential limit of finite sets.
\end{definition}Recall that the full subcategory of solid abelian groups is closed under all limits and higher derived limits, colimits, and extensions, as well as all internal $\underline{\on{Ext}}^i(N,-)$ for $N$ a light condensed abelian group.  For $S\in\on{ProfSet}^{light}$, the object $\mathbb{Z}[S]^\square$ is solid, and is compact and projective as a solid abelian group, even in the strong sense of internal hom: $\underline{\on{Ext}}^i(\mathbb{Z}[S]^\square, M)=0$ for all $i>0$ and $M$ solid and $\underline{\on{Hom}}(\mathbb{Z}[S]^\square,-):\on{Solid}\to\on{Solid}$ preserves all colimits.  Here $\underline{\on{Ext}}^i$ as before is the internal Ext in light condensed abelian groups.  If $(e_i)_{i\in I}$ is a $\mathbb{Z}$-basis of the free abelian group $\on{Cont}(S,\mathbb{Z})$, then the corresponding map
$$\mathbb{Z}[S]^\square \to \prod_I \mathbb{Z}$$
is an isomorphism of light condensed (solid) abelian groups.  Note that $I$ is countable.

The theory can be extended to spectra, as noted in \cite{ClausenAnalytic}.  We give a brief exposition here.

\begin{definition}
Let $X$ be a light condensed spectrum.  Say that $X$ is \emph{solid} if the homotopy group $\pi_n X$ is solid as a light condensed abelian group for all $n\in\mathbb{Z}$.
\end{definition}

\begin{theorem}\label{solidproperties}
\begin{enumerate}
\item The full subcategory $\on{Sp}^\square\subset \on{CondSp}^{light}$ of solid spectra is closed under all limits and colimits, and if $M,N\in\on{CondSp}^{light}$ with $N\in \on{Sp}^\square$, we have
$$\underline{\on{Hom}}(M,N) \in \on{Sp}^\square$$
as well.
\item If $S=\varprojlim_n S_n$ is a light profinite set, then the free solid spectrum on $S$ identifies with $\varprojlim \mathbb{S}[S_n]$.  Given a $\mathbb{Z}$-basis $(e_i)_{i\in I}$ of $\on{Cont}(S,\mathbb{Z})$, we can furthermore identify this with $\prod_I\mathbb{S}$ via the map $S\to \prod_I \mathbb{S}$, unique up to homotopy, which is given by $e_i$ on $\pi_0$ of the $i^{th}$ component.
\end{enumerate}
\end{theorem}
\begin{proof}
To show closure under colimits, it suffices to show closure under finite colimits and closure under infinite direct sums.  Infinite direct sums are t-exact in light condensed spectra (it's a general property of toposes), so that case reduces to the analogous fact for solid abelian groups.  The case of finite colimits reduces to that of shifts and cones, and those follow from the fact that the collection of solid abelian groups is an abelian subcategory closed under extensions.

To show closure under limits, similarly we need only show closure under products.  For this can filter by the Postnikov tower to reduce to the fact that solid abelian groups are closed under higher derived limits.  A similar argument shows the closure under internal Hom.

For 2, note that
$$\varprojlim_n \mathbb{S}[S_n] = \underline{\on{Hom}}(\varinjlim_n \prod_{S_n} \mathbb{S}, \mathbb{S}).$$
We have $\pi_0 \varinjlim \prod_{S_n}\mathbb{S} = \varinjlim \prod_{S_n}\mathbb{Z}=C(S;\mathbb{Z})$.  Given a $\mathbb{Z}$-basis $(e_i)_{i\in I}$ of the latter abelian group, the induced map
$$\oplus_{i\in I}\mathbb{S} \to \varinjlim_n \prod_{S_n}\mathbb{S}$$
is an isomorphism on $\pi_\ast$ and hence an isomorphism of spectra.  Taking $\underline{\on{Hom}}(-,\mathbb{S})$ in light condensed spectra, we deduce that
$$\varprojlim_n \mathbb{S}[S_n]\overset{\sim}{\rightarrow}\prod_{i\in I} \mathbb{S}.$$
In particular $\varprojlim_n \mathbb{S}[S_n]$ is connective.  (Recall that countable products are exact by \ref{replete}.) Using that $\mathbb{Z}$ is pseudocoherent as an $\mathbb{S}$-module and that countable products of hypercovers are hypercovers in light condensed anima by repleteness (\ref{replete}), we further deduce that
$$\mathbb{Z}\otimes_{\mathbb{S}}\varprojlim_n \mathbb{S}[S_n] = \varprojlim_n \mathbb{Z}[S_n].$$

To prove that $\mathbb{S}[S]^\square= \varprojlim \mathbb{S}[S_n]$, we need to show that if $M$ is a solid spectrum, then
$$\on{Hom}(\varprojlim_n \mathbb{S}[S_n],M)\overset{\sim}{\to} \on{Hom}(\mathbb{S}[S],M).$$
By connectivity, for this we can assume $M$ is connective.  Then filtering by the Postnikov tower and shifting, we can assume $M$ is concentrated in degree $0$.  Then in particular $M$ is a $\mathbb{Z}$-module, and by adjunction and the result $\mathbb{Z}\otimes_{\mathbb{S}}\varprojlim_n \mathbb{S}[S_n] = \varprojlim_n \mathbb{Z}[S_n]$ of the previous paragraph we reduce to the fact that $\mathbb{Z}[S]^\square$ is the free solid abelian group on $S$ also in the derived sense (e.g.\ by the projectivity discussed above).
\end{proof}

Let us now give some examples of solid spectra.

\begin{example}\label{solidexamples}
\begin{enumerate}
\item Recall the unique geometric morphism $\delta^\ast: \on{An}\to\on{CondAn}^{light}$, which is fully faithful (\ref{etale}).  It induces a fully faithful embedding $\delta^\ast:\on{Sp}\to\on{CondSp}^{light}$; the essential image consists of the \emph{discrete} light condensed spectra.  Every discrete light condensed spectrum is solid.  Indeed, taking homotopy groups, this reduces to the fact that every discrete light condensed abelian group is solid, which is easy to check from the definition.
\item Given any diagram of spectra, we can use $\delta^\ast$ to put it inside light condensed spectra, then take the limit there.  The result is solid by \ref{solidproperties}.  Note that this is generally very different from taking the limit just in spectra and then using $\delta^\ast$ on the result; morally we are considering the limit of spectra ``equipped with the inverse limit topology'', not with the ``discrete topology''.
\item The light condensed Morava E-theory spectrum $e^\ast \mathbb{E}_n$ from \ref{lighte} is solid.  Indeed, by Lemma \ref{eisetale} it is $(p,v_1,\ldots,v_{n-1})$-complete and its mod $(p,v_1,\ldots,v_{n-1})$-reduction is discrete, so this follows from 2.
\end{enumerate}
\end{example}

The closure of solid spectra under limits and colimits implies a further closure property.  Let $L:\on{Sp}\to\on{Sp}$ be any accessible localization functor, with associated full $\infty$-subcategory of local objects $L\on{Sp}\subset\on{Sp}$.  For any presentable stable $\infty$-category $\mathcal{C}$, we get an induced accessible localization of $\mathcal{C}$, namely the quotient gotten by killing the cocomplete full subcategory generated by the objects $A\otimes X$ where $X\in\mathcal{C}$ and $A\in\on{Sp}$ satisfies $LA=0$.  Explicitly, $L\mathcal{C}\subset\mathcal{C}$ identifies with the full subcategory of $Y\in \mathcal{C}$ such that the mapping spectrum $\on{map}(X,Y)\in\on{Sp}$ lies in $L\on{Sp}$ for all $X\in\mathcal{C}$.  More abstractly, comparing universal properties we see that
$$L\mathcal{C}=\mathcal{C}\otimes_{\on{Sp}}L\on{Sp}.$$

In particular, we can apply this with $\mathcal{C}=\on{CondSp}^{light}$.  Explicitly, $X\in\on{CondSp}^{light}$ lies in $L\on{CondSp}^{light}$ if and only if $X(S)\in L\on{Sp}$ for all $S\in\on{ProfSet}^{light}$, and the localization $L$ of $\on{CondSp}^{light}$ can be obtained by applying $L$ on the level of presheaves and then hypersheafifying the result.   Then we have the following result showing that this $L$-localization passes to solid spectra.

\begin{lemma}
Let $L$ be an accessible localization of spectra, and let $X\in \on{CondSp}^{light}$.
\begin{enumerate}
\item If $LX=0$, then $L(X^\square)=0$.
\item If $X$ is solid, then $LX$ is solid.
\end{enumerate}
\end{lemma}
\begin{proof}
For 1, recall that $\on{ker}(L)\subset\on{CondSp}^{light}$ is generated under colimits by the objects $A\otimes Y$ for $Y\in\on{CondSp}^{light}$ and $A\in\on{Sp}$ with $LA=0$.  Since $(-)^\square$ preserves colimits, it suffices to show the claim for such $X=A\otimes Y$.  But again because $(-)^\square$ commutes with colimits we have
$$(A\otimes Y)^\square = A\otimes (Y^\square),$$
another basic object in $\on{ker}(L)$, as required.

Now for 2, denote by $X'$ the fiber of the map $X \to LX$.  Then $X'\in \on{ker}(L)$, and $X' \to X$ is the terminal map to $X$ from an object in $\on{ker}(L)$.  Since $X$ is solid, we get a factoring $X' \to (X')^\square \to X$.  Applying 1, we have that $(X')^\square \in \on{ker}(L)$.   Thus, by the universal property of $X' \to X$ we deduce that $X'$ is a retract of $(X')^\square$.  It follows (again by closure under colimits, or limits) that $X'$ is also solid, hence so is $LX = \on{cofib}(X'\to X)$, as desired.
\end{proof}

In particular, the $L$-localization of solid spectra $\on{SolidSp}\to L\on{SolidSp}$ is calculated by the $L$-localization on the level of condensed spectra.

\begin{remark}\label{localizesolid}
Another compatibility property is that for $X\in\on{CondSp}^{light}$, we have $L(X(\ast)) = (LX)(\ast)$.  In other words the $L$-localization of the underlying spectrum of $X$ can be calculated as the underlying spectrum of $LX$.  This holds because hypersheafification on the site of light profinite sets does not change the value on $\ast$ (every hypercover of $\ast$ is split).  Strangely, this can be useful, in conjunction with the previous lemma, for calculating $L$-localizations of spectra.
\end{remark}

Now we will discuss the solid tensor product, in manner analogous to the discussion of the (derived) solid tensor product of solid abelian groups in \cite{ClausenCondensed}.  By adjunction, another way of stating the closure under internal Hom in part 1 of this theorem is the following: if $M$ is a light condensed spectrum with $M^\square=0$, then $(M\otimes_\mathbb{S} N)^{\square}=0$ for any light condensed spectrum $N$.  In other words, the localization of light condensed spectra to solid spectra is gotten by killing a tensor ideal.  It follows that there is a unique symmetric monoidal structure on solid spectra making solidification into a symmetric monoidal functor; this is called the \emph{solid tensor product}.  Explicitly, it is gotten by solidifying the tensor product of light condensed spectra:
$$M\otimes^\square_{\mathbb{S}}M = (M\otimes_\mathbb{S} N)^\square.$$

The most basic calculation, which follows from the definition, is the following: if $S,T\in\on{ProfSet}^{light}$, then
$$\mathbb{S}[S]^\square\otimes^\square_{\mathbb{S}}\mathbb{S}[T]^{\square} \overset{\sim}{\rightarrow}\mathbb{S}[S\times T]^\square.$$
Using this we get the following, which is the true starting point for solid tensor product calculations.
\begin{lemma}\label{tensorproduct}
Let $I$ and $J$ be countable sets.  Then
$$\prod_I \mathbb{S} \otimes^\square_{\mathbb{S}} \prod_J \mathbb{S} \overset{\sim}{\rightarrow} \prod_{I\times J}\mathbb{S}.$$
\end{lemma}
\begin{proof}
Choose light profinite sets $S$ and $T$ with an $I$-indexed basis of $C(S,\mathbb{Z})$ and a $J$-indexed basis of $C(T,\mathbb{Z})$.  Then we get an $I\times J$-indexed basis of $C(S\times T;\mathbb{Z})=C(S;\mathbb{Z})\otimes_\mathbb{Z} C(T;\mathbb{Z})$.  On the other hand by Theorem \ref{solidproperties} this data induces $\mathbb{S}[S]^\square = \prod_I \mathbb{S}$ and $\mathbb{S}[T]^\square = \prod_J \mathbb{S}$ and $\mathbb{S}[S\times T]^{\square}=\prod_{I\times J}\mathbb{S}$.  Comparing we deduce the claim.
\end{proof}

More general solid tensor product calculations can in principle be reduced to this one, using the fact that the solid tensor product commutes with colimits in each variable (a fact which is obvious from the definition).  A very simple example is the following: on the full subcategory of $\on{CondSp}^{light}$ given by the discrete spectra (those in the essential image of the fully faithful functor $\delta^\ast:\on{Sp}\to\on{CondSp}^{light}$), which lies inside $\on{SolidSp}$ by \ref{solidexamples}, the solid tensor product, the light condensed tensor product, and the usual tensor product of spectra all agree (via obvious comparison maps).  Another example is the following, expressing a situation in which the solid tensor product behaves as a ``completed tensor product''.

\begin{lemma}\label{tensorlimit}
Suppose $(N_n)_{n\in\mathbb{N}}$ and $(M_m)_{m\in\mathbb{N}}$ are two towers of connective spectra, with the following property: for all $n\in\mathbb{N},m\in \mathbb{N}$, and $d\in\mathbb{N}$, we have that $\pi_d N_n$ and $\pi_d M_m$ are finitely generated as abelian groups.  View the inverse limits $N=\varprojlim_n N_n$ and $M=\varprojlim_m M_m$ as solid spectra, as in Example \ref{solidexamples}.  Then
$$N\otimes^\square_\mathbb{S} M \overset{\sim}{\longrightarrow} \varprojlim_m (N\otimes^{\square}_\mathbb{S} M_m)$$
(which then also identifies with $\varprojlim_{n,m} N_n\otimes^\square_{\mbb{S}} M_m$ by applying the same result again to the other variable.)
\end{lemma}
\begin{proof}
Let $N' = \prod_n N_n$ and $M'=\prod_m M_m$.  Using the ``shift minus identity'' fiber sequences $N \to N'\to N'$ and $M\to M'\to M'$ we reduce to showing that
$$N'\otimes^\square_\mathbb{S} M' \overset{\sim}{\rightarrow} \prod_{n,m} N_n\otimes_\mathbb{S}M_m.$$
Now, using the hypothesis on the homotopy groups, each $N_n$ and $M_m$ admits a resolution, in $\on{Sp}_{\geq 0}$, where each term is a finite product of copies of $\mathbb{S}$.  By \ref{replete} a countable product of resolutions is a resolution in $\on{CondSp}^{light}_{\geq 0}$.  In this way we reduce to the case where each $N_n$ and $M_m$ is a finite product of copies of $\mathbb{S}$.  Then by reindexing the claim reduces to Lemma \ref{tensorproduct}.
\end{proof}

\begin{example}\label{tensorexamples}
\begin{enumerate}
\item The $p$-complete sphere $\mathbb{S}_{\widehat{p}}$ in $\on{CondSp}^{light}$ is idempotent with respect to the solid tensor product:
$$\mathbb{S}_{\widehat{p}}\overset{\sim}{\rightarrow}\mathbb{S}_{\widehat{p}}\otimes^\square_\mathbb{S}\mathbb{S}_{\widehat{p}}.$$
Indeed, Lemma \ref{tensorlimit} implies that $\mathbb{S}_{\widehat{p}}\otimes^\square_\mathbb{S}\mathbb{S}_{\widehat{p}}$ is $p$-complete, so we can test mod $p$, when we reduce to $\mathbb{S}/p\overset{\sim}{\rightarrow}\mathbb{S}/p$.
In particular, we deduce that $\on{Mod}_{\mathbb{S}_{\widehat{p}}}(\on{SolidSp})$ is a full subcategory of $\on{SolidSp}$ via the forgetful functor.  It contains, but is much bigger than, the category of $p$-complete solid spectra: it is closed under arbitrary (not $p$-completed) colimits inside $\on{SolidSp}$.
\item If $X$ is a connective spectrum each of whose homotopy groups is finitely generated, or a countable inverse limit of such spectra in $\on{SolidSp}$, then for any light profinite set $S=\varprojlim_n S_n$, we have
$$X\otimes^\square_\mathbb{S} \mathbb{S}[S]^\square\overset{\sim}{\rightarrow}\varprojlim_n X[S_n],$$
and if $I$ is a countable set, then
$$X\otimes^\square_\mathbb{S}\prod_I\mathbb{S}\overset{\sim}{\rightarrow} \prod_I X.$$
\end{enumerate}
\end{example}

Only slightly more involved is the following example, which will be used later.

\begin{proposition}\label{tensorwithe}
Recall the light condensed Morava E-theory spectrum $E_n:=e^\ast \mathbb{E}_n$ from Deinition \ref{lighte}, which is solid by \ref{solidexamples}.  If $S=\varprojlim_m S_m$ is a light profinite set, we have
$$E_n\otimes^\square_\mathbb{S} \mathbb{S}[S]^\square\overset{\sim}{\rightarrow}\varprojlim_m E_n[S_m],$$
and if $I$ is a countable set, then
$$E_n\otimes^\square_\mathbb{S}\prod_I\mathbb{S}\overset{\sim}{\rightarrow} \prod_I E_n.$$
\end{proposition}
\begin{proof}
It suffices to prove the second statement (e.g.\ by the Milnor sequence).  Let $e_n = \tau_{\geq 0}E_n$.  Then $e_n$ is $p,u_1,\ldots,u_{n-1}$-complete (where $\pi_0E_n \simeq W(\mathbb{F}_{p^n})[[u_1,\ldots,u_{n-1}]]$) and $e_n/p,u_1,\ldots,u_{n-1}$ is discrete with finite homotopy.  Then we can apply \ref{tensorproduct} to deduce $e_n\otimes^\square_\mathbb{S}\prod_I\mathbb{S}=\prod_Ie_n$.  We have $E_n=e_n[\sigma^{-1}]$, so $E_n\otimes^\square(-) = (e_n\otimes^\square(-))[\sigma^{-1}]$.  On the other hand we also have $\prod_I E_n = (\prod_I e_n)[\sigma^{-1}]$ because inverting $\sigma$ just periodizes the pre-periodic homotopy of $\prod_I e_n$.  Comparing we conclude.\end{proof}

Now, in this paper we have been studying etale sheaves of spectra on $\ast/G$ for various light condensed groups $G$.  By definition, this $\infty$-category $\on{Sh}_{et}(\ast/G;\on{Sp})$ is a full subcategory of $\on{Sh}_v(\ast/G;\on{Sp})$, or in other words, the $\infty$-category of spectrum objects in the $\infty$-topos $\on{CondAn}^{light}_{/(\ast/G)}$.

Our goal in this section is to study an $\infty$-category which is intermediate between $\on{Sh}_{et}(\ast/G;\on{Sp})$ and $\on{Sh}_{v}(\ast/G;\on{Sp})$, and to which the duality of Section \ref{dualsec} extends.

\begin{definition}\label{solidbg}
For a light condensed group $G$, define
$$\on{Sh}_\square(\ast/G;\on{Sp})\subset \on{Sh}_{v}(\ast/G;\on{Sp})$$
to be the full subcategory of those $v$-sheaves of spectra on $\ast/G$ such that on pullback along $\ast\to \ast/G$, the resulting light condensed spectrum is solid.
\end{definition}

In other words, this is the $\infty$-category of solid spectra equipped with a $G$-action.

\begin{remark}
A similar definition works with $\ast/G$ replaced by any light condensed anima which admits a surjection from a discrete set.  Using the relative solid theory of Fargues-Scholze \cite{FarguesGeometrization} Ch.\ VII, one can even extend to arbitrary light condensed anima.  We will not pursue this here.
\end{remark}

\begin{remark}\label{moduleperspective}
The inclusion $\on{Sh}_\square(\ast/G;\on{Sp})\subset \on{Sh}_{v}(\ast/G;\on{Sp})$ has a left adjoint, again called solidification, such that pullback along $\ast\to\ast/G$ commutes with solidification.  To see this, and for some other purposes, it can be convenient to take a different perspective.  Note that $\on{Sh}_v(\ast/G;\on{An})$ identifies, by descent, with the $\infty$-category $\on{Mod}_G(\on{CondAn}^{light})$, viewing $G$ as an associative monoid object in $\on{CondAn}^{light}$ with cartesian product symmetric monoidal structure.  Passing to spectrum objects, we deduce that
$$\on{Sh}_v(\ast/G;\on{Sp})=\on{Mod}_{\mathbb{S}[G]}(\on{CondSp}^{light}).$$
In these terms, Definition \ref{solidbg} translates into saying that an $\mathbb{S}[G]$-module is solid iff its underlying light condensed spectrum is solid, and the fact that solidification of $\mathbb{S}[G]$-modules commutes with passing to the underlying light condensed spectrum follows formally from the fact that solidification of light condensed spectra is a symmetric monoidal localization.  In particular,
$$\on{Sh}_\solid(\ast/G;\on{Sp}) = \on{Mod}_{\mathbb{S}[G]^\solid}(\on{SolidSp}).$$
Let us take the convention that we use \emph{left} modules unless otherwise indicated.
\end{remark}

\begin{remark}
We will more precisely be interested in $p$-adic solid coefficients.  We take this in the broadest possible sense: that of modules over the solid ring $\mathbb{S}_{\widehat{p}}$.  Recall from Example \ref{tensorexamples} that this is idempotent, so if $\mathcal{C}$ is any $\on{SolidSp}$-linear presentable $\infty$-category (such as $\on{Sh}_\solid(\ast/G;\on{Sp})$) then $\mathbb{S}_{\widehat{p}}$-modules in $\mathcal{C}$ forms a full subcategory of $\mathcal{C}$ via the forgetful functor.
\end{remark}

Before proving that duality extends to the solid category, we need some lemmas.  The first concerns some general module theory in the light condensed setting.  Suppose $R$ is a light condensed ring, or more precisely an object in $\on{Alg}(\on{CondSp}^{light})$.  Then $R(\ast)$ is a usual ring spectrum (an object in $\on{Alg}(\on{Sp})$), and there is a functor
$$\on{Mod}_{R(\ast)}(\on{Sp}) \to \on{Mod}_R(\on{CondSp}^{light})$$
sending a (left) $R(\ast)$-module $M$ to the $R$-module
$$R\otimes_{R(\ast)}M.$$
This is the left adjoint to the obvious functor $N\mapsto N(\ast)$ in the other direction.

It is a basic fact that this functor $\on{Mod}_{R(\ast)} \to \on{Mod}_R(\on{CondSp}^{light})$ is fully faithful.  Indeed, its right adjoint preserves colimits (as every cover of $\ast$ is split); but the unit of the adjunction is an isomorphism on $R(\ast)$ by construction, hence is an isomorphism in general.  In particular, we can view $\on{Perf}(R(\ast))$ as a full subcategory of $\on{Mod}_R(\on{CondSp}^{light})$.  A translation is the following.

\begin{definition}
For a light condensed ring $R$, say that a (left) $R$-module $M\in\on{Mod}_R(\on{CondSp}^{light})$ is \emph{perfect} if it lies in the thick subcategory generated by the free rank one $R$-module $R$.  Denote this full subcategory by
$\on{Perf}(R)\subset \on{Mod}_R(\on{CondSp}^{light})$.
\end{definition}

Thus, by the above remarks, $\on{Perf}(R)$ identifies with the usual $\infty$-category of perfect $R(\ast)$-modules via $M\mapsto M(\ast)$.

Now, here is the first lemma.

\begin{lemma}\label{checkperfmodp}
Let $R$ be a connective $p$-complete light condensed ring spectrum, and $M$ a connective $p$-complete left $R$-module.  Suppose $\mathbb{F}_p \otimes_\mathbb{S} M \in \on{Perf}(\mathbb{F}_p\otimes_\mathbb{S}R)$.  Then $M\in \on{Perf}(R)$.
\end{lemma}
\begin{proof}
Consider the natural map
$$M(\ast)\otimes_{R(\ast)} R\rightarrow M.$$
Applying $\mathbb{F}_p\otimes_\mathbb{S}-$, we get the analogous map for the $\mathbb{F}_p\otimes_{\mathbb{S}}R$-module $\mathbb{F}_p\otimes_{\mathbb{S}}M$, an isomorphism by hypothesis.  Since both sides are connective, it follows that our map exhibits the target as the $p$-completion of the source.

Thus we reduce to proving the analogous lemma in the setting of usual (not condensed) ring spectra and modules.  In fact, using the existence of homomorphism $\mathbb{F}_p\otimes_\mathbb{S}R \to (\pi_0R)/p$, it suffices to show the following: if $R$ is a $p$-complete connective ring spectrum and $M$ is a $p$-complete connective $R$-module such that $M_0 = R_0\otimes_{R}M$ is a perfect $R_0=(\pi_0R)/p$-module, then $M\in\on{Perf}(R)$.  In the case of $E_\infty$-rings this is a special case of \cite{LurieHA} Prop.\ 8.3.5.5; let us now give the analogous argument in the associative ring case.

We proceed by induction on the Tor-amplitude of the perfect $R_0$-module $M_0$.  Recall that this is the smallest value $b-a$ such that for all right $R_0$-modules $N$ in degree $0$, the tensor product $N\otimes_{R_0} M_0$ lives in degrees $[a,b]$.  By shifting we can assume the range of degrees is $[0,b]$, and then the Tor-amplitude is $b$.

First assume the Tor-amplitude $b$ is $>0$.  Taking $N=R_0$ we deduce that $M_0$ lives in degrees $[0,b]$.  Being a perfect complex, its lowest nonzero homotopy group is finitely presented.  In particular we can choose an $R_0$-module map $f_0:R_0^d \to M_0$ which is surjective on $\pi_0$.  Lift $f_0$ arbitrarily to an $R$-module map $f:R^d \to M$ and let $M'=\on{Fib}(R^d\to M)$.  Since $M$ is $p$-complete, by derived Nakayama we have that $f$ is surjective on $\pi_0$, hence $M'$ is connective.  It is also $p$-complete as $R$ and $M$ are.  Moreover $M'_0 = \on{Fib}(R_0^d\to M_0)$ has Tor-amplitude $b-1$ by the long exact sequence in Tor, so we can conclude by induction.

It remains to handle the case where $M_0$ is of Tor amplitude $0$.  Being in particular a perfect complex concentrated in degree $0$, we have that $M_0$ is a finitely presented $R_0$-module.  By the Tor amplitude condition, it is also flat.  Therefore it is finitely generated projective.  Choose $\alpha_0:R_0^d\to M_0$ and $\beta_0:M_0\to R_0^d$ such that $\alpha_0\circ \beta_0 = \on{id}_{M_0}$, and let $e_0=\beta_0 \circ \alpha_0$, so that $e_0$ is an idempotent $d\times d$ matrix over $R_0$.

It is elementary to see that idempotent $d\times d$ matrices always lift along square-zero extensions (of associative rings).  Thus we can lift $e_0$ to $(\pi_0R)/p^2$, then to $(\pi_0R)/p^3$, etc., and then in the limit to an idempotent $d\times d$ matrix over $\varprojlim_n (\pi_0R)/p^n$.  As $R$ is (derived) $p$-complete, the map $\pi_0R \to \varprojlim_n (\pi_0R)/p^n$ is another example of a square zero extension, and it follows that we can lift $e_0$ to an idempotent $n\times n$ matrix over $\pi_0R$.  Thus, we can lift $e_0$ to an $R$-linear map $e:R^d\to R^d$ which is idempotent up to homotopy: $e^2\sim e$.  This is enough to guarantee that the map
$$R^d \to e^{-1}R^d \oplus (1-e)^{-1}R^d$$
is an isomorphism.  To finish the proof, it suffices to show that $M$ is isomorphic to $e^{-1}R^d$.

As in the inductive argument above, we can lift $\alpha_0:R_0^d\to M_0$ to an $R$-module map $\alpha:R^d\to M$.  Replacing $\alpha$ by $\alpha\circ e$, we can assume that $\alpha\circ e \sim \alpha$.  This allows us to make a map $e^{-1}R^d\to M$ which reduces to the given isomorphism $e_0^{-1}R_0^d\to M_0$ when we apply $\mathbb{F}_p\otimes_\mathbb{S}-$.  As both $e^{-1}R^d$ and $M$ are connective and $p$-complete, this implies that the map $e^{-1}R^d\to M$ is an isomorphism, as required.
\end{proof}

\begin{example}\label{perfexample}
Suppose $G$ is a light profinite group such that the augmentation module $\mathbb{F}_p$ is a perfect $\mathbb{F}_p[[G]]$-module, where $\mathbb{F}_p[[G]]=\varprojlim_N \mathbb{F}_p[G/N]$ denotes the usual completed group ring.  (Recall from \cite{LazardGroupes} that every compact $p$-adic Lie group has an open normal subgroup satisfying this condition.)
Then $\mathbb{S}_{\wh{p}}\in\on{Perf}(\mathbb{S}_{\wh{p}}[G]^\solid)$.  Indeed, by the lemma it suffices to show that $\mathbb{F}_p\in\on{Perf}(\mathbb{F}_p[G]^\solid)$.  By hypothesis $\mathbb{F}_p$ admits a finite resolution by finitely generated projective $\mathbb{F}_p[[G]]$-modules.  But by definition $\mathbb{F}_p[G]^\solid$ is exactly $\mathbb{F}_p[[G]]$ equipped with its natural light profinite structure.  Every finitely generated projective $\mathbb{F}_p[[G]]$-module gets its induced functorial light profinite topology, and a surjection of light profinite sets is surjective in solid modules by definition, so it follows that the same resolution shows $\mathbb{F}_p\in\on{Perf}(\mathbb{F}_p[G]^\solid)$, as desired.
\end{example}

Duality concerns the study of the functor $\pi_\ast:\on{Sh}_\square(\ast/G;\mathbb{S}_{\widehat{p}})\to \on{Sh}_\square(\ast;\mathbb{S}_{\widehat{p}})$, which is by definition right adjoint to the pullback $\pi^\ast$.  In terms of the identification $\on{Sh}_\solid(\ast/G;\mbb{S}_{\widehat{p}}) = \on{Mod}_{\mbb{S}_{\widehat{p}}[G]^\solid}(\on{SolidSp})$ of Remark \ref{moduleperspective}, $\pi^\ast$ corresponds to the restriction of scalars functor associated to the augmentation map $\mathbb{S}_{\widehat{p}}[G]^\solid\to \mathbb{S}_{\widehat{p}}$.  It follows that $\pi_\ast$ corresponds to the solid internal hom out of $\mathbb{S}_{\widehat{p}}$ in $\mbb{S}_{\wh{p}}[G]^\solid$-modules:

$$\pi_\ast M = \on{map}_{\mbb{S}_{\wh{p}}[G]^\solid}(\mbb{S}_{\wh{p}},M).$$

As usual, discussion of duality starts with identifying situations where $\pi_\ast$ commutes with colimits and satisfies the projection formula.

\begin{proposition}\label{smallgood}
Let $p$ be a prime.  Suppose $G$ is a light profinite group such that $\mathbb{F}_p\in\on{Perf}(\mathbb{F}_p[[G]])$.  Then
$$\pi_\ast: \on{Sh}_\solid(\ast/G;\mbb{S}_{\wh{p}})\to\on{Sh}_\solid(\ast;\mbb{S}_{\wh{p}})$$
preserves colimits and satisfies the projection formula.
\end{proposition}
\begin{proof}
We use the above formula $\pi_\ast M = \on{map}_{\mbb{S}_{\wh{p}}[G]^\solid}(\mbb{S}_{\wh{p}},M)$.  Note that the solid internal hom out of $\mbb{S}_{\wh{p}}[G]^\square$ identifies with the forgetful functor, and this clearly commutes with colimits and satisfies the projection formula.  Thus it suffices to show that $\mathbb{S}_{\wh{p}}$ lies in the thick subcategory generated by $\mbb{S}_{\wh{p}}[G]^\square$.  But this was proven in Example \ref{perfexample}.
\end{proof}

As in the discussion in \ref{sixsec}, the basic idea of the six functor formalism is that, in cases where $\pi_\ast$ fails to preserve colimits and satisfy the projection formula, one wants a replacement $\pi_!$.  Here we will implement this idea in a somewhat different way, following the idea of \cite{KleinDualizing}.

\begin{definition}
Let $p$ be a prime, and suppose $G$ is a light profinite group which admits an open subgroup $N$ such that $\mathbb{F}_p\in \on{Perf}(\mathbb{F}_p[[N]])$.  Define
$$\pi_!:\on{Sh}_\solid(\ast/G;\mbb{S}_{\wh{p}})\to\on{Sh}_\solid(\ast;\mbb{S}_{\wh{p}})$$
to be the left Kan extension of the restriction of $\pi_\ast$ to the full subcategory of compact objects in $\on{Sh}_\solid(\ast/G;\mbb{S}_{\wh{p}})$.  Thus $\pi_!$ comes with a natural transformation $\pi_!\to \pi_\ast$.
\end{definition}

\begin{remark}\label{compactlygenerated}
Note that $\on{SolidSp}$ is compactly generated with compact generator $\prod_I \mathbb{S}$ for countably infinite $I$. Hence $\on{Mod}_{\mbb{S}_{\wh{p}}[G]^\solid}(\on{SolidSp})$ is also compactly generated, with compact generator $\mbb{S}_{\wh{p}}[G]^\solid\otimes^\solid_{\mbb{S}}(\prod_I \mathbb{S})$.  Thus this is in principle a sensible definition.  Moreover by Lemma \ref{tensorproduct} the full subcategory of compact objects $\on{SolidSp}^{\aleph_0}\subset\on{SolidSp}$ is closed under solid tensor product, so we also see that the action of $\on{SolidSp}$ on $\on{Mod}_{\mbb{S}_{\wh{p}}[G]^\solid}(\on{SolidSp})$ restricts to an action on the level of compact objects.\end{remark}

\begin{lemma}\label{handleonlowershriek}
Let $p$ be a prime.  Suppose $G$ is a light profinite group with an open subgroup $H$ such that $\mathbb{F}_p\in \on{Perf}(\mathbb{F}_p[[H]])$.  Write $f:\ast/H\to\ast/G$ for the map induced by the inclusion, and continue to write $\pi:\ast/G\to\ast$ for the projection.  For any $M\in\on{Sh}_\square(\ast/G;\mbb{S}_{\wh{p}})$ which is in the essential image of $f_\ast$, we have:
\begin{enumerate}
\item for any $X\in\on{Sh}_\square(\ast;\mbb{S}_{\wh{p}})$, the natural map $X\otimes \pi_\ast M\to\pi_\ast(X\otimes M)$ is an isomorphism;
\item the natural map $\pi_! M\to\pi_\ast M$ is an isomorphism.
\end{enumerate}
\end{lemma}
\begin{proof}
Note that $f$ satisfies the projection formula because it's finite etale.  Since $\pi':\ast/H\to\ast$ satisfies the projection formula by Proposition \ref{smallgood}, we deduce 1.  For 2, consider the natural transformation
$$\pi_!\circ f_\ast \to \pi_\ast \circ f_\ast = \pi'_\ast.$$
It suffices to show this is an isomorphism.  Again because $f$ is finite etale, $f_\ast$ preserves colimits; thus so does $\pi_!\circ f_\ast$.  So does $\pi'_\ast$ by Proposition \ref{smallgood}.  Thus it suffices to show that this natural transformation is an isomorphism on the compact generator $\prod_I\mathbb{S}\otimes\mbb{S}_{\wh{p}}[H]^\solid$ of Remark \ref{compactlygenerated}.  For that, in turn, since $\pi_!\to\pi_\ast$ is an isomorphism on compact objects it suffices to show $f_\ast (\prod_I\mathbb{S}\otimes^\solid\mbb{S}_{\wh{p}}[H]^\solid)\simeq \prod_I\mathbb{S}\otimes^\solid\mbb{S}_{\wh{p}}[G]^\solid$.  By projection formula for $f_\ast$, this reduces to $f_\ast\mbb{S}_{\wh{p}}[H]^\solid\simeq \mbb{S}_{\wh{p}}[G]^\solid$ which follows by ambidexterity of $f$ (which holds due to $f$ being finite etale).
\end{proof}

\begin{proposition}
Let $p$ be a prime.  Suppose $G$ is a light profinite group which admits an open subgroup $H$ such that $\mathbb{F}_p\in \on{Perf}(\mathbb{F}_p[[H]])$.  The functor $\pi_!:\on{Sh}_\solid(\ast/G;\mbb{S}_{\wh{p}})\to\on{Sh}_\solid(\ast;\mbb{S}_{\wh{p}})$ has the natural structure of a $\on{Sh}_\solid(\ast;\mbb{S}_{\wh{p}})$-linear functor.
\end{proposition}
\begin{proof}
As in the proof of the previous lemma, note that the compact generator of $\on{Sh}_\solid(\ast/G;\mbb{S}_{\wh{p}})$ lies in the essential image of $f_\ast$. Thus, by the previous lemma we have that $\pi_\ast$, when restricted to compact objects, satisfies the projection formula, hence is canonically $\on{Sh}_\solid(\ast;\mbb{S}_{\wh{p}})^{\aleph_0}$-linear on compact objects.  Passing to Ind-categories we deduce the claim.
\end{proof}

\begin{remark}\label{moritaperspective}
By the Morita theory of \cite{LurieHA} Prop.\ 7.1.2.4, it follows that in module theoretic terms, we can describe $\pi_!$ as follows.  Define
$$\mathbb{D}_G=\pi_\ast(\mathbb{S}_{\wh{p}}[G]^\solid)=\on{map}_{\mathbb{S}_{\widehat{p}}[G]^\square}(\mathbb{S}_{\widehat{p}},\mathbb{S}_{\widehat{p}}[G]^\square),$$
where the right hand side denotes the solid mapping spectrum in solid left $\mathbb{S}_{\widehat{p}}[G]^\square$-modules.  We view $\mathbb{D}_G$ as a right $\mathbb{S}_{\widehat{p}}[G]^\square$-module using the (unit) bimodule structure of $\mathbb{S}_{\widehat{p}}[G]^\square$.  In other words, $\mathbb{D}_G$ is the value on $\mathbb{S}_{\widehat{p}}[G]^\square$ of the right adjoint to the functor $\mathbb{S}_{\widehat{p}}\otimes^\square_{\mathbb{S}}-$ from right modules to bimodules (using the augmentation left module structure on $\mathbb{S}_{\widehat{p}}$).  Then
$$\pi_!M = \mathbb{D}_G\otimes^\square_{\mathbb{S}[G]^\square}M$$
as a functor from $\on{Mod}_{\mbb{S}_{\wh{p}}[G]^\solid}(\on{SolidSp})\to\on{Mod}_{\mathbb{S}_{\wh{p}}}(\on{SolidSp})$.
\end{remark}

Now we compare this $\pi_!$ with the etale $\pi_!$ defined earlier in Section \ref{sixsec}.  For this we pass to the $p$-complete setting.  Recall that the natural pullback functor
$$\on{Sh}_{et}(\ast/G;\on{Sp}_{\wh{p}})\to\on{Sh}_v(\ast/G;\on{Sp}_{\wh{p}})$$
is fully faithful.  The essential image consists of those $p$-complete v-sheaves whose mod $p$ reduction is etale as a $v$-sheaf of spectra (\ref{etinsidev}).  By Example \ref{solidexamples}, any such is solid. Thus we can naturally view
$$\on{Sh}_{et}(\ast/G;\on{Sp}_{\wh{p}})\subset\on{Sh}_\solid(\ast/G;\on{Sp}_{\wh{p}})$$
as the full subcategory of those $p$-complete solid sheaves on $\ast/G$ whose mod $p$ reduction is etale (``discrete'') on pullback to $\ast$.  Trivially, pullback functors on both sides are compatible with this inclusion.  Now we show compatibility with $\pi_!$ functors.

\begin{proposition}
Let $p$ be a prime.  Suppose $G$ is a light profinite group which admits an open subgroup $H$ such that $\mathbb{F}_p\in \on{Perf}(\mathbb{F}_p[[H]])$.  The functor
$$\pi_!:\on{Sh}_\solid(\ast/G;\on{Sp}_{\wh{p}})\to\on{Sh}_\solid(\ast;\on{Sp}_{\wh{p}})$$
sends $\on{Sh}_{et}(\ast/G;\on{Sp}_{\wh{p}})$ inside $\on{Sh}_{et}(\ast;\on{Sp}_{\wh{p}})=\on{Sp}_{\widehat{p}}$, and on that level canonically identifies with the $\pi_!$ from the etale theory in a unique manner compatible with the natural transformation to $\pi_\ast$ (see the proof for more precision).
\end{proposition}
\begin{proof}
Let us use a superscript ``et'' to denote the etale variants of the various functors.  We have a natural transformations
$$\pi^{et}_!\to\pi^{et}_\ast \to \pi_\ast\leftarrow \pi_!$$
of functors $\on{Sh}_{et}(\ast/G;\on{Sp}_{\wh{p}})\to \on{Sh}_\solid(\ast;\on{Sp}_{\wh{p}})$.
For any $H\subset G$ as in the statement, if $f:\ast/H\to\ast/G$, then $f^{et}_\ast=f_\ast$ because $f$ is finite etale (on underlying spectra, $f_\ast$ is just a finite direct sum).  We claim the following: if $M\in \on{Sh}_{et}(\ast/G;\on{Sp}_{\wh{p}})$ lies in the essential image of $f_\ast$, then:
\begin{enumerate}
\item $\pi^{et}_\ast(M)\overset{\sim}{\rightarrow}\pi_\ast(M)$;
\item $\pi^{et}_!(M)\overset{\sim}{\rightarrow}\pi^{et}_\ast(M)$;
\item $\pi_!(M)\overset{\sim}{\rightarrow}\pi_\ast(M)$.
\end{enumerate}
For claim 1, it suffices to show that $\pi'_\ast$ preserves the etale full subcategories, where $\pi':\ast/H\to\ast$.  This follows by resolving $\ast/H$ by copies of the light profinite sets $H^n$ and using a Postnikov tower argument. Claim 2 follows from \ref{chausproper}. Finally, Claim 3 follows from Proposition \ref{handleonlowershriek}.
Combining 1 and 3, it follows in particular that $\pi_!(M)$ is etale for any such $M$.  But by Example \ref{countableexamples}, the etale category is generated under colimits by objects of the form $M=f_\ast N$, so this implies that $\pi_!(M)$ is etale for any etale $M$, as claimed.

For the second claim, note that since both $\pi_!$ and $\pi_!^{et}$ preserve colimits, natural transformations out of either of them are uniquely determined by their effect on any full subcategory of compact generators.  Again by \ref{countableexamples} such can be chosen of the form $f_\ast(N)$, whence the unique identification $\pi_!=\pi_!^{et}$ compatible with the map to $\pi_\ast$, as claimed.
\end{proof}

Now we turn to the discussion of the right adjoint $\pi^!$ to $\pi_!$.  In terms of the Morita-theoretic perspective from Remark \ref{moritaperspective}, this is described as follows:

$$\pi^!X = \on{map}(\mathbb{D}_G,X).$$

In other words, we take the solid mapping spectrum from $\mathbb{D}_G$ to $X$, equipped with the left $\mathbb{S}_{\wh{p}}[G]^\solid$-module structure coming from the right module structure on $\mathbb{D}_G$.  The following result shows the compatibility of this $\pi^!$ with the etale one.

\begin{proposition}
Let $p$ be a prime.  Suppose $G$ is a light profinite group which admits an open subgroup $H$ such that $\mathbb{F}_p\in\on{Perf}(\mathbb{F}_p[[H]])$.

We have that $\pi^!$ sends etale sheaves to etale sheaves, and the canonical comparison map $(\pi^!)^{et}\to\pi^!$ is an isomorphism on etale sheaves.
\end{proposition}
\begin{proof}
By the above formula for $\pi^!$, for the first claim it suffices to show that the solid mapping spectrum out of $\mathbb{D}_G$ sends etale sheaves (on $\ast$) to etale sheaves (on $\ast$).  Recall that $\mathbb{D}_G = \pi_\ast \mathbb{S}_{\wh{p}}[G]^\solid$.  By ambidexterity for $\ast/H\to\ast/G$, it follows that $\mathbb{D}_G\simeq\mathbb{D}_H$ as solid spectra.  Thus we reduce to the case where $\mathbb{F}_p\in\on{Perf}(\mathbb{F}_p[[G]])$, which by Example \ref{perfexample} implies that $\mathbb{S}_{\wh{p}}\in\on{Perf}(\mathbb{S}_{\wh{p}}[G]^\solid)$.  It follows that $\mathbb{D}_G$ lies in the thick subcategory (of solid spectra) generated by $\mathbb{S}_{\wh{p}}[G]^\solid$.  Thus it suffices to show that the internal mapping spectrum out of $\mathbb{S}[S]$, for any light profinite set $S$, sends etale spectra to etale spectra.  But this follows from \ref{etale}.

We have just shown that $\pi^!$ sends etale sheaves to etale sheaves.  As the same is true for $\pi_!$ by the previous result, it follows a fortiori by the defining universal properties that $\pi^!=(\pi^!)^{et}$ on etale sheaves, finishing the proof.
\end{proof}

\begin{corollary}\label{compareduality}
Let $p$ be a prime.  Suppose $G$ is a light profinite group which admits an open subgroup $H$ such that $\mathbb{F}_p\in\on{Perf}(\mathbb{F}_p[[H]])$.  Then:
\begin{enumerate}
\item The etale dualizing sheaf $\pi^!(\mathbb{S}_{\wh{p}})$ for $\ast/G$ from \ref{dualsec} identifies with the solid mapping spectrum
$$\on{map}(\mathbb{D}_G,\mathbb{S}_{\wh{p}})$$
equipped with $G$-action coming from the right module structure on $\mathbb{D}_G$ (\ref{moritaperspective}).
\item Suppose $G$ is a virtual (mod $p$) Poincaré duality group.  Then $\mathbb{D}_G$ is itself etale, and identifies with the dual (= inverse) of the etale dualizing sheaf.
\end{enumerate}
\end{corollary}
\begin{proof}
Part 1 follows by applying the previous proposition to the unit etale sheaf $\mbb{S}_{\wh{p}}$.  For part 2, note that in the proof of the previous proposition we showed that as a solid spectrum, $\mathbb{D}_G$ lies in the thick subcategory generated by the $M=\mathbb{S}_{\wh{p}}[S]^\square$ for light profinite $S$.  But for any such $M$, we have
$$M\overset{\sim}{\rightarrow}\on{map}(M,\on{map}(M,\mbb{S}_{\wh{p}}))$$
by definition of solidification (see \ref{solidproperties}).  It follows that we can read 1 in reverse, getting $\mathbb{D}_G \simeq \on{map}(e^\ast\pi^!\mathbb{S}_{\wh{p}},\mbb{S}_{\wh{p}})$ as underlying $p$-complete spectra.  Since $G$ is a duality group, $\pi^!\mathbb{S}_{\wh{p}}$ is invertible, indeed a shift of the unit.  (Indeed, we see as in the proof of \ref{padicsmooth} that $\pi^!\mbb{S}_{\wh{p}}$ is bounded below, so by \ref{compactinvertible} it suffices to see that $\pi^!\mbb{F}_p$ is one dimensional in a single degree, but this is exactly the mod $p$ duality condition.)
\end{proof}

We ended the previous section with the theorem expressing that $K(n)$-local spectra are the same thing as $\mathbb{G}_n$-semilinear $K(n)$-local $E_n$-modules.  This result does not wholly extend to the solid context, but at least half of it does:

\begin{proposition}\label{halfofdescent}
Let $p$ be a prime and $n\geq 1$.  For $X\in \on{SolidSp}$, we have
$$L_{K(n)}X = \pi_\ast(L_{K(n)}(X\otimes^\square \mathbb{E}_n)),$$
where $\pi:\ast/\mathbb{G}_n\to\ast$ and we view $\mathbb{E}_n\in\on{Sh}_\square(\ast/\mathbb{G}_n;\on{Sp})$.
\end{proposition}
\begin{proof}
Recall from \cite{LazardGroupes} that any compact $p$-adic Lie group $G$ admits an open normal subgroup $N$ with $\mbb{F}_p\in\on{Perf}(\mathbb{F}_p[[N]])$.  Then $\ast/G = (\ast/N)_{G/N}$, and it follows from Lemma \ref{smallgood} and Tate vanishing in $K(n)$-local homotopy theory (\cite{GreenleesTate}, \cite{KuhnTate}) that $\pi_!\to\pi_\ast$ is an isomorphism in the $K(n)$-local context, and $\pi_\ast$ satisfies the projection formula.  It follows that the claim reduces to the case $X=\mathbb{S}$, which is \ref{globalsec}.
\end{proof}

\begin{remark}\label{nottheotherhalf}
This means that the functor
$$L_{K(n)}\on{SolidSp} \to \on{Mod}_{\mbb{E}_n}(\on{Sh}_{\square}(\ast/\mbb{G}_n;L_{K(n)}\on{Sp}))$$
defined by $X\mapsto L_{K(n)}(X\otimes^{\square}\mbb{E}_n)$ is fully faithful.  As hinted at above, in contrast to the discrete setting it is not essentially surjective.  To show this, let us exhibit two non-isomorphic objects in the target each of which map to $E_n$ under $\pi_\ast$.

The first is the obvious one, $L_{K(n)}(E_n\otimes^{\square}\mbb{E}_n)$, which maps to $E_n$ under $\pi_\ast$ by the above result.  To produce the second one, start with $\mbb{S}[G]^\square \otimes^\square \mbb{E}_n$, which is $K(n)$-local by Proposition \ref{tensorwithe}.  Note that by definition $\mbb{S}[G]^\square$ is the solidification of $e_\natural \mbb{S}$, where $e_\natural$ denotes the left adjoint to $e^\ast$ on v-sheaves.  Thus by the projection formula for $e_\natural$ (which is formal by slice topos nonsense) we deduce
$$\mbb{S}[G]^\square \otimes^\square \mbb{E}_n\simeq \mbb{S}[G]^\square \otimes^\square E_n.$$
(We have left the $\mbb{E}_n$-module category, but that's okay because we only need to be able to calculate $\pi_\ast$.)  But now by the projection formula for $\pi_\ast$ we get that
$$\pi_\ast (\mbb{S}[G]^\square\otimes^{\solid} E_n)\simeq (e^\ast\mbb{D}_G)\otimes^\square E_n.$$
Thus if we take our object $\mbb{S}[G]^\square\otimes^\solid\mbb{E}_n$ and tensor it with $e^\ast \pi^!\mbb{S}_{\wh{p}}$, this cancels the $e^\ast\mbb{D}_G$ by Corollary \ref{compareduality} and so we get another object whose $\pi_\ast$ gives $E_n$.  These objects are not isomorphic, even on underlying solid spectra, because the first one is etale in the $K(n)$-local sense while the second one is not.
\end{remark}

Here is a corollary in ordinary homotopy theory.  Of course, it can also be proved directly, but it's not completely obvious.

\begin{corollary}
Let $p$ be a prime, $n\geq 1$, and $I$ a countable set.  Then for the spectrum $\prod_I \mathbb{S} \in \on{Sp}$, we have
$$L_{K(n)}\prod_I\mathbb{S} = \prod_I L_{K(n)}\mbb{S}.$$
\end{corollary}
\begin{proof}
By Remark \ref{localizesolid}, it suffices to show that the $K(n)$-localization of the \emph{solid} spectrum $\prod_I\mathbb{S}$ identifies with $\prod_I L_{K(n)}\mbb{S}$.  But \ref{tensorwithe} shows that
$$\prod_I\mathbb{S}\otimes^\solid \mbb{E}_n = \prod_I \mbb{E}_n.$$
In particular this is already $K(n)$-local, and as $\pi_\ast$ commutes with limits we deduce the claim from Proposition \ref{halfofdescent}.
\end{proof}

\begin{remark}
The result also holds without cardinality restriction on $I$.  On the other hand, it seems, based on the precise way the telescope conjecture fails (\cite{BurklundTelescope}), that for $n\geq 2$ the analogous statement is false for $L_{T(n)}$.
\end{remark}

\section{Application to self-duality of Morava E-theory, following Beaudry-Goerss-Hopkins-Stojanoska}\label{bghssec}

Hopkins proved the intriguing result that the $K(n)$-local spectrum $E_n$ is \emph{self-dual} up to a shift --- not in the strong sense of categorical dualizability (which would be false, essentially because $E_n$ is an infinite extension of $L_{K(n)}\mathbb{S}$) --- but in the weaker sense that the internal Hom in $K(n)$-local spectra from $E_n$ to $L_{K(n)}\mathbb{S}$ identifies with $\Sigma^{-n^2}E_n$ (see \cite{StricklandDuality} for an account).  A more structured version of this result, taking into account the action of the Morava stabilizer group $\mathbb{G}_n$, was proved in \cite{BeaudryLinearization}.  Here we revisit these results from the perspective of this paper.

We start with a formulation of this self-duality in the setting of etale sheaves.  This is essential a formal consequence of the Galois structure.  Recall the Morava stabilizer group $\mathbb{G}_n$ and the structured Morava E-theory $\mathbb{E}_n$, which is a commutative algebra object in $\on{Sh}_{et}(\ast/\mathbb{G}_n;L_{K(n)}\on{Sp})$ (\ref{lighte}).  Recall also the dualizing object 

$$L_{K(n)}\pi^!(\mbb{S}_{\wh{p}}) \in \on{Sh}_{et}(\ast/\mathbb{G}_n;L_{K(n)}\on{Sp})$$
for $\pi:\ast/\mbb{G}_n\to \ast$, which is invertible by \ref{padicsmooth} and identifies with $L_{K(n)}\mbb{S}_{\wh{p}}^{\on{ad}_{\mbb{G}_n}}$ by \ref{maintheorem}.

\begin{lemma}
Let $p$ be a prime and $n\geq 1$.  Consider the map
$$c:\mathbb{E}_n\to L_{K(n)}\pi^!(\mathbb{S}_{\wh{p}})$$
in $\on{Sh}_{et}(\ast/\mathbb{G}_n;L_{K(n)}\on{Sp})$ which is adjoint (via \ref{padicsmooth}) to the inverse of the isomorphism
$$L_{K(n)}\mathbb{S}\overset{\sim}{\rightarrow} \pi_\ast\mathbb{E}_n$$
from \ref{globalsec}.

Then the composition with multiplication
$$\mathbb{E}_n\otimes\mathbb{E}_n\to\mbb{E}_n\overset{c}{\rightarrow} L_{K(n)}\pi^!(\mathbb{S}_{\wh{p}})$$
is a perfect pairing in the weak sense that it induces
$$\mathbb{E}_n\overset{\sim}{\rightarrow}\underline{\on{map}}(\mbb{E}_n,L_{K(n)}\pi^!(\mbb{S}_{\wh{p}})),$$
where $\underline{\on{map}}$ denotes the internal mapping spectrum in $\on{Sh}_{et}(\ast/\mathbb{G}_n;L_{K(n)}\on{Sp})$.
\end{lemma}
\begin{proof}
As $\underline{\on{map}}(\mbb{E}_n,-)$ gives the right adjoint to the forgetful functor from $\mathbb{E}_n$-modules to underlying $K(n)$-local etale sheaves, another interpretation of the map
$$\mbb{E}_n\rightarrow \underline{\on{map}}(\mbb{E}_n,L_{K(n)}\pi^!(\mbb{S}_{\wh{p}}))$$
we're trying to prove is an isomorphism is that it is the unique $\mathbb{E}_n$-module map which restricts to $c$ on the unit of $\mathbb{E}_n$.  By Theorem \ref{kndescent}, it suffices to show that this map induces an isomorphism of $K(n)$-local spectra on applying $\pi_\ast$.  But this is exactly how $c$ was chosen.
\end{proof}

It is important to note that there is no a priori relation between this self-duality of $\mathbb{E}_n$ in the etale category and the self-duality of $E_n$ proved by Hopkins.  This is because the internal hom in etale sheaves does not generally commute with pullbacks (the relevant pullback here being along $e:\ast\to\ast/\mbb{G}_n$).  Indeed, in a purely algebraic Galois setting, say an infinite $G$-Galois extension $K\to L$ of characteristic $p$ fields where $G$ is a mod $p$ Poincare duality group, the analog of this proposition still holds ($L$ is weakly self-dual as a $K$-module in etale sheaves on $\ast/G$ up to a twist) by the same proof (see Remark \ref{alternatedescent}), but it is clearly not the case that $L$ is weakly self-dual as a $K$-module forgetting the $G$-action.

To recover Hopkins' result, we therefore need some ingredient which is special to the situation of $L_{K(n)}\mbb{S}\to E_n$.  In our presentation this ingredient will be Proposition \ref{tensorwithe}, which stems from the fact that although $E_n$ is an infinite extension of $L_{K(n)}\mbb{S}$, it still has enough finiteness properties for some purposes (it is 2-periodic with profinite homotopy groups).

To give the precise argument addressing this, note that internal hom in v-sheaves does commute with arbitrary pullbacks.  This is a tautological feature of slice $\infty$-topoi.  Thus it would suffice to prove the analog of the above proposition taking the internal hom
$$\underline{\on{hom}}_v(\mbb{E}_n,L_{K(n)}\pi^!\mbb{S}_{\wh{p}})$$
in $v$-sheaves instead.  For that it suffices to show that the internal hom in v-sheaves already lies in the $K(n)$-local etale full subcategory (as then it will a fortiori identify with the internal Hom in $K(n)$-local etale shaves).  As by definition etale sheaves are detected on pullback to any cover, and $e^\ast \pi^!\mbb{S}_{\wh{p}}$ is a shift of $L_{K(n)}\mbb{S}$, it suffices to show the following.

\begin{lemma}
The internal hom in light condensed spectra
$$\underline{\on{hom}}_v(E_n,L_{K(n)}\mbb{S})$$
is $K(n)$-locally etale, i.e.\ it is discrete mod $(p,v_1,\ldots,v_{n-1})$.
\end{lemma}
\begin{proof}
As $L_{K(n)}\mbb{S} = \pi_\ast \mbb{E}_n$ (\ref{globalsec}), we have
$$\underline{\on{hom}}_v(E_n,L_{K(n)}\mbb{S}) \simeq \pi_\ast \underline{\on{hom}}_v(E_n,\mbb{E}_n).$$
Now, consider the $\mbb{E}_n$-module (in $K(n)$-local $v$-sheaves on $\ast/\mbb{G}_n$) given by $L_{K(n)}(\mbb{E}_n\otimes E_n)$.  By the Galois property \ref{differentgalois}, we have
$$L_{K(n)}(\mathbb{E}_n\otimes E_n) \simeq  L_{K(n)}(\mathbb{E}_n\otimes e_\ast \mbb{S}).$$
as $\mbb{E}_n$-modules.  Since $\mathbb{E}_n$ is a $K(n)$-local $\mathbb{E}_n$-module, it follows that
$$\underline{\on{hom}}_v(E_n,\mbb{E}_n) \simeq \underline{\on{hom}}_v(e_\ast \mbb{S},\mbb{E}_n).$$
Now we use the solid formalism.  Since $\mbb{E}_n$ is solid, this internal hom in $v$-sheaves is also solid.  Consider the induced natural map
$$\underline{\on{hom}}_v(e_\ast\mbb{S},\mbb{S})\otimes^\square\mathbb{E}_n\rightarrow \underline{\on{hom}}_v(e_\ast \mbb{S},\mbb{E}_n)$$
We claim this is an isomorphism.  Indeed we can check on underlying solid spectra (i.e.\ after pullback to $\ast$), and then the claim is given by \ref{tensorwithe}.  (Once again, this is where we use a special feature of Morava E-theory as compared to a general Galois extension.)  Now by definition of solidification $\underline{\on{hom}}_v(e_\ast\mbb{S},\mbb{S})$ identifies with the solid sheaf on $\ast/G$ corresponding to $\mathbb{S}[G]^\solid$ viewed as a left module over itself.  Then in Remark \ref{nottheotherhalf} we already calculated that $\pi_\ast(\mbb{S}[G]^\solid\otimes^\solid\mbb{E}_n)$ is $K(n)$-locally etale, finishing the proof.
\end{proof}

Combining these two lemmas, we get the following, giving the equivariant generalization of Hopkins' result (already obtained by \cite{BeaudryLinearization} in different language).

\begin{theorem}
Let $p$ be a prime and $n\geq 1$.  There is a canonical isomorphism
$$\mathbb{E}_n\overset{\sim}{\rightarrow}\underline{\on{map}}_v(\mbb{E}_n,L_{K(n)}\pi^!(\mbb{S}_{\wh{p}})),$$
where $\underline{\on{map}}_v$ denotes the internal mapping spectrum in $\on{Sh}_v(\ast/\mathbb{G}_n;\on{Sp})$.
\end{theorem}
\begin{proof}
By the first lemma, we get the analog of this with $v$ replaced by $et$.  By the second lemma, the two different notions of internal mapping object agree.
\end{proof}

\begin{remark}
In \cite{BeaudryLinearization}, this result is fruitfully used to understand self-duality properties of spectra of the form $(E_n)^{hG}$, for $G\subset \mbb{G}_n$ a suitable finite subgroup.  For their analysis, one needs to get a good handle on the dualizing object $\pi^!(\mbb{S}_{\wh{p}})$.  A first step is the \emph{linearization} of this problem, provided by our main theorem \ref{maintheorem} (which they independently showed holds upon restriction to suitable finite subgroups).   As a second step, they transfer from the $p$-adic J-homomorphism to the real J-homomorphism using an integral structure on the adjoint representation, again restricted to a suitable finite subgroup.  This is accomplished using the reciprocity law from \cite{ClausenJ} or Section \ref{jsec} (which they reprove a special case of).  In the third step, they use  \emph{orientation} properties of $E_n$ with respect to the real J-homomorphism in conjunction with some very nice analysis of characteristic classes of representations to pin down the relevant twists.

It would be interesting to understand what kind of orientation properties one has directly for the $p$-adic J-homomorphism, without going over to the real $J$-homomorphism.
\end{remark}

\printbibliography

@book{SerreLie,
  title={Lie algebras and Lie groups: 1964 lectures given at Harvard University},
  author={Serre, Jean-Pierre},
  year={2009},
  publisher={Springer}
}

@article {RognesGalois,
    AUTHOR = {Rognes, John},
     TITLE = {Galois extensions of structured ring spectra. {S}tably
              dualizable groups},
   JOURNAL = {Mem. Amer. Math. Soc.},
  FJOURNAL = {Memoirs of the American Mathematical Society},
    VOLUME = {192},
      YEAR = {2008},
    NUMBER = {898},
     PAGES = {viii+137},
   MRCLASS = {55P43 (55M05 55P35 57T05)},
  MRNUMBER = {2387923},
MRREVIEWER = {Alberto\ Cavicchioli},
       DOI = {10.1090/memo/0898},
}

@article{GrothendieckAnalytic,
     author = {Grothendieck, Alexander},
     title = {Techniques de construction en g\'eom\'etrie analytique. {I.} {Description} axiomatique de l'espace de {Teichm\"uller} et de ses variantes},
     journal = {S\'eminaire Henri Cartan},
     pages = {1--33},
     publisher = {Secr\'etariat math\'ematique},
     volume = {13},
     number = {1},
     year = {1960-1961},
     zbl = {0142.33503},
}

@misc{ScholzeDiamonds,
      title={Etale cohomology of diamonds}, 
      author={Peter Scholze},
      year={2022},
      eprint={1709.07343},
      archivePrefix={arXiv},
      primaryClass={math.AG},
      url={https://arxiv.org/abs/1709.07343}, 
}

@article{LurieSAG,
  title={Spectral algebraic geometry},
  author={Lurie, Jacob},
  year={2018},
  url={https://www.math.ias.edu/~lurie/papers/SAG-rootfile.pdf}
}

@misc{LurieHA,
  title={Higher algebra},
  author={Lurie, Jacob},
  year={2017},
  url={https://www.math.ias.edu/~lurie/papers/HA.pdf}
}

@book{LurieHTT,
  title={Higher topos theory},
  author={Lurie, Jacob},
  year={2009},
  publisher={Princeton University Press}
}

@article {BhattProEtale,
    AUTHOR = {Bhatt, Bhargav and Scholze, Peter},
     TITLE = {The pro-\'etale topology for schemes},
   JOURNAL = {Ast\'erisque},
  FJOURNAL = {Ast\'erisque},
    NUMBER = {369},
      YEAR = {2015},
     PAGES = {99--201},
      ISBN = {978-2-85629-805-3},
   MRCLASS = {14F05 (14F20 14F35 14H30 18B25)},
  MRNUMBER = {3379634},
MRREVIEWER = {Pieter\ Belmans},
}

@article {MondalReplete,
    AUTHOR = {Mondal, Shubhodip and Reinecke, Emanuel},
     TITLE = {On {P}ostnikov completeness for replete topoi},
   JOURNAL = {Homology Homotopy Appl.},
  FJOURNAL = {Homology, Homotopy and Applications},
    VOLUME = {27},
      YEAR = {2025},
    NUMBER = {1},
     PAGES = {179--196},
   MRCLASS = {18N60 (14A20 14F06 18F10)},
  MRNUMBER = {4883680},
}

@article{HoyoisLefschetz,
  title={A quadratic refinement of the Grothendieck--Lefschetz--Verdier trace formula},
  author={Hoyois, Marc},
  journal={Algebraic \& Geometric Topology},
  volume={14},
  number={6},
  pages={3603--3658},
  year={2015},
  publisher={Mathematical Sciences Publishers}
}

@article{ClausenHyper,
  title={Hyperdescent and {\'e}tale K-theory},
  author={Clausen, Dustin and Mathew, Akhil},
  journal={Inventiones mathematicae},
  volume={225},
  pages={981--1076},
  year={2021},
  publisher={Springer}
}

@misc{BarwickPyknotic,
      title={Pyknotic objects, I. Basic notions}, 
      author={Clark Barwick and Peter Haine},
      year={2019},
      eprint={1904.09966},
      archivePrefix={arXiv},
      primaryClass={math.AG},
      url={https://arxiv.org/abs/1904.09966}, 
}

@article{AtiyahDuality,
  title={Thom complexes},
  author={Atiyah, Michael Francis},
  journal={Proceedings of the London Mathematical Society},
  volume={3},
  number={1},
  pages={291--310},
  year={1961},
  publisher={Oxford University Press}
}

@article {CarmeliCardinality,
    AUTHOR = {Carmeli, Shachar and Schlank, Tomer M. and Yanovski, Lior},
     TITLE = {Ambidexterity and height},
   JOURNAL = {Adv. Math.},
  FJOURNAL = {Advances in Mathematics},
    VOLUME = {385},
      YEAR = {2021},
     PAGES = {Paper No. 107763, 90},
      ISSN = {0001-8708,1090-2082},
   MRCLASS = {18N60 (55U35)},
  MRNUMBER = {4246977},
       DOI = {10.1016/j.aim.2021.107763},
       URL = {https://doi.org/10.1016/j.aim.2021.107763},
}

@article {MilnorDuality,
    AUTHOR = {Milnor, John and Spanier, Edwin},
     TITLE = {Two remarks on fiber homotopy type},
   JOURNAL = {Pacific J. Math.},
  FJOURNAL = {Pacific Journal of Mathematics},
    VOLUME = {10},
      YEAR = {1960},
     PAGES = {585--590},
   MRCLASS = {57.00},
  MRNUMBER = {117750},
MRREVIEWER = {J.\ F.\ Adams}
}

@article{AtiyahOrientation,
  title={Clifford modules},
  author={Atiyah, Michael F and Bott, Raoul and Shapiro, Arnold},
  journal={Topology},
  volume={3},
  pages={3--38},
  year={1964},
  publisher={Pergamon}
}

@article{LazardGroupes,
  title={Groupes analytiques $ p $-adiques},
  author={Lazard, Michel},
  journal={Publications Math{\'e}matiques de l'IH{\'E}S},
  volume={26},
  pages={5--219},
  year={1965}
}

@article {SerreDimension,
    AUTHOR = {Serre, Jean-Pierre},
     TITLE = {Sur la dimension cohomologique des groupes profinis},
   JOURNAL = {Topology},
  FJOURNAL = {Topology. An International Journal of Mathematics},
    VOLUME = {3},
      YEAR = {1965},
     PAGES = {413--420},
   MRCLASS = {14.50 (20.80)},
  MRNUMBER = {180619},
MRREVIEWER = {Michael\ I.\ Rosen}
}

@book{DixonAnalytic,
  title={Analytic pro-p groups},
  author={Dixon, John D and Du Sautoy, Marcus PF and Mann, Avinoam and Segal, Dan},
  number={61},
  year={2003},
  publisher={Cambridge University Press}
}

@article{Pstragowski,
  title={P-Adic Analytic Groups, Harvard Math 291Y, Fall 2023},
  author={Pstragowski, Piotr},
  url={https://www.kurims.kyoto-u.ac.jp/~piotr/291y_notes.pdf}
}

@article{MathewGalois,
  title={The Galois group of a stable homotopy theory},
  author={Mathew, Akhil},
  journal={Advances in Mathematics},
  volume={291},
  pages={403--541},
  year={2016},
  publisher={Elsevier}
}

@inproceedings{MathewDescent,
  title={Examples of descent up to nilpotence},
  author={Mathew, Akhil},
  booktitle={Geometric and Topological Aspects of the Representation Theory of Finite Groups: PIMS Summer School and Workshop, July 27-August 5, 2016},
  pages={269--311},
  year={2018},
  organization={Springer}
}

@article{BhattProjectivity,
  title={Projectivity of the Witt vector affine Grassmannian},
  author={Bhatt, Bhargav and Scholze, Peter},
  journal={Inventiones mathematicae},
  volume={209},
  pages={329--423},
  year={2017},
  publisher={Springer}
}

@article{BousfieldLocalization,
  title={The localization of spectra with respect to homology},
  author={Bousfield, Aldridge K},
  journal={Topology},
  volume={18},
  number={4},
  pages={257--281},
  year={1979}
}

@misc{EfimovKtheory,
      title={K-theory and localizing invariants of large categories}, 
      author={Alexander I. Efimov},
      year={2025},
      eprint={2405.12169},
      archivePrefix={arXiv},
      primaryClass={math.KT},
      url={https://arxiv.org/abs/2405.12169}, 
}

@misc{DranishnikovDimension,
      title={Cohomological dimension theory of compact metric spaces}, 
      author={A. N. Dranishnikov},
      year={2005},
      eprint={math/0501523},
      archivePrefix={arXiv},
      primaryClass={math.GN},
      url={https://arxiv.org/abs/math/0501523}, 
}

@article {NeemanPhantom,
    AUTHOR = {Neeman, Amnon},
     TITLE = {On a theorem of {B}rown and {A}dams},
   JOURNAL = {Topology},
  FJOURNAL = {Topology. An International Journal of Mathematics},
    VOLUME = {36},
      YEAR = {1997},
    NUMBER = {3},
     PAGES = {619--645},
   MRCLASS = {18E30 (18A40 55U99)},
  MRNUMBER = {1422428},
MRREVIEWER = {David\ J.\ Green},
}

@article{GrothendieckTohoku,
  title={Sur quelques points d'alg{\`e}bre homologique},
  author={Grothendieck, Alexandre},
  journal={Tohoku Mathematical Journal, Second Series},
  volume={9},
  number={2},
  pages={119--183},
  year={1957},
  publisher={Mathematical Institute, Tohoku University}
}

@book{GaitsgoryDerivedI,
  title={A study in derived algebraic geometry: Volume I: correspondences and duality},
  author={Gaitsgory, Dennis and Rozenblyum, Nick},
  volume={221},
  year={2019},
  publisher={American Mathematical Society}
}

@misc{ArinkinLanglands,
      title={The stack of local systems with restricted variation and geometric Langlands theory with nilpotent singular support}, 
      author={D. Arinkin and D. Gaitsgory and D. Kazhdan and S. Raskin and N. Rozenblyum and Y. Varshavsky},
      year={2022},
      eprint={2010.01906},
      archivePrefix={arXiv},
      primaryClass={math.AG},
      url={https://arxiv.org/abs/2010.01906}, 
}

@misc{RamziRigid,
      title={Locally rigid $\infty$-categories}, 
      author={Maxime Ramzi},
      year={2024},
      eprint={2410.21524},
      archivePrefix={arXiv},
      primaryClass={math.CT},
      url={https://arxiv.org/abs/2410.21524}, 
}

@article {NikolausTC,
    AUTHOR = {Nikolaus, Thomas and Scholze, Peter},
     TITLE = {On topological cyclic homology},
   JOURNAL = {Acta Math.},
  FJOURNAL = {Acta Mathematica},
    VOLUME = {221},
      YEAR = {2018},
    NUMBER = {2},
     PAGES = {203--409},
      ISSN = {0001-5962,1871-2509},
   MRCLASS = {55U35 (16E40 18E30 19D99)},
  MRNUMBER = {3904731},
MRREVIEWER = {Geoffrey\ M. L. Powell},
       DOI = {10.4310/ACTA.2018.v221.n2.a1},
       URL = {https://doi.org/10.4310/ACTA.2018.v221.n2.a1},
}

@incollection {AdemTransfer,
    AUTHOR = {Adem, A. and Cohen, R. L. and Dwyer, W. G.},
     TITLE = {Generalized {T}ate homology, homotopy fixed points and the
              transfer},
 BOOKTITLE = {Algebraic topology ({E}vanston, {IL}, 1988)},
    SERIES = {Contemp. Math.},
    VOLUME = {96},
     PAGES = {1--13},
 PUBLISHER = {Amer. Math. Soc., Providence, RI},
      YEAR = {1989},
   MRCLASS = {55N25 (18G30 19D55 55P42)},
  MRNUMBER = {1022669},
MRREVIEWER = {J.\ P. C. Greenlees},
}

@misc{HeyerSixFunctors,
      title={6-Functor Formalisms and Smooth Representations}, 
      author={Claudius Heyer and Lucas Mann},
      year={2024},
      eprint={2410.13038},
      archivePrefix={arXiv},
      primaryClass={math.CT},
      url={https://arxiv.org/abs/2410.13038}, 
}

@misc{DauserUniqueness,
      title={Uniqueness of six-functor formalisms}, 
      author={Adam Dauser and Josefien Kuijper},
      year={2024},
      eprint={2412.15780},
      archivePrefix={arXiv},
      primaryClass={math.AG},
      url={https://arxiv.org/abs/2412.15780}, 
}

@inproceedings{QuillenQ,
  title={Higher algebraic K-theory: I},
  author={Quillen, Daniel},
  booktitle={Higher K-Theories: Proceedings of the Conference held at the Seattle Research Center of the Battelle Memorial Institute, from August 28 to September 8, 1972},
  pages={85--147},
  year={1972},
  organization={Springer}
}

@article{BarwickQ,
  title={On exact infty-categories and the Theorem of the Heart},
  author={Barwick, Clark},
  journal={Compositio Mathematica},
  volume={151},
  number={11},
  pages={2160--2186},
  year={2015},
  publisher={London Mathematical Society}
}

@article{SuslinK,
  title={On the K-theory of local fields},
  author={Suslin, Andrei A},
  journal={Journal of pure and applied algebra},
  volume={34},
  number={2-3},
  pages={301--318},
  year={1984},
  publisher={North-Holland}
}

@misc{BraunlingLocalCompact,
      title={Local compactness as the K(1)-local dual of finite generation}, 
      author={Oliver Braunling},
      year={2023},
      eprint={2301.05943},
      archivePrefix={arXiv},
      primaryClass={math.KT},
      url={https://arxiv.org/abs/2301.05943}, 
}

@misc{ClausenJ,
      title={p-adic J-homomorphisms and a product formula}, 
      author={Dustin Clausen},
      year={2012},
      eprint={1110.5851},
      archivePrefix={arXiv},
      primaryClass={math.AT},
      url={https://arxiv.org/abs/1110.5851}, 
}

@misc{ClausenArtin,
      title={A K-theoretic approach to Artin maps}, 
      author={Dustin Clausen},
      year={2017},
      eprint={1703.07842},
      archivePrefix={arXiv},
      primaryClass={math.KT},
      url={https://arxiv.org/abs/1703.07842}, 
}

@article{HopkinsNilpotence,
  title={Nilpotence and stable homotopy theory II},
  author={Hopkins, Michael J and Smith, Jeffrey H},
  journal={Annals of mathematics},
  volume={148},
  number={1},
  pages={1--49},
  year={1998},
  publisher={JSTOR}
}

@misc{GoerssMfg,
      title={Quasi-coherent sheaves on the moduli stack of formal groups}, 
      author={Paul G. Goerss},
      year={2008},
      eprint={0802.0996},
      archivePrefix={arXiv},
      primaryClass={math.AT},
      url={https://arxiv.org/abs/0802.0996}, 
}

@article{CartierModules,
  title={Modules associ{\'e}s {\`a} un groupe formel commutatif. Courbes typiques},
  author={Cartier, Pierre},
  journal={CR Acad. Sci. Paris S{\'e}r. AB},
  volume={265},
  pages={A129--A132},
  year={1967}
}

@incollection {MumfordBiextensions,
    AUTHOR = {Mumford, David},
     TITLE = {Bi-extensions of formal groups},
 BOOKTITLE = {Algebraic {G}eometry ({I}nternat. {C}olloq., {T}ata {I}nst.
              {F}und. {R}es., {B}ombay, 1968)},
    SERIES = {Tata Inst. Fundam. Res. Stud. Math.},
    VOLUME = {4},
     PAGES = {307--322},
 PUBLISHER = {Tata Inst. Fund. Res., Bombay},
      YEAR = {1969},
   MRCLASS = {14.50},
  MRNUMBER = {257089},
MRREVIEWER = {J.\ Dieudonn\'e},
}

@article{LurieE,
  title={Elliptic cohomology II: orientations},
  author={Lurie, Jacob},
  year={2020},
  url={https://www.math.ias.edu/~lurie/papers/Elliptic-II.pdf}

}

@article{DevinatzE,
  title={Homotopy fixed point spectra for closed subgroups of the Morava stabilizer groups},
  author={Devinatz, Ethan S and Hopkins, Michael J},
  journal={Topology},
  volume={43},
  number={1},
  pages={1--47},
  year={2004},
  publisher={Elsevier}
}

@article {KuhnTate,
    AUTHOR = {Kuhn, Nicholas J.},
     TITLE = {Tate cohomology and periodic localization of polynomial
              functors},
   JOURNAL = {Invent. Math.},
  FJOURNAL = {Inventiones Mathematicae},
    VOLUME = {157},
      YEAR = {2004},
    NUMBER = {2},
     PAGES = {345--370},
   MRCLASS = {55P60 (55N22 55P65 55P91 55P92)},
  MRNUMBER = {2076926},
MRREVIEWER = {J.\ P. C. Greenlees},
}

@misc{BurklundTelescope,
      title={$K$-theoretic counterexamples to Ravenel's telescope conjecture}, 
      author={Robert Burklund and Jeremy Hahn and Ishan Levy and Tomer M. Schlank},
      year={2023},
      eprint={2310.17459},
      archivePrefix={arXiv},
      primaryClass={math.AT},
      url={https://arxiv.org/abs/2310.17459}, 
}

@article {GreenleesTate,
    AUTHOR = {Greenlees, J. P. C. and Sadofsky, Hal},
     TITLE = {The {T}ate spectrum of {$v_n$}-periodic complex oriented
              theories},
   JOURNAL = {Math. Z.},
  FJOURNAL = {Mathematische Zeitschrift},
    VOLUME = {222},
      YEAR = {1996},
    NUMBER = {3},
     PAGES = {391--405},
   MRCLASS = {55N22},
  MRNUMBER = {1400199},
MRREVIEWER = {Mark\ Hovey},
       
}

@misc{MorPicard,
      title={Picard and Brauer groups of $K(n)$-local spectra via profinite Galois descent}, 
      author={Itamar Mor},
      year={2023},
      eprint={2306.05393},
      archivePrefix={arXiv},
      primaryClass={math.AT},
      url={https://arxiv.org/abs/2306.05393}, 
}

@article{ClausenCondensed,
  title={Lectures on Condensed Mathematics},
  author={Clausen, Dustin and Scholze, Peter},
  url={https://www.math.uni-bonn.de/people/scholze/Condensed.pdf}
}

@article{ClausenAnalytic,
  title={Lectures on Analytic Geometry},
  author={Clausen, Dustin and Scholze, Peter},
  url={https://www.math.uni-bonn.de/people/scholze/Analytic.pdf}
}

@misc{FarguesGeometrization,
      title={Geometrization of the local Langlands correspondence}, 
      author={Laurent Fargues and Peter Scholze},
      year={2024},
      eprint={2102.13459},
      archivePrefix={arXiv},
      primaryClass={math.RT},
      url={https://arxiv.org/abs/2102.13459}, 
}

@article{KleinDualizing,
  title={The dualizing spectrum of a topological group},
  author={Klein, John R},
  journal={Mathematische Annalen},
  volume={319},
  number={3},
  pages={421--456},
  year={2001},
  publisher={Springer}
}

@article {StricklandDuality,
    AUTHOR = {Strickland, N. P.},
     TITLE = {Gross-{H}opkins duality},
   JOURNAL = {Topology},
  FJOURNAL = {Topology. An International Journal of Mathematics},
    VOLUME = {39},
      YEAR = {2000},
    NUMBER = {5},
     PAGES = {1021--1033},
   MRCLASS = {55P60 (55M05 55N22 55P42)},
  MRNUMBER = {1763961},
MRREVIEWER = {Mark\ Hovey},
}

@article {BeaudryLinearization,
    AUTHOR = {Beaudry, Agn\`es and Goerss, Paul G. and Hopkins, Michael J.
              and Stojanoska, Vesna},
     TITLE = {Dualizing spheres for compact {$p$}-adic analytic groups and
              duality in chromatic homotopy},
   JOURNAL = {Invent. Math.},
  FJOURNAL = {Inventiones Mathematicae},
    VOLUME = {229},
      YEAR = {2022},
    NUMBER = {3},
     PAGES = {1301--1434},
   MRCLASS = {55P92 (20E18 22E41 55R25 55R50 55U30)},
  MRNUMBER = {4462627},
MRREVIEWER = {Hans-Werner\ Henn},
}

\end{document}